\documentclass{memo-l}
\usepackage{amssymb}
\usepackage{stmaryrd}
\usepackage{enumerate}
\usepackage[english]{babel}
\usepackage[arrow,matrix,curve]{xy}
\usepackage{hyperref}

\DeclareMathOperator{\coind}{coind}
\DeclareMathOperator{\colim}{colim}
\DeclareMathOperator{\End}{End}
\DeclareMathOperator{\Ext}{Ext}    
\DeclareMathOperator{\Hom}{Hom}
\DeclareMathOperator{\gl}{gl}
\DeclareMathOperator{\Ho}{Ho}
\DeclareMathOperator{\incl}{incl}
\DeclareMathOperator{\ind}{ind}
\DeclareMathOperator{\Id}{Id}
\DeclareMathOperator{\lv}{lv}
\DeclareMathOperator{\map}{map}
\DeclareMathOperator{\mo}{-mod}
\DeclareMathOperator{\Or}{Or}
\DeclareMathOperator{\op}{op}
\DeclareMathOperator{\proj}{proj}
\DeclareMathOperator{\tr}{tr}
\DeclareMathOperator{\res}{res}
\DeclareMathOperator{\Sp}{Sp}
\DeclareMathOperator{\Sym}{Sym}
\DeclareMathOperator{\st}{st}
\DeclareMathOperator{\sk}{sk}
\DeclareMathOperator{\sh}{sh}
\DeclareMathOperator{\Th}{Th}
\DeclareMathOperator{\un}{un}
\DeclareMathOperator{\Conj}{Conj}
\DeclareMathOperator{\Vect}{Vect}
\DeclareMathOperator{\Wirth}{Wirth}
\DeclareMathOperator{\Ker}{Ker}

\newcommand{\sm}{\wedge} 
\newcommand{\tensor}{\otimes} 
\newcommand{\iso}{\cong} 
\newcommand{\braced}[1]{\lbrace #1\rbrace}
\newcommand{\td}[1]{\langle #1\rangle}
\renewcommand{\to}{\longrightarrow}
\newcommand{\ul}{\underline} 
\newcommand{\bs}{\backslash} 
\newcommand{\upi}{\underline{\pi}}
\newcommand{\GH}{\mathcal{GH}}
\newcommand{\uEG}{\underline{E}G}
\newcommand{\gh}[1]{\llbracket #1\rrbracket}

\newcommand{\mA}{\mathbb A}
\newcommand{\mC}{\mathbb C}
\newcommand{\mF}{\mathbb F}
\newcommand{\mN}{\mathbb N}
\newcommand{\mQ}{\mathbb Q}
\newcommand{\mR}{\mathbb R}
\newcommand{\mS}{\mathbb S}
\newcommand{\mZ}{\mathbb Z}

\newcommand{\bA}{\mathbf A}
\newcommand{\bE}{\mathbf E}
\newcommand{\bL}{\mathbf L}
\newcommand{\bO}{\mathbf O}
\newcommand{\bT}{\mathbf{T}}
\newcommand{\bku}{\mathbf{ku}}
\newcommand{\bKU}{{\mathbf{KU}}}

\newcommand{\cA}{\mathcal A}
\newcommand{\cB}{\mathcal B}
\newcommand{\cC}{\mathcal C}
\newcommand{\cD}{\mathcal D}
\newcommand{\cE}{\mathcal E}
\newcommand{\cF}{\mathcal F}
\newcommand{\cH}{\mathcal H}
\newcommand{\cK}{\mathcal K}
\newcommand{\cM}{\mathcal M}
\newcommand{\cP}{\mathcal P}
\newcommand{\cS}{\mathcal S}
\newcommand{\cT}{\mathcal T}
\newcommand{\cU}{\mathcal U}
\newcommand{\cV}{\mathcal V}
\newcommand{\cX}{\mathcal X}
\newcommand{\cY}{\mathcal Y}
\newcommand{\cZ}{\mathcal Z}

\newcommand{\Fin}{\mathcal Fin}
\newcommand{\Com}{\mathcal Com}

\numberwithin{section}{chapter}
\numberwithin{equation}{section}
\newtheorem{thm}[equation]{Theorem}
\newtheorem{prop}[equation]{Proposition}
\newtheorem{cor}[equation]{Corollary}

\theoremstyle{definition}
\newtheorem{defn}[equation]{Definition}
\newtheorem{rk}[equation]{Remark}
\newtheorem{eg}[equation]{Example}
\newtheorem{con}[equation]{Construction}

\makeindex

\begin{document}

\frontmatter

\title[Proper Equivariant Stable Homotopy Theory]
{Proper Equivariant Stable Homotopy Theory}

\author{Dieter Degrijse}
\address{Keylane, Copenhagen, Denmark}
\email{dieterdegrijse@gmail.com}

\author{Markus Hausmann}
\address{Mathematisches Institut, Universit\"at Bonn, Germany}
\email{hausmann@math.uni-bonn.de}

\author{Wolfgang L{\"u}ck}
\address{Mathematisches Institut, Universit\"at Bonn, Germany}
\email{wolfgang.lueck@him.uni-bonn.de}

\author{Irakli Patchkoria}
\address{Department of Mathematics, University of Aberdeen, UK}
\email{irakli.patchkoria@abdn.ac.uk}

\author{Stefan Schwede}
\address{Mathematisches Institut, Universit\"at Bonn, Germany}
\email{schwede@math.uni-bonn.de}


\thanks{
  All five authors were supported by the
Hausdorff Center for Mathematics at the University of Bonn
(DFG GZ 2047/1, project ID 390685813)
and by the Centre for Symmetry and Deformation at the University of Copenhagen
(CPH-SYM-DNRF92);
we would like to thank these two institutions for their hospitality, support and 
the stimulating atmosphere. Hausmann, Patchkoria and Schwede were partially supported
by the DFG Priority Programme 1786 `Homotopy Theory and Algebraic Geometry'.
Work on this monograph was funded by the ERC Advanced Grant `KL2MG-interactions'
of L{\"u}ck (Grant ID 662400), granted by the European Research Council.
Patchkoria was supported by the Shota Rustaveli National Science Foundation Grant 217-614.
Patchkoria and Schwede would like to thank the Isaac Newton Institute for Mathematical Sciences
for support and hospitality during the programme `Homotopy harnessing higher structures',
when work on this paper was undertaken (EPSRC grant number EP/R014604/1).
We would also like to thank Bob Oliver and S{\o}ren Galatius for helpful conversations
on topics related to this project, and the anonymous referee for his or her careful reading
and the many useful comments.
}



\maketitle

\begin{abstract}
  This monograph introduces a framework for genuine
  proper equivariant stable homotopy theory for Lie groups. 
  The adjective `proper' alludes to the feature that equivalences
  are tested on compact subgroups, and that the objects are built
  from equivariant cells with compact isotropy groups;
  the adjective `genuine' indicates that 
  the theory comes with appropriate transfers and Wirthm{\"u}ller isomorphisms,
  and the resulting equivariant cohomology theories support the analog
  of an $R O(G)$-grading.

  Our model for genuine proper $G$-equivariant stable homotopy theory
  is the category of orthogonal $G$-spectra; the equivalences
  are those morphisms that induce isomorphisms of equivariant stable homotopy
  groups for all compact subgroups of $G$.
  This class of $\pi_*$-isomorphisms is part of a symmetric monoidal stable model structure,
  and the associated tensor triangulated homotopy category is compactly generated.
  Consequently, every orthogonal $G$-spectrum represents
  an equivariant cohomology theory on the category of $G$-spaces.
  These represented cohomology theories are designed to only depend on the
  `proper $G$-homotopy type', tested by fixed points under all compact subgroups.
  
  An important special case of our theory are infinite discrete groups.
  For these, our genuine equivariant theory is related to finiteness properties
  in the sense of geometric group theory;
  for example, the $G$-sphere spectrum is a compact object in our
  triangulated equivariant homotopy category if the universal space
  for proper $G$-actions has a finite $G$-CW-model.
  For discrete groups, the represented equivariant cohomology theories
  on finite proper $G$-CW-complexes admit a more explicit description in
  terms of parameterized equivariant homotopy theory, suitably stabilized by $G$-vector bundles.
  Via this description, we can identify the previously defined $G$-cohomology
  theories of equivariant stable cohomotopy and equivariant K-theory as
  cohomology theories represented by specific orthogonal $G$-spectra.
\end{abstract}

\tableofcontents

\mainmatter

\chapter*{Introduction}

This monograph explores proper equivariant stable homotopy theory for Lie groups. 
The theory generalizes the well established `genuine' equivariant stable
homotopy theory for {\em compact} Lie groups, and the adjective `proper'
indicates that the theory is only sensible to fixed point information
for compact subgroups. In other words, hardwired into
our theory are bootstrap arguments that reduce questions to
equivariant homotopy theory for compact Lie groups. 
Nevertheless, there are many new features that have no direct analog in
the compact case.

\smallskip

Equivariant stable homotopy theory has a long tradition, originally motivated
by geometric questions about symmetries of manifolds.
Certain kinds of features can be captured by `naive' stable equivariant theories,
obtained one way or other by formally inverting suspension on some
category of equivariant spaces. The naive stable theory works in broad generality
for general classes of topological groups,
and it can be modeled by sequential spectra of $G$-spaces,
or functors from a suitable orbit category to spectra
(interpreted either in a strict sense, or in an $\infty$-categorical context).

A refined version of equivariant stable homotopy theory,
usually referred to as `genuine', was traditionally only available
for compact Lie groups of equivariance.
The genuine theory has several features not available in the naive theory,
such as transfer maps, dualizability, stability under ‘twisted suspension’ (i.e., smash product
with linear representation spheres),
an extension of the $\mZ$-graded cohomology
groups to an $R O(G)$-graded theory, and an equivariant refinement of additivity
(the so called Wirthm{\"u}ller isomorphism).
The homotopy theoretic foundations of this theory were laid by
tom Dieck \cite{tomDieck:bordism, tomDieck:orbittypenI, tomDieck:orbittypenII, tomDieck:transformation},
May \cite{lewis-may:steinberger:equivariant_stable, greenlees-may:generalized_Tate, greenlees-may:completion}
and Segal \cite{segal:equivariant_K, segal:ICM, segal:some_equivariant}
and their students and collaborators since the 70’s.
A spectacular recent application was the solution,
by Hill, Hopkins and Ravenel \cite{hill-hopkins-ravenel:kervaire},
to the Kervaire invariant 1 problem.
This monograph extends genuine equivariant stable homotopy theory
to Lie groups that need not be compact;
this includes infinite discrete groups as an important special case.

A major piece of our motivation for studying equivariant homotopy theory
for infinite discrete groups and not necessarily compact Lie groups
comes from the Baum-Connes Conjecture and the Farrell-Jones Conjectures.
The Baum-Connes Conjecture was originally formulated in \cite{baum-connes:geometric_K},
and subsequently considered in the formulation stated in 
\cite[Conjecture 3.15 on page 254]{baum-connes-higson:classifying_spaces}.
The Farrell-Jones Conjecture was formulated in \cite[1.6 on page 257]{farrell-jones:isomorphism_conjectures};
two survey articles about these isomorphism conjectures are 
\cite{luck-reich:BC_and_FJ} and \cite{valette:introduction}.
Roughly speaking, these conjectures identify the theory of interest
-- topological K-groups of reduced group $C^*$-algebras or algebraic K-and L-groups of group rings --
with certain equivariant homology theories, applied to classifying spaces of
certain families of subgroups.
Many applications of these conjectures
to group homology, geometry, or classification results of $C^*$-algebras
are based on computations of the
relevant equivariant homology theories,
and of their cohomological analogues.
Even if one is interested only in non-equivariant (co)homology
of classifying spaces, it is useful to invoke equivariant homotopy theory on
the level of spectra; examples are
\cite{echterhoff-luck-phillips-walters:structure_crossed, langer-luck:group_cohomology, luck:rational_computations}.
Our book provides a systematic framework for such calculations,
capable of capturing extra pieces of structure like induction, transfers, restriction,
multiplication, norms, global equivariant features, and gradings beyond naive $\mZ$-grading.

For example, one might wish to use the Atiyah-Hirzebruch spectral sequence
to compute K-theory or stable cohomotopy of $B G$ for an infinite discrete group $G$; 
this will often be a spectral sequence with differentials of arbitrary length.
However, the completion theorems for equivariant K-theory \cite{luck-oliver:completion}
and for equivariant stable cohomotopy \cite{luck:segal_infinite}
provide an alternative line of attack:
one may instead compute equivariant cohomology of $\underline{E}G$
by the equivariant Atiyah-Hirzebruch spectral sequence
(see Construction \ref{con:equivariant AHSS}),
and then complete the result to obtain the non-equivariant cohomology of $B G$.
Even if $B G$ is infinite dimensional,
$\underline{E}G$ might well be finite-dimensional and relatively small,
in which case the equivariant Atiyah-Hirzebruch spectral sequence is easier to analyze.
For example, for a virtually torsion free group,
this spectral sequence has a vanishing line at the virtual cohomological dimension.

Among other things, our formalism provides a definition of
the equivariant homotopy groups $\pi_*^G$ for infinite groups $G$.
These equivariant homotopy groups have features which
reflect geometric group theoretic properties of $G$.
As the group $G=\mZ$ already illustrates, the sphere spectrum
need not be connective with respect to $\pi_*^G$,
compare Example \ref{eg:rational Z-equivariant}.
If the group $G$ is virtually torsion free,
then the equivariant Atiyah-Hirzebruch spectral sequence
shows that the equivariant homotopy groups $\pi_*^G(\mS)$ vanish below
the negative of the virtual cohomological dimension of~$G$.\\

We conclude this introduction with a summary of the highlights of this monograph.

$\bullet$ 
Our model for proper equivariant stable homotopy theory of a Lie group $G$
is the category of orthogonal $G$-spectra, i.e., orthogonal spectra
equipped with a continuous action of $G$.
This pointset level model is well-established, explicit, 
and has nice formal properties; for example, orthogonal $G$-spectra
are symmetric monoidal under the smash product of orthogonal spectra,
endowed with the diagonal $G$-action.

$\bullet$ 
All the interesting homotopy theory is encoded in the notion
of stable equivalences for orthogonal $G$-spectra.
We use the {\em $\pi_*$-isomorphisms}, defined as those morphisms
of orthogonal $G$-spectra that induce isomorphisms on $\mZ$-graded $H$-equivariant homotopy
groups, for all compact subgroups~$H$ of $G$.
These $H$-equi\-variant homotopy groups are based on a complete universe of
orthogonal $H$-re\-presentations.
In~\cite[Prop.\,6.5]{fausk:prospectra}, Fausk has extended these $\pi_*$-isomorphisms
to a stable model structure on the category of orthogonal $G$-spectra,
via an abstract Bousfield localization procedure.
We develop a different (but Quillen equivalent)
model structure in Theorem \ref{thm:stable}, also with the $\pi_*$-isomorphisms
as weak equivalences, that gives better control over the stable fibrations;
in particular, the stably fibrant objects in our model structure are 
the $G$-$\Omega$-spectra as defined in Definition~\ref{def:stable fibrations}
below.
Our model structure is compatible with the smash product of
orthogonal $G$-spectra, and compatible with restriction to closed subgroups,
see Theorem \ref{thm:stable}.

$\bullet$ 
A direct payoff of the stable model structure is that the homotopy category
$\Ho(\Sp_G)$ comes with a triangulated structure. 
This structure is made so that mapping cone sequences of proper $G$-CW-complexes
become distinguished triangles in $\Ho(\Sp_G)$.
As a consequence, every orthogonal $G$-spectrum represents a $G$-equivariant
cohomology theory on proper $G$-spaces, see Construction \ref{con:cohomology from E}. 
The triangulated homotopy category $\Ho(\Sp_G)$ comes with a
distinguished set of small generators, the suspension spectra of
the homogeneous spaces $G/H$ for all compact subgroups~$H$ of~$G$,
see Corollary \ref{cor:Ho_G compactly generated}.
This again has certain direct payoffs, such as Brown representability
of homology and cohomology theories on~$\Ho(\Sp_G)$, and a non-degenerate t-structure.

$\bullet$
The triangulated categories $\Ho(\Sp_G)$ enjoy a large amount
of functoriality in the group $G$:
every continuous homomorphism $\alpha:K\to G$ gives rise
to a restriction functor $\alpha^*:\Sp_G\to\Sp_K$ that in turn
admits an exact total left derived functor $L\alpha^*:\Ho(\Sp_G)\to\Ho(\Sp_K)$, 
see Theorem \ref{thm:homomorphism adjunctions spectra}.
These derived functors assemble to a contravariant pseudo-functor
from the category of Lie groups and continuous homomorphisms
to the category of triangulated categories and exact functors.
Moreover, conjugate homomorphisms and homotopic homomorphisms induce
isomorphic derived functors.

$\bullet$ 
The proper equivariant stable homotopy theory should be thought of
as a `weak homotopy invariant' of the Lie group $G$. 
More precisely, we show in Theorem \ref{thm:weak equivalence}
that for every continuous homomorphism $\alpha:K\to G$ between
Lie groups that is a weak equivalence of underlying spaces, 
the derived restriction functor $L\alpha^*:\Ho(\Sp_G)\to\Ho(\Sp_K)$
is an equivalence of tensor triangulated categories. 
A special case is the inclusion of a maximal compact subgroup
$M$ of an almost connected Lie group $G$.
In this case, the restriction functor $\res^G_M:\Ho(\Sp_G)\to\Ho(\Sp_M)$
is an equivalence. 
So for almost connected Lie groups, our theory reduces to the classical
case of compact Lie groups.
In this sense, the new mathematics in this memoir
is mostly about infinite versus finite component groups.

$\bullet$ 
When $G$ is discrete, the heart of the preferred t-structure on $\Ho(\Sp_G)$
has a direct and explicit algebraic description: it is equivalent
to the abelian category of $G$-Mackey functors in the sense of
Martinez-P{\'e}rez and Nucinkis \cite{martinezperez-nucinkis:cohomological_dimension},
see Theorem \ref{thm:embed_MG_into_Ho(Sp)}. 
In particular, every $G$-Mackey functor is realized, essentially uniquely,
by an Eilenberg-Mac\,Lane spectrum in~$\Ho(\Sp_G)$.
The cohomology theory represented by the Eilenberg-Mac\,Lane spectrum
is Bredon cohomology, see Example~\ref{eg:HM_represents_Bredon}.

$\bullet$ 
For discrete groups~$G$, the {\em rational} $G$-equivariant
stable homotopy theory is completely algebraic: the rationalization of 
$\Ho(\Sp_G)$ is equivalent to the derived category of rational $G$-Mackey functors,
see Theorem \ref{thm:rational SH}.
When $G$ is infinite, the abelian category of rational $G$-Mackey functors
is usually not semisimple, so in contrast to the well-known case of finite groups,
a rational $G$-spectrum is not generally classified in $\Ho(\Sp_G)$
by its homotopy group Mackey functors alone.

$\bullet$ 
Our theory is the analog, for general Lie groups,
of `genuine' equivariant stable homotopy theory; 
for example, the equivariant cohomology theories arising from orthogonal $G$-spectra
have a feature analogous to an `$R O(G)$-grading' in the compact case.
In the present generality, however, representations should be replaced
by equivariant real vector bundles over the universal space~$\uEG$
for proper $G$-actions, and so the analog of an
$R O(G)$-grading is a grading by the Grothendieck group~$K O_G(\uEG)$
of such equivariant vector bundles, see Remark~\ref{rk:RO(G) grading}.

$\bullet$  
For discrete groups we identify the cohomology theories represented by $G$-spectra 
in more concrete terms via fiberwise equivariant homotopy theory,
see Theorem \ref{thm:represented equals bundle}.
This allows us to compare our approach to equivariant cohomology theories 
that were previously defined by different means.
For example, we show in Example \ref{eg:represent cohomotopy}
that for discrete groups the theory represented by the $G$-sphere spectrum coincides,
for finite proper $G$-CW-complexes, with equivariant cohomotopy
as defined by the third author in \cite{luck:burnside_ring}. 
In Theorem \ref{thm:repr K-theory} we show that if $G$ is a discrete group, 
then the equivariant K-theory based on $G$-vector bundles defined by
the third author and Oliver
in \cite{luck-oliver:completion} is also representable in $\Ho(\Sp_G)$.\medskip

{\bf Conventions.}
Throughout this memoir, a {\em space} is a {\em compactly generated space}\index{space!compactly generated} 
in the sense of~\cite{mccord:classifying_spaces},
i.e., a $k$-space (also called {\em Kelley space})
that satisfies the weak Hausdorff condition.
Two extensive resources with background material about compactly generated spaces
are Section~7.9 of tom Dieck's textbook \cite{tomDieck:algebraic_topology}
and Appendix A of the fifth author's book \cite{schwede:global}.
Two other influential -- but unpublished -- sources about compactly generated
spaces are the Appendix~A of Gaunce Lewis's thesis \cite{lewis:thesis}
and Neil Strickland's preprint \cite{strickland:CGWH}.
We denote the category of compactly generated spaces
and continuous maps by $\bT$.

\chapter{Equivariant spectra}

\section{Orthogonal \texorpdfstring{$G$}{G}-spectra}

In this section we recall the basic objects of our theory, 
orthogonal spectra and orthogonal $G$-spectra, where $G$ is a Lie group.
We start in Proposition \ref{prop:Gmod} 
with a quick review of the $\Com$-model structure 
for $G$-spaces, i.e., the model structure where equivalences and fibrations
are tested on fixed points for compact subgroups of $G$.
The homotopy category of this model structure is equivalent to
the category of proper $G$-CW-complexes and equivariant homotopy classes of $G$-maps.
Proposition \ref{prop:homomorphism adjunctions spaces} 
records how the $\Com$-model structures interact with
restriction along a continuous homomorphism between Lie groups.
We recall orthogonal $G$-spectra in Definition \ref{def:Gspec},
and we end this section with several examples.

\medskip

We let $G$ be a topological group, which we take to mean a group
object in the category $\bT$ of compactly generated spaces.
So a topological group is a compactly generated
space equipped with an associative and unital multiplication
\[ \mu \ : \  G \times G \ \to \ G \]
that is continuous with respect to the 
compactly generated product topology, and such that the shearing map
\[   G \times G \ \to \  G \times G \ ,\quad (g,h)\ \longmapsto \ (g, g h)\]
is a homeomorphism (again for the compactly generated product topology). 
This implies in particular that inverses exist in $G$,
and that the inverse map $g\mapsto g^{-1}$ is continuous.
A {\em $G$-space}\index{G-space@$G$-space} is then a compactly generated space $X$ 
equipped with an associative and unital action 
\[ \alpha\ : \ G  \times X \ \to \ X \]
that is continuous with respect to the compactly generated product topology.
We write $G\bT$ for the category of $G$-spaces and continuous $G$-maps.
The forgetful functor from $G$-spaces to compactly generated spaces 
has both a left and a right adjoint,
and hence limits and colimits of $G$-spaces are created in the underlying
category $\bT$.

\begin{rk}
We mostly care about the case when $G$ is a Lie group. Then the underlying
space of $G$ is locally compact Hausdorff, and for every compactly generated space $X$,
the space $G\times X$ is a $k$-space (and hence compactly generated)
in the usual product topology. So for Lie groups, the potential ambiguity about
continuity of the action disappears.
\end{rk}

For $G$-spaces $X$ and $Y$ we write $\map(X,Y)$\index{mapping space} for the
space of continuous maps with the function space topology internal
to $\bT$ (i.e., the Kelleyfied compact-open topology).
The group $G$ acts continuously on $\map(X,Y)$ by conjugation, and
these constructions are related by adjunctions, i.e., natural homeomorphisms
of $G$-spaces
\[ \map(X,\map(Y,Z))\ \iso \ \map(X\times Y,Z) \ ;\]
in particular, the category of $G$-spaces is cartesian closed.

The category $G\bT$ is enriched, tensored and cotensored over the category 
$\bT$ of spaces, as follows.
The tensor and cotensor of a $G$-space $X$ and a space $K$
are the product $X\times K$ and the function space $\map(K,X)$,
respectively, both with trivial $G$-action on $K$.
The enrichment in $\bT$ is given by the
$G$-fixed point space $\map^G(X,Y)$, i.e., the space of $G$-equivariant continuous maps.

A $G$-space $X$ is {\em proper}\index{G-space@$G$-space!proper} if the map 
\[ G\times X \ \to \ X\times X \ , \quad (g,x)\ \longmapsto \ (g x,x) \]
is proper, i.e., inverse images of compact subsets are compact.
Because the inverse image of a diagonal point $(x,x)$ is homeomorphic to
the isotropy group of $x$, all isotropy groups of proper $G$-spaces are compact.
Any irrational rotation action of the group $\mZ$ on a circle shows that
the converse is not generally true, i.e., a continuous $G$-action with compact stabilizers
need not be proper in general.

A {\em $G$-CW-complex}\index{G-CW-complex@$G$-CW-complex} is a $G$-space $X$ equipped with an
exhaustive increasing filtration
\[ \emptyset \ =\  X_{-1}\ \subseteq\ X_0\ \subseteq\ X_1\ \subseteq\ \dots\
  \subseteq\ X_n\ \subseteq\ \dots \]
by closed $G$-invariant subspaces such that
\begin{itemize}
  \item each $X_n$ is obtained from $X_{n-1}$
by attaching a disjoint union of $G$-spaces $G/H_i\times D^n$ along 
 $G/H_i\times S^{n-1}$, for some set $\{H_i\}_{i\in I}$ of closed subgroups of $G$, and
\item the space $X$ has the colimit topology (weak topology) of the sequence.
\end{itemize}
A $G$-CW-complex is {\em finite-dimensional}\index{G-CW-complex@$G$-CW-complex!finite-dimensional}
if the filtration stabilizes, i.e., $X_n=X$ for some $n\geq 0$.
A $G$-CW-complex is {\em of finite type}
if in each dimension $n$, the number of equivariant cells is finite.
A $G$-CW-complex is {\em finite}\index{G-CW-complex@$G$-CW-complex!finite}
if it is finite-dimensional and of finite type, i.e., if the total number of
equivariant cells is finite.
Equivalently, a $G$-CW-complex $X$ is finite if and only if the action is cocompact,
i.e., the orbit space $X/G$ is compact.
For a $G$-CW-complex, properness of the action is equivalent to
the requirement that all of the isotropy groups are compact,
see for example \cite[Theorem 1.23]{luck:transformation}.\index{G-CW-complex@$G$-CW-complex!proper}

\begin{defn}  Let $G$ be a Lie group.
  \begin{enumerate}[(i)]
  \item 
    A map $f \colon X \to Y$ of $G$-spaces is a $\Com$-\emph{equivalence}
    \index{Com-equivalence@$\Com$-equivalence} 
    if for every compact subgroup~$H$ of~$G$ the map $f^H \colon X^H \to Y^H$ 
    is a weak equivalence of topological spaces.
  \item 
    A map $f \colon X \to Y$ of $G$-spaces
    is a $\Com$-\emph{fibration}\index{Com-fibration@$\Com$-fibration}  
    if for every compact subgroup $H$ of $G$ the map $f^H \colon X^H \to Y^H$ 
    is a Serre fibration of topological spaces.
  \item 
    A map $i \colon A \to B$ of $G$-spaces
    is a $\Com$-\emph{cofibration}\index{Com-cofibration@$\Com$-cofibration} 
    if it has the left lifting property with respect to every map which 
    is a $\Com$-equivalence and $\Com$-fibration.
  \item 
    A map $i \colon A \to B$ of $G$-spaces
    is a $G$-\emph{cofibration}\index{G-cofibration@$G$-equivalence} 
    if it has the left lifting property with respect to every map 
    $f:X\to Y$ such that $f^\Gamma:X^\Gamma\to Y^\Gamma$ is a weak equivalence
    and Serre fibration for every closed subgroup $\Gamma$ of $G$.
  \end{enumerate}
\end{defn}

We alert the reader that our use of the expression `$G$-cofibration' is different from
the usage in some older papers on the subject, where this term refers to
the larger class of $G$-maps with the equivariant homotopy extension property.
In this book, morphisms with the homotopy extension property
will be referred to as {\em h-cofibrations}, see Definition \ref{def:h-cofibration} below.

Clearly, every $\Com$-cofibration of $G$-spaces is a $G$-cofibration.
The following proposition is a special case 
of \cite[Prop.\,2.11]{fausk:prospectra} or \cite[Prop.\,B.7]{schwede:global}.
For the definition of a cofibrantly generated model category see 
for example \cite[Sec.\,2.1]{hovey:model_categories}. 

\begin{prop}\label{prop:Gmod}
  Let $G$ be a Lie group.
  \begin{enumerate}[\em (i)]
  \item The $\Com$-equivalences, $\Com$-fibrations and $\Com$-cofibrations 
    form a proper, topological, cofibrantly generated model structure on the category 
    of $G$-spaces, the $\Com$-model structure.\index{Com-model structure@$\Com$-model structure}
    A morphism $i:A\to B$ is a $\Com$-cofibration if and only if it is a $G$-cofibration
    and the stabilizer group of every point in $B-i(A)$ is compact.
  \item
  The set of maps
  \[ G/H \times i^k \colon G/H \times \partial D^k \to G/H \times D^k \]
  serves as a set of generating cofibrations
  for the $\Com$-model structure, as $H$ ranges over all 
  compact subgroups of~$G$  and $k\geq 0$. The set of maps 
  \[ G/H \times j^k \colon G/H \times D^k\times\{0\}
  \to G/H \times D^k \times [0,1] \]
  serves as a set of generating acyclic cofibrations, as $H$ ranges over all 
  compact subgroups of~$G$  and $k\geq 0$. 
\item For every $G$-cofibration $i:A\to B$ and every $\Com$-cofibration $j:K\to L$ 
  of based $G$-spaces, the pushout product map
  \[ i \boxempty j = (i\times L)\cup(B \times j)
  \ :\  (A \times L) \cup_{A \times K} (B \times K) \ \to \ B \times L \]
  is a $\Com$-cofibration. If moreover one of $i$ or $j$ is a $\Com$-equivalence, 
  then $i \boxempty j$ is also a $\Com$-equivalence.
  \end{enumerate}
\end{prop}
\begin{proof}
As we already mentioned, parts (i) and (ii) are proved in detail in
\cite[Prop.\,B.7]{schwede:global}.
Since smash product has an adjoint in each variable, 
it preserves colimits in each variable.
So it suffices to check the pushout product properties in (iii)
when the maps $f$ and $g$ are from the sets of generating (acyclic) cofibrations,
compare \cite[Cor.\,4.2.5]{hovey:model_categories}.
The set of inclusions of spheres into discs is closed under pushout product,
in the sense that $i^k\Box i^l$ is homeomorphic to $i^{k+l}$.
Similarly, the pushout product of $i^k$ with $j^l$ is isomorphic to $j^{k+l}$.
So all claims reduce to the fact that for every pair of closed subgroups 
$\Gamma$ and $H$ of $G$ such that $H$ is compact, the $G$-space $G/\Gamma\times G/H$
with the diagonal $G$-action is $\Com$-cofibrant. 
Indeed, this product is $G$-homeomorphic 
to $G \times_H (G/\Gamma)$, with $G$-action only on the left factor of $G$. 
Illman's theorem \cite[Thm.\,7.1]{illman:triangulation_Lie} 
implies that $G/\Gamma$ admits the structure of an $H$-CW complex;
hence $G \times_H (G/\Gamma)$ admits the structure of a $G$-CW-complex,
and its isotropy groups are compact. 
\end{proof}

Every proper $G$-CW-complex is in particular $\Com$-cofibrant.
On the other hand, every
$\Com$-cofibrant $G$-space is $G$-equivariantly homotopy equivalent
to a proper $G$-space that admits the structure of a $G$-CW-complex.
So for all practical purposes, $\Com$-cofibrant $G$-spaces are
as good as proper $G$-CW-complexes.
Since all $G$-spaces are fibrant in the $\Com$-model structure,
the abstract homotopy category $\Ho^{\Com}(G\bT)$,
defined as the localization at the class of $\Com$-equivalences,
is equivalent to the concrete homotopy category 
of proper $G$-CW-complexes and equivariant homotopy classes of $G$-maps.

We denote by $\uEG$ a universal proper $G$-space,\index{universal proper $G$-space}
i.e., a universal $G$-space for the family of compact subgroups of~$G$.
It is characterized up to $G$-homotopy equivalence by the following properties:
\begin{enumerate}[(i)]
\item  $\uEG$ admits the structure of a $G$-CW complex.
\item The $H$-fixed point space $(\uEG)^H$ is contractible 
    if $H$ is compact, and empty otherwise.
\end{enumerate}
The existence of $\uEG$ follows for example from \cite[Thm.\,1.9]{luck:survey}.
We note that for every $G$-space $X$, the projection 
$\uEG\times X \to X$ is a  $\Com$-equivalence. 
Indeed, taking $H$-fixed points for a compact subgroup~$H$ of~$G$ we have
\[ (\uEG\times X)^H \ \cong\ (\uEG)^H\times X^H \]
which maps by a homotopy equivalence to~$X^H$ since  $(\uEG)^H$ is contractible.
The unit of the cartesian product (the one-point $G$-space)
is not $\Com$-cofibrant unless the group $G$ is compact;  
since the unique map $\uEG\to \ast$ is a $\Com$-equivalence,
$\uEG$ is a cofibrant replacement of the one-point $G$-space.

Here are some examples of universal spaces for proper actions.
\begin{itemize}
\item   If the Lie group $G$ is compact, then the one-point $G$-space is a model for $\uEG$.  
\item
  Suppose that $G$ has no non-trivial compact subgroups;
  for example, $G$ could be discrete and torsion free.
  Then any universal $G$-space for free $G$-actions,
  usually denoted by $E G$, is also a universal $G$-space for proper $G$-actions.
\item  If $G$ is discrete and can be exhausted by an ascending sequence of finite subgroups,
  then it has a 1-dimensional model for $\uEG$, see Example \ref{eg:1-dim for locally finite}.
\item Suppose that $G$ is {\em almost connected}, i.e., its group of path components is finite.
  Then $K$ has a maximal compact subgroup, i.e., a compact subgroup $K$ 
  such that every compact subgroup is subconjugate to $K$.
  In this situation, the orbit space $G/K$ is a model for $\uEG$.
  We refer to the discussion immediately preceding Theorem \ref{thm:reduction}
  for more details and references.
\end{itemize}

\begin{eg}[A universal proper $D_\infty$-space]\label{eg:ED_infty}\index{infinite dihedral group}
We offer the case $G=D_\infty=\mZ\ltimes\mZ/2$  of the infinite dihedral group
as an explicit example to keep in mind; the group $D_\infty$ is also isomorphic to
the free product $\mZ/2\ast\mZ/2$.
Besides the trivial subgroup, $D_\infty$ has two other conjugacy classes of finite subgroups,
both of order 2; the subgroup $H_1$ generated by $(0, 1+2\mZ)$
and the subgroup $H_2$ generated by $(1,1+2\mZ)$ are representatives.

The groups $\mZ$ and $\mZ/2$ act on the real line
by translation and reflection at the origin, respectively.
These two actions conspire into an action of
the semidirect product $D_\infty=\mZ\ltimes \mZ/2$ on $\mR$.
The real line has a 1-dimensional $D_\infty$-CW-structure
with 3 equivariant cells, as follows.
The 0-skeleton is the union of the transitive subsets $\mZ$ and $1/2+\mZ$;
so each of these two invariant subspaces is an equivariant 0-cell,
and the respective stabilizers are the conjugacy classes of $H_1$ and $H_2$.
There is a single equivariant 1-cell, with trivial stabilizer.
The fixed point spaces of this $D_\infty$-action on $\mR$
are either empty (for infinite subgroups), they are single points
(for two-element subgroups), or all of $\mR$ (for the trivial subgroup).
So the real line is a universal proper $D_\infty$-space.
\end{eg}

We briefly discuss how change of group functors interact with
the $\Com$-model structures. We let $ \alpha \colon K \to G$ 
be a continuous homomorphism between Lie groups.
Restriction of scalars along $\alpha$ is a functor\index{restriction functor}
\[ \alpha^* \ : \ G \bT \ \to \ K \bT \ ; \]
here $\alpha^*(X)$ has the same underlying space as $X$, with $K$
acting through the homomorphism $\alpha$.
The restriction functor $\alpha^*$ has a left adjoint\index{induction functor}
\[  G \times_{\alpha} - \ :\ K \bT\  \to  \ G\bT\] 
and a right adjoint
\[ \map^{K,\alpha}(G,-) \ :\ K \bT\  \to\  G\bT\ .\] 
For any based $K$-space $X$, the $G$-space $G \times_{\alpha} X$ is the quotient 
of $G\times X$ by the equivalence relation $(g \alpha(k), x) \sim (g, k x)$. 
The $G$-space $\map^{K,\alpha}(G,X)$ is the space of those
continuous maps $f:G\to X$ that satisfy $k\cdot f(g)=f(\alpha(k)\cdot g)$
for all $(k,g)\in K\times G$, with $G$ acting by
\[ (\gamma\cdot f)(g) \ = \ f(g\gamma)\ .\]
An important special case is when $\alpha$ is the inclusion of a closed
subgroup $\Gamma$ of $G$. In that case we write $\res^G_\Gamma$ for the
restriction functor, and we simplify the notation for the left and
right adjoint to $G\times_\Gamma-$ and $\map^\Gamma(G,-)$, respectively.

\begin{prop}\label{prop:H-Gamma-CW}
  Let $H$ and $\Gamma$ be closed subgroups of a Lie group $G$,
  such that moreover $H$ is compact. 
  Then the $(H\times \Gamma)$-action on $G$ given 
  by $(h,\gamma) \cdot g = h g \gamma^{-1}$ 
  underlies an $(H \times \Gamma)$-CW-complex.
\end{prop}
\begin{proof}
We claim that the $(H \times \Gamma)$-action is {\em proper}, i.e., the map
\[  (H\times\Gamma) \times G\ \to \ G \times G \ , \quad
((h,\gamma),g)\ \longmapsto \ ( h g \gamma^{-1}, g) \]
is a proper map in the sense that preimages of compact sets are compact.  
Indeed, we can factor this map as the composite of three proper maps, namely
the inclusion of the closed subspace $H \times \Gamma\times G$
into $H \times G \times G$,
followed by the homeomorphism 
\[ H \times G \times G  \ \xrightarrow{\ \cong\ }\  
H \times G \times G, \quad (h, \gamma, g) \ \longmapsto \ (h, h g \gamma^{-1}, g) \]
and the projection of $H \times G \times G$ to the last two factors.
Since the~$(H\times\Gamma)$-action on~$G$ is also smooth, 
Theorem~I of~\cite{illman:triangulation_proper} provides an $(H\times\Gamma)$-equivariant triangulation
of~$G$, and hence the desired equivariant CW-structure,
by~\cite[Prop.\,11.5]{illman:triangulation_proper}.
\end{proof}

We recall that a {\em Quillen adjunction}\index{Quillen adjunction}
is an adjoint functor pair $(F,G)$ between model categories such that
the left adjoint~$F$ preserves cofibrations and 
the right adjoint~$G$ preserves fibrations.
Equivalent conditions are to require that
the left adjoint~$F$ preserves cofibrations and acyclic cofibrations;
or that the right adjoint~$G$ preserves fibrations and acyclic fibrations.

\begin{prop} \label{prop:homomorphism adjunctions spaces} 
Let $ \alpha \colon K \to G$ be a continuous homomorphism between Lie groups.
\begin{enumerate}[\em (i)]
\item The restriction functor $\alpha^*:G\bT\to K\bT$
takes $\Com$-equivalences of $G$-spaces to
$\Com$-equivalences of $K$-spaces.\index{restriction functor}
\item The adjoint functor pair\index{induction functor}
\[\xymatrix@C=12mm{  
G \times_{\alpha } -\ : \ 
    K \bT\ \ar@<0.5ex>[r]    &
    \ G \bT \ : \ \alpha^*  \ar@<0.5ex>[l]  } \]
is a Quillen adjunction with respect to the $\Com$-model structures. 
\item If the image of $\alpha$ is closed in $G$
and the kernel of $\alpha$ is compact, then the adjoint functor pair
\[\xymatrix@C=12mm{  
\map^{K,\alpha}(G,-)\ : \ K \bT\ \ar@<-0.5ex>[r]  &
\ G \bT  \ : \ \alpha^*   \ar@<-0.5ex>[l]
} \]
is a Quillen adjunction with respect to the $\Com$-model structures. 
\end{enumerate}
\end{prop}
\begin{proof}
For every compact subgroup 
$L\leq K$ the image $\alpha(L)$ is a compact subgroup of $G$,
and for every based $G$-space $X$, 
we have the equality $(\alpha^*(X))^L= X^{\alpha(L)}$. 
So the restriction functor takes $\Com$-equivalences of $G$-spaces to
$\Com$-equivalences of $K$-spaces, and
it takes $\Com$-fibrations of $G$-spaces to
$\Com$-fibrations of $K$-spaces. This establishes parts (i) and (ii).

As a left adjoint, the restriction functor $\alpha^*$ preserves colimits.
So for part (iii) we only have to check
that the restriction of each of the generating 
$G$-$\Com$-cofibrations specified in Proposition~\ref{prop:Gmod} (ii)
is a $\Com$-cofibration of $K$-spaces.
This amounts to the claim that for every compact subgroup~$H$ of~$G$ the 
$K$-space $\alpha^*(G/H)$ is $\Com$-cofibrant.
We let $\Gamma=\alpha(K)$ denote the image of $\alpha$, 
which is a closed subgroup of $G$ by hypothesis.
Since~$G$ admits a $(\Gamma\times H)$-CW-structure by 
Proposition \ref{prop:H-Gamma-CW},
the orbit space $G/H$ inherits a $\Gamma$-CW-structure.
We let $\beta:K\to \Gamma$ denote the same homomorphism as
$\alpha$, but now considered as a continuous epimorphism onto its image.
For every closed subgroup $\Delta$ of $\Gamma$ we have
\[ \beta^*(\Gamma/\Delta)\ \iso \ K/\beta^{-1}(\Delta)\ . \]
Since the restriction functor $\beta^*:\Gamma\bT\to K\bT$ 
commutes with colimits and products with spaces, this shows that
it takes $\Gamma$-CW-complexes to $K$-CW-complexes.
In particular, $\alpha^*(G/H)=\beta^*(\res^G_\Gamma(G/H))$
admits the structure of a $K$-CW-complex.
The $K$-stabilizer group of a coset $g H$ is $\alpha^{-1}(H^g)$.
Since $H$ is compact and the kernel of $\alpha$ is compact by
hypothesis, all stabilizer groups of $\alpha^*(G/H)$ are compact.
So $\alpha^*(G/H)$ is $\Com$-cofibrant as a $K$-space.
\end{proof}

\begin{rk}
One should beware that restriction to a closed subgroup does {\em not}
preserve general equivariant cofibrations without an isotropy condition.
This should be contrasted with the fact that h-cofibrations
(i.e., maps with the equivariant homotopy extension property)
are stable under restriction to closed subgroups.
For example, the left translation action makes $\mR/\mZ$ an $\mR$-CW-complex
and a cofibrant $\mR$-space; the $\mR$-space $\mR/\mZ$ is {\em not} $\Com$-cofibrant, however, because
the stabilizer group $\mZ$ is not compact.
On the other hand, if $\Gamma$ is the additive subgroup of $\mR$ generated by an irrational number,
then the underlying $\Gamma$-action on the circle $\mR/\mZ$
is not proper, and $\mR/\mZ$ is neither a $\Gamma$-CW-complex nor cofibrant as a 
$\Gamma$-space.
\end{rk}

In the application to orthogonal $G$-spectra, we will also need the
{\em based} version of the $\Com$-model structure, and the modification
of some of the previous results to the based context.
We write $\bT_*$ for the category of based compactly generated spaces.
A {\em based $G$-space} is a $G$-space equipped with a $G$-fixed basepoint;
we write $G \bT_*$ for the category of based $G$-spaces 
and based continuous $G$-maps.\index{Com-model structure@$\Com$-model structure!based version}

A standard result in model category theory lets us lift the $\Com$-model structure
from unbased to based $G$-spaces.
A morphism in $G\bT_*$ is a $\Com$-equivalence, $\Com$-fibration or 
$\Com$-cofibration if and only if it is so as an unbased $G$-map,
see \cite[Prop.\,1.1.8]{hovey:model_categories}.
We will freely use the based version of Propositions \ref{prop:Gmod} 
and \ref{prop:homomorphism adjunctions spaces} in what follows.

\medskip

After discussing equivariant spaces, we now move on to equivariant spectra.
An {\em inner product space}\index{inner product space}
is a finite-dimensional real vector space equipped with a scalar product.
We denote by $\bL(V,W)$ the space of linear isometric embeddings
between two inner product spaces, topologized as the Stiefel manifold of
$\dim(V)$-frames in $W$.

\begin{con}\label{con:O}
We let $V$ and $W$ be inner product spaces.
The {\em orthogonal complement vector bundle}
over the space $\bL(V,W)$ of linear isometric embeddings has total space
\[ \xi(V,W) \ = \ 
\{\, (w,\varphi) \in W\times\bL(V,W) \ | \ w\perp\varphi(V)\,\} \ .\]
The structure map $\xi(V,W)\to\bL(V,W)$
is the projection to the second factor.
The vector bundle structure of~$\xi(V,W)$ is 
as a vector subbundle of the trivial vector bundle $W\times\bL(V,W)$,
and the fiber over $\varphi:V\to W$ is the orthogonal complement $W-\varphi(V)$ 
of the image of $\varphi$.

We let $\bO(V,W)$ be the Thom space of the orthogonal complement bundle,
i.e., the one-point compactification of the total space of $\xi(V,W)$.
Up to non-canonical homeomorphism, we can describe the space $\bO(V,W)$ 
differently as follows.
If the dimension of $W$ is smaller than the dimension of $V$,
then the space $\bL(V,W)$ is empty and $\bO(V,W)$ consists
of a single point at infinity. 
If $\dim V=m$ and $\dim W=m+n$,
then $\bL(V,W)$ is homeomorphic to the homogeneous space $O(m+n)/O(n)$
and $\bO(V,W)$ is homeomorphic to $O(m+n)\ltimes_{O(n)}S^n$.

The Thom spaces $\bO(V,W)$ are the morphism spaces of a based topological category $\bO$.
Given a third inner product space $U$, the bundle map 
\[ \xi(V,W) \times \xi(U,V) \ \to \ \xi(U,W) \ , \quad
((w,\varphi),\,(v,\psi)) \ \longmapsto \ (w+\varphi(v),\,\varphi\psi)\]
covers the composition map $\bL(V,W)\times\bL(U,V)\to \bL(U,W)$.
Passage to Thom spaces gives a based map
\[ \circ \ : \ \bO(V,W) \sm \bO(U,V) \ \to \ \bO(U,W) \]
which is clearly associative, and is the composition in the category $\bO$. 
The identity of $V$ is $(0,\Id_V)$ in $\bO(V,V)$.
\end{con}

\begin{defn}\label{def:Gspec} 
Let $G$ be a Lie group.\index{orthogonal $G$-spectrum}
An {\em orthogonal $G$-spectrum} is a based continuous functor from $\bO$ 
to the category $G\bT_*$ of based $G$-spaces.
A {\em morphism} of orthogonal spectra is a natural transformation of functors.
We denote the category of orthogonal $G$-spectra by $\Sp_G$. 
\end{defn}  

A continuous functor to based $G$-spaces is the same as a $G$-object
of continuous functors. So orthogonal $G$-spectra could equivalently
be defined as orthogonal spectra equipped with a continuous $G$-action.
Since we will not consider any other kind of spectra in this memoir,
we will often drop the adjective `orthogonal'; in other words,
we use `$G$-spectrum' as a synonym for `orthogonal $G$-spectrum'.
\medskip

If $V$ and $W$ are inner product spaces, we define a distinguished based continuous map
\begin{equation} \label{eq:define_i_V W}
 i_{V,W}\ : \ S^V \ \to \ \bO(W,V\oplus W)\text{\quad by\quad} 
v\ \longmapsto \ ( (v,0), (0,-))\ ,   
\end{equation}
the one-point compactification of the fiber over the embedding
$(0,-):W\to V\oplus W$ as the second summand.
If $X$ is an orthogonal spectrum, we refer to the composite
\[ \sigma_{V,W}\ : \ 
S^V\sm X(W)\ \xrightarrow{i_{V,W}\sm X(W)} \  
\bO(W,V\oplus W)\sm X(W) \ \to\  X(V\oplus W)   \]
as the {\em structure map} of~$X$.\index{structure map!of an orthogonal $G$-spectrum}

Limits and colimits in enriched functor categories are created objectwise.
In particular, all small limits and colimits in $\Sp_G$ exist and are
created `levelwise'. Moreover, limits and colimits of based $G$-spaces
are created on underlying non-equivariant based spaces; hence all limits and
colimits in~$\Sp_G$ are created in the category
of underlying non-equivariant orthogonal spectra.
By \cite[Sec.\,3.8]{kelly:enriched_category_theory} 
we conclude that the category $\Sp_G$ is enriched complete and cocomplete.

\begin{rk}
If $G$ is compact, then the above definition is equivalent to the original 
definition of orthogonal $G$-spectra given by Mandell and May 
in \cite{mandell-may:equivariant_orthogonal}, in the following sense.
In \cite[II.2]{mandell-may:equivariant_orthogonal}, 
Mandell and May define $G$-equivariant orthogonal spectra indexed on a 
$G$-universe $\cU$.\index{G-universe@$G$-universe}
Such a $G$-spectrum is a collection of $G$-spaces indexed on those representations 
that embed into $\cU$, 
together with certain equivariant structure maps. 
It follows from \cite[Thm.\,II.4.3, Thm.\,V.1.5]{mandell-may:equivariant_orthogonal} 
that for any $G$-universe $\cU$, the category of orthogonal $G$-spectra indexed 
on $\cU$ and the category of orthogonal $G$-spectra as in Definition \ref{def:Gspec} 
are equivalent. This shows that universes are not really relevant 
for the pointset level definition of an orthogonal $G$-spectrum. 
However, they become important when one considers the homotopy theory 
of orthogonal $G$-spectra. 
\end{rk}

Here are some basic examples of orthogonal $G$-spectra;
further examples will be discussed along the way.

\begin{eg}[Suspension spectra]
The {\em sphere spectrum} $\mS$ is the orthogonal spectrum\index{suspension spectrum} 
given by $\mS(V)= S^V$.
The orthogonal group $O(V)$ acts on $V$ and hence 
on the one-point compactification $S^V$. 
The structure maps are the canonical homeomorphisms $S^V \wedge S^W \cong S^{V\oplus W}$.
The {\em $G$-sphere spectrum} $\mS_G$ is the orthogonal sphere spectrum\index{suspension spectrum}  
equipped with trivial $G$-action.
We will show in Example \ref{eg:represent cohomotopy}
that for discrete groups,
the sphere spectrum represents $G$-equivariant stable cohomotopy as
defined by the third author in \cite{luck:burnside_ring}. 

More generally we consider a based $G$-space $A$.
The {\em suspension spectrum} $\Sigma^\infty A$ is defined by 
$(\Sigma^\infty A)(V)= S^V\sm A$. The group $G$ acts through the second factor 
and the orthogonal groups act through the first factor. 
The structure maps are given by the canonical homeomorphisms 
$S^V\sm (S^W \wedge A)  \cong  S^{V\oplus W}\sm A$. 
The sphere spectrum $\mS$ is isomorphic to $\Sigma^\infty S^0$.
\end{eg}

\begin{eg}[Trivial $G$-spectra] 
  Every orthogonal spectrum becomes an orthogonal $G$-spectrum
  by letting $G$ act trivially.
For example, the $G$-sphere spectrum $\mS_G$ arises in this way.
This construction derives to an exact functor from global stable homotopy theory
to $G$-equivariant stable homotopy theory,
compare Theorem~\ref{thm:change to groups} below.\index{global spectrum} 
The $G$-equivariant cohomology theories that arise in this way from
global stable homotopy types have additional structure and special properties,
i.e., they form `equivariant cohomology theories' for all Lie groups and 
not just for a particular group and its subgroups.
We return to this class of examples in more detail in Section \ref{sec:global}.
\end{eg}

\begin{con}\label{con:smash with G-space}
We let $G$ and $K$ be Lie groups.
Every continuous based functor $F:G\bT_\ast\to K\bT_\ast$
between the categories of based equivariant spaces gives rise to a continuous functor
\[ F\circ- \ : \ \Sp_G \ \to \ \Sp_K \]
from orthogonal $G$-spectra to orthogonal $K$-spectra by postcomposition: 
it simply takes an orthogonal $G$-spectrum $X$ to the composite
\[ \bO \ \xrightarrow{\ X \ }\ G\bT_*\ \xrightarrow{\ F \ } \ K\bT_\ast\ . \]
If $A$ is a based $G$-space, then smashing with $A$ and taking based maps
out of $A$ are two such functors (for $K=G$). So for every orthogonal $G$-spectrum~$X$,
we can define two new orthogonal $G$-spectra $X\sm A$ and $X^A$
by smashing with $A$ or taking maps from $A$ levelwise.
More explicitly, we have
\[  (X\sm A)(V) =  X(V) \sm A \text{\quad respectively\quad}
(X^A)(V) = X(V)^A = \map_*(A,X(V)) \]
for an inner product space~$V$. 
The structure maps and actions of the orthogonal groups 
do not interact with~$A$.
Just as the functors $-\sm A$ and $\map_*(A,-)$ are adjoint
on the level of based $G$-spaces, the two functors
just introduced are an adjoint pair on the level of orthogonal $G$-spectra.
\end{con}

The previous Construction~\ref{con:smash with G-space}
provides tensors and cotensors for the category of orthogonal $G$-spectra
over the closed symmetric monoidal category (under smash product) of based $G$-spaces.
There is also  enrichment of orthogonal $G$-spectra in based $G$-spaces as follows.
The mapping space~$\map(X,Y)$ between two orthogonal $G$-spectra is
the space of morphisms between the underlying non-equivariant orthogonal
spectra of~$X$ and~$Y$; on this mapping space, the group~$G$ acts by conjugation.

\begin{eg}[Free spectra and evaluation on representations] \label{ex:free spectra}
  Every real inner product space $V$ gives rise 
  to a representable functor $\bO(V,-):\bO \to \bT_*$,
  which we denote by $F_V$. For example, the sphere spectrum $\mS$ is
  isomorphic to the representable functor $\bO(0,-)$.
  We turn $F_V$ into a $G$-orthogonal spectrum by giving it the trivial action. 
  As a consequence of the enriched Yoneda lemma, the functor
  \[ G\bT_*\ \to\  \Sp_G\ , \quad A\ \longmapsto\  F_V\sm A \]
  is left adjoint to the evaluation functor at $V$.\index{orthogonal $G$-spectrum!free}
  
  More generally, we let $H$ be a closed subgroup of $G$ and $V$ an $H$-representation. 
  Then the evaluation~$X(V)$ is an $(H\times H)$-space 
  by the `external' $H$-action on~$X$ and the `internal' $H$-action from the action 
  on~$V$ and the $O(V)$-functoriality of~$X$. 
  We consider $X(V)$ as an $H$-space via the diagonal $H$-action. 
  Via this action, if $W$ is another $H$-representation, the structure map
  \[ S^V\wedge X(W)\to X(V\oplus W) \]
  becomes $H$-equivariant, with $H$ acting diagonally on the domain. 
  Moreover, the resulting evaluation functor
  \[ -(V)\ :\ \Sp_G\ \to \ H\bT_* \]
  also has a left adjoint which sends a based $H$-space $A$ 
  to the orthogonal $G$-spectrum $G\ltimes_H (F_V\wedge A)$. 
  Here $H$ acts on $F_V$ by precomposition with the $H$-action on~$V$.
\end{eg}

\section{The stable model structure}
\label{sec:stable}

In this relatively long section
we establish the stable model structure on the category $\Sp_G$
of orthogonal $G$-spectra and investigate how it interacts with the smash product. 
We begin by recalling the equivariant homotopy groups $\pi_*^H$ 
for {\em compact} Lie groups $H$ in Construction \ref{con:equivariant homotopy groups compact},
which are used to define the $\pi_*$-isomorphisms of orthogonal $G$-spectra,
see Definition \ref{def:pi_*-isomorphism}.
As an application of our theory 
we will later also define equivariant homotopy groups 
for non-compact Lie groups, see Definition~\ref{def:G homotopy groups} below, 
but these do not play a role in the construction of the model structure. 
Theorem \ref{thm:bifree smash preserves} and Corollary \ref{cor:Gamma2G}
show that certain induction constructions preserve $\pi_*$-isomorphisms.
The stable model structure on the category of orthogonal $G$-spectra
in Theorem \ref{thm:stable} is the main result of this section.
This model structure is Quillen equivalent to the one previously obtained by Fausk 
in~\cite[Prop.\,6.5]{fausk:prospectra}; 
our model structure has more cofibrations and we provide an explicit characterization
of the stable fibrations by certain homotopy pullback requirements. 
The proof of the stable model structure proceeds 
by localizing a certain level model structure.

We also review the smash product of orthogonal spectra in Definition \ref{def:smash product}
and establish various homotopical properties of the smash product of orthogonal $G$-spectra,
with diagonal $G$-action.
Theorem \ref{thm:smashing with cofibrants} shows that
smashing with a fairly broad class of orthogonal $G$-spectra
that we call `quasi-flat' preserves $\pi_*$-isomorphisms;
Proposition \ref{prop:smash is monoidal model} (`pushout product property')
and Proposition \ref{prop:monoid axiom} ('monoid axiom')
show that the smash product of orthogonal $G$-spectra interacts well with the
stable model structure.
The final result of this section is Proposition \ref{prop:Thom is invertible},
which shows that the Thom space of a $G$-vector bundle over
the universal $G$-space $\uEG$ for proper actions is smash invertible in $\Ho(\Sp_G)$.
This result is the first indication that the role of the representation ring $R O(G)$
in the realm of compact Lie groups is now taken by $K O_G(\uEG)$,
the Grothendieck group of $G$-vector bundles over $\uEG$.

\begin{defn} 
  Let $H$ be a compact Lie group.\index{G-universe@$G$-universe!complete} 
  A {\em complete $H$-universe} is an orthogonal $H$-representation
  of countably infinite dimension such that 
  every finite-dimensional $H$-representation embeds into it.
\end{defn}

For every compact Lie group $H$, we choose a complete $H$-universe $\cU_H$.
Up to equivariant isometry, such a complete $H$-universe is given by
\[ \cU_H \ \iso \ \bigoplus_{\lambda\in\Lambda} \bigoplus_\mN\lambda \ , \]
where $\Lambda$ is a set of representatives of all irreducible $H$-representations.
We let $s(\cU_H)$ denote the poset, under inclusion,
of finite-dimensional $H$-subrepresentations of $\cU_H$.

\begin{con}[Equivariant homotopy groups]\label{con:equivariant homotopy groups compact}
  Let $k$ be any integer, $X$ an orthogonal $G$-spectrum and $H$ a compact subgroup of $G$;
  we define the $H$-equivariant homotopy group~$\pi_k^H(X)$. 
  We start with the case $k\geq 0$.
  We recall that for an orthogonal $H$-re\-presen\-ta\-tion $V$, 
  we let $H$ act diagonally on~$X(V)$, through the two $H$-actions on~$X$ and on~$V$.
  For every $V\in s(\cU_H)$ we consider the set
  \[  [S^{V\oplus \mR^k}, X(V)]^H \]
  of $H$-equivariant homotopy classes of based $H$-maps
  from $S^{V\oplus\mR^k}$ to $X(V)$.
  We can stabilize by increasing $V\subset W$ along the maps
  \[
    [S^{V\oplus\mR^k}, X(V)]^H \ \to \  [S^{W\oplus\mR^k},X(W)]^H
  \]
  defined as follows.
  We let $V^\perp=W-V$ denote the orthogonal complement of $V$ in~$W$.
  The stabilization sends the homotopy class of $f:S^{V\oplus\mR^k}\to X(V)$
  to the homotopy class of the composite 
  \begin{align*}
    S^{W\oplus\mR^k}\ \iso \ S^{V^\perp} \sm S^{V\oplus \mR^k}
    \ &\xrightarrow{S^{V^\perp}\sm f} \ 
        S^{V^\perp}\sm X(V) \\ 
      &\xrightarrow{\ \sigma_{V^\perp,V}}\  
        X(V^\perp\oplus V) \ \iso \ X(W) \ ,
  \end{align*}
  where the two unnamed homeomorphisms use the preferred linear isometry
  \[
    V^\perp\oplus V  \ \iso \  W\ , \quad (w,v)\ \longmapsto \ (w+v) \ .
  \]
  These stabilization maps define a functor on the poset $s(\cU_H)$.
  The {\em $k$-th equivariant homotopy group} $\pi_k^H(X)$ is then defined as
  \begin{equation}\label{eq:define pi^H}\index{equivariant homotopy groups}
    \pi_k^H(X) \ = \ \colim_{V\in s(\cU_H)}\,    [S^{V\oplus\mR^k}, X(V)]^H \ .
  \end{equation}
  The abelian group structure arises from the pinch addition in the source variable,
  based on a $G$-fixed unit vector in $V$, for large enough $V$.
  For $k<0$, the definition of $\pi_k^H(X)$ is the same,
  but with  $[S^{V\oplus\mR^k}, X(V)]^H$ replaced by $[S^V, X(V\oplus\mR^{-k})]^H$.
\end{con}

\begin{defn}\label{def:pi_*-isomorphism} 
Let $G$ be a Lie group.
A morphism $f \colon X \to Y$ of orthogonal $G$-spectra 
is a {\em $\pi_*$-isomorphism} if for every compact subgroup $H$ of~$G$
and every integer $k$, the induced map\index{pistar-isomorphism@$\pi_*$-isomorphism}
\[ \pi_k^H(f) \ :\ \pi_k^H (X) \ \to\ \pi_k^H(Y)  \]
is an isomorphism. 
\end{defn}

In the case of compact groups, this definition recovers 
the notion of $\pi_*$-isomorphism 
from \cite[Sec.\,III.3]{mandell-may:equivariant_orthogonal}
or \cite[Def.\,3.1.12]{schwede:global}.

\begin{con}[Loop and suspension isomorphism]
  An important special case of Construction~\ref{con:smash with G-space}
  is when $A=S^1$ is a $1$-sphere with trivial action.
  The \emph{suspension}
  $X \sm S^1$ is defined by
  \[ (X \sm S^1)(V) \ = \ X(V) \sm S^1 \ , \]
  the smash product of the $V$-th level of $X$ with the sphere $S^1$.
  The \emph{loop spectrum} $\Omega X=X^{S^1}$ is defined by
  \[ (\Omega X)(V) \ = \ \Omega X(V) \ = \ \map_*(S^1,X(V)) \ , \]
  the based mapping space from  $S^1$ to the $V$-th level of $X$.
  
  We define the {\em loop isomorphism}\index{loop isomorphism}
  \begin{equation}\label{eq:loop iso} 
    \alpha\ : \  \pi^H_k(\Omega X)\ \to\ \pi^H_{k+1} (X)  \ .
  \end{equation}
  For $k\geq 0$, we represent a given class in $\pi_k^H(\Omega X)$ by a based $H$-map 
  $f:S^{V\oplus \mR^k}\to \Omega X(V)$ and let $\tilde f :S^{V\oplus \mR^{k+1}}\to X(V)$
  denote the adjoint of $f$, which represents an element of
  $\pi_{k+1}^H(X)$.
  For $k< 0$, we represent a class in $\pi_k^H(\Omega X)$ by a based $H$-map 
  $f:S^V\to \Omega X(V\oplus\mR^{-k})$ and let
  $\tilde f :S^{V\oplus \mR}\to X(V\oplus\mR^{-k})\iso X( (V\oplus\mR)\oplus \mR^{-(k+1)})$
  denote the adjoint of $f$, which represents an element of
  $\pi_{k+1}^H(X)$. Then we can set
  \[ \alpha[f]\ = \ [\tilde f] \ .\]
  The loop isomorphism is indeed bijective, by straightforward adjointness.

  Next we define the {\em suspension isomorphism}\index{suspension isomorphism!for equivariant homotopy groups}
  \begin{equation}\label{eq:suspension iso}
    -\sm S^1 \ : \ \pi^H_k (X)\ \to \ \pi^H_{k+1}(X\sm S^1) \ .   
  \end{equation}
  For $k\geq 0$ we represent a given class in $\pi_k^H(X)$
  by a based $H$-map $f:S^{V\oplus \mR^k}\to X(V)$;
  then $f\sm S^1 : S^{V\oplus \mR^{k+1}}\to X(V)\sm S^1$
  represents a class in $\pi_{k+1}^H(X \sm S^1)$.
  For $k< 0$ we represent a given class in $\pi_k^H(X)$
  by a based $H$-map $f:S^V\to X(V\oplus\mR^{-k})$;
  then $f\sm S^1 : S^{V\oplus \mR}\to X(V\oplus\mR^{-k})\sm S^1\iso
  X( (V\oplus\mR)\oplus\mR^{-(k+1)})\sm S^1$
  represents a class in $\pi_{k+1}^H(X \sm S^1)$. Then we set
  \[  [f]\sm S^1 \ = \ [f\sm S^1] \ .\]
  The suspension isomorphism is indeed bijective, 
  see for example~\cite[Prop.\,3.1.30]{schwede:global}.
\end{con}

Next we recall the concept of an h-cofibration of orthogonal $G$-spectra;
this notion occurs in the proof of the Theorem \ref{thm:bifree smash preserves},
and we will use it at several later points.
In the case when $G$ is a compact Lie group, the basic properties of 
h-cofibrations are discussed 
in \cite[III Thm.\,3.5]{mandell-may:equivariant_orthogonal}.

\begin{defn}\label{def:h-cofibration}
A morphism $i \colon A \to B$ of orthogonal $G$-spectra is an {\em h-cofibration} 
if it has the homotopy extension property:\index{h-cofibration!of orthogonal $G$-spectra} 
for every orthogonal $G$-spectrum $X$, every morphism 
$\varphi:B\to X$ and every homotopy $H\colon A\sm [0,1]_+ \to X$
starting with $\varphi\circ i$, there exists a homotopy
$\bar H:B\sm [0,1]_+ \to X$ starting with $\varphi$ that satisfies
$\bar H\circ(i\sm[0,1]_+)=H$.
\end{defn} 

There is a universal test case for the homotopy extension property,
and a morphism $i \colon A \to B$ is an h-cofibration 
if and only if the canonical morphism 
$(A\sm [0,1]_+) \cup_{i} B \to B\sm [0,1]_+ $ admits a retraction. 
For every continuous homomorphism $\alpha:K\to G$ between Lie groups,
the restriction functor $\alpha^*:\Sp_G\to\Sp_K$ preserves colimits
and smash products with based spaces; so if $i$ is an h-cofibration of $G$-spectra,
then $\alpha^*(i)$  is an h-cofibration of $K$-spectra.
In particular, restriction to a closed subgroup preserves h-cofibrations.

Similarly, if $i:A\to B$ is an h-cofibration of orthogonal $G$-spectra,
then for every compact subgroup $H$ of $G$ 
and every orthogonal $H$-re\-presen\-tation $V$, 
the $H$-equivariant map $i(V) \colon A(V) \to B(V)$ is an h-cofibration 
of based $H$-spaces. This uses that the evaluation functors also commute 
with colimits and smash products with based spaces. 

\medskip

In the next theorem we consider two Lie groups~$K$ and~$\Gamma$.
We call a $(K\times \Gamma)$-space {\em bifree}
if the underlying $K$-action is free and
the underlying $\Gamma$-action is free.
Equivalently, the stabilizer group of every point 
intersects both of the two subgroups $K\times\{1\}$
and~$\{1\}\times \Gamma$ only in the neutral element. 
The compact subgroups of~$K\times \Gamma$ with this property are precisely the
graphs of all continuous monomorphisms $\alpha:L\to \Gamma$,
defined on compact subgroups of~$K$.
In the following theorem we turn the left $\Gamma$-action on $A$
into a right action by setting $a\cdot \gamma = \gamma^{-1}\cdot a$,
for $(a,\gamma)\in A\times \Gamma$.

\begin{thm}\label{thm:bifree smash preserves}
Let $\Gamma$ and $K$ be Lie groups and $A$ 
a bifree $\Com$-cofibrant $(K\times \Gamma)$-space. Then the functor
\[ A_+\sm_\Gamma - \ : \ \Sp_\Gamma\ \to \ \Sp_K \]
takes $\pi_*$-isomorphisms of orthogonal $\Gamma$-spectra to $\pi_*$-isomorphisms
of orthogonal $K$-spectra.
\end{thm}
\begin{proof}
We start with the special case when the group $K$ is compact.
Since the functor $A_+\sm_\Gamma-$ commutes with mapping cones,
and since mapping cone sequences give rise to long exact sequences
of equivariant homotopy groups~\cite[Prop.\,3.1.36]{schwede:global}, 
it suffices to show the following special case: 
we let~$X$ be any orthogonal $\Gamma$-spectrum
that is $\Gamma$-$\pi_*$-trivial, i.e., 
all of whose equivariant homotopy groups, for all compact subgroups of~$\Gamma$, vanish.
Then $A_+\sm_\Gamma X$ is $K$-$\pi_*$-trivial.
For this we assume first that $A$ is a finite-dimensional 
proper $(K\times \Gamma)$-CW-complex,
with skeleta $A^n$. We argue by induction over the dimension of~$A$.
The induction starts with~$A^{-1}$, which is empty, and there is nothing to show.
Then we let $n\geq 0$ and assume the claim for~$A^{n-1}$.
By hypothesis there is a pushout square of  $(K\times \Gamma)$-spaces:
\[ \xymatrix{ 
\coprod_{j\in J} \, (K\times \Gamma)/\Delta_j\times \partial D^n
  \ar[r]\ar[d]&
\coprod_{j\in J} \,  (K\times \Gamma)/\Delta_j\times  D^n
 \ar[d]\\
A^{n-1}\ar[r] & A^n } \]
Here $J$ is an indexing set of the $n$-cells of the equivariant CW-structure
and $\Delta_j$ is a compact subgroup of $K\times \Gamma$.
Since the $(K\times \Gamma)$-action on $A$ is bifree, 
each of the subgroups $\Delta_j$
must be the graph of a continuous monomorphism $\alpha_j:L_j\to \Gamma$
defined on a compact subgroup $L_j$ of~$K$.

The inclusion $A^{n-1}\to A^n$
is an h-cofibration of $(K\times\Gamma)$-spaces,
so the morphism $A^{n-1}_+\sm_\Gamma X\to A^n_+\sm_\Gamma X$
is an h-cofibration of orthogonal $K$-spectra.
The long exact homotopy group sequence \cite[Cor.\,3.1.38]{schwede:global}
thus reduces the inductive step
to showing that the $K$-equivariant homotopy groups of the cofiber
$(A^n_+\sm_\Gamma X)/(A^{n-1}_+\sm_\Gamma X)$ vanish.
This cofiber is isomorphic to 
\[ \bigvee_{j\in J} \,  (K\times \Gamma/\Delta_j)_+\sm_\Gamma X \sm S^n  \ .\]
Since equivariant homotopy groups take wedges
to sums~\cite[Cor.\,3.1.37 (i)]{schwede:global}
and reindex upon smashing with $S^n$
(by the suspension isomorphism~\eqref{eq:suspension iso}, 
compare~\cite[Prop.\,3.1.30]{schwede:global}), 
it suffices to
consider an individual wedge summand without any suspension. 
In other words, we may show that the orthogonal $K$-spectrum
\[  (K\times \Gamma)/\Delta_+\sm_\Gamma X \ \cong \ 
K\ltimes_L \alpha^*(X)   \]
is $K$-$\pi_*$-trivial, where $L$ is a closed subgroup of~$K$ and $\Delta$ is
the graph of a continuous monomorphism $\alpha:L\to \Gamma$.
Now $X$ is $\Gamma$-$\pi_*$-trivial by hypothesis, 
so $\alpha^*(X) $ is $L$-$\pi_*$-trivial.
Since $K$ and $L$ are compact,  \cite[Cor.\,3.2.21]{schwede:global}
lets us conclude that $K\ltimes_L\alpha^*(X)$ is $K$-$\pi_*$-trivial. 
This completes the inductive step.

Now we suppose that $A$ is a proper $(K\times\Gamma)$-CW-complex,
possibly infinite dimensional. As already noted above, the morphisms
\[ A^{n-1}_+ \sm_\Gamma X \ \to \  A^n_+ \sm_\Gamma X   \]
induced by the skeleton inclusions 
are h-cofibrations of orthogonal $K$-spectra.
Since $A_+\sm_\Gamma X$ is a colimit of the sequence of spectra $A^n_+\sm_\Gamma X$,
each $A^n_+\sm_\Gamma X$ is $K$-$\pi_*$-trivial, 
and a colimit of $\pi_*$-isomorphisms over a sequence of h-cofibrations
is another $\pi_*$-isomorphism
(compare \cite[III Thm.\,3.5 (v)]{mandell-may:equivariant_orthogonal}
or \cite[Prop.\,3.1.41]{schwede:global}),
we conclude that $A_+\sm_\Gamma X$ is $K$-$\pi_*$-trivial.
A general $\Com$-cofibrant $(K\times \Gamma)$-space is
$(K\times \Gamma)$-homotopy equivalent
to a proper $(K\times \Gamma)$-CW-complex, so this concludes the proof
in the special case where $K$ is compact.

Now we treat the general case. 
We let $L$ be any compact subgroup of $K$.
The underlying $(L\times\Gamma)$-space of $A$ is again bifree,
and it is $\Com$-cofibrant as an $(L\times \Gamma)$-space by
Proposition \ref{prop:homomorphism adjunctions spaces} (iii).
So the composite functor $\res^K_L\circ (A_+\sm_\Gamma-):\Sp_\Gamma\to\Sp_L$
preserves $\pi_*$-isomorphisms by the special case above.
Since $\pi_*$-isomorphisms of $K$-spectra can be tested on all
compact subgroups of $K$, this proves the claim.
\end{proof}

\begin{cor}\label{cor:Gamma2G}
Let $\Gamma$ be a closed subgroup of a Lie group $G$. 
The induction functor
\[ G\ltimes_\Gamma \ : \ \Sp_\Gamma \ \to \ \Sp_G \]
takes $\pi_*$-isomorphisms of orthogonal $\Gamma$-spectra 
to $\pi_*$-isomorphisms of orthogonal $G$-spectra.
\end{cor}
\begin{proof}
We let $H$ be a compact subgroup of~$G$.
Then the $(H\times \Gamma)$-action on $G$ given 
by $(h,\gamma) \cdot g = h g \gamma^{-1}$ 
underlies a proper $(H \times \Gamma)$-CW-complex,
by Proposition \ref{prop:H-Gamma-CW}.
In particular, $G$ is $\Com$-cofibrant as an $(H\times \Gamma)$-space.
The action is bifree, so Theorem~\ref{thm:bifree smash preserves}   
applies and shows that the functor $\res^G_H\circ (G\ltimes_\Gamma -)$
takes $\pi_*$-isomorphisms of orthogonal $\Gamma$-spectra
to $\pi_*$-isomorphisms of orthogonal $H$-spectra. 
Since $H$ was an arbitrary compact subgroup of $G$, 
this proves the claim.
\end{proof}

We let $H$ be a compact subgroup of a Lie group $G$, 
and $X$ is an orthogonal $G$-spectrum. 
We recall from Example \ref{ex:free spectra} that if $V$ is an $H$-representation, 
then we equip the evaluation~$X(V)$ with the diagonal $H$-action 
of the two $H$-actions on $X$ and on $V$.

\begin{defn} Let $G$ be a Lie group and $f \colon X \to Y$ a morphism 
  of orthogonal $G$-spectra.
  \begin{enumerate}[(i)]
  \item 
    The morphism $f$ is a \emph{level equivalence}\index{level equivalence!of orthogonal $G$-spectra} 
    if $f(V)^H \colon X(V)^H \to Y(V)^H$ is a weak equivalence for every 
    compact subgroup~$H$ of~$G$ and every orthogonal $H$-representation~$V$.
  \item 
    The morphism $f$ is a \emph{level fibration}\index{level fibration!of orthogonal $G$-spectra}  
    if $f(V)^H \colon X(V)^H \to Y(V)^H$ is a Serre fibration for every 
    compact subgroup~$H$ of~$G$ and every orthogonal $H$-re\-presentation $V$.
  \end{enumerate}
\end{defn}

It follows from the definition that if $f \colon X \to Y$ 
is a level equivalence (level fibration), then $f(V) \colon X(V) \to Y(V)$ 
is an $H$-weak equivalence ($H$-fibration) for every compact subgroup $H$, 
simply because every orthogonal $H$-representation
is also a $K$-representation for every $K \leq H$.

In order to define the cofibrations of orthogonal $G$-spectra,
we recall the {\em skeleton filtration},
a functorial way to write an orthogonal spectrum as a 
sequential colimit of spectra which are made from the information
below a fixed level. The word `filtration' should be used with caution
because the maps from the skeleta to the orthogonal spectrum need not be injective.

\begin{con}[Skeleton filtration of orthogonal spectra]\index{skeleton!of an orthogonal spectrum} 
We let $\bO_{\leq m}$ denote the full topological subcategory 
of $\bO$ whose objects are the inner product spaces of dimension at most $m$.
We write $\Sp_{\leq m}$ for the category of continuous based functors
from $\bO_{\leq m}$ to $\bT_*$.
Restriction to the subcategory $\bO_{\leq m}$ defines a functor
\[ (-)_{\leq m}\ : \ \Sp\ \to \ \Sp_{\leq m} \ . \]
This functor has a left adjoint
\[ l_m\ : \ \Sp_{\leq m}\ \to \ \Sp  \ ,\]
an enriched left Kan extension.
The {\em $m$-skeleton} of an orthogonal spectrum $X$ is
\[ \sk^m X\ = \ l_m(X_{\leq m}) \ ,\]
the extension of the restriction of $X$ to $\bO_{\leq m}$.
The skeleton comes with a natural morphism $i_m:\sk^m X\to X$, the counit of
the adjunction $(l_m,(-)_{\leq m})$. 
Kan extensions along a fully faithful functor do not change the values
on the given subcategory \cite[Prop.\,4.23]{kelly:enriched_category_theory},
so the value
\[ i_m(V)\ :\ (\sk^m X)(V)\ \to \ X(V) \]
is an isomorphism for all inner product spaces $V$ of dimension at most $m$.
The {\em $m$-th latching space} of $X$ is the based $O(m)$-space
\[ L_m X \ = \ (\sk^{m-1} X)(\mR^m) \ ;\]\index{latching space!of an orthogonal spectrum}  
it comes with a natural based $O(m)$-equivariant map
\[  \nu_m=i_{m-1}(\mR^m)\ :\ L_m X\ \to \ X(\mR^m) \ , \]
the {\em $m$-th latching map}. 
We set $\sk^{-1} X=\ast$, the trivial orthogonal spectrum,
and $L_0 X=\ast$, a one-point space.\index{latching map!of an orthogonal spectrum}  

The different skeleta are related by natural morphisms
$j_m:\sk^{m-1} X\to \sk^m X$, for all $m\geq 0$, 
such that $i_m\circ j_m=i_{m-1}$.
The sequence of skeleta stabilizes to $X$ in a strong sense: 
the maps $j_m(V)$ and $i_m(V)$ are isomorphisms as soon as $m >\dim(V)$.
In particular, $X(V)$ is a colimit, with respect to the maps $i_m(V)$, 
of the sequence of maps $j_m(V)$. 
Since colimits in the category of orthogonal spectra are created objectwise,
the orthogonal spectrum $X$ is a colimit, 
with respect to the morphisms $i_m$, of the sequence of morphisms $j_m$. 

Moreover, each skeleton is built from the previous one in a systematic way
controlled by the latching map. We write $G_m$
for the left adjoint to the evaluation functor 
\[ \text{ev}_{\mR^m}\ : \ \Sp \ \to \ O(m)\bT_*\ . \]
Then the commutative square
\begin{equation}  \begin{aligned}\label{eq:skeleton_pushout}
\xymatrix@C=12mm{ G_m L_m X \ar[r]^-{G_m\nu_m} \ar[d] & G_m X(\mR^m) \ar[d]\\
\sk^{m-1} X \ar[r]_-{j_m} & \sk^m X}    
  \end{aligned}\end{equation}
is a pushout of orthogonal spectra, see \cite[Prop.\,C.17]{schwede:global}.
The two vertical morphisms are instances of the adjunction counit.
\end{con}

Since the skeleta and latching objects
are continuous functors in the orthogonal spectrum,
and since the latching morphisms are natural,
actions of groups go along free for the ride.
More precisely, the skeleta of (the underlying orthogonal spectrum of)
an orthogonal $G$-spectrum inherit a continuous $G$-action by functoriality.
In other words, the skeleta and the various morphisms between them
lift to endofunctors and natural transformations on the category of
orthogonal $G$-spectra.
If $X$ is an orthogonal $G$-spectrum, then the $O(m)$-space $L_m X$ 
comes with a commuting action by $G$, again by functoriality of the latching space.
Moreover, the latching morphism $\nu_m:L_m X\to X(\mR^m)$ is
$(G\times O(m))$-equivariant.
Since colimits of orthogonal $G$-spectra are created in the 
underlying category of orthogonal spectra, the square \eqref{eq:skeleton_pushout}
is a pushout square of orthogonal $G$-spectra.

\begin{defn}
  Let $G$ be a Lie group.
  A morphism $i:A\to B$ of orthogonal $G$-spectra is a \emph{cofibration} \index{cofibration!of orthogonal $G$-spectra} 
  if for every $m\geq 0$ the latching map 
  \[ \nu_m i\ = \ \nu_m^B \cup i(\mR^m)  \ :\ L_m B \cup_{L_m A} A(\mR^m) \ \to\  B(\mR^m) \]
  is a $\Com$-cofibration of $(G\times O(m))$-spaces and, moreover, the
  action of $O(m)$ is free away from the image of $\nu_m i$.
\end{defn}

If $X$ is an orthogonal spectrum and $V$ and $W$ are inner product spaces, we write
\[ \tilde\sigma_{V,W}\ : \ X(W)\ \to \ \map_*(S^V,X(V\oplus W))  \]
for the adjoint of the structure map $\sigma_{V,W}:S^V\sm X(W)\to X(V\oplus W)$.
Here $\map_*(-,-)$ denotes the space of {\em based} continuous maps.

\begin{defn}\label{def:stable fibrations}
    A morphism $f:X\to Y$ of orthogonal $G$-spectra is a {\em stable fibration}
    if it is a level fibration,\index{stable fibration!of orthogonal $G$-spectra} 
    and moreover, for every compact subgroup $H$ of~$G$ and all $H$-representations~$V$ and~$W$ 
    the square 
    \begin{equation}  \begin{aligned}\label{eq:stable fibration square}
        \xymatrix@C=15mm{ X(W)^H \ar[d]_{f(V)^H} \ar[r]^-{(\tilde\sigma_{V,W})^H} & 
          \map_*^H(S^V, X(V\oplus W)) \ar[d]^{\map_*^H(S^V,f(V\oplus W))} \\
          Y(W)^H \ar[r]_-{(\tilde\sigma_{V,W})^H} & \map_*^H(S^V, Y(V\oplus W))}
      \end{aligned}\end{equation}
    is homotopy cartesian.  
    An orthogonal $G$-spectrum is a \emph{$G$-$\Omega$-spectrum}\index{G-Omega-spectrum@$G$-$\Omega$-spectrum}
    if for every compact subgroup $H$ of~$G$ and all $H$-representations~$V$ and~$W$ 
    the  map
    \[      (\tilde\sigma_{V,W})^H \ : \  X(W)^H \ \to \ \map_*^H(S^V,X(V\oplus W)) \]
    is a weak equivalence.
\end{defn}

We note that an orthogonal $G$-spectrum is a $G$-$\Omega$-spectrum
precisely when the unique morphism to any trivial spectrum is a stable fibration.
In other words, $G$-$\Omega$-spectra come out as the fibrant
objects in the stable model structure on~$\Sp_G$.

\begin{prop}\label{prop:stable zero is level zero}
  Let $G$ be a Lie group.
  Every $\pi_*$-isomorphism that is also a stable fibration 
  is a level equivalence.
\end{prop}
\begin{proof}
  This is a combination of Proposition~4.8 and Corollary~4.11
  of~\cite[Ch.\,III]{mandell-may:equivariant_orthogonal}.
  In more detail, we let $f:X\to Y$ be a $\pi_*$-isomorphism and a stable fibration,
  and we consider a compact subgroup~$H$ of~$G$.
  Since~$f$ is a stable~fibration, \cite[III Prop.\,4.8]{mandell-may:equivariant_orthogonal}
  shows that the morphism $\res^G_H(f)$ of underlying orthogonal $H$-spectra
  has the right lifting property with respect to a certain set~$K$ 
  of morphisms specified in  \cite[III Def.\,4.6]{mandell-may:equivariant_orthogonal}.
  Since $f$ is also a $\pi_*$-isomorphism, \cite[III Cor.\,4.11]{mandell-may:equivariant_orthogonal}
  then shows that for every $H$-representation~$V$ the map
  $f(V)^H:X(V)^H\to Y(V)^H$ is a weak equivalence.
\end{proof}

Now we name explicit sets of generating cofibrations and
generating acyclic cofibrations for the stable model structure on $\Sp_G$.
We fix once and for all a complete set $\cV_H$ 
of representatives of isomorphism classes of finite-dimensional 
orthogonal $H$-representations, for every compact Lie group~$H$.
We let $I^G_{\lv}$ denote the set of morphisms
\[
  ( G \ltimes_H F_V )\wedge  \partial D^k_+\ \to\  (G \ltimes_H F_V )\wedge D^k_+\ ,
\]
for all $k\geq 0$, where $H$ runs through all compact subgroups of~$G$
and $V$ runs through all representations in $\cV_H$. 
Here $F_V$ is the free spectrum in level $V$, see Example \ref{ex:free spectra}. 
Similarly, we let $J^G_{\lv}$ denote the set of morphisms
\begin{equation}\label{eq:gen_level_acyc_cofibration}
( G \ltimes_H F_V )\wedge  (D^k\times\{0\})_+
 \to 
 (G \ltimes_H F_V)\wedge(D^k\times[0,1])_+   \ , 
\end{equation}
with~$(H,V,k)$ running through the same set as for $I^G_{\lv}$.

Every morphism of orthogonal $G$-spectra $j:A\to B$
factors through the mapping cylinder as the composite
\[
  A\ \xrightarrow{c(j)} \  Z(j) = (A\sm [0,1]_+)\cup_j B \ \xrightarrow{r(j)}\ B
\]
where $c(j)$ is the `front' mapping  cylinder inclusion and~$r(j)$ 
is the projection, which is a homotopy equivalence.
In our applications we will assume that both $A$ and $B$ are cofibrant,
and then the morphism $c(j)$ is a cofibration,
compare \cite[Lemma 3.4.10]{hovey-shipley-smith:symmetric_spectra}.
We then define $\cZ(j)$
 as the set of all pushout product maps
\[ c(j)\Box i^k_+ \ : \  A\wedge D^k_+ \cup_{A\wedge \partial D^k_+} Z(j)\wedge\partial D^k_+
\ \to \ Z(j)\wedge D^k_+ \]
for $k\geq 0$, where $i^k:\partial D^k\to D^k$ is the inclusion.

Let~$H$ be a compact subgroup of the Lie group $G$.
For a pair of $H$-representations $V$ and $W$, a morphism of orthogonal $H$-spectra
\begin{equation}\label{def:lambda}
\lambda_{H,V,W} \colon F_{V \oplus W}S^V \ \to\  F_W  
\end{equation}
is defined as the adjoint of the $H$-map \eqref{eq:define_i_V W}
\[ i_{V,W}\ : \ S^V \ \to \ \bO(W,V\oplus W)\ = \ F_W(V\oplus W) \ ,\]
compare \cite[Sec.\,III.4]{mandell-may:equivariant_orthogonal}. 
So $\lambda_{H,V,W}$ represents taking $H$-fixed points 
of the adjoint structure map:
\[ (\tilde\sigma_{V,W})^H \ :\ X(W)^H\ \to\  \map^H(S^V,X(V\oplus W)) \ . \]
The morphism~$\lambda_{H,V,W}$ is a $\pi_*$-isomorphism of orthogonal $H$-spectra 
by \cite[III Lemma 4.5]{mandell-may:equivariant_orthogonal}.

We set
\[ \cK^G \ = \ \bigcup_{H,V,W} \cZ( G\ltimes_H \lambda_{H,V,W} ) \ ,\]
the set of all pushout products of sphere inclusions $\partial D^k\to D^k$
with the mapping cylinder inclusions of the morphisms $G\ltimes_H \lambda_{H,V,W}$.
Here the union is over a set of triples $(H,V,W)$ consisting of
a compact subgroup $H$ of~$G$ and two $H$-representations~$V$ and~$W$
from the set $\cV_H$ of representatives of isomorphism classes of $H$-representations.
We let
\[ J^G_{\st} \ = \ J^G_{\lv} \cup \cK^G \]
stand for the union of $J^G_{\lv}$ and $\cK^G$. 
The sets $I^G_{\lv}$ and $J^G_{\st}$ will serve as sets of generating cofibrations 
and acyclic cofibrations for the stable model structure on $\Sp_G$.

\begin{prop}\label{prop:J_st detects stable fibrations} 
A morphism of orthogonal $G$-spectra 
is a stable fibration if and only if it has the right lifting property
with respect to the set~$J^G_{\st}$.\index{stable fibration!of orthogonal $G$-spectra}
\end{prop}
\begin{proof}
The right lifting property with respect to~$J^G_{\lv}$ 
is equivalent to being a level~fibration.
By \cite[Prop.\,1.2.16]{schwede:global}, the right lifting property 
of a morphism $f:X\to Y$ 
with respect to $J^G_{\st}=J^G_{\lv}\cup\cK^G$ is then equivalent to the additional
requirement that the square of mapping spaces
\[ 
 \xymatrix@C=30mm{ 
\map^G(G\ltimes_H F_W,X)\ar[r]^-{\map^G(G\ltimes_H \lambda_{H,V,W},X)} 
\ar[d]_{\map^G(G\ltimes_H F_W,f)} &
\map^G(G\ltimes_H F_{V\oplus W}S^V,X) \ar[d]^{\map^G(G\ltimes_H F_{V\oplus W}S^V,f)} \\
\map^G(G\ltimes_H F_W,Y)\ar[r]_-{\map^G(G\ltimes_H \lambda_{H,V,W},Y)} & 
\map^G(G\ltimes_H F_{V\oplus W}S^V,Y) }       
 \]
is homotopy cartesian, where $\map^G(-,-)$ is the space of
morphisms of orthogonal $G$-spectra.
This proves the claim because the orthogonal $G$-spectrum~$G\ltimes_H F_W$ 
represents the functor~$X\mapsto X(W)^H$,
the orthogonal $G$-spectrum $G\ltimes_H (F_{V\oplus W}S^V)$
represents the functor~$X\mapsto \map_*^H(S^V,X(V\oplus W))$,
and the morphism $G\ltimes_H\lambda_{H,V,W}$ represents 
$H$-fixed points of the adjoint structure map
$\tilde\sigma_{V,W}:X(W)\to \map_*(S^V,X(V\oplus W))$.
\end{proof}

\begin{prop}\label{prop:J is acyclic cof} 
Every morphism in $J^G_{\st}$ is a $\pi_*$-isomorphism and a cofibration. 
\end{prop}
\begin{proof}
All morphisms in~$J_{\lv}^G$ are cofibrations.
They are also level equivalences,
and hence also $\pi_*$-isomorphisms 
by \cite[III Lemma 3.3]{mandell-may:equivariant_orthogonal}, 
applied to the underlying orthogonal $H$-spectra,
for all compact subgroups~$H$.

Since for any compact subgroup $H$ of $G$, the orthogonal $G$-spectra $G \ltimes_H F_{V\oplus W}S^V$ and $G \ltimes_H F_W$ are cofibrant, 
the morphisms in $\cK^G$ are cofibrations.
The morphism $\lambda_{H,V.W}$ is a $\pi_*$-isomorphism of orthogonal $H$-spectra
by \cite[III Lemma 4.5]{mandell-may:equivariant_orthogonal}. So the mapping cylinder inclusion $c(\lambda_{H,V,W})$
is then also a $\pi_*$-isomorphism,
because it differs from $\lambda_{H,V,W}$ only by a homotopy equivalence of
orthogonal $H$-spectra. The pushout product $c(G\ltimes_H\lambda_{H,V,W})\boxempty i^k_+$
is isomorphic to $G\ltimes_H (c(\lambda_{H,V,W})\boxempty i^k_+)$. 
By \cite[III Sec.\,4]{mandell-may:equivariant_orthogonal}, 
the morphism $c(\lambda_{H,V,W})\boxempty i^k_+$ 
is a $\pi_*$-isomorphism of orthogonal $H$-spectra. 
Now Corollary \ref{cor:Gamma2G} 
implies that $G\ltimes_H (c(\lambda_{H,V,W})\boxempty i^k_+)$ is a $\pi_*$-isomorphism. 
\end{proof}

Now we assemble the ingredients and construct the stable model structure. 
As we already mentioned, the following 
model structure is Quillen equivalent to the one established by Fausk
in~\cite[Prop.\,6.5]{fausk:prospectra}. 
We refrain from comparing the two model structures,
since that is not relevant for our purposes. 

\begin{thm}[Stable model structure]  \label{thm:stable} 
  Let $G$ be a Lie group.\index{stable model structure!for orthogonal $G$-spectra|(}
  \begin{enumerate}[\em (i)]
  \item 
    The $\pi_*$-isomorphisms, stable~fibrations and cofibrations 
    form a model structure on the category of orthogonal $G$-spectra, 
    the {\em stable model structure}.
  \item 
    Every cofibration of orthogonal $G$-spectra is an h-cofibration.\index{h-cofibration!of orthogonal $G$-spectra} 
  \item
    Let $i:A\to B$ be a cofibration of orthogonal $G$-spectra and
    $j:K\to L$ a $G$-cofibration of based $G$-spaces.
    Then the pushout-product morphism
    \[ i\boxempty j \ = \ (i\sm L)\cup(B\sm j)\ : \ 
    A\sm L\cup_{A\sm K}B\sm K \ \to \ B\sm L\]
    is a cofibration of orthogonal $G$-spectra. 
    If moreover $i$ is a $\pi_*$-isomorphism or $j$ is a $\Com$-equivalence,
    then $i \boxempty j$ is a $\pi_*$-isomorphism.
  \item 
    The stable model structure is proper, stable, topological and cofibrantly generated.
  \item
    The fibrant objects in the stable model structure 
    are the $G$-$\Omega$-spectra.\index{G-Omega-spectrum@$G$-$\Omega$-spectrum}
  \end{enumerate}
\end{thm}
\begin{proof}
(i) 
We start by establishing a level model structure for orthogonal $G$-spectra.
Orthogonal $G$-spectra are continuous based functors from $\bO$
to the category $G\bT_*$, so we can employ the machinery 
from \cite[Prop.\,C.23]{schwede:global}
that produces level model structures on enriched functor categories.
Here the base category is $\cV=G\bT_*$, with smash product as monoidal structure.
The index category is $\cD=\bO$, where the group $G$ acts trivially on
all morphism spaces. The category $\cD^*$ of enriched functors
from $\bO$ to $G\bT_*$ then becomes $\Sp_G$. 
The dimension function on $\bO$ is the vector space
dimension, and then the abstract skeleta of 
\cite[Con.\,C.13]{schwede:global} specialize to the skeleta above.

To apply \cite[Prop.\,C.23]{schwede:global} we need to specify 
a model structure on the category of $O(m)$-objects in $G\bT_*$,
i.e., on the category of based $(G\times O(m))$-spaces.
We use the $\cC(m)$-projective model structure,
in the sense of \cite[Prop.\,B.7]{schwede:global},
where $\cC(m)$ is the family of those closed subgroups of $G\times O(m)$
that are the graph of a continuous homomorphism to $O(m)$
defined on some compact subgroup of $G$.
Equivalently, a closed subgroup $\Delta$ of $G\times O(m)$ belongs to
$\cC(m)$ if and only if it is compact and $\Delta\cap (1\times O(m))$
consists only of the neutral element.
The restriction functor from $(G\times O(m+n))$-spaces to $(G\times O(m))$-spaces
takes $\cC(m+n)$-fibrations to  $\cC(m)$-fibrations,
so its left adjoint
\[ O(m+n)\ltimes_{O(m)}- \ : \ (G\times O(m))\bT_*\ \to \ (G\times O(m+n))\bT_*  \]
preserves acyclic cofibrations.
This establishes the consistency condition of \cite[Def.\,C.22]{schwede:global}. 
With respect to the $\cC(m)$-projective model structures on $(G\times O(m))$-spaces, 
the level equivalences, level fibrations
and cofibrations in the sense of \cite[Prop.\,C.23]{schwede:global} 
are precisely the level equivalences, level fibrations,
and cofibrations of orthogonal $G$-spectra.
So \cite[Prop.\,C.23]{schwede:global} shows that
the level equivalences, level fibrations and cofibrations 
form a model structure on the category $\Sp_G$, the {\em level model structure}.

The right lifting property against the set $I^G_{\lv}$ 
detects the level acyclic fibrations, simply by the adjunction
\[ \Sp_G( (G\ltimes_H F_V)\sm K, X) \ \iso \ \bT_*(K,X(V)^H)\ , \]
where $K$ is a non-equivariant based space.
So the set $I^G_{\lv}$ serves as a set of generating cofibrations. 

Before proceeding with the stable model structure, we prove part (ii).
If $X$ is any orthogonal $G$-spectrum, then evaluation at the point $0\in[0,1]$
is a level equivalence and level fibration $X^{[0,1]}\to X$, by direct inspection.
So every cofibration has the left lifting property with respect to
this evaluation morphism, which means that every cofibration is an h-cofibration.

Now we continue with the stable model structure.
The $\pi_*$-isomorphisms satisfy the 2-out-of-3 property (MC2)
and the classes of $\pi_*$-isomorphisms, stable fibrations and 
cofibrations are closed under retracts (MC3).
The level model structure shows that every morphism of orthogonal $G$-spectra
can be factored as a cofibration followed by a level equivalence that is
also a level fibration.
Level equivalences are in particular $\pi_*$-isomorphisms,
and for them the square \eqref{eq:stable fibration square}
is homotopy cartesian. 
So level acyclic fibrations are also stable fibrations.
Hence the level model structure provides one of the factorizations as required by MC5.

For the other half of the factorization axiom MC5
we exploit that the set $J^G_{\st}$ detects the stable fibrations,
compare Proposition \ref{prop:J_st detects stable fibrations}. 
We apply the small object argument 
(see for example~\cite[7.12]{dwyer-spalinski:model_categories} or~\cite[Thm.\,2.1.14]{hovey:model_categories})
to the set~$J_{\st}^G$.
All morphisms in $J^G_{\st}$ are cofibrations and $\pi_*$-isomorphisms 
by Proposition~\ref{prop:J is acyclic cof}. 
The small object argument provides a functorial factorization
of every morphism $\varphi:X\to Y$ of orthogonal $G$-spectra
as a composite
\[  X \ \xrightarrow{\ i\ }\ W \ \xrightarrow{\ q\ } \ Y \]
where $i$ is a sequential composition of cobase changes of coproducts
of morphisms in $J^G_{\st}$, 
and $q$ has the right lifting property with respect 
to $J^G_{\st}$; in particular, the morphism~$q$ is a stable fibration.
All morphisms in~$J^G_{\st}$ are $\pi_*$-isomorphisms and cofibrations,
hence also h-cofibrations.
The class of h-cofibrations that are simultaneously $\pi_*$-isomorphisms
is closed under coproducts, cobase changes and sequential compositions 
by~\cite[III Thm.\,3.5]{mandell-may:equivariant_orthogonal}.
So the morphism~$i$ is a cofibration and a $\pi_*$-isomorphism.

Now we show the lifting properties MC4. 
By Proposition~\ref{prop:stable zero is level zero}
a morphism that is both a stable fibration and a $\pi_*$-isomorphism
is a level equivalence, and hence an acyclic fibration
in the level model structure. So every morphism that is
simultaneously a stable fibration and a $\pi_*$-isomorphism has the
right lifting property with respect to cofibrations.
Now we let $j:A\to B$ be a cofibration that is also a $\pi_*$-isomorphism and 
we show that it has the left lifting property with respect to stable fibrations.
We factor~$j=q\circ i$, via the small object argument for $J^G_{\st}$,
where $i:A\to W$ is a $J^G_{\st}$-cell complex
and $q:W\to B$ is a stable fibration, 
see Proposition \ref{prop:J_st detects stable fibrations}.
Then $q$ is a $\pi_*$-isomorphism since $j$ and~$i$ are,
so $q$ is an acyclic fibration in the level model structure,
again by Proposition~\ref{prop:stable zero is level zero}.
Since $j$ is a cofibration, a lifting in
\[\xymatrix{
A \ar[r]^-i \ar[d]_j & W \ar[d]^q_(.6)\sim \\
B \ar@{=}[r] \ar@{..>}[ur] & B }\]
exists. Thus $j$ is a retract of the morphism~$i$ that has the left lifting property
with respect to stable fibrations.
But then $j$ itself has this lifting property.
This finishes the verification of the model category axioms
for the stable model structure.

(iii) 
  Pushouts of orthogonal spectra and smash products with based spaces
  are formed levelwise. So 
  \begin{align*}
    L_m(B\sm L)&\cup_{L_m(A\sm L\cup_{A\sm K}B\sm K)} ( A\sm L\cup_{A\sm K}B\sm K)(\mR^m)\\
    &= \  L_m(B)\sm L \cup_{L_m(A)\sm L\cup_{L_m(A)\sm K}L_m(B)\sm K} \\
    &  \hspace{4cm}  ( A(\mR^m)\sm L\cup_{A(\mR^m)\sm K}B(\mR^m)\sm K)\\
    &\iso \
    (L_m(B)\cup_{L_m(A)} A(\mR^m))\sm L \cup_{(L_m(B)\cup_{L_m(A)} A(\mR^m))\sm K} B(\mR^m)\sm K\ .
  \end{align*}
  Moreover, $(B\sm L)(\mR^m)= B(\mR^m)\sm L$.
  Under these identifications, the $m$-th latching map for the morphism $i\boxempty j$
  becomes the pushout product of $\nu_m i:L_m B\cup_{L_m A}A(\mR^m)\to B(\mR^m)$
  with the map $j:K\to L$.
  By hypothesis, $\nu_m(i)$ is a $\Com$-cofibration of $(G\times O(m))$-spaces.
  Since $j$ is a $G$-cofibration, it is also a $(G\times O(m))$-cofibration
  for the trivial $O(m)$-action.
  So 
  \[ \nu_m(i\boxempty j) \ = \ \nu_m(i)\boxempty j \]
  is a $\Com$-cofibration of $(G\times O(m))$-spaces
  by Proposition \ref{prop:Gmod} (iii).
  Also by hypothesis, the group $O(m)$ acts freely away from the image of $\nu_m(i)$;
  hence it also acts freely off the image of $\nu_m(i)\boxempty j$.
  This proves that $i\boxempty j$ is a cofibration of orthogonal $G$-spectra.

  Now we suppose in addition that $i$ is a $\pi_*$-isomorphism or $j$ is 
  a $\Com$-equivalence.
  Since the morphism $i\boxempty j$ is a cofibration by the above,
  it is in particular an h-cofibration. 
  So to show that $i\boxempty j$ is a $\pi_*$-isomorphism, 
  the long exact homotopy group sequence allows us to show that its cofiber 
  $(B/A)\sm (K/L)$ is $\pi_*$-isomorphic to the trivial spectrum.
  Now we let $H$ be a compact subgroup of $G$.
  In the case where $i$ is a $\pi_*$-isomorphism, its long exact homotopy group sequence 
  shows that $B/A$ is $H$-$\pi_*$-trivial; since $K/L$ is a cofibrant based $H$-space,
  the smash product $(B/A)\sm(K/L)$ is $H$-$\pi_*$-trivial,
  by \cite[III Thm.\,3.11]{mandell-may:equivariant_orthogonal} 
  or \cite[Prop.\,3.2.19]{schwede:global}.
  If $j$ is a $\Com$-equivalence, then the underlying based $H$-space of $K/L$ 
  is equivariantly contractible.
  So $(B/A)\sm (K/L)$ is $H$-equivariantly contractible, and hence $\pi_*$-trivial. 

  (iv) Alongside with the proof of the model structure we have also specified sets 
  of generating cofibrations $I^G_{\lv}$
  and generating acyclic cofibrations $J^G_{\st}$.
  Sources and targets of all morphisms in these sets are small with
  respect to sequential colimits of cofibrations. So the 
  model structure is cofibrantly generated.

  Left properness of the stable model structure follows from
  the fact that every cofibration is in particular an h-cofibration
  of orthogonal $G$-spectra. So for every compact subgroup~$H$, the underlying morphism
  of orthogonal $H$-spectra is an h-cofibration,
  and pushout along it preserves $\pi_*$-isomorphisms of orthogonal $H$-spectra,
  by \cite[III Thm.\,3.5 (iii)]{mandell-may:equivariant_orthogonal}
  or \cite[Cor.\,3.1.39]{schwede:global}.
  Right properness follows from the fact that stable fibrations
  are in particular level fibrations, and hence the natural morphism from
  the strict fiber to the homotopy fiber is a level equivalence.
  Source, target and homotopy fiber of any morphism of orthogonal $G$-spectra
  are related by a long exact sequences of equivariant homotopy groups 
  (see for example \cite[Prop.\,3.1.36]{schwede:global}),
  so the five lemma concludes the argument.

The loop functor $\Omega:\Sp_G\to\Sp_G$ 
and the suspension functor $-\sm S^1:\Sp_G\to\Sp_G$ 
preserve $\pi_*$-isomorphisms by the loop isomorphism \eqref{eq:loop iso} 
and the suspension isomorphism \cite[Prop.\,3.1.30]{schwede:global}.
Moreover, the loop functor preserves stable fibrations
by direct inspection. So the adjoint functor pair
\[ \xymatrix@C=15mm{ -\sm S^1 \ : \ \Sp_G \ \ar@<0.5ex>[r] & 
\ \Sp_G \ar@<0.5ex>[l] \ : \ \Omega } \]
is a Quillen adjunction with respect to the stable model structure,
and these functors model the model categorical suspension and loop functors. 
Furthermore, the unit $\eta:X \to \Omega (X\sm S^1)$ 
and counit $( \Omega X)\sm S^1 \to X$ 
of the adjunction are $\pi_*$-isomorphisms \cite[Prop.\,3.1.25]{schwede:global}.
Hence the adjunction is a Quillen equivalence, which proves stability of the stable
model structure.

Every cofibration of non-equivariant spaces is in particular
a $G$-cofibration when given the trivial $G$-action.
So the stable model structure is topological as a special case of part (iii).
Part (v) is clear from the definitions.
\end{proof}

\begin{rk}[Relation to previous stable model structures] 
For compact Lie groups,
the proper equivariant stable homotopy theory reduces to
the `genuine' equivariant stable homotopy theory.
In this special case, several stable model structures 
have already been constructed that complement the
$\pi_*$-isomorphisms by different classes of cofibrations and fibrations.
We explain how our stable model structure relates
to the previous ones for compact Lie groups.

For compact Lie groups $H$, model structures on orthogonal $H$-spectra
with $\pi_*$-isomorphisms as weak equivalences were established by
Mandell and May~\cite{mandell-may:equivariant_orthogonal}, Stolz~\cite{stolz:thesis},
Brun, Dundas and Stolz \cite{brun-dundas-stolz:equivariant_smash},
and Hill, Hopkins and Ravenel~\cite{hill-hopkins-ravenel:kervaire}.

The cofibrations in these model structures
each admit a characterization in terms of the latching maps, 
with different conditions on the allowed isotropy away from the image;
however, these characterizations are not explicitly stated in the other papers.
A morphism of orthogonal $H$-spectra $i:A\to B$ is an `$\mS$-cofibration'
in the sense of Stolz \cite[Def.\,2.3.4]{stolz:thesis}
and Brun-Dundas-Stolz \cite[Def.\,2.9.11]{brun-dundas-stolz:equivariant_smash}
precisely when the latching morphism
\[ \nu_m i\ = \ \nu_m^B \cup i(\mR^m) \ :\ L_m B \cup_{L_m A} A(\mR^m)  \ \to\  B(\mR^m) \]
is a cofibration of $(H\times O(m))$-spaces, with no additional constraint
on the isotropy.
The morphism $i$ is a `q-cofibration'
in the sense of Mandell and May \cite[III Def.\,2.3]{mandell-may:equivariant_orthogonal}
precisely when the latching morphism $\nu_m i$ 
is a cofibration of $(H\times O(m))$-spaces 
and additionally the isotropy group of every point that is not in the image
of $\nu_m i$ is the graph of a continuous homomorphism 
$K\to O(m)$, for some closed subgroup $K$ of $H$, 
that admits an extension to a continuous homomorphism defined on $H$.
In particular, Stolz' $\mS$-model structure has more cofibrations than
our model structure, and we have more cofibrations than Mandell and May.
For finite groups, our cofibrations of orthogonal $H$-spectra
specialize to the `complete cofibrations' in the sense of Hill, Hopkins
and Ravenel, i.e., to the variant of the {\em positive complete cofibrations}
of \cite[B.63]{hill-hopkins-ravenel:kervaire} where the positivity condition 
is dropped.
So for finite groups~$H$, the stable model structure 
specified in Theorem \ref{thm:stable} is `essentially'
the positive complete model structure of
\cite[B.4.1]{hill-hopkins-ravenel:kervaire}.
\end{rk}

\index{stable model structure!for orthogonal $G$-spectra|)}

\begin{eg}[No compact subgroups]
We already emphasized that when $G$ is compact, 
our theory just returns the well-known $G$-equivariant stable homotopy
theory, based on a complete $G$-universe.
There is another extreme where we also recover a well-known homotopy theory.
Indeed, suppose that the only compact subgroup of $G$ 
is the trivial subgroup. For example, $G$ could be discrete and torsion
free, or the additive group of $\mR^n$.
The category of orthogonal $G$-spectra is isomorphic to
the category of module spectra over the spherical group ring~$\mS[G]$ 
-- this is a pointset level statement and holds for all Lie groups~$G$.
But if the trivial group is the only compact subgroup of~$G$,
then a morphism of orthogonal $G$-spectra is a $\pi_*$-isomorphism 
or stable fibration
if and only if the underlying morphism of non-equivariant orthogonal spectra
is a $\pi_*$-isomorphism or stable fibration, respectively.
So not only is the category $\Sp_G$ isomorphic to module spectra over~$\mS[G]$,
also the model structure is the one on modules over 
an orthogonal ring spectrum, lifted along the forgetful functor,
compare~\cite[Thm.\,12.1 (i)]{mandell-may-schwede-shipley:diagram_spectra}. 
So in particular,
\[ \Ho(\Sp_G)\ \cong \ \Ho(\mS[G]\text{-mod}) \]
i.e., the stable $G$-equivariant homotopy category `is'
the homotopy category of module spectra over~$\mS[G]$.
\end{eg}

The category $\Sp$ of orthogonal spectra supports a symmetric monoidal {\em smash product},
which is an example of a convolution product considered
by category theorist Day in \cite{day:closed};
like the tensor product of abelian groups, 
it can be introduced via a universal property or as a specific construction.
The indexing category $\bO$ for orthogonal spectra
was introduced in Construction \ref{con:O}.
A based continuous functor
\[ \oplus \ : \  \bO\sm\bO \ \to \ \bO\]
is defined on objects by orthogonal direct sum, and on morphism spaces by
\begin{align*}
   \bO(V,W)\sm \bO(V',W')\ &\to \ \bO(V\oplus V', W\oplus W')\\
(w,\varphi)\sm (w',\varphi')\quad &\longmapsto \quad ((w,w'),\varphi\oplus\varphi')\ .
\end{align*}
A {\em bimorphism} $b:(X,Y)\to Z$\index{bimorphism!of orthogonal spectra}
from a pair of orthogonal spectra $(X,Y)$
to an orthogonal spectrum $Z$ is a natural transformation 
\[ b \ : \ X\bar\sm Y \ \to \ Z\circ\oplus \]
of continuous functors $\bO\sm\bO\to\bT_*$;
here $X\bar\sm Y$ is the `external smash product'  defined by
$(X\bar\sm Y)(V,W)=X(V)\sm Y(W)$. A bimorphism thus consists of
based continuous maps 
\[ b_{V,W} \ : \ X(V) \sm  Y(W) \ \to \ Z(V\oplus W) \]
for all inner product spaces $V$ and $W$ that form morphisms of orthogonal
spectra in each variable separately.
A smash product of two orthogonal spectra is now a universal example
of a bimorphism from $(X,Y)$.

\begin{defn}\label{def:smash product}
A {\em smash product}\index{smash product!of orthogonal spectra}
of two orthogonal spectra $X$ and $Y$ is a pair $(X\sm Y,i)$ 
consisting of an orthogonal spectrum $X\sm Y$
and a universal bimorphism $i:(X,Y)\to X\sm Y$,
i.e., a bimorphism 
such that for every orthogonal spectrum $Z$ the map
\[ 
 \Sp(X\sm Y,Z) \ \to \ \text{Bimor}((X,Y),Z) \ , \
f\mapsto f i = \{f(V\oplus W)\circ i_{V,W}\}_{V,W} 
 \]
is bijective. 
\end{defn}

A smash product of two orthogonal spectra can be constructed as an enriched Kan extension
of the external smash product $X\bar\sm Y:\bO\sm\bO\to\bT_*$
along the continuous functor $\oplus : \bO\sm\bO \to \bO$.
This boils down to presenting $(X\sm Y)(\mR^n)$
as a quotient space of the wedge, over $0\leq k\leq n$, of the
$O(n)$-spaces 
\[ O(n)\ltimes_{O(k)\times O(n-k)}X(\mR^k)\sm Y(\mR^{n-k})\ .\]
However, we feel that this explicit construction does not give much insight beyond
showing the existence of an object with the desired universal property.
Anyhow, Day's general theory \cite{day:closed} shows
that the smash product $X \wedge Y$ supports preferred natural associativity
isomorphisms $(X\sm Y)\sm Z\iso X\sm(Y\sm Z)$,
symmetry isomorphisms $X\sm Y\iso Y\sm X$ and unit isomorphisms
$X\sm \mS\iso X\iso\mS\sm X$, see also \cite[Con.\,C.9]{schwede:global}.
Moreover, there exists an adjoint internal function orthogonal spectrum.
All this data makes the smash product into a closed symmetric monoidal structure
on the category of orthogonal spectra.

If $X$ and $Y$ are orthogonal $G$-spectra, then the smash product $X \wedge Y$ 
inherits the diagonal $G$-action, and $G$ acts on the internal function spectrum
by conjugation. 
So the category $\Sp_G$ forms a closed symmetric monoidal category
under smash product.
We will now show that the category $\Sp_G$ of orthogonal $G$-spectra,
equipped with the stable model structure and the smash product, 
is a monoidal model category in the sense 
of \cite[Def. 4.2.6]{hovey:model_categories}.
We also show that the stable model structure 
satisfies the monoid axiom \cite[Def. 3.3]{schwede-shipley:algebres_modules}. 
This allows us to automatically lift the stable model structure 
to the categories of module $G$-spectra and ring $G$-spectra.

\begin{defn}
  Let $G$ be a Lie group.
  An orthogonal $G$-spectrum $X$ is {\em quasi-flat}\index{quasi-flat!orthogonal $G$-spectrum}
  if for every compact subgroup $H$ of $G$ and every $m\geq 0$,
  the latching map $\nu_m:L_m X \to X(\mR^m)$
  is an $(H\times O(m))$-cofibration.
\end{defn}

An orthogonal $G$-spectrum $X$ is quasi-flat precisely if
for every compact subgroup $H$ of $G$,
the underlying orthogonal $H$-spectrum is $H$-flat in the sense of
\cite[Def.\,3.5.7]{schwede:global}.
When $G$ is compact, `quasi-flat' is the same as `$G$-flat';
in this case, the $G$-flat orthogonal spectra are the cofibrant objects 
in the $\mS$-model structure of Stolz \cite[Thm.\,2.3.27]{stolz:thesis}
and Brun-Dundas-Stolz  \cite[Def.\,2.9.11]{brun-dundas-stolz:equivariant_smash}.
Every quasi-cofibrant orthogonal $G$-spectrum 
in the sense of Definition \ref{def:quasi-cofibrant} below is in particular quasi-flat.

\begin{thm} \label{thm:smashing with cofibrants} 
Let $G$ be a Lie group. 
For every quasi-flat orthogonal $G$-spectrum $X$, the functor $-\sm X$
preserves $\pi_*$-isomorphisms of orthogonal $G$-spectra. 
\end{thm}
\begin{proof}
We let $H$ be any compact subgroup of $G$. 
Then the underlying orthogonal $H$-spectrum of $X$ is $H$-flat
in the sense of \cite[Def.\,3.5.7]{schwede:global}.
Now we let $f \colon A \to B$ be a $\pi_*$-isomorphism of orthogonal $G$-spectra. 
Since the $G$-action on a smash product is defined diagonally, 
we have $\res^G_H(A\sm X)=\res^G_H(A)\sm\res^G_H(X)$, and similarly for $B\sm X$.
Since $X$ is $H$-flat, the morphism $\res^G_H(f\wedge X)$ is a $\pi_*$-isomorphism
of orthogonal $H$-spectra by \cite[Thm.\,3.5.10]{schwede:global}.
Since $H$ was an arbitrary compact subgroup of $G$, this proves the claim.
\end{proof}

\begin{prop}\label{prop:smash is monoidal model}
  Let $G$ be a Lie group.\index{pushout product property!for equivariant smash product}
  \begin{enumerate}[\em (i)]
  \item 
    Let $i:A\to B$ and $j:K\to L$ be cofibrations of orthogonal $G$-spectra.
    Then the pushout-product morphism
    \[ i\boxempty j \ = \ (i\sm L)\cup(B\sm j)\ : \ 
    A\sm L\cup_{A\sm K}B\sm K \ \to \ B\sm L\]
    is a cofibration of orthogonal $G$-spectra. 
    If, in addition, $i$ or $j$ is a $\pi_*$-isomorphism, then so is $i \boxempty j$.
  \item   The  category $\Sp_G$ equipped 
    with the stable model structure is a monoidal model category 
    under the smash product of orthogonal $G$-spectra.
\end{enumerate}
\end{prop}
\begin{proof} 
  (i)
  We start with the claim that only involves cofibrations.
  It suffices to check the statement for the generating cofibrations.
  The pushout product, in the category of spaces, of two sphere inclusions
  is homeomorphic to another sphere inclusion.
  So it suffices to show that for all compact subgroups $H$ and $K$ of $G$,
  all $H$-representations $V$ and all $K$-representations $W$, the $G$-spectrum
  \begin{equation}\label{eq:smash gens} 
    (G \ltimes_H F_V ) \wedge (G \ltimes_K F_W)\ \iso \ 
    \Delta^*( (G\times G)\ltimes_{H\times K} F_{V\oplus W})
  \end{equation}
  is cofibrant, where $\Delta:G\to G\times G$ is the diagonal embedding.
  Since $(G\times G)\ltimes_{H\times K} F_{V\oplus W}$ is a cofibrant $(G\times G)$-spectrum,
  Theorem \ref{thm:adjunctions spectra pointset} (ii) below 
  shows that the orthogonal $G$-spectrum \eqref{eq:smash gens} is cofibrant.

  Now we suppose that in addition the morphism $i$ is a $\pi_*$-isomorphism, 
  the other case being analogous.
  Since $i$ is a cofibration, the long exact homotopy group sequence 
  (see \cite[Cor.\,3.1.38]{schwede:global}) 
  shows that its cofiber $B/A$ is $\pi_*$-trivial. 
  Since $j$ is a cofibration, its cofiber $L/K$ is cofibrant, hence $G$-flat,
  so the smash product $(B/A)\sm (K/L)$ is $\pi_*$-trivial 
  by Theorem \ref{thm:smashing with cofibrants}.
  Since the morphism $i\boxempty j$ is a $G$-cofibration with cofiber isomorphic
  to $(B/A)\sm (K/L)$, its long exact homotopy group sequence 
  shows that $i\boxempty j$ is a $\pi_*$-isomorphism.

  (ii) The pushout product property is established in part (i).\index{universal proper $G$-space}
  The suspension spectrum $\Sigma^\infty_+\uEG$ is a cofibrant replacement
  of the $G$-sphere spectrum $\mS_G$, the monoidal unit object.
  For every compact subgroup $H$ of $G$ the underlying $H$-space of $\uEG$ 
  is $H$-equivariantly contractible.
  So for every orthogonal $G$-spectrum $X$, the projection
  \[ X\sm \uEG_+ \ \iso \ X\sm \Sigma^\infty_+\uEG \ \to \ X \]
  is a homotopy equivalence of underlying orthogonal $H$-spectra, 
  and thus induces an isomorphism on $\pi_*^H$. 
  Since $H$ was any compact subgroup, the projection is a $\pi_*$-isomorphism.
  This establishes the unit axiom of \cite[Def.\,4.2.6]{hovey:model_categories}. 
\end{proof}

\begin{prop}[Monoid axiom]\label{prop:monoid axiom}\index{monoid axiom!for equivariant smash product}
  Let $G$ be a Lie group and $i \colon A \to B$ 
  a cofibration of orthogonal $G$-spectra which is also a $\pi_*$-isomorphism.
  \begin{enumerate}[\em (i)]
  \item For every orthogonal $G$-spectrum $Y$, 
    the morphism $i \wedge Y \colon  A \wedge Y  \to B \wedge Y$ 
    is an h-cofibration and a $\pi_*$-isomorphism.
  \item Let $\cD$ denote the class of maps of the form $i \wedge Y$, 
    where $i$ is a stably acyclic cofibration and $Y$ any orthogonal $G$-spectrum. 
    Then any map in the class $\cD$-cell (maps obtained as transfinite compositions 
    of cobase changes of small coproducts of morphisms in $\cD$) 
    is a $\pi_*$-isomorphism.
  \end{enumerate}
\end{prop}
\begin{proof}
  The class of h-cofibrations which are also $\pi_*$-isomorphisms is closed
  under transfinite compositions, coproducts and cobase changes
  by \cite[III Thm.\,3.5]{mandell-may:equivariant_orthogonal}.
  Hence part (ii) is a consequence of part (i). 

  The proof of part (i) is very similar to the proof of the corresponding statement 
  in the non-equivariant case, 
  compare \cite[Prop.\,12.5]{mandell-may-schwede-shipley:diagram_spectra}. 
  For the sake of completeness we provide the details here. 
  Since $i \colon A \to B$ is a cofibration, 
  the cofiber $B/A$ is cofibrant. Let $\alpha \colon Y^c \to Y$ 
  be a cofibrant approximation of $Y$. Then $(B/A) \wedge \alpha$ 
  is a $\pi_*$-isomorphism by Theorem \ref{thm:smashing with cofibrants}. 
  Furthermore the cofiber $B/A$ is $\pi_*$-isomorphic to the trivial $G$-spectrum,
  by the long exact sequence of homotopy groups, 
  see \cite[III Thm.\,3.5 (vi)]{mandell-may:equivariant_orthogonal}. 
  Using again Theorem \ref{thm:smashing with cofibrants}, 
  we see that $(B/A) \wedge Y^c$ and hence $(B/A) \wedge Y$ are $\pi_*$-isomorphic 
  to the trivial $G$-spectrum. 
  Now the morphism $i \wedge Y \colon  A \wedge Y  \to B \wedge Y$ 
  is an h-cofibration and its cofiber is isomorphic to $B/A \wedge Y$. 
  Since $(B/A) \wedge Y$ is $\pi_*$-isomorphic to the trivial $G$-spectrum,
  the long exact homotopy group sequence \cite[Cor.\,3.1.28]{schwede:global} 
  shows that the map $i \wedge Y$ is a $\pi_*$-isomorphism.
\end{proof}

The previous proposition almost immediately implies 
that the stable model structure on $\Sp_G$ 
lifts to the category of orthogonal ring $G$-spectra and the category 
of module spectra over an orthogonal ring $G$-spectrum $R$,
by the results of the fifth author and Shipley
\cite[Thm.\,4.1]{schwede-shipley:algebres_modules}.
We will not go into further details here.

\medskip

For compact Lie groups, the spheres of linear representations
become invertible objects in the genuine equivariant stable homotopy category.
In our more general context, the role of linear representations
is taken up by equivariant vector bundles over $\uEG$,
the universal $G$-space for proper actions.
We recall that a {\em $G$-vector bundle}\index{G-vector bundle@$G$-vector bundle}\index{equivariant vector bundle|see{$G$-vector bundle}}  
is a map of $G$-spaces $\xi \colon E \to X$ equipped with the structure of a real vector bundle, 
and such that the map $g\cdot-:\xi_x\to \xi_{g x}$ is $\mR$-linear 
for all $(g,x)\in G\times X$.
By one-point compactifying the fibers we obtain a $G$-equivariant 
fiber bundle $S^\xi \to X$ with fibers the spheres of dimension equal 
to the dimension of $\xi$. 
This bundle has two preferred $G$-equivariant sections
\[ s_0 ,  s_{\infty} \ :\ X \ \to\ S^\xi \]
which send a point in $X$ to the zero element and the point at infinity,
respectively, in the corresponding fiber.

\begin{prop}\label{prop:s_infty is cofibration} 
  Let $G$ be a Lie group and $(X,A)$ a relative proper $G$-CW-pair.
  Then for every $G$-vector bundle~$\xi$ over $X$,
  the $G$-map
  \[ s_0\cup s_{\infty}\cup \incl\ : \ 
    X\times\{0,\infty\}\cup_{A\times\{0,\infty\}} S^{\xi|_A}\ \to \ S^\xi \]
  is a $\Com$-cofibration of~$G$-spaces.
\end{prop}
\begin{proof}
  The $G$-vector bundle~$\xi$ admits a $G$-invariant euclidean metric by
  the real analog of
  \cite[Lemma 1.4]{luck-oliver:completion}.
  We choose a relative $G$-CW-structure on $(X,A)$ with skeleta~$X^n$ 
  and such that $X^{-1}=A$.
  We let~$\xi^n:E^n\to X^n$ denote the restriction 
  of the given euclidean vector bundle to~$X^n$.
  In a first step we show that the inclusion
  $S(E^{n-1})\to S(E^n)$ of the total spaces of the sphere bundles 
  is a $\Com$-cofibration of $G$-spaces.
  Lemma 1.1 (iii) of~\cite{luck-oliver:completion} provides a pushout square of $G$-spaces:
  \[ \xymatrix{ 
      \coprod_{j\in J} ( G\times_{H_j} S(V_j))\times \partial D^n \ar[r] \ar[d] & S(E^{n-1}) \ar[d] \\
      \coprod_{j\in J} (G\times_{H_j} S(V_j))\times D^n) \ar[r]  & S(E^n) } \]
  Here $J$ is an indexing set of the equivariant $n$-cells of $X$,
  $H_j$ is the stabilizer group of the cell indexed by~$j$, 
  $V_j$ is an orthogonal representation of~$H_j$, and $S(V_j)$
  is its unit sphere. In particular, each of the groups~$H_j$
  is compact by our hypotheses.
  The pushout arises from choices of characteristic maps for the equivariant $n$-cells
  of $X$ and choices of trivializations of~$\xi$ over each equivariant cell.
  Since $H_j$ is compact, the unit sphere $S(V_j)$ admits an $H_j$-CW-structure
  by Illman's theorem~\cite[Thm.\,7.1]{illman:triangulation_Lie},
  so it is $H$-cofibrant. Hence  $G\times_{H_j} S(V_j)$ is a $\Com$-cofibrant $G$-space, 
  and so the left and right vertical maps
  in the pushout square are $\Com$-cofibrations of $G$-spaces.
  Since $\Com$-cofibrations are closed under sequential colimits,
  this proves that the inclusion
  $S(\xi|_A)=S(E^{-1})\to \colim_n S(E^n)=S(E)$ is a $\Com$-cofibration.
  
  The fiberwise one-point compactification participates in a pushout square
  of $G$-spaces:
  \[ \xymatrix@C=15mm{ 
      S(E)\times \{0,\infty\}\cup_{S(\xi|_A)\times \{0,\infty\}}
      S(\xi|_A)\times [0,\infty] \ar[d] \ar[r] &
      X\times\{0,\infty\}\cup_{A\times\{0,\infty\}} S^{\xi|_A}\ar[d]^{s_0\cup s_\infty\cup\incl} \\
      S(E)\times [0,\infty]\ar[r]& S^\xi }\]
  Here the lower horizontal map crushes $S(E)\times\{0\}$ and
  $S(E)\times\{\infty\}$ to the sections at 0 and $\infty$, respectively.
  Since the inclusion $S(\xi|_A)\to S(E)$ is a $\Com$-cofibration, 
  so is the left vertical map, and hence also the right vertical map.
\end{proof}

We let $\xi:E\to X$ be a $G$-vector bundle over a $G$-space $X$.
By dividing out the image of the section at infinity $s_{\infty}:X\to S^\xi$, 
we get a based $G$-space
\begin{equation}\label{eq:define_Thom_space}
 \Th(\xi) \ = \ S^\xi / s_\infty(X)\ ,   
\end{equation}
the Thom space of $\xi$.\index{Thom space}

\begin{prop}\label{prop:Thom is invertible} 
  Let $G$ be a Lie group and~$\xi$
  a $G$-vector bundle over $\uEG$. 
  \begin{enumerate}[\em (i)]
  \item For every compact subgroup~$H$ of~$G$ and every 
    $H$-fixed point~$x\in(\uEG)^H$, the composite map
    \[ S^{\xi_x}\ \xrightarrow{\incl} \ S^{\xi}\ \xrightarrow{\proj} \ \Th(\xi) \]
    is a based $H$-equivariant homotopy equivalence.
  \item The endofunctors~$-\sm \Th(\xi)$ and~$\map_*(\Th(\xi),-)$
    of the category of orthogonal $G$-spectra 
    preserve and detect $\pi_*$-isomorphisms.
  \item For every orthogonal $G$-spectrum~$X$ the adjunction unit
    \[ \eta_X \ : \ X \ \to \ \map_*(\Th(\xi),X\sm \Th(\xi)) \]
    is a $\pi_*$-isomorphism. 
  \item The adjoint functor pair
    \[\xymatrix@C=12mm{  
   -\sm \Th(\xi)\ : \ \Sp_G\ \ar@<0.5ex>[r]
   & \ \Sp_G \ : \    \map_*(\Th(\xi), -) \ar@<0.5ex>[l]} \]
    is a Quillen equivalence.
    Consequently, the suspension spectrum of the 
    Thom space~$\Th(\xi)$ is an invertible object in~$\Ho(\Sp_G)$.  
  \end{enumerate}
\end{prop}
\begin{proof}
  (i) The restriction of the $G$-space $\uEG$
  is $H$-equivariantly contractible.
  Therefore, by the homotopy invariance theorem 
  \cite[Thm.\,1.2]{luck-oliver:completion}, 
  the underlying $H$-vector bundle of $\xi$ is 
  $H$-equivariantly isomorphic to the trivial bundle 
  $\xi_x \times \uEG \to \uEG$. 
  This implies that the underlying $H$-space of the Thom space $\Th(\xi)$ 
  is $H$-equivariantly isomorphic to $S^{\xi_x} \wedge \uEG_+$.
  Since $\uEG$ is $H$-equivariantly contractible, the claim follows.
  
  (ii) We let~$H$ be any compact subgroup of~$G$.
  We choose an $H$-fixed point~$x\in (\uEG)^H$.
  Part~(i) shows that the fiber inclusion induces an $H$-equivariant 
  based homotopy equivalence 
  \[\varphi\colon S^{\xi_x} \ \to \   \res^G_H(\Th(\xi)) \ .\]
  This map induces homotopy equivalences of orthogonal $H$-spectra
  \[ \varphi_* \ : \ \res^G_H(X)\sm S^{\xi_x} \ \to \ \res^G_H( X\sm \Th(\xi)) \ \]
  and
  \[ \varphi_* \ : \ \res_H^G(\map_*(\Th(\xi),X))\to \Omega^{\xi_x}(\res_H^G(X))\ . \]
  Since $H$ is compact, the fiber~$\xi_x$ can be endowed with an $H$-invariant
  inner product, making it an orthogonal $H$-representation.
  Now the representation sphere~$S^{\xi_x}$ admits the structure of a finite based
  $H$-CW-complex. So the functor~$-\sm S^{\xi_x}$ preserves $\pi_*$-isomorphisms 
  by~\cite[III Thm.\,3.11]{mandell-may:equivariant_orthogonal} 
  or~\cite[Prop.\,3.2.19 (ii)]{schwede:global},
  and the functor~$\Omega^{\xi_x}$ preserves $\pi_*$-isomorphisms 
  by~\cite[III Prop.\,3.9]{mandell-may:equivariant_orthogonal} 
  or~\cite[Prop.\,3.1.40 (ii)]{schwede:global}.
  
  Furthermore, if~$f:X\to Y$ is a morphism of orthogonal $H$-spectra such that 
  $f\sm S^{\xi_x}:X\sm S^{\xi_x}\to Y\sm S^{\xi_x}$ is a $\pi_*$-isomorphism, 
  then $\Omega^{\xi_x}(f\sm S^{\xi_x})$ is a $\pi_*$-isomorphism by the previous paragraph.
  Since the adjunction unit~$\eta^{\xi_x}_X:X\to\Omega^{\xi_x}(X\sm S^{\xi_x})$ 
  is a $\pi_*$-isomorphism 
  (by~\cite[III Lemma 3.8]{mandell-may:equivariant_orthogonal} 
  or~\cite[Prop.\,3.1.25 (ii)]{schwede:global}),
  the original morphism~$f$ is a $\pi_*$-isomorphism. 
  So smashing with $S^{\xi_x}$ detects $\pi_*$-isomorphisms.  
  By the same argument, using that the adjunction counit
  $\epsilon^{\xi_x}_X:(\Omega^{\xi_x}X)\wedge S^{\xi_x}\to X$ is a $\pi_*$-isomorphism 
  (see \cite[Prop.\,3.1.25 (ii)]{schwede:global}), 
  it follows that $\Omega^{\xi_x}$ detects $\pi_*$-isomorphisms of $H$-orthogonal spectra.
  
  Since $H$ was any compact subgroup of $G$,
  this shows that smashing with $\Th(\xi)$ and taking $\map_*(\Th(\xi),-)$ 
  detect and preserve $\pi_*$-isomorphisms  of orthogonal $G$-spectra.

  (iii)  Again we let~$H$ be any compact subgroup of~$G$, 
  and we choose an $H$-fixed point~$x\in (\uEG)^H$.
  The fiber inclusion
  $\varphi\colon S^{\xi_x}\to \res^G_H(\Th(\xi))$
  is a based $H$-equivariant homotopy equivalence by part~(i). 
  The following square of orthogonal $H$-spectra commutes:
  \[ \xymatrix@C=20mm{ 
      X \ar[r]\ar[d] & \map_*(\Th(\xi),X\sm \Th(\xi))\ar[d]^{\map_*(\varphi,\Id)}\\
      \map_*(S^{\xi_x}, X\sm S^{\xi_x})\ar[r]_-{\map_*(\Id,X\sm \varphi)} &
      \map_*(S^{\xi_x},X\sm \Th(\xi)) } \]
  The two morphisms starting at~$X$ are the adjunction units.  
  The left vertical morphism is a $\pi_*$-isomorphism
  by~\cite[III Lemma 3.8]{mandell-may:equivariant_orthogonal}
  or~\cite[Prop.\,3.1.25 (ii)]{schwede:global};
  the right vertical and lower horizontal morphisms are homotopy equivalences,
  hence $\pi_*$-isomorphisms. So the upper horizontal morphism is also a
  $\pi_*$-isomorphism.

  (iv) The section at infinity is a $\Com$-cofibration of $G$-spaces 
  by Proposition \ref{prop:s_infty is cofibration} for $A=\emptyset$;  
  so the Thom space~$\Th(\xi)$ is $\Com$-cofibrant as a based $G$-space.
  The adjoint functors thus form a Quillen pair 
  by Theorem \ref{thm:stable} (iii).
  Since the right adjoint $\map_*(\Th(\xi),-)$ 
  preserves and detects all $\pi_*$-isomorphisms
  and the adjunction unit is a $\pi_*$-isomorphism, 
  the Quillen pair is a Quillen equivalence.
\end{proof}

\section{The \texorpdfstring{$G$}{G}-equivariant stable homotopy category}

\index{G-equivariant stable homotopy category@$G$-equivariant stable homotopy category|(} 

In Section \ref{sec:stable} we showed that
the \emph{$G$-equivariant stable homotopy category} $\Ho(\Sp_G)$
is the homotopy category of a stable model structure,
so it is naturally a triangulated category,
for example by~\cite[Sec.\,7.1]{hovey:model_categories} 
or~\cite[Thm.\,A.12]{schwede:order_topological}.
This section discusses those aspects of our theory that are most
conveniently phrased in terms of the triangulated structure.

We record in Proposition \ref{prop:rephtpy} 
that for every compact subgroup $H$ of a Lie group $G$,
the suspension spectrum of $G/H$ represents the functor $\pi_0^H$ on $\Ho(\Sp_G)$;
a direct consequence is the fact that
for varying compact subgroups $H$, the suspension spectra of the orbits $G/H$
form a set of small weak generators for the stable $G$-homotopy category,
see  Corollary \ref{cor:Ho_G compactly generated}.
In other words, the triangulated category $\Ho(\Sp_G)$ is
{\em compactly generated}.
As we explain thereafter, this has various formal, but rather useful, consequences,
such as Brown representability (see Corollary \ref{cor:generators for GH_F}),
a t-structure (see Corollary \ref{cor:t-structure}),
and Postnikov sections (see Remark \ref{rk:postnikov}).
In Section \ref{sec:Mackey}  we will return to this
preferred t-structure in the special case
of {\em discrete} groups $G$; we will then show that its
heart is equivalent to the abelian category of $G$-Mackey functors.
In particular, for discrete groups $G$, every $G$-Mackey functor
has an associated Eilenberg-Mac\,Lane spectrum.\medskip

In the following we write
\begin{equation}\label{eq:gamma_G}
 \gamma_G \ : \ \Sp_G \ \to \ \Ho(\Sp_G)   
\end{equation}
for the localization functor at the class of $\pi_*$-isomorphisms; 
so $\gamma_G$ initial among functors from $\Sp_G$ 
that send $\pi_*$-isomorphisms to isomorphisms.

\begin{con}[Triangulated structure on $\Ho(\Sp_G)$]\label{con:define triangles}
The suspension isomorphism~\eqref{eq:suspension iso} 
between $\pi_k^G(X)$ and~$\pi_{k+1}^G(X\sm S^1)$
shows that the pointset level suspension of orthogonal $G$-spectra
preserves $\pi_*$-isomorphism, so it passes to a functor 
\[ [1]\ = \ \Ho(-\sm S^1) \ : \  \Ho(\Sp_G)\ \to \ \Ho(\Sp_G) \]
by the universal property of the localization.
In other words, the shift functor is characterized by the relation
\[ [1]\circ \gamma_G \ = \ \gamma_G\circ (-\sm S^1) \ : \ \Sp_G \ \to \ \Ho(\Sp_G)\ .\]
By~\cite[III Lemma 3.8]{mandell-may:equivariant_orthogonal} 
or~\cite[Prop.\,3.1.25 (ii)]{schwede:global},
the adjunction unit~$\eta:X\to\Omega(X\sm S^1)$ 
is a $\pi_*$-isomorphism of orthogonal $G$-spectra;
so at the level of the stable homotopy category, suspension becomes inverse
to looping; in particular, the shift functor $[1]$
is an auto-equivalence of the category $\Ho(\Sp_G)$.

The distinguished triangles in $\Ho(\Sp_G)$ are
defined from mapping cone sequences as follows.
We let $f:X\to Y$ be a morphism of orthogonal $G$-spectra.
The \emph{reduced mapping cone} $C f$ is defined by
\[ C f \ = \ (X\sm [0,1])\cup_f Y \ . \]
Here the unit interval $[0,1]$ is based by $0\in [0,1]$, 
so that $X\sm [0,1]$ is the reduced cone of $X$. 
The mapping cone comes with an embedding $i:Y\to C f$
and a projection $p:C f\to X\sm S^1$.
As $S^1$ is the one-point compactification of $\mR$,
the projection sends $Y$ to the basepoint and is given on $X\sm [0,1]$ 
by $p(x,z)=x\sm t(z)$ where
\begin{equation}  \label{eq:define_t}
 t\ :\ [0,1]\ \to\  S^1\text{\qquad is \qquad}
t(z)\ =\ \frac{2z-1}{z(1-z)}\ .   
\end{equation}
What is relevant about the map $t$ is not the precise formula, 
but that it passes to a homeomorphism between the quotient space 
$[0,1]/\{0,1\}$ and $S^1=\mR\cup\{\infty\}$.
Then the image in~$\Ho(\Sp_G)$ of the sequence  
\begin{equation} \label{eq:mapping_cone_triangle}
 X \ \xrightarrow{\ f\ }\ Y \ \xrightarrow{\ i \ } \ C f \ 
\xrightarrow{\ p \ }\ X\sm S^1   
\end{equation}
is a distinguished triangle.
More generally, a triangle in $\Ho(\Sp_G)$
is distinguished if and only if it is isomorphic in $\Ho(\Sp_G)$
to such a mapping cone triangle for some morphism of orthogonal $G$-spectra~$f$.
\end{con}

\begin{rk}[Integer shifts in $\Ho(\Sp_G)$]\label{rk:[k]}\index{shift!of an orthogonal $G$-spectrum}
The shift functor on $\Ho(\Sp_G)$ is an auto-equivalence, but {\em not}
an automorphism of $\Ho(\Sp_G)$, so we fix a convention of
what we mean by integer shifts.
For $k\geq 0$ we set
\[ [k]\ = \ \Ho(-\sm S^k) \ :\ \Ho(\Sp_G)\to \Ho(\Sp_G) \ .\]
For $k<0$ we observe that the functor $\Omega^{-k}:\Sp_G\to\Sp_G$
also preserves $\pi_*$-isomorphisms, because looping shifts equivariant
homotopy groups by the loop isomorphism \eqref{eq:loop iso}.
So for negative values of $k$ we define
\[ [k]\ = \ \Ho(\Omega^{-k}) \ :\ \Ho(\Sp_G)\to \Ho(\Sp_G) \ .\]
Since positive and negative shift are not inverse to each other
on the nose, we specify natural isomorphisms
\begin{equation} \label{eq:define_s_k}
 s_k\ : \ X[k][1]\ \iso \ X[k+1]   
\end{equation}
as endofunctors on $\Ho(\Sp_G)$, for all integers $k$.
For $k\geq 0$, we let $s_k$ be induced by the canonical homeomorphism
\[ S^k\sm S^1\ \iso\ S^{k+1}\ , \quad x\sm y \ \longmapsto \ (x,y)\ . \]
For $k<0$, we let $s_k$ be induced by the natural morphism
of orthogonal $G$-spectra
\[ 
\text{eval} \ : \ (\Omega^{-k} X)\sm S^1\ \to\ \Omega^{-(k+1)} X\ , \quad 
 \text{eval}(f\sm t)(z)\ = \ f(t\sm z)   
\]
that evaluates the last coordinate,
where $f\in \Omega^{-k} X$, $t\in S^1$ and $z\in S^{-(1+k)}$.
The evaluation morphism is a $\pi_*$-isomorphism,
see for example \cite[Prop.\,3.1.25 (ii)]{schwede:global}.
\end{rk}

It is now a formal procedure to extend the isomorphisms $s_k$
to a preferred system of natural isomorphisms $X[k][l]\iso X[k+l]$.
We omit the proof of the following proposition.

\begin{prop}\label{prop:coherent isos}
  Let $G$ be a Lie group.
  There is a unique collection of natural isomorphisms
  \[ t_{k,l}\ : \  X[k][l]\ \xrightarrow{\ \iso \ }\  X[k+l] \ ,\]
  for all integers $k$ and $l$,
  of endofunctors of $\Ho(\Sp_G)$, subject to the following conditions:
  \begin{enumerate}[\em (a)]
  \item $t_{1,-1}\star [1]=[1]\star t_{-1,1}:[1]\circ[-1]\circ[1]\to [1]$;
  \item $t_{k,0}=t_{0,k}=\Id_{[k]}$ for every integer $k$;
  \item $t_{k,1}=s_k$ for every integer $k$; and
  \item for every triple of integers $k,l,m$, the following square commutes:
    \[ \xymatrix{
        X[k][l][m]\ar[r]^-{t^{X[k]}_{l,m}}\ar[d]_{t^X_{k,l}[m]} &
        X[k][l+m]\ar[d]^{t^X_{k,l+m}} \\
        X[k+l][m] \ar[r]_-{t^X_{k+l,m}} &   X[k+l+m]
      }  \]
  \end{enumerate}
\end{prop}

Now we turn to the topic of `compact generation' for triangulated categories.
The small objects in the sense of the following definition
are most commonly called `compact' objects; since we already use 
the adjective `compact' in a different sense,
we prefer to use `small'.

\begin{defn}
Let $\cT$ be a triangulated category which has all set indexed sums.
An object $C$ of~$\cT$ is {\em small} (sometimes called {\em finite}
or {\em compact}\,) if for every family $\{X_i\}_{i\in I}$ of objects
the canonical map
\[ \bigoplus_{i\in I}\, \cT(C,\, X_i) \ \to \ \cT(C,\, \bigoplus_{i\in I} X_i)
\]
is an isomorphism. 
A set $\cS$ of objects of $\cT$ is a set of {\em weak generators}
if the following condition holds: if $X$ is an object such that
the groups $\cT(C[k],X)$ are trivial for all $k\in\mZ$ and all $C\in\cS$, 
then $X$ is a zero object.
The triangulated category $\cT$ is {\em compactly generated}
if it has all set indexed sums and a set of small weak generators.
\end{defn}

For every compact subgroup~$H$ of~$G$ we define a tautological homotopy class
\begin{equation} \label{eq:define_tautological}
 u_H \ \in \ \pi_0^H(\Sigma^\infty_+ G/H)   
\end{equation}
as the class represented by the distinguished coset~$e H$ in $G/H$;
indeed, $e H$ is an $H$-fixed point of $G/H$, 
so it gives rise to a based $H$-map 
\[ S^0 \ \to \  G/H_+ \ = \ (\Sigma^\infty_+ G/H)_0 \]
by sending the non-basepoint to $e H$.
For orthogonal $G$-spectra $X$ and~$Y$, we will denote
the Hom abelian group $\Ho(\Sp_G)(X, Y)$ by $[X, Y]^G$.

\begin{prop} \label{prop:rephtpy} \index{equivariant homotopy groups}
  Let $G$ be a Lie group and $H$ a compact subgroup of $G$. 
  Then for every orthogonal $G$-spectrum $X$, the evaluation map
  \[
    [\Sigma^{\infty}_+ G/H, X]^G \ \cong\ \pi_0^H(X)\ , \quad
    [f] \ \longmapsto \ f_*(u_H)
  \]
  is an isomorphism. The suspension spectrum~$\Sigma^\infty_+ G/H$
  is a small object in $\Ho(\Sp_G)$.\index{suspension spectrum}
\end{prop}
\begin{proof}
  Source and target of the evaluation map take $\pi_*$-isomorphisms of orthogonal
  $G$-spectra to isomorphisms of groups, so it suffices to show the
  claim in the special case when $X$ is fibrant in the stable model structure,
  i.e., a $G$-$\Omega$-spectrum.
  Since $H$ is compact, the orthogonal $G$-spectrum $\Sigma^\infty_+ G/H$ is cofibrant;
  so for stably fibrant~$X$ the localization functor $\gamma_G:\Sp_G\to\Ho(\Sp_G)$
  induces a bijection
  \[
    \Sp_G(\Sigma^\infty_+ G/H, X) /  \text{homotopy}\ \xrightarrow{\ \cong \ }\
    [\Sigma^\infty_+ G/H,X]^G
  \]
  from the set of homotopy classes of morphisms in~$\Sp_G$ to the set
  of morphisms in the homotopy category~$\Ho(\Sp_G)$.
  The suspension spectrum $\Sigma^\infty_+ G/H$ represents the $H$-fixed points
  in level~0, so the left hand side bijects with the path components of
  the space $X(0)^H$.
  Since $X$ is a $G$-$\Omega$-spectrum, all the maps in the colimit
  system for~$\pi_0^H(X)$ are bijections, and hence the canonical map
  \[ \pi_0( X(0)^H) \ = \ [S^0, X(0)]^H \ \to \ \colim_{V\in s(\cU_H)}\,  
    [S^V,X(V)]^H\ = \ \pi_0^H(X) \]
  is bijective. We omit the straightforward verification that the
  combined bijection between $[\Sigma^\infty_+ G/H,X]^G$ and~$\pi_0^H(X)$
  coincides with evaluation at the class~$u_H$.

  For every compact Lie group~$H$, the functor~$\pi_*^H$ takes
  wedges of orthogonal $G$-spectra 
  to directs sums~\cite[III Thm.\,3.5 (ii)]{mandell-may:equivariant_orthogonal}. 
  So formation of wedges preserves $\pi_*$-isomorphisms,
  and the wedge of any family $\{X_i\}_{i\in I}$ of orthogonal $G$-spectra 
  is a coproduct in $\Ho(\Sp_G)$. 
  We have a commutative square
  \[
    \xymatrix{  {\bigoplus_{i\in I}} [ \Sigma^\infty_+ G/H ,\, X_i ]^G \ar[r]\ar[d] &
      [ \Sigma^\infty_+ G/H ,\, \bigvee_{i\in I} X_i ]^G \ar[d] \\
      {\bigoplus_{i\in I}} \pi^H_0(X_i) \ar[r] &
      \pi_0^H\left( \bigvee_{i\in I}X_i\right) }
  \]
  in which the vertical maps are evaluation at $u_H$.
  The lower horizontal map is an isomorphism,
  hence so is the upper horizontal map.
  This shows that $\Sigma^\infty_+ G/H$ is small as an object 
  of the triangulated category $\Ho(\Sp_G)$.
\end{proof}

Essentially by definition, an orthogonal $G$-spectrum is a zero object
in~$\Ho(\Sp_G)$ if and only if its $H$-equivariant homotopy groups
vanish for all compact subgroups~$H$ of~$G$. So 
Proposition~\ref{prop:rephtpy} directly implies:

\begin{cor}\label{cor:Ho_G compactly generated} 
Let~$G$ be a Lie group.
The triangulated stable homotopy category~$\Ho(\Sp_G)$
has infinite sums and the suspension spectra~$\Sigma^\infty_+ G/H$
for all compact subgroups $H$ of~$G$ form a set of small weak generators.
In particular, the  triangulated stable homotopy category~$\Ho(\Sp_G)$ 
is compactly generated.
\end{cor}

A contravariant functor $E$ from a triangulated category~$\cT$ 
to the category of abelian groups 
is called {\em cohomological} if for every distinguished\index{cohomological functor}
triangle $(f,g,h)$ in $\cT$ the sequence of abelian groups
\[ \xymatrix{ E(\Sigma A) \ar[r]^-{E(h)} & E(C) \ar[r]^-{E(g)} & 
E(B) \ar[r]^-{E(f)} &E(A)}\]
is exact.
Dually, a covariant functor $F$ from~$\cT$ to the category of abelian groups 
is called {\em homological} if for every distinguished\index{homological functor}
triangle $(f,g,h)$ in $\cT$ the sequence of abelian groups
\[ \xymatrix{ F(A) \ar[r]^-{F(f)} & F(B) \ar[r]^-{F(g)} & 
F(C) \ar[r]^-{F(h)} &F(\Sigma A)}\]
is exact.
The fact that the triangulated category~$\Ho(\Sp_G)$ 
is compactly generated has various useful
consequences that we summarize in the next corollary.

\begin{cor}\label{cor:generators for GH_F}
  Let $G$ be a Lie group.\index{Brown representability!for the equivariant stable homotopy category}
  \begin{enumerate}[\em (i)]
  \item Every cohomological functor $E$ on~$\Ho(\Sp_G)$ 
    that takes sums to products is representable, i.e., 
    there is an orthogonal  $G$-spectrum $Y$ 
    and a natural isomorphism $E\cong [-,Y]^G$. 
  \item Every homological functor $F$ on $\Ho(\Sp_G)$ that takes products
    to products is representable, i.e., there is an orthogonal $G$-spectrum $X$ 
and a natural isomorphism $F\cong [X,-]^G$. 
  \item An exact functor~$F:\Ho(\Sp_G)\to\cS$ to another triangulated category
    has a right adjoint if and only if it preserves sums.
  \item An exact functor~$F:\Ho(\Sp_G)\to\cS$ to another triangulated category
    has a left adjoint if and only if it preserves products.
  \end{enumerate}
\end{cor}
\begin{proof}
  Part~(i) is a direct consequence of being compactly generated, see for example
  \cite[Thm.\,3.1]{neeman:Grothendieck_duality}
  or \cite[Thm.\,A]{krause:brown_coherent}.
  A proof of part~(ii) of this form of Brown representability can be found in
  \cite[Thm.\,8.6.1]{neeman:triangulated_categories}
  or \cite[Thm.\,B]{krause:brown_coherent}.
  Part~(iii) is a formal consequences of part~(i):
  if $F$ preserves sums, then for every object~$X$ of~$\cS$ the functor
  \[ \cS(F(-), X) \ : \ \Ho(\Sp_G)^{\text{op}} \ \to \ \cA b \]
  is cohomological and takes sums to products.
  Hence the functor is representable by an orthogonal $G$-spectrum $R X$
  and an isomorphism
  \[  [A, R X]^G \ \cong \ \cS(F A, X) \ , \]
  natural in~$A$.
  Once this representing data is chosen, the assignment~$X\mapsto R X$ extends 
  canonically to a functor~$R:\cS\to\Ho(\Sp_G)$ that is right adjoint to~$F$.
  In much the same way, part~(iv) is a formal consequence of part~(ii).
\end{proof}

The preferred set of generators $\{\Sigma^\infty_+ G/H\}$
of the stable $G$-homotopy category has another special property,
it is `positive' in the following sense: for all compact subgroups
$H$ and~$K$ of~$G$, and all $n<0$,
\begin{equation}\label{eq:positivity}
 [\Sigma^\infty_+  G/K[n],\ \Sigma^\infty_+  G/H]^G \ \cong \ 
 \pi_n^K(\Sigma^\infty_+  G/H) \ = \ 0 \ ,
\end{equation}
because the underlying orthogonal $K$-spectrum of
$\Sigma^\infty_+  G/H$ is the suspension spectrum of a $K$-space.
A set of positive compact generators in this sense
automatically gives rise to a non-degenerate t-structure, as we shall now recall. 
When~$G$ is discrete, the heart of the t-structure is equivalent to 
the abelian category of $G$-Mackey functors, 
see Theorem \ref{thm:embed_MG_into_Ho(Sp)} below.

A `t-structure' as introduced by\index{t-structure!on a triangulated category} 
Beilinson, Bernstein and Deligne in \cite[Def.\,1.3.1]{beilinson-bernstein-deligne:faisceaux_pervers}
axiomatizes the situation in the derived category of an abelian category
given by cochain complexes whose cohomology vanishes 
in positive respectively negative dimensions.

\begin{defn}\label{def:t-structure}
  A {\em t-structure} on a triangulated category~$\cT$
  is a pair of full subcategories $(\cT_{\geq 0},\cT_{\leq 0})$
  satisfying the following three conditions, where $\cT_{\geq n}= \cT_{\geq 0}[n]$ 
and $\cT_{\leq n}=\cT_{\leq 0}[n]$:
\begin{enumerate} 
\item For all $X\in\cT_{\geq 0}$ and all $Y\in\cT_{\leq -1}$ we have $\cT(X,Y)=0$.
\item $\cT_{\geq 0}\subset \cT_{\geq -1}$ and $\cT_{\leq 0}\supset \cT_{\leq -1}$. 
\item For every object~$X$ of~$\cT$ there is a distinguished triangle
\[ A \ \to \ X \ \to \ B \ \to \ A[1]  \]
such that $A\in\cT_{\geq 0}$ and $B\in\cT_{\leq -1}$.
\end{enumerate}
A t-structure is {\em non-degenerate}
if $\bigcap_{n\in\mZ} \cT_{\leq n}=\{0\}$ and $\bigcap_{n\in\mZ} \cT_{\geq n}=\{0\}$.
The {\em heart} of the t-structure is the full subcategory 
\[ \cH \ = \ \cT_{\geq 0}\cap \cT_{\leq 0} \ ;  \]
it is an abelian category
by \cite[Thm.\,1.3.6]{beilinson-bernstein-deligne:faisceaux_pervers}.
\end{defn}

The original definition of t-structures is formulated slightly
differently in `cohomological' notation, 
motivated by derived categories of cochain complexes as the main examples.
We are mainly interested in spectra, where a homological
(as opposed to {\em co\,}homological) grading is more common,
and the definition above is adapted to the homological setting.

\begin{defn}
  Let $G$ be a Lie group.
  An orthogonal $G$-spectrum $X$ is {\em connective}\index{connective orthogonal $G$-spectrum}
  if the homotopy group $\pi_n^H(X)$ is trivial for 
  every compact subgroup $H$ of~$G$ and every $n<0$.
  An orthogonal $G$-spectrum $X$ is {\em coconnective}\index{coconnective orthogonal $G$-spectrum}
  if the homotopy group $\pi_n^H(X)$ is trivial for 
  every compact subgroup $H$ of~$G$ and every $n>0$.
\end{defn}

\begin{cor}\label{cor:t-structure}\index{t-structure!on the equivariant stable homotopy category} 
  Let $G$ be a Lie group.
  The classes of connective $G$-spectra and coconnective $G$-spectra 
  form a non-degenerate t-structure on~$\Ho(\Sp_G)$
  whose heart consists of those orthogonal $G$-spectra~$X$ such that $\pi_n^H(X)=0$
  for all compact subgroups~$H$ of~$G$ and all~$n\ne 0$.
\end{cor}
\begin{proof}
  We use the more general arguments 
  of Beligiannis and Reiten \cite[Ch.\,III]{beligiannis-reiten:torsion_theories} 
  who systematically investigate torsion pairs 
  and t-structures in triangulated categories that are generated by small objects. 
  By Corollary \ref{cor:Ho_G compactly generated}  the set 
  \[ \cP\ =\ \{\Sigma^\infty_+ G/H\}_{H\in\Com} \]
  is a set of small weak generators for the triangulated category $\Ho(\Sp_G)$.
  We let $\cY$ be the class of $G$-spectra $Y$ such that 
  \[ [P[n], Y]^G \ = \ 0  \]
  for all $P\in\cP$ and all $n\geq 0$.
  The representability result of Proposition \ref{prop:rephtpy} 
  shows that these are precisely those $G$-spectra such that
  $\pi_n^H(Y)=0$ for all compact subgroups $H$ of $G$ and all $n\geq 0$.
  Hence $\cY[1]$ is the class of coconnective $G$-spectra.
  We let $\cX$ be the `left orthogonal' to $\cY$, i.e., 
  the class of $G$-spectra $X$ such that $[X,Y]^G=0$ for all $Y\in\cY$.
  Since the objects of $\cP$ are small in $\Ho(\Sp_G)$
  by Proposition \ref{prop:rephtpy}, Theorem III.2.3 of
  \cite{beligiannis-reiten:torsion_theories} shows that 
  the pair $(\cX,\cY)$ is a torsion pair
  in the sense of \cite[Def.\,I.2.1]{beligiannis-reiten:torsion_theories}.
  This simply means that the pair $(\cX,\cY[1])$ is a t-structure
  in the sense of Definition \ref{def:t-structure}, 
  see \cite[Prop.\,I.2.13]{beligiannis-reiten:torsion_theories}.
  
  It remains to show that $\cX$ coincides with the class of connective $G$-spectra.
  This needs the positivity property \eqref{eq:positivity} 
  of the set $\cP$ of small generators, which lets us apply
  \cite[Prop.\,III.2.8]{beligiannis-reiten:torsion_theories},
  showing that $\cX$ coincides with the class of those $G$-spectra $X$ such that
  $[\Sigma^\infty_+ G/H, X[n]]^G=0$ for all $H\in\Com$ and $n\geq 1$.
  Since the latter group is isomorphic to $\pi_{-n}^H(X)$, 
  this shows that $\cX$ is precisely the class of connective $G$-spectra.
  The t-structure is non-degenerate because spectra with trivial $\Com$-equivariant
  homotopy groups are zero objects in $\Ho(\Sp_G)$.
\end{proof}

\begin{rk}[Postnikov sections]\label{rk:postnikov}\index{Postnikov sections!for orthogonal $G$-spectra}
In every t-structure and for every integer $n$, 
the inclusion $\cT_{\leq n}\to \cT$ 
has a left adjoint $\tau_{\leq n}:\cT\to\cT_{\leq n}$,
by \cite[Prop.\,1.3.3]{beilinson-bernstein-deligne:faisceaux_pervers}.
For the standard t-structure on the $\Ho(\Sp_G)$, 
given by the connective and coconnective $G$-spectra, the truncation functor
\[ \tau_{\leq n} \ : \ \Ho(\Sp_G) \ \to \ \Ho(\Sp_G)_{\leq n} \ , \]
left adjoint to the inclusion, provides a `Postnikov section':
For every orthogonal $G$-spectrum~$X$ 
and every compact subgroup $H$ of $G$,
the $G$-spectrum $\tau_{\leq n}X$ satisfies $\pi^H_k(\tau_{\leq n}X)=0$ for $k>n$ 
and the adjunction unit $X\to X_{\leq n}$ induces an isomorphism on $\pi_k^H$
for every $k\leq n$.
\end{rk}
\index{G-equivariant stable homotopy category@$G$-equivariant stable homotopy category|)}

\section{Change of groups}
\label{sec:change of groups}

This longish section is devoted to studying how
the proper equivariant stable homotopy theory varies with the ambient Lie group.
More precisely, we investigate different aspects of functoriality of the
stable model structure on $\Sp_G$ and of the triangulated homotopy category $\Ho(\Sp_G)$
for continuous group homomorphisms between Lie groups.

Theorem \ref{thm:adjunctions spectra pointset}  records how the restriction functor
$\alpha^*:\Sp_G\to\Sp_K$ along a continuous homomorphism $\alpha:K\to G$
interacts with the stable model structures:
on the one hand, $\alpha^*$ preserves $\pi_*$-isomorphisms whenever
the kernel of $\alpha$ has no non-trivial compact subgroup,
and then $\alpha^*$ is a right Quillen functor for the stable model structures.
On the other hand, 
$\alpha^*$ is a left Quillen functor for the stable model structures
whenever the image of $\alpha$ is closed and its kernel is compact.
Corollary \ref{cor:restriction stable} specializes this to the
inclusion of a closed subgroup,
in which case the restriction functor and its left adjoint both preserve $\pi_*$-isomorphisms.

Theorem \ref{thm:reduction} explains why for almost connected Lie groups
(i.e., those with finitely many path components),
our theory reduces to the classical case of {\em compact} Lie groups:
almost connected Lie groups have maximal compacts subgroups, and
restriction to a maximal compact subgroup is a Quillen equivalence
between the stable model categories of equivariant spectra.

In the remaining part of this section we switch our focus
to functoriality of the triangulated equivariant stable homotopy category.
Despite the fact that the restriction functor
$\alpha^*:\Sp_G\to\Sp_K$ along a continuous homomorphism $\alpha:K\to G$
need not be a Quillen functor for our model structures in general,
we show in Theorem \ref{thm:homomorphism adjunctions spectra}
that it always admits a total left derived functor $L\alpha^*:\Ho(\Sp_G)\to\Ho(\Sp_K)$;
moreover, $L\alpha^*$ is an exact functor of triangulated categories and has a right adjoint.
Morally speaking, the functor $L\alpha^*$ is determined by these properties
and the fact that it `commutes with suspension spectra',
see Theorem \ref{thm:homomorphism adjunctions spectra} for the precise statement.
When the kernel of $\alpha$ has non-trivial compact subgroups,
$L\alpha^*$ will typically {\em not} preserve products,
as Remark \ref{rk:inflation and products} illustrates.
The derived functors $L\alpha^*$ enjoy a specific kind of lax functoriality:
pairs of composable continuous homomorphisms give rise to
a preferred natural transformation
$\td{\alpha,\beta}:L\beta^*\circ L\alpha^*\Longrightarrow L(\alpha\beta)^*$,  
see Construction \ref{con:lax transfo};
and these transformations in turn satisfy a coherence condition,
see Proposition \ref{prop:derived cocycle}.

Finally, we show that proper $G$-equivariant stable homotopy theory
is `homotopy invariant' in the Lie group $G$, in two specific ways.
On the one hand, if $\alpha:K\to G$ is a continuous homomorphism
and a weak equivalence of underlying topological spaces,
then $L\alpha^*:\Ho(\Sp_G)\to\Ho(\Sp_K)$ is an equivalence of triangulated
categories, see Theorem \ref{thm:weak equivalence}.
On the other hand, a homotopy through continuous group homomorphisms
gives rise to a specific natural isomorphism between the
left derived functors, see Theorem \ref{thm:homotopic homomorphisms}.
Both of these homotopy invariance statements are new phenomena,
since they are trivially true for {\em compact} Lie groups.
Indeed, a multiplicative weak equivalence between compact Lie groups
is already an isomorphism; and two homotopic continuous homomorphisms
between compact Lie groups are already conjugate.\medskip

As we discussed in Proposition \ref{prop:homomorphism adjunctions spaces}, 
a continuous homomorphism $\alpha:K\to G$ between
Lie groups gives rise to functors
between the associated categories of equivariant spaces:
\[\xymatrix@C=18mm{  
K \bT\ \ar@/^1pc/@<.5ex>[r]^-{G \ltimes_{\alpha } -} 
\ar@/_1pc/@<-.5ex>[r]_-{\map^{K,\alpha}(G,-)} &
\ G \bT  \ar[l]_-{\alpha^* } } \]
Here $G\times_\alpha -$ is left adjoint to $\alpha^*$,
and  $\alpha^*$ is left adjoint to $\map^{K,\alpha}(G,-)$.
Levelwise application of these functors gives rise to
analogous adjoint functor pairs between categories of equivariant spectra
\[\xymatrix@C=18mm{  
\Sp_K\ \ar@/^1pc/@<.5ex>[r]^-{G \ltimes_{\alpha } -} 
\ar@/_1pc/@<-.5ex>[r]_-{\map^{K,\alpha}(G,-)} &
\ \Sp_G  \ar[l]_-{\alpha^* } }\ , \]
see Construction \ref{con:smash with G-space}.
The following theorem records how these functors interact 
with the stable model structures on $\Sp_K$ and $\Sp_G$.

\medskip

We call a continuous homomorphism $\alpha:K\to G$
between Lie groups {\em quasi-injective}\index{quasi-injective}
if the restriction of $\alpha$ to every compact subgroup of $K$ is
injective. Equivalently, the kernel of $\alpha$ has
no non-trivial compact subgroups.

\begin{thm} \label{thm:adjunctions spectra pointset} 
Let $ \alpha \colon K \to G$ be a continuous homomorphism between Lie groups.
\begin{enumerate}[\em (i)]
\item If $\alpha$ is quasi-injective, 
  then the restriction functor $\alpha^*:\Sp_G\to \Sp_K$
  preserves $\pi_*$-isomorphisms and stable fibrations.
  In particular, the adjoint functor pair $(G\ltimes_\alpha-,\alpha^*)$
  is a Quillen pair with respect to the stable model structures.
\item If $\alpha$ has a closed image and a compact kernel,
  then the adjoint functor pair $(\alpha^*,\map^{K,\alpha}(G,-))$
  is a Quillen pair with respect to the stable model structures.
\end{enumerate}
\end{thm}
\begin{proof}
  (i) We let $f:X\to Y$ be a $\pi_*$-isomorphism or a stable fibration
  of orthogonal $G$-spectra.
  The definitions of $\pi_*$-isomorphism and stable fibrations only refer to
  compact subgroups, so to show that $\alpha^*(f)$
  is a $\pi_*$-isomorphism or a stable fibration of orthogonal $K$-spectra,
  we can restrict to all compact subgroups $L$ of $K$.
  Since $L$ is compact, the restriction of $\alpha$ to $L$ is injective,
  hence a closed embedding, and hence an isomorphism of Lie groups onto its image
  $H=\alpha(L)$. So $\alpha$ induces a natural isomorphism between $\pi_*^L(\alpha^*(X))$
  and $\pi_*^H(X)$, which shows that $\pi_*^L(\alpha^*(f))$ is an isomorphism.
  Since $L$ was an arbitrary compact subgroup of $K$, this proves that
  $\alpha^*(f)$ is a $\pi_*$-isomorphism. The argument for stable fibrations is similar,
  by using the isomorphism $\alpha:L\iso H$ to translate 
  the commutative square \eqref{eq:stable fibration square}
  for given $L$-representations $V$ and $W$ into an analogous square
  for the $H$-representations $(\alpha^{-1})^*(V)$ and $(\alpha^{-1})^*(W)$.

  (ii) In a first step we show that the restriction functor $\alpha^*:\Sp_G\to \Sp_K$
  takes all cofibrations of orthogonal $G$-spectra to
  cofibrations of orthogonal $K$-spectra.
  Restriction along $\alpha$ only changes the group actions, 
  but it does not change the underlying orthogonal spectra. 
  Hence the skeleta and latching objects of $\alpha^*(X)$ are the same
  as for $X$, but with action restricted along the homomorphism
  $\alpha\times O(m):K\times O(m)\to G\times O(m)$.
  Since the image of $\alpha$ is closed in $G$, 
  the image of $\alpha\times O(m)$ is closed in $G\times O(m)$.
  The kernel of $\alpha\times O(m)$ is $\ker(\alpha)\times 1$,
  which is compact by hypothesis.
  So restriction along $\alpha\times O(m)$ takes 
  $\Com$-cofibrations of $(G\times O(m))$-spaces
  to $\Com$-cofibrations of $(K\times O(m))$-spaces,
  by Proposition \ref{prop:homomorphism adjunctions spaces}.
  So if $i:A\to B$ is a cofibration of orthogonal $G$-spectra,
  then $\nu_m i$ is a $\Com$-cofibration of $(G\times O(m))$-spaces,
  and $\nu_m(\alpha^*(i))=(\alpha\times O(m))^*(\nu_m i)$
  is a $\Com$-cofibration of $(K\times O(m))$-spaces.
  Moreover, the $O(m)$-action is unchanged, so it still acts
  freely off the image of $\nu_m(\alpha^*(i))$.
  This shows that $\alpha^*(i)$ is a cofibration of orthogonal $K$-spectra.

  It remains to show that $\alpha^*$ takes
  cofibrations of orthogonal $G$-spectra that are also $\pi_*$-isomorphisms 
  to $\pi_*$-isomorphisms of orthogonal $K$-spectra.
  Here we treat two special cases first. 
  If $\alpha$ is the inclusion of a closed subgroup $\Gamma$ of $G$, 
  then $\alpha^*=\res^G_\Gamma$ preserves all $\pi_*$-isomorphisms by part (i).
  If $\alpha$ is surjective, then we verify 
  that $\alpha^*$ takes the generating acyclic cofibrations 
  of the stable model structure on orthogonal $G$-spectra to 
  acyclic cofibrations of orthogonal $K$-spectra.
  The generating acyclic cofibrations $J^G_{\lv}$ 
  of the level model structure \eqref{eq:gen_level_acyc_cofibration}
  are $G$-equivariant homotopy equivalences, so $\alpha^*$ takes
  them to $K$-equivariant homotopy equivalences, which are in particular
  $\pi_*$-isomorphisms.
  For the other generating acyclic cofibrations in $\cK^G$ 
  we recall from \eqref{def:lambda} the $\pi_*$-isomorphism
  \[
    G\ltimes_H \lambda_{H,V,W}\ : \ G\ltimes_H ( F_{V\oplus W} S^V ) \ \to \ G\ltimes_H  F_W \ ;
  \]
  here $H$ is a compact subgroup of $G$, and $V$ and $W$ are $H$-representations.
  Since the kernel of $\alpha$ is compact, the group $L=\alpha^{-1}(H)$ 
  is then a compact subgroup of $K$.
  An isomorphism of orthogonal $K$-spectra
  \[ K\ltimes_L F_{(\alpha|_L)^*(V)} \ \cong \   \alpha^*( G \ltimes_H F_V ) \]
  is given levelwise by sending $k\sm x$ to $\alpha(k)\sm x$.
  We conclude that $\alpha^*$ takes the $\pi_*$-isomorphism $G\ltimes_H \lambda_{H,V,W}$
  to the $\pi_*$-isomorphism of orthogonal $K$-spectra
  $K\ltimes_L \lambda_{L,\alpha^*(V),\alpha^*(W)}$.
  The inflation functor $\alpha^*$ commutes with formation of mapping cylinders
  and levelwise smash product with spaces, 
  so we conclude that $\alpha^*(\cK^G)\subset \cK^K$.
  This completes the proof that $\alpha^*$ preserves stable acyclic cofibrations
  if $\alpha$ is surjective with compact kernel.

  In the general case we factor $\alpha$ as the composite
  \[ K \ \xrightarrow{\ \beta\ }\ \Gamma \ \xrightarrow{\incl}\ G\ , \]
  where $\Gamma=\alpha(K)$ is the image of $\alpha$, and $\beta$ is the
  same map as $\alpha$, but with image $\Gamma$.
  The restriction homomorphism factors as $\alpha^*=\beta^*\circ\res^G_\Gamma$,
  and each of the two functors is a left Quillen functor by the special cases
  treated above.
\end{proof}

\begin{rk}
  Both hypotheses on the continuous homomorphism $\alpha$ imposed 
  in Theorem \ref{thm:adjunctions spectra pointset} (ii)
  are really necessary. For example, if $\alpha:\mZ\to U(1)$
  is the continuous homomorphism that takes the generator
  to $e^{2\pi i x}$ for an irrational real number $x$,
  then $\alpha^*(U(1))$ is not cofibrant as a $\mZ$-space, and 
  $\alpha^*(\Sigma^\infty_+ U(1))$ is not cofibrant as an orthogonal $\mZ$-spectrum.
  So we cannot drop the hypothesis that the image of $\alpha$ is closed.
  
  If $\alpha:G\to e$ is the unique homomorphism to the trivial group, 
  then $\alpha^*(\mS)=\mS_G$ is the $G$-sphere spectrum.
  This $G$-sphere spectrum is cofibrant as an orthogonal $G$-spectrum
  precisely when $G$ is compact. 
  So we cannot drop the hypothesis that the kernel of $\alpha$ is compact.
\end{rk}

For easier reference we spell out the important special case of 
Theorem \ref{thm:adjunctions spectra pointset} for the inclusion of a closed subgroup.
The induction functor $G\ltimes_\Gamma-$ preserves $\pi_*$-isomorphisms
by Corollary \ref{cor:Gamma2G}.

\begin{cor}\label{cor:restriction stable}\index{induction functor}\index{restriction functor}
  Let $\Gamma$ be a closed subgroup of a Lie group $G$. 
  The restriction functor $\res^G_\Gamma \colon \Sp_G \to \Sp_\Gamma$ 
  preserves $\pi_*$-isomorphisms, cofibrations and stable fibrations.
  Hence the two adjoint functor pairs
  \[
    \xymatrix@C=15mm{  
      \Sp_\Gamma\ \ar@<0.5ex>[r]^-{G\ltimes_\Gamma} &\ \Sp_G  \ar@<0.5ex>[l]^-{\res^G_\Gamma}} 
    \text{\qquad and\qquad}
    \xymatrix@C=15mm{  
      \Sp_\Gamma\ \ar@<-0.5ex>[r]_-{\map^\Gamma(G,-)} &\ \Sp_G \ar@<-0.5ex>[l]_-{\res^G_\Gamma}}
  \]
  are Quillen adjunctions with respect to the two stable model structures.
  Moreover, the induction functor $G\ltimes_\Gamma \colon \Sp_\Gamma \to \Sp_G$ 
  preserves $\pi_*$-isomorphisms.
\end{cor}

A celebrated theorem of Cartan, Iwasawa \cite{iwasawa:types_groups}
and Malcev \cite{malcev:theory_Lie}
says that every connected Lie group $G$ has a 
maximal compact subgroup, i.e., a compact subgroup $K$ 
such that every compact subgroup is subconjugate to $K$.
Moreover, $G$ is homeomorphic as a topological space to $K\times \mR^n$ 
for some $n\geq 0$ 
(but there is typically no Lie group isomorphism between $G$ and $K\times\mR^n$).
In particular, the inclusion $K\to G$ is a homotopy equivalence
of underlying spaces.
A comprehensive exposition with further references can be found in
Borel's {\em S{\'e}minaire Bourbaki} article \cite{borel:sous_groupes_compacts}.
A maximal compact subgroup with these properties exists
more generally when the Lie group~$G$ is {\em almost connected}, i.e., 
when it has finitely many path components,
and even for locally compact topological groups
whose component group is compact, see \cite[Thm.\,A.5]{abels:parallelizability}.

\begin{thm}[Reduction to maximal compact subgroups] \label{thm:reduction}
  Let $G$ be an almost connected Lie group, and $K$ a maximal compact subgroup of~$G$. 
  Then the restriction functor
  \[ \res^G_K \colon  \Sp_G\ \to\ \Sp_K \]\index{maximal compact subgroup}
  is a left and right Quillen equivalence for the stable model structures.
\end{thm}
\begin{proof}
  Since every compact subgroup of $G$ is subconjugate to $K$,
  the restriction functor detects $\pi_*$-isomorphisms.
  Since restriction and induction preserve
  $\pi_*$-isomorphisms in full generality,
  we may show that for every orthogonal $K$-spectrum $Y$,
  the adjunction unit $Y\to G\ltimes_K Y$ is a $\pi_*$-isomorphism
  of orthogonal $K$-spectra.

  At this point we need additional input, namely a theorem of Abels 
  \cite[Thm.\,A.5]{abels:parallelizability} 
  that provides a subspace $E$ of~$G$ with the following properties:
  \begin{enumerate}[(i)]
  \item The space~$E$ is invariant under conjugation by $K$.
  \item Under the conjugation action, the space $E$ is $K$-equivariantly
    homeomorphic to a finite-dimensional linear $K$-representation. 
  \item The multiplication map~$E\times K\to G$, $(e,k)\mapsto e k$
    is a homeomorphism.
  \end{enumerate}
  The multiplication homeomorphism $E\times K \cong G$
  is left $K$-equivariant for the diagonal $K$-action on the source,
  by conjugation on $E$ and by left translation on~$K$.
  The multiplication homeomorphism is right $K$-equivariant for 
  the trivial $K$-action on~$E$ and the right translation action on~$K$.
  So in particular, the morphism
  \[  E_+\sm Y \ \to \ G\ltimes_K Y \ , \quad (e,y)\ \longmapsto \ [e,y] \]
  is an isomorphism of orthogonal $K$-spectra, where $K$ acts diagonally on
  the source, by conjugation on~$E$, and by the given action on~$Y$.
  Since $E$ is $K$-homeomorphic to a finite-dimensional linear $K$-representation,
  it is $K$-equivariantly contractible.
  So the adjunction unit is a $K$-equivariant homotopy equivalence
  of underlying orthogonal $K$-spectra. 
\end{proof}

\begin{eg}\label{eg:1-dim for locally finite}
By the very construction, the proper $G$-equivariant homotopy theory
is assembled from the equivariant homotopy theories of the 
compact subgroups. We now discuss a situation where this relationship is
especially tight, and where morphisms in~$\Ho(\Sp_G)$ can be calculated
directly from morphisms in~$\Ho(\Sp_H)$ for finite subgroups~$H$.

Let $G$ be a countable locally finite group, i.e., a discrete group\index{locally finite group} 
that has an exhaustive sequence of finite subgroups
\[ H_0 \ \subseteq\  H_1 \ \subseteq\  \dots \ \subseteq\  H_n \ \subseteq\ \dots\  , \]
i.e., so that $G=\bigcup H_n$.
The various restriction functors
\[ \res^G_{H_n}\ : \ \Ho(\Sp_G)\ \to \ \Ho(\Sp_{H_n}) \]
are then compatible. For all orthogonal $G$-spectra $X$ and $Y$,
the restriction maps thus assemble into a group homomorphism
\[ \res^G\ : \ [X,Y]^G \ \to \ {\lim} \, [X,Y]^{H_n}\ , \]
where the inverse limit on the right hand side is formed along restriction maps.
We have simplified the notation by suppressing the restriction functors
$\res^G_{H_n}$ on the right hand side.
The hypothesis that the sequence $H_n$ exhausts $G$ implies that for every
finite subgroup $K$ of~$G$ the colimit of the sequence of sets
\[ (G/H_0)^K \ \to \  (G/H_1)^K \ \to \ \dots \ \to \  (G/H_n)^K \ \to \ \dots\]
is a single point.
Thus the mapping telescope, in the category of $G$-spaces, 
of the sequence of $G$-spaces $G/H_n$ is $\Com$-equivalent to the one-point $G$-space.
Since the mapping telescope comes to us as a 1-dimensional $G$-CW-complex,
it is a 1-dimensional $G$-CW-model for~$\uEG$.

For all cofibrant $G$-spectra $X$, the mapping telescope,
in the category of orthogonal $G$-spectra, of the sequence 
\[ X\sm (G/H_0)_+ \ \to \  X\sm (G/H_1)_+ \ \to \ \dots \ \to \  X\sm (G/H_n)_+
\ \to \ \dots\]
is hence $\pi_*$-isomorphic to~$X$.
The mapping telescope models an abstract homotopy colimit 
in the triangulated category~$\Ho(\Sp_G)$.
So applying the functor~$[-,Y]^G$ yields a short exact sequence of abelian groups
\begin{align*}
   0 \ \to \ {\lim}^1 \, [ X\sm (G/H_n)_+\sm S^1, Y]^G \ &\to \  [X,Y]^G \\
  \to \ &{\lim}\, [X\sm (G/H_n)_+,Y]^G \ \to \ 0 \ . 
\end{align*}
We rewrite $X\sm (G/H_n)_+$ as $G\ltimes_{H_n}(\res^G_{H_n} X)$
and use the adjunction isomorphism between restriction and induction 
to identify the group $[X\sm (G/H_n)_+,Y]^G$ of morphisms in~$\Ho(\Sp_G)$ with the group
$[X,Y]^{H_n}$ of morphisms in~$\Ho(\Sp_{H_n})$. The maps in the tower then become
restriction maps, so we have shown:
\end{eg}

\begin{cor}\label{cor:lim_lim^1_sequence} 
Let $G$ be a discrete group and $\{H_n\}_{n\geq 0}$
an ascending exhaustive sequence of finite subgroups of~$G$.
Then for all orthogonal $G$-spectra $X$ and~$Y$ there is a short exact sequence
\[ 0 \ \to \ {\lim}^1 \, [X\sm S^1,Y]^{H_n} \ \to \ 
 [X,Y]^G \ \to \ 
{\lim}\, [ X,Y]^{H_n} \ \to \ 0 \ . \]  
Here the inverse and derived limit are formed along restriction maps,
and so is the map from $[X,Y]^G$ to the inverse limit.
\end{cor}

\begin{rk}
  The {\em $\Fin$-orbit category}\index{Fin-orbit category@$\Fin$-orbit category}
  of a discrete group $G$ is
  the full subcategory of the category of $G$-sets with objects $G/H$ 
  for all finite subgroups $H$ of $G$.
  As we hope to make precise in future work, the underlying $\infty$-category
  of the stable model category of orthogonal $G$-spectra is a limit,
  for $G/H$ ranging through the $\Fin$-orbit category,
  of the $\infty$-categories of genuine $H$-spectra.

  If $G$ is an ascending union of finite subgroups $H_n$ as in Example \ref{eg:1-dim for locally finite},
  the limit diagram can be modified to show that
  the underlying $\infty$-category of orthogonal $G$-spectra
  is an inverse limit of the tower of $\infty$-categories associated to orthogonal $H_n$-spectra.
  The short exact sequence of Corollary \ref{cor:lim_lim^1_sequence} is a consequence of
  this more refined relationship.
  
  An interesting special case of this\index{Pr{\"u}fer group}
  is the Pr{\"u}fer group $C_{p^\infty}$ for a prime number $p$, i.e.,
  the group of $p$-power torsion elements in $U(1)$, with the discrete topology.
  Since $C_{p^\infty}$ is the union of its subgroups $C_{p^n}$ for $n\geq 0$,
  the underlying $\infty$-category of orthogonal $C_{p^\infty}$-spectra is
  an inverse limit of the $\infty$-categories of genuine $C_{p^n}$-spectra.
  Hence the  $\infty$-category of genuine proper orthogonal $C_{p^\infty}$-spectra is
  equivalent to the $\infty$-category of genuine $C_{p^\infty}$-spectra
  in the sense of Nikolaus and Scholze \cite[Def.\,II.2.15]{nikolaus-scholze:TC},
  on which their notion of
  {\em genuine $p$-cyclotomic spectra} is based,\index{p-cyclotomic spectrum@$p$-cyclotomic spectrum}
  compare \cite[Def.\,II.3.1]{nikolaus-scholze:TC}.
\end{rk}

Now we compare the equivariant stable homotopy categories for varying Lie groups.
We let $ \alpha \colon K \to G$ be a continuous homomorphism between Lie groups.
By Theorem \ref{thm:adjunctions spectra pointset} (i),
the restriction functor $\alpha^*:\Sp_G\to\Sp_K$ preserves $\pi_*$-isomorphisms
if $\alpha$ happens to be quasi-injective, 
but not when the kernel of $\alpha$ has a non-trivial compact subgroup.
In that case there cannot be an induced functor 
$\Ho(\alpha^*):\Ho(\Sp_G)\to\Ho(\Sp_K)$ such that $\Ho(\alpha^*)\circ\gamma_G$
is equal to $\gamma_K\circ\alpha^*$,
where $\gamma_G$ and $\gamma_K$ are the localization functors of \eqref{eq:gamma_G}. 
However, the next best thing is true: restriction along $\alpha$ has a total
left derived functor, see Theorem \ref{thm:homomorphism adjunctions spectra} below.
For the convenience of the reader we briefly review this concept.
The arguments we use in the rest of this section have a substantial overlap with Chapter VII
on `Deformable functors and their approximations'
of the book \cite{dwyer-hirschhorn-kan-smith} by Dwyer, Hirschhorn, Kan and Smith;
for the convenience of the reader, we give a largely self-contained exposition of the relevant parts,
adapted to our context.

We let $(\cC,w)$ be a relative category, i.e., a category $\cC$
equipped with a distinguished class $w$ of morphisms that
we call {\em weak equivalences}. An important special case\index{relative category} 
of relative categories are the ones underlying model categories.
A functor between relative categories is {\em homotopical}
if it takes weak equivalences to weak equivalences.

\begin{defn}
The {\em homotopy category} of a relative category $(\cC,w)$\index{homotopy category!of a relative category} 
is a functor $\gamma_\cC:\cC\to \Ho(\cC)$ that sends all weak equivalences to
isomorphisms and initial among such functors.
\end{defn}

Explicitly, the universal property of the homotopy category
$\gamma_\cC:\cC\to \Ho(\cC)$ is as follows. For every functor $\Phi:\cC\to \cX$
that sends all weak equivalences to isomorphisms, there is a unique functor
$\bar\Phi$ such that $\bar\Phi\circ\gamma_\cC=\Phi$.
In the generality of relative categories, a homotopy category need not always exist
(with small hom set, or in the same Grothendieck universe, that is).
In the examples we care about, the relative category is underlying a model category,
and then already Quillen \cite[I Thm.\,1']{quillen:homotopical_algebra}
constructed a homotopy category as the quotient of the category
of cofibrant-fibrant objects by an explicit homotopy relation on morphisms.

There are many interesting functors between relative categories
that are not homotopical, but still induce interesting functors
between the homotopy categories.
Often, extra structure on the relative categories is used to
define and study such `derived' functors, 
for example a model category structure. 
We are particularly interested in `left derived functors'.

In the following we will compose (or `paste') functors and natural
transformations, and we introduce notation for this.
Let $\nu:F\Longrightarrow F':\cC\to\cD$ be a natural transformation 
between two functors, and let $E:\cB\to\cC$ and $G:\cD\to\cE$ be functors.
We write $\nu\star E:F\circ E\Longrightarrow F'\circ E$
and $G\star\nu:G\circ F\Longrightarrow G\circ F'$
for the natural transformation with components
\[ (\nu\star E)_X \ = \ \nu_{E(X)} \ : F(E(X)) \ \to \ F'(E(X)) \]
and
\[ (G\star\nu)_Y \ = \ G(\nu_Y) \ : G(F(Y)) \ \to \ G(F'(Y)) \ ,\]
respectively. If $\mu:G\Longrightarrow G'$ is another natural transformation
between functors from $\cD$ to $\cE$, then the following
{\em interchange relation} holds:
\begin{equation}\label{eq:interchange}
(G'\star \nu)\circ (\mu\star F)\ = \ (\mu\star F')\circ (G\star \nu)\ ;
\end{equation}
this relation is just a restatement of naturality.

\begin{defn}\index{left derived functor|(}
Let $F:\cC\to\cD$ be a functor between relative categories.
A {\em total left derived functor} of $F$ is a pair
$(L,\tau)$ consisting of a functor $L:\Ho(\cC)\to \Ho(\cD)$ and
a natural transformation $\tau:L \circ \gamma_\cC\Longrightarrow \gamma_\cD\circ F$ 
with the following universal property:  for every pair $(\Phi,\kappa)$
consisting of a functor $\Phi:\Ho(\cC)\to \Ho(\cD)$ and
a natural transformation $\kappa:\Phi\circ \gamma_\cC\Longrightarrow \gamma_\cD\circ F$,
there is a unique natural transformation $\bar\kappa:\Phi\Longrightarrow L$ such that
$\kappa= \tau\circ (\bar\kappa \star\gamma_\cC)$.
A functor between relative categories is {\em left derivable}
if it admits a total left derived functor.
\end{defn}

\begin{eg}
Suppose that $F:\cC\to\cD$ is a homotopical functor between relative categories.
Then the composite $\gamma_\cD\circ F$ takes all weak equivalences to isomorphisms. 
Hence the universal property of the homotopy category provides a unique functor
$\Ho(F):\Ho(\cC)\to\Ho(\cD)$ such that $\Ho(F)\circ\gamma_\cC=\gamma_\cD\circ F$.
Then the pair $(\Ho(F),\Id)$ is a total left derived functor of $F$.
If $G:\cD\to\cE$ is another homotopical functor, then the composite
$G F:\cC\to \cE$ is also homotopical.
Moreover, 
\[ \Ho(G)\circ\Ho(F)\circ\gamma_\cC\ = \ \Ho(G)\circ\gamma_\cD\circ F
\ = \ \gamma_\cE \circ G\circ F\ = \ \Ho(G F)\circ \gamma_\cC\ . \]
The universal property of the homotopy category thus shows that
$\Ho(G)\circ\Ho(F) = \Ho(G F)$.

Conversely, suppose that $F:\cC\to\cD$ is a functor between model
categories which admits a total left derived functor $(L,\tau)$,
and such that $\tau$ is a natural isomorphism.
We claim that then $F$ must be homotopical.
Indeed, for every weak equivalence $f$ in $\cC$,
the morphism $\gamma_\cC(f)$ is an isomorphism in $\Ho(\cC)$,
and hence $L(\gamma_\cC(f))$ is an isomorphism in $\Ho(\cD)$.
Since $\tau:L\circ\gamma_\cC\Longrightarrow \gamma_\cD\circ F$ 
is a natural isomorphism, the morphism $\gamma_\cD(F(f))$ is an isomorphism in $\Ho(\cD)$.
In model categories, the localization functor detects weak equivalences,
so $F(f)$ is a weak equivalence in $\cD$.
\end{eg}

We recall an important result of Quillen that implies 
that a functor between model categories admits 
a total left derived functor if it takes weak equivalences between
cofibrant object to weak equivalences.
Moreover, Maltsiniotis observed that any such total left derived functor 
is automatically an {\em absolute} right Kan extension,
i.e., it remains a right Kan extension after postcomposition with any functor:

\begin{thm}[Quillen \cite{quillen:homotopical_algebra},
Maltsiniotis \cite{maltsiniotis:theorem_quillen}]\label{thm:absolute}
Let $\cC$ be a model category and $F:\cC\to\cX$ a functor 
that takes weak equivalences between cofibrant objects to isomorphisms.
Then $F$ admits a right Kan extension $(L,\tau)$ along the localization
functor $\gamma_\cC:\cC\to\Ho(\cC)$.
Moreover:
\begin{enumerate}[\em (i)]
\item The morphism $\tau_X:L(\gamma_\cC(X))\to F X$ is 
an isomorphism for every cofibrant object $X$ of $\cC$.
\item Every right Kan extension of $F$ along $\gamma_\cC$ is an absolute right Kan extension.
\end{enumerate}
\end{thm}

The following proposition is a direct consequence of the absolute Kan extension
property.

\begin{prop}\label{prop:homotopical after absolute}
Let $F:\cC\to\cD$ and $G:\cD\to\cE$ be composable
functors between relative categories such that $F$ is absolutely left derivable
and $G$ is homotopical. 
Let $(L F,\tau_F)$ be a total left derived functor of $F$.
Then the pair $(\Ho(G)\circ L F, \Ho(G)\star\tau_F)$ 
is an absolute left derived functor of $G F$.
\end{prop}
\begin{proof}
  We let $\psi:\Ho(\cE)\to \cY$ be any functor.
  Since $(L F,\tau_F)$ is an absolute right Kan extension of $\gamma_\cD\circ F$
  along $\gamma_\cC$, the pair $(\psi\circ \Ho(G)\circ L F,(\psi\circ\Ho(G))\star \tau_F)$
  is a right Kan extension of 
  $\psi\circ\Ho(G)\circ \gamma_\cD\circ F=\psi\circ \gamma_\cE\circ G\circ F$ along $\gamma_\cC$.
  This precisely means that $(\Ho(G)\circ L F,\Ho(G)\star\tau_F)$ 
  is an absolute total left derived functor of $G F$.
\end{proof}
\index{left derived functor|)}

Our next goal is to show that restriction of equivariant spectra
along a continuous homomorphism between Lie groups
has a total left derived functor.
To this end, the following class of equivariant spectra will be useful.

\begin{defn}\label{def:quasi-cofibrant}
  Let $G$ be a Lie group. An orthogonal $G$-spectrum $X$
  is {\em quasi-cofibrant}\index{quasi-cofibrant!orthogonal $G$-spectrum}  
  if for every compact subgroup $H$ of $G$ the underlying $H$-spectrum of $X$
  is cofibrant.
\end{defn}

\begin{eg}
Restriction to a closed subgroup preserves cofibrancy, 
by Corollary \ref{cor:restriction stable};
so every cofibrant orthogonal $G$-spectrum is quasi-cofibrant.
If the Lie group $G$ is itself compact, then every quasi-cofibrant 
$G$-spectrum is already cofibrant, so in the compact case the two notions
coincide.

If the Lie group $G$ is not compact, then `quasi-cofibrant'
is a strictly more general concept.
Indeed, the $G$-equivariant sphere spectrum $\mS_G$ is cofibrant if and only 
if $G$ is compact; hence for every Lie group $G$, $\mS_G$ is quasi-cofibrant.
More generally, we let $\Gamma$ be a closed subgroup of $G$.
For every compact subgroup $H$ of $G$, the $H$-action
on the coset space $G/\Gamma$ by translation is smooth;
so Illman's theorem \cite[Thm.\,7.1]{illman:triangulation_Lie} 
provides an $H$-CW-structure on $G/\Gamma$.
In particular, $G/\Gamma$ is cofibrant as an $H$-space, and hence
the suspension spectrum $\Sigma^\infty_+ G/\Gamma$ is quasi-cofibrant.
\end{eg}

The next proposition provides a characterization of quasi-cofibrant spectra
in terms of cofibrant spectra.

\begin{prop}
  Let $G$ be a Lie group. An orthogonal $G$-spectrum $X$ is quasi-cofibrant  
  if and only if for every $\Com$-cofibrant $G$-space $B$
  the orthogonal $G$-spectrum $X\sm B_+$ is cofibrant.
\end{prop}
\begin{proof}
For one implication we let $X$ be an orthogonal $G$-spectrum $X$ 
such that $X\sm B_+$ is cofibrant as an orthogonal $G$-spectrum 
for every $\Com$-cofibrant $G$-space $B$.
For every compact subgroup $H$ of $G$, the homogeneous $G$-space $B=G/H$ 
is $\Com$-cofibrant, so in particular the $G$-spectrum
$X\sm G/H_+$ is cofibrant. The underlying $H$-spectrum of $X\sm G/H_+$ 
is then cofibrant by Corollary \ref{cor:restriction stable}.
The two $H$-equivariant morphisms
\[ X \ \xrightarrow{x\mapsto x\sm e H}\ X\sm G/H_+ \ \xrightarrow{x\sm g H\mapsto x}\ X 
 \]
witness that $X$ is an $H$-equivariant retract of $X\sm G/H_+$,
and so $X$ is itself cofibrant as an $H$-spectrum. Hence $X$ is quasi-cofibrant.

For the other implication we let $X$ be a quasi-cofibrant 
orthogonal $G$-spectrum, and we let $K$ denote the class of 
those morphisms $i:A\to B$ of $G$-spaces  
such that $X\sm f_+:X\sm A_+\to X\sm B_+$ is a cofibration of orthogonal $G$-spectra.
We claim that $K$ contains all $\Com$-cofibrations of $G$-spaces;
for $A=\emptyset$ this proves that $X\sm B_+$ is cofibrant.

For every compact subgroup $H$ of $G$, the underlying $H$-spectrum of
$X$ is cofibrant by assumption. So the $G$-spectrum
\[ X\sm G/H_+ \ \iso \ G\ltimes_H \res^G_H(X) \]
is $G$-cofibrant.  
Since the stable model structure on $\Sp_G$ is $G$-topological, the morphism
\[ X\sm (G/H\times i^k)_+ \  :\ X\sm (G/H\times\partial D^k)_+ \ \to
\ X\sm(G/H\times D^k)_+ \]
is a cofibration of $G$-spectra.
This shows that the generating $\Com$-cofibrations belong to the class $K$.
Because cofibrations are closed under cobase change, coproducts, sequential colimits,
and retracts, and because $X\sm(-)_+$ preserves colimits,
the class $K$ is closed under cobase change, coproducts, sequential colimits
and retracts. So all $\Com$-cofibrations belong to the class $K$.
\end{proof}

The following theorem shows that restriction of equivariant spectra
along a continuous homomorphism between Lie groups
has a total left derived functor, and it collects many important properties of the
left derived functor.
Among other things, the derived functor
`commutes with suspension spectra'. 
To make this precise we observe that the suspension spectrum functor 
$\Sigma^\infty_+:G\bT\to \Sp_G$ is fully homotopical, 
i.e., it takes $\Com$-weak equivalences to $\pi_*$-isomorphisms,
for example by \cite[Prop.\,3.1.44]{schwede:global}.
The space level restriction functor $\alpha^*:G\bT\to K\bT$
is also fully homotopical for $\Com$-equivalences,
compare Proposition \ref{prop:homomorphism adjunctions spaces} (i). 
Hence the functor
\[ \Sigma^\infty_+\circ \alpha^* \ = \ \alpha^*\circ\Sigma^\infty_+ \ : \ 
G\bT \ \to \ \Sp_K\]
is homotopical. So there is a unique functor
\[ \Ho(\Sigma^\infty_+\circ\alpha^*) \ : \ \Ho^{\Com}(G\bT)\ \to \ \Ho(\Sp_K)\]
such that
\[ \Ho(\Sigma^\infty_+\circ\alpha^*) \circ \gamma_G^{\un}\ = \ 
\gamma_K\circ \Sigma^\infty_+\circ \alpha^* \ : \ G\bT\ \to \ \Ho(\Sp_K)\ , \]
where $\gamma_G^{\un}:G\bT\to\Ho^{\Com}(G\bT)$ is the localization functor.

\index{left derived functor!of restriction|(}
\begin{thm} \label{thm:homomorphism adjunctions spectra} \index{stable homotopy category!equivariant|see{$G$-equivariant stable homotopy category}}\index{G-equivariant stable homotopy category@$G$-equivariant stable homotopy category}
  Let $ \alpha \colon K \to G$ be a continuous homomorphism between Lie groups.
  \begin{enumerate}[\em (i)]
  \item The restriction functor $\alpha^*:\Sp_G\to \Sp_K$
    takes quasi-cofibrant orthogonal $G$-spectra to quasi-cofibrant
    orthogonal $K$-spectra, and it takes $\pi_*$-isomorphisms 
    between quasi-cofibrant orthogonal $G$-spectra to
    $\pi_*$-isomorphisms of orthogonal $K$-spectra.\index{restriction functor}
  \item 
    The restriction functor $\alpha^*:\Sp_G\to\Sp_K$\index{restriction functor}
    has a total left derived functor $(L\alpha^*,\alpha_!)$.
    For every quasi-cofibrant $G$-spectrum $X$, the morphism
    $\alpha_!:(L\alpha^*)(X)\to \gamma_K(\alpha^*(X))$ is an isomorphism in $\Ho(\Sp_K)$.
  \item 
    The derived functor $L\alpha^*$  preserves sums and has a right adjoint.
  \item 
    There is a unique natural transformation
    \[
      \sigma \ : \ (L\alpha^*)\circ [1] \ \Longrightarrow \  [1]\circ (L\alpha^*)
    \]
    of functors $\Ho(\Sp_G)\to\Ho(\Sp_K)$ such that 
    \begin{align} \label{eq:define_sigma}
      ([1]\star \alpha_!)\circ (\sigma\star\gamma_G)  \ = \ 
      &\alpha_!\star(-\sm S^1)\ :\\ 
      (L \alpha^*)&\circ [1]\circ\gamma_G\ \to\ \gamma_K \circ \alpha^*\circ (-\sm S^1)\ . \nonumber
    \end{align}
    Moreover, the transformation $\sigma$ is a natural isomorphism
    and the pair $(L\alpha^*,\sigma)$ is an exact functor of triangulated categories.
  \item
    There is a unique natural transformation
    \[
      \nu\ : \ (L\alpha^*)\circ \Ho(\Sigma^\infty_+) \ \Longrightarrow \  
      \Ho(\Sigma^\infty_+\circ \alpha^*)
    \]
    of functors $\Ho^{\Com}(G\bT)\to\Ho(\Sp_K)$ such that 
    \[ 
    \nu\star\gamma_G^{\un}  \ = \   \alpha_!\star \Sigma^\infty_+\ :\ 
    (L \alpha^*)\circ \Ho(\Sigma^\infty_+)\circ\gamma_G^{\un}\ \to\ 
    \Ho(\Sigma^\infty_+\circ \alpha^*) \circ\gamma_G^{\un}\ .
    \]
    Moreover, $\nu$ is a natural isomorphism.
  \item 
    If $\alpha$ is quasi-injective, then the universal natural transformation
    $\alpha_!:(L\alpha^*)\circ\gamma_G\Longrightarrow \gamma_K\circ\alpha^*$ 
    is an isomorphism, and $L\alpha^*$  preserves products and has a left adjoint.
  \end{enumerate}
\end{thm}
\begin{proof}
  (i) For the first claim we let $X$ be a quasi-cofibrant orthogonal $G$-spectrum
  and $L$ a compact subgroup of $K$. Then the $L$-spectrum $\res^K_L(\alpha^*(X))$
  is the same as $(\alpha|_L)^*(X)$, where $\alpha|_L:L\to G$ is the restricted
  homomorphism. Since $L$ is compact, $\alpha|_L$ has closed image and compact kernel,
  so $(\alpha|_L)^*(X)$ is a cofibrant $L$-spectrum by 
  Theorem \ref{thm:adjunctions spectra pointset} (ii).
  Since $L$ was any compact subgroup of $K$, this proves that $\alpha^*(X)$
  is quasi-cofibrant.

  Now we let $f:X\to Y$ be a $\pi_*$-isomorphism 
  between quasi-cofibrant orthogonal $G$-spectra. 
  We let $L$ be a compact subgroup of $K$.
  We factor the restriction  $\alpha|_L:L\to G$ as
  \[ L \ \xrightarrow{\ \beta\ } \ H \ \xrightarrow{\incl} \ G\ ,\]
  where $H=\alpha(L)$ is the image of $\alpha$, 
  and $\beta$ is the same map as $\alpha|_L$, but with target $H$.
  The group $H$ is compact since $L$ is;
  since $X$ and $Y$ are quasi-cofibrant, their underlying $H$-spectra
  are cofibrant.
  Since $f$ is a $\pi_*$-isomorphism of quasi-cofibrant $G$-spectra, 
  $\res^G_H(f)$ is a $\pi_*$-isomorphism between cofibrant $H$-spectra.
  The continuous epimorphism $\beta:L\to H$
  satisfies the hypotheses of Theorem \ref{thm:adjunctions spectra pointset} (ii) 
  because $L$ is compact;
  so $\beta^*:\Sp_H\to\Sp_L$ is a left Quillen functor for the stable model structures.
  In particular, $\beta^*$ takes $\pi_*$-isomorphisms between cofibrant
  $H$-spectra to $\pi_*$-isomorphisms, by Ken Brown's lemma 
  \cite[Lemma 1.1.12]{hovey:model_categories}.
  So the morphism 
  \[ \res^K_L(\alpha^*(f))\ = \ \beta^*(\res^G_H(f)) \]
  is a $\pi_*$-isomorphism of orthogonal $L$-spectra.
  Since $L$ was an arbitrary compact subgroup of $K$, this proves the last claim.

  (ii) Part (i)
  shows that the restriction functor $\alpha^*:\Sp_G\to \Sp_K$
  takes $\pi_*$-isomorphisms between cofibrant orthogonal $G$-spectra to
  $\pi_*$-isomorphisms of orthogonal $K$-spectra.
  Given this, Quillen's result \cite[I.4, Prop.\,1]{quillen:homotopical_algebra}
  provides the left derived functor and shows that 
  $\alpha_!:(L\alpha^*)(X)\to \alpha^*(X)$ is an isomorphism whenever $X$ is cofibrant,
  compare also Theorem \ref{thm:absolute}.
  
  If $X$ is quasi-cofibrant, we choose a $\pi_*$-isomorphism $f:Y\to X$
  from a cofibrant orthogonal $G$-spectrum.
  Then $\gamma_G(f)$ is an isomorphism in $\Ho(\Sp_G)$, so
  the upper horizontal morphism in the commutative square
  \[
    \xymatrix@C=20mm{ 
      (L\alpha^*)(Y)\ar[r]^-{(L\alpha^*)(\gamma_G(f))} \ar[d]_{\alpha_!^Y} & 
      (L\alpha^*)(X)\ar[d]^{\alpha_!^X}\\ 
      \gamma_K(\alpha^*(Y))\ar[r]_--{\gamma_K(\alpha^*(f))} & \gamma_K(\alpha^*(X))  }
  \]
  is an isomorphism in $\Ho(\Sp_K)$.
  The morphism $\alpha^Y_!$ is an isomorphism because $Y$ is cofibrant.
  The morphism $\gamma_K(\alpha^*(f))$ 
  is an isomorphism because $\alpha^*$ preserves $\pi_*$-isomorphisms
  between quasi-cofibrant spectra, by part (i).
  So $\alpha^X_!$ is an isomorphism.

  (iii)
  We exploit that coproducts in $\Ho(\Sp_G)$ and $\Ho(\Sp_K)$
  are modeled by wedges of equivariant spectra, because formation
  of wedges is fully homotopical.
  We let $\{X_i\}_{i\in I}$ be a family of cofibrant orthogonal $G$-spectra;
  then the wedge $\bigvee X_i$ is also cofibrant.
  So the vertical morphisms in the commutative square
  \[
    \xymatrix{ 
      \bigvee (L\alpha^*)(X_i)\ar[r]^-{\kappa} \ar[d]_{\bigvee \alpha_!^{X_i}} & 
      (L\alpha^*)(\bigvee X_i)\ar[d]^{\alpha_!^{\bigvee X_i}}\\ 
      \bigvee \gamma_K(\alpha^*(X_i))\ar[r]_--{\iso} & \gamma_K(\alpha^*(\bigvee X_i))}
  \]
  are isomorphisms in $\Ho(\Sp_K)$ by part (i).
  Since $\alpha^*$ preserves colimits, the lower morphism is an isomorphism,
  and hence so is the canonical morphism $\kappa$.
  This proves that $L\alpha^*$ preserves sums.
  Corollary \ref{cor:generators for GH_F} (iii) then provides 
  a right adjoint for $L\alpha^*$.
  
  (iv)
  The suspension functor is fully homotopical, 
  so Proposition \ref{prop:homotopical after absolute} 
  shows that the pair $([1]\circ L\alpha^*,[1]\star \alpha_!)$ 
  is an absolute left derived functor of the functor
  \[  \alpha^*(-)\sm S^1\ = \ \alpha^*(-\sm S^1)\ : \ \Sp_G\ \to \ \Sp_K \ .\]
  The universal property of $([1]\circ L\alpha^*,[1]\star\alpha_!)$ 
  thus provides a unique natural transformation 
  $\sigma : (L\alpha^*)\circ [1] \Longrightarrow [1]\circ (L\alpha^*)$
  that satisfies the relation specified in the statement of the theorem.
  The diagram
  \[
    \xymatrix{ 
      (L\alpha^*)(X[1])\ar[r]^-{\sigma_X} \ar[d]_{\alpha_!^{X\sm S^1}} & 
      ((L\alpha^*)(X))[1]\ar[d]^{\alpha_!^X[1]}\\ 
      \gamma_K(\alpha^*(X\sm S^1))\ar@{=}[r] & \gamma_K(\alpha^*(X))[1]}
  \]
  commutes by construction of $\sigma$.
  If $X$ is cofibrant, so is $X\sm S^1$, so both vertical morphisms
  are isomorphisms in $\Ho(\Sp_K)$, by part (ii).
  Hence $\sigma$ is an isomorphism for cofibrant $G$-spectra;
  in $\Ho(\Sp_G)$, every object is isomorphic to a cofibrant spectrum,
  so $\sigma$ is a natural isomorphism.
  
  Every distinguished triangle in $\Ho(\Sp_G)$ is isomorphic to
  the mapping cone triangle \eqref{eq:mapping_cone_triangle} associated with
  a morphism $f:X\to Y$ between cofibrant $G$-spectra.
  So to show that the pair $(L\alpha^*,\sigma)$ preserves distinguished triangles,
  it suffices to show exactness for these special ones. 
  We contemplate the commutative diagram:
  \[
    \xymatrix@C=14mm{ 
      (L\alpha^*)(X) \ar[d]_{\alpha_!^X}
      \ar[r]^-{(L\alpha^*)(f)} & (L\alpha^*)(Y)\ar[r]^-{(L\alpha^*)(i)}\ar[d]_{\alpha_!^Y} & 
      (L\alpha^*)(C f)\ar[r]^-{\sigma_X\circ (L\alpha^*)(p)}\ar[d]^{\alpha_!^{C f}} &
      (L\alpha^*)(X)\sm S^1\ar[d]^{\alpha_!^X\sm S^1}\\
      \gamma_K(\alpha^*(X)) \ar[r]_-{\gamma_K(\alpha^*(f))} &
      \gamma_K(\alpha^*(Y))\ar[r]_-{\gamma_K(\alpha^*(i))} &
      \gamma_K(\alpha^*(C f))\ar[r]_-{\gamma_K(\alpha^*(p))} & \gamma_K(\alpha^*(X))\sm S^1}
  \]
  Since $X$ and $Y$ are cofibrant, so are $C f$ and $X\sm S^1$;
  hence all vertical morphisms are isomorphisms by part (ii).
  The lower triangle is distinguished because the pointset level 
  restriction functor $\alpha^*$ 
  commutes with formation of mapping cones and suspension.
  So the upper triangle is distinguished.
  
  (v)
  The existence and characterization of $\nu$ are just
  the universal property of the pair $(\Ho(\Sigma^\infty_+\circ\alpha^*),\Id)$
  which is a total left derived functor of the homotopical functor
  $\Sigma^\infty_+\circ\alpha^*$.
  The characterizing property means that when we specialize the transformation
  $\nu$ to a $G$-space $A$, we have
  \[ \nu_A \ = \ \alpha_!^{\Sigma^\infty_+ A} \ : \ 
    (L\alpha^*)(\Sigma^\infty_+ A)\ \to 
    \gamma_K(\alpha^*(\Sigma^\infty_+ A))\ . \]
  If $A$ is $\Com$-cofibrant as a $G$-space, then $\Sigma^\infty_+ A$
  is cofibrant as a $G$-spectrum. 
  So in that case, the morphism $\alpha_!^{\Sigma^\infty_+A}$ is
  an isomorphism in $\Ho(\Sp_K)$, by part (ii).
  Hence $\nu$ is an isomorphism for $\Com$-cofibrant $G$-spaces;
  in $\Ho^{\Com}(G\bT)$, every object is isomorphic to a $\Com$-cofibrant $G$-space,
  so $\nu$ is a natural isomorphism.  
  
  (vi)
  If $\alpha$ is quasi-injective, then the restriction functor $\alpha^*$
  is fully homotopical and a right Quillen functor
  by Theorem \ref{thm:adjunctions spectra pointset} (i).
  The universal natural transformation $\alpha_!$
  is an isomorphism because $\alpha^*$ is fully homotopical.
  The functor $L\alpha^*$ has a left adjoint because $\alpha^*$ is a right Quillen functor;
  the left adjoint is a total left derived functor of $G\ltimes_\alpha-:\Sp_K\to\Sp_G$.
\end{proof}

\begin{rk}[Derived inflation and products]\label{rk:inflation and products}\index{G-equivariant stable homotopy category@$G$-equivariant stable homotopy category}
  As the previous theorem shows, the derived functor $L\alpha^*:\Ho(\Sp_G)\to\Ho(\Sp_K)$
  of restriction along a continuous homomorphism $\alpha:K\to G$
  always has a right adjoint, and it has a left adjoint 
  whenever $\alpha$ is quasi-injective.
  While the pointset level restriction functor $\alpha^*:\Sp_G\to\Sp_K$
  preserves products, 
  its left derived functor $L\alpha^*$ does {\em not} preserve products in general.
  
  The simplest example is inflation along the unique homomorphism $p:C_2\to e$
  from a group with two elements to a trivial group.
  In the non-equivariant stable homotopy category the canonical map
  \[ {\bigoplus}_{k < 0}\,  H\mF_2[k] \ \to \ {\prod}_{k < 0}\,  H\mF_2[k]  \]
  from the coproduct to the product of infinitely many desuspended copies
  of the mod-2 Eilenberg-Mac\,Lane spectrum is an isomorphism.
  Since $L p^*$ and equivariant homotopy groups preserves coproducts, 
  the canonical map
  \begin{equation}\label{eq:L_of_product}
    {\bigoplus}_{k<0} \, \pi_0^{C_2} \left( (L p^*)( H\mF_2[k] ) \right)
    \ \to \ \pi_0^{C_2} \left( (L p^*)\left( {\bigoplus}_{k< 0} \, H\mF_2 [k]\right) \right)     
  \end{equation}
  is an isomorphism of abelian groups.
  For every non-equivariant spectrum  $X$, the $C_2$-spectrum $L p^*(X)$
  has `constant geometric fixed points' and the geometric fixed point map 
  $\Phi:\pi_0^{C_2}(L p^*(X))\to\Phi^{C_2}_0(L p^*(X))\iso\pi_0^e(X)$
  has a section, compare \cite[Ex.\,4.5.10]{schwede:global};
  hence the isotropy separation sequence 
  (see \cite[(3.3.9)]{schwede:global}) splits. 
  So the $C_2$-equivariant homotopy groups decompose as
  \begin{align*}
    \pi_0^{C_2} (L p^*(X) ) \ &\iso \ 
                                \pi_0^{C_2} ( L p^*(X)\sm (E C_2)_+) \oplus \Phi_0^{C_2}(L X)  \\
                                &\iso \ \pi_0^e(X\sm (B C_2)_+) \oplus\  \pi_0^e(X)  \ .
  \end{align*}
  The second step uses the Adams isomorphism and the fact that $L p^*(X)$
  has trivial $C_2$-action as a naive $C_2$-spectrum.
  When $X=H\mF_2[k]$ for negative~$k$, then the second summand is trivial
  and hence
  \[ \pi_0^{C_2}( L p^*( H\mF_2 [k]))\ \iso \ 
    \pi_0^e(H\mF_2[k]\sm (B C_2)_+) \ \iso \ H_{-k}(BC_2,\mF_2)  \ .\]
  So the group \eqref{eq:L_of_product} is a countably infinite sum of
  copies of $\mF_2$. On the other hand, 
  \begin{align*}
    \pi_0^{C_2} \left({\prod}_{k< 0} \, L p^*( H\mF_2[k] ) \right) \ 
    &\iso \ {\prod}_{k< 0} \, \pi_0^{C_2} (  L p^*( H\mF_2[k] ))    \\
    \iso \ {\prod}_{k< 0} &\, \pi_0^e \left( H\mF_2[k]\sm (B C_2)_+ \right)   
      \ \iso \  {\prod}_{k< 0}\, H_{-k}(B C_2, H\mF_2 )   \ ,  
  \end{align*}
  again by the split isotropy separation sequence.
  This is an infinite product of copies of $\mF_2$, so the canonical map
  \[ L p^*\left({\prod}_{k<0} \, H\mF_2[k] \right) \ \to \ 
    {\prod}_{k< 0}\, L p^*( H\mF_2[k]) \]
  is not a $\pi_*$-isomorphism of $C_2$-spectra.
\end{rk}
\index{left derived functor!of restriction|)}

\begin{eg}\label{eg:conjugation invariance}
  We give a rigorous formulation of the idea that 
  `inner automorphisms act as the identity'.
  For a Lie group $G$ and $g\in G$, we let
  \[ c_g^* \ : \ \Sp_G \ \to \ \Sp_G \]
  be restriction along the inner automorphism $c_g:G\to G$, $c_g(\gamma)=g^{-1}\gamma g$.
  Since the restriction functor $c_g^*$ is fully homotopical,
  the induced functor
  \[  \Ho(c_g^*) \ : \ \Ho(\Sp_G)\ \to \ \Ho(\Sp_G)\]
  is also a total left derived functor, 
  with respect to the identity natural transformation.
  We exhibit a specific natural isomorphism
  between $\Ho(c_g^*)=L c_g^*$ and the identity functor.
  We let
  \[ l_g \ : \ c_g^*\ \Longrightarrow \ \Id_{\Sp_G}
  \]
  denote the natural isomorphism of functors whose value $l_g^X:c_g^*(X)\to X$
  at an orthogonal $G$-spectrum $X$ is left multiplication by $g$.
  This induces a natural isomorphism of functors
  on $\Ho(\Sp_G)$
  \[ L(l_g)\ : \ \Ho(c_g^*) \ \Longrightarrow \ \Id_{\Ho(\Sp_G)}  \ .\]
\end{eg}

\index{left derived functor|(}
Our next aim is to show that total left derived functors
organize themselves into a `lax functor' (whenever they exist).
Loosely speaking this means that while $L G\circ L F$ need not be isomorphic
to $L(G F)$, there is a preferred natural transformation 
$\td{G,F}:L G \circ L F\Longrightarrow L(G F)$, and these natural transformations
satisfy certain coherence conditions. This is surely well known among experts,
but we were unable to find a complete reference.
There is something to show here because:
\begin{itemize}
\item Not all functors between relative categories are left derivable.
\item If two composable functors between relative categories 
are left derivable, the composite need not be left derivable.
\item If $F:\cC\to\cD$ and $G:\cD\to\cE$ are left derivable functors 
such that $G F$ is also left derivable,
then $L(G F)$ need not be isomorphic to the composite $L G\circ L F$. 
\end{itemize}

\begin{con}
For every relative category $(\cC,w)$, we choose a homotopy category
 $\gamma_\cC:\cC\to\Ho(\cC)$.
We also choose a total left derived functor $(L F,\tau_F)$
for every left derivable functor $F:\cC\to\cD$ between relative
categories. 
As we shall now explain, these choices determine all the remaining coherence data,
without the need to make any further choices.

We let $F:\cC\to\cD$ and $G:\cD\to\cE$ be left derivable functors 
between relative categories. If the composite $G F:\cC\to\cE$
is also left derivable, then the universal property of $(L(G F),\tau_{G F})$
provides a unique natural transformation
\[ \td{G,F}\ : \ L G\circ L F \ \Longrightarrow \ L(G F)   \]
such that
\begin{equation} \label{eq:td_relation}
\tau_{G F}\circ (\td{G,F}\star\gamma_\cC)
 \ = \ 
(\tau_G\star F)\circ(L G\star\tau_F )\ :\ 
L G\circ L F\circ\gamma_\cC \ \to\ \gamma_\cE \circ G\circ F\ . 
\end{equation}
\end{con}
\index{left derived functor|)}

We apply the previous construction to restriction functors
along continuous group homomorphisms.
In a nutshell, the ultimate outcome is that the assignment $\alpha\mapsto L\alpha^*$
extends to a contravariant pseudo-functor 
from the category of Lie groups and continuous homomorphisms
to the 2-category of triangulated categories, 
exact functors, and exact transformations.

\index{left derived functor!of restriction|(}
\begin{con}\label{con:lax transfo}
We let $\alpha:K\to G$ and $\beta:J\to K$ be two composable
continuous homomorphisms between Lie groups.
We let $(L\alpha^*,\alpha_!)$, $(L\beta^*,\beta_!)$
and $(L(\alpha\beta)^*,(\alpha\beta)_!)$
be total left derived functors of $\alpha^*$, $\beta^*$ and $(\alpha\beta)^*$, 
respectively.
The universal property of $(L(\alpha\beta)^*,(\alpha\beta)_!)$
provides a unique natural transformation
\[
 \td{\alpha,\beta}\ : \ L\beta^*\circ L\alpha^* \ \Longrightarrow \
L(\alpha\beta)^*  \]
such that
\begin{align}\label{eq:td(a,b)}
(\beta_!\star\alpha^*)\circ(L\beta^*\star\alpha_!)
 \ &= \ 
(\alpha\beta)_!\circ (\td{\alpha,\beta}\star\gamma_G)\ :\\ 
& L\beta^*\circ L\alpha^*\circ\gamma_G \ \Longrightarrow\ 
\gamma_J \circ\beta^*\circ\alpha^*\ =  \ \gamma_J\circ (\alpha\beta)^* \ .\nonumber
\end{align}
The natural transformations so obtained satisfy a coherence condition:
if $\gamma:M\to J$ is yet another continuous homomorphism, 
then the following square of natural transformations commutes:
\[ \xymatrix@C=20mm{ 
L\gamma^*\circ L\beta^*\circ L\alpha^*
\ar@{=>}[r]^-{L\gamma^*\star \td{\alpha,\beta}} \ar@{=>}[d]_{\td{\beta,\gamma}\star L\alpha^*} & L\gamma^*\circ L(\alpha\beta)^* \ar@{=>}[d]^{\td{\alpha\beta,\gamma}}\\
L(\beta\gamma)\circ L\alpha^*\ar@{=>}[r]_-{\td{\alpha,\beta\gamma}} &
L(\alpha\beta\gamma)^*} \]
Indeed, this is an instance of the general coherence 
for total left derived functors 
that we spell out in the following Proposition \ref{prop:derived cocycle},
applied to the left derivable functors
$\alpha^*:\Sp_G\to\Sp_K$, $\beta^*:\Sp_K\to\Sp_J$, and $\gamma^*:\Sp_J\to\Sp_M$.
\end{con}

\begin{prop}\index{G-equivariant stable homotopy category@$G$-equivariant stable homotopy category}
For all composable continuous homomorphisms $\alpha:K\to G$ and $\beta:J\to K$ 
between Lie groups, the natural transformation 
$\td{\alpha,\beta}:L\beta^*\circ L\alpha^*\Longrightarrow L(\alpha\beta)^*$
is exact and a natural isomorphism.
\end{prop}
\begin{proof}
Exactness of the  transformation $\td{\alpha,\beta}$
means that the following diagram of natural transformations commutes:
\[ \xymatrix@C=15mm{ 
L\beta^*\circ L\alpha^* \circ [1]\ar@{=>}[r]^-{L\beta^*\star \sigma^\alpha} 
\ar@{=>}[d]_{\td{\alpha,\beta}\star [1]}&
L\beta^*\circ [1] \circ L\alpha^* \ar@{=>}[r]^-{\sigma^\beta\star L\alpha^*} &
[1]\circ L\beta^* \circ L\alpha^*  \ar@{=>}[d]^{[1]\star \td{\alpha,\beta}}\\
L(\alpha\beta)^*\circ [1]\ar@{=>}[rr]_-{\sigma^{\alpha\beta}} && 
[1]\circ L(\alpha\beta)^*} \]
Since suspension is fully homotopical,
the pair $([1]\circ L(\alpha\beta)^*,[1]\star(\alpha\beta)_!)$ is a
total left derived functor of  $(-\sm S^1)\circ(\alpha\beta)^*$,
by Proposition \ref{prop:homotopical after absolute}.
The universal property allows us to check the commutativity 
of the diagram after precomposition with the localization functor 
$\gamma_G:\Sp_G\to \Ho(\Sp_G)$ and postcomposition with the natural transformation
$[1]\star(\alpha\beta)_!$.
This is a straightforward, but somewhat lengthy, calculation with 
the various defining properties.
We start with the observation:
\begin{align}\label{eq:subproblem}
 ([1]\star \beta_!\star\alpha^*)
\circ ([1]\star L\beta^*\star \alpha_!) 
&= \  [1]\star (( \beta_!\star\alpha^*)\circ (L\beta^*\star \alpha_!) )  
\\
_\eqref{eq:td(a,b)} &= \  [1]\star ( (\alpha\beta)_!\circ (\td{\alpha,\beta}\star\gamma_G) )
\nonumber \\
&= \  
([1]\star  (\alpha\beta)_!)\circ ([1]\star\td{\alpha,\beta}\star\gamma_G )\nonumber
\end{align}
The final relation is then obtained as follows:
\begin{align*}
([1]\star(\alpha\beta)_!)&\circ
 \left( (\sigma^{\alpha\beta}\circ (\td{\alpha,\beta}\star[1]) )\star\gamma_G \right)
  \\ 
&= \ 
([1]\star(\alpha\beta)_!)\circ(\sigma^{\alpha\beta}\star\gamma_G ) \circ (\td{\alpha,\beta}\star[1]\star\gamma_G)  \\
_\eqref{eq:define_sigma} 
   &= \ 
((\alpha\beta)_!\star(-\sm S^1)) \circ (\td{\alpha,\beta}\star\gamma_G\star(-\sm S^1))  \\
   &= \ 
( (\alpha\beta)_! \circ (\td{\alpha,\beta}\star\gamma_G)) \star(-\sm S^1)  \\
_\eqref{eq:td(a,b)}
   &= \ 
((\beta_!\star\alpha^*)\circ(L\beta^*\star\alpha_!)) \star(-\sm S^1)  \\
   &= \ 
(\beta_!\star\alpha^*\star(-\sm S^1))\circ (L\beta^*\star\alpha_!\star(-\sm S^1)) 
  \\
   &= \ 
(\beta_!\star(-\sm S^1)\star\alpha^*)\circ (L\beta^*\star\alpha_!\star(-\sm S^1)) 
  \\
_\eqref{eq:define_sigma}   &= \ 
( (  ([1]\star \beta_!)\circ (\sigma^\beta\star\gamma_K) )
\star\alpha^*)\circ (L\beta^*\star(([1]\star \alpha_!)\circ (\sigma^\alpha\star\gamma_G)))  
  \\
   &= \ 
 ([1]\star \beta_!\star\alpha^*)\circ (\sigma^\beta\star\gamma_K\star\alpha^*) 
\circ 
(L\beta^*\star[1]\star \alpha_!)\circ (L\beta^*\star\sigma^\alpha\star\gamma_G)
  \\
_\eqref{eq:interchange}   &= \ 
 ([1]\star \beta_!\star\alpha^*)\circ ([1]\star L\beta^*\star \alpha_!) 
\circ (\sigma^\beta\star L\alpha^*\star\gamma_G) \circ (L\beta^*\star\sigma^\alpha\star\gamma_G)
  \\
_\eqref{eq:subproblem}   &= \ 
([1]\star  (\alpha\beta)_!)\circ ([1]\star\td{\alpha,\beta}\star\gamma_G )
\circ (\sigma^\beta\star L\alpha^*\star\gamma_G) 
\circ (L\beta^*\star\sigma^\alpha\star\gamma_G)
  \\
   &= \ 
([1]\star  (\alpha\beta)_!)\circ \left( 
 ( ([1]\star\td{\alpha,\beta})\circ (\sigma^\beta\star L\alpha^*)
\circ (L\beta^*\star\sigma^\alpha))\star\gamma_G\right) 
\end{align*}
To prove that $\td{\alpha,\beta}$ is an isomorphism, 
we consider a quasi-cofibrant orthogonal $G$-spectrum $X$ and contemplate the
following commutative diagram in $\Ho(\Sp_L)$:
\[ \xymatrix@C=10mm{ 
(L\beta^*)((L\alpha^*)(X))
\ar[rr]^-{(L\beta^*)(\alpha_!^X)}_-\iso 
\ar[d]_{\td{\alpha,\beta}^X} && 
(L\beta^*)(\alpha^*(X))\ar@{=}[r] &
(L\beta^*)(\alpha^*(X))\ar[d]_\iso^{\beta_!^{\alpha^*(X)}} \\ 
L(\alpha\beta)^*(X)\ar[rr]^-\iso_-{(\alpha\beta)_!^X} && 
(\alpha\beta)^*(X) \ar@{=}[r] &
\beta^*(\alpha^*(X))} \]
The orthogonal $K$-spectrum $\alpha^*(X)$ is quasi-cofibrant by  
Theorem \ref{thm:homomorphism adjunctions spectra} (i).
Since $X$ and $\alpha^*(X)$ are quasi-cofibrant,
the morphisms $\alpha^X_!$, $(\alpha\beta)^X_!$ and $\beta^{\alpha^*(X)}_!$
are isomorphisms by Theorem \ref{thm:homomorphism adjunctions spectra} (ii).
So the morphism $\td{\alpha,\beta}^X$ is also an isomorphism.
Every orthogonal $G$-spectrum is isomorphic in $\Ho(\Sp_G)$
to a quasi-cofibrant spectrum, so this proves that $\td{\alpha,\beta}$ is
a natural isomorphism.
\end{proof}

The following proposition records the coherence property that the
natural transformations $\td{G,F}$ enjoy;
the condition is essentially saying that these transformations
make the assignment $F\mapsto L F$ into a lax functor
from the `category' of relative categories and left derivable functors
to the 2-category of categories. 
The caveat is that the composite of left derivable functors need not 
be left derivable, so we don't really have a category of these.
As we indicate in Remark \ref{rk:further coherence} below,
this implies that all further coherence conditions 
between tuples of composable derivable functors are automatically satisfied.

\begin{prop}\label{prop:derived cocycle}
Let $F:\cC\to\cD$, $G:\cD\to\cE$ and $H:\cE\to\cF$ be composable
functors between relative categories such that
$F$, $G$, $H$ and $G F$, $H G$ and $H G F$ are left derivable.
Then
\[ \td{H G,F} \circ  (\td{H,G}\star L F) \ = \ 
\td{H, G F}\circ(L H\star \td{G, F})
 \]
as natural transformations 
$L H\circ L G\circ L F\Longrightarrow L(H G F)$.
\end{prop}
\begin{proof}
The various defining relations and the interchange relation
provide the following equalities of natural transformations
between functors $\cC\to \Ho(\cF)$:
\begin{align*}
\tau_{H G F}\circ (( \td{H G,F} \circ & (\td{H,G}\star L F))\star\gamma_\cC) \\ 
&= \ \tau_{H G F}\circ  (\td{H G,F}\star\gamma_\cC) \circ (\td{H,G}\star (L F\circ\gamma_\cC)) \\ 
\eqref{eq:td_relation} &= \ (\tau_{H G}\star F)\circ  (L(H G)\star\tau_F) \circ (\td{H,G}\star (L F\circ\gamma_\cC)) \\ 
\eqref{eq:interchange} &= \ (\tau_{H G}\star F)\circ  
( \td{H,G}\star(\gamma_\cD\circ F))\circ( (LH \circ L G)\star\tau_F) \\ 
 &= \ (\tau_{H G}\star F)\circ  
( (\td{H,G}\star\gamma_\cD)\star F)\circ( (LH \circ L G)\star\tau_F) \\ 
 &= \ 
((\tau_{H G}\circ  (\td{H,G}\star\gamma_\cD))\star F)\circ( (L H \circ L G)\star\tau_F) \\ 
\eqref{eq:td_relation} &= \ 
(((\tau_H\star G)\circ (L H\star\tau_G))\star F)\circ( (L H \circ L G)\star\tau_F) \\ 
&= \ (\tau_H\star(G F))\circ  
( L H\star \tau_G\star F)\circ( (L H \circ L G)\star\tau_F) \\ 
&= \ (\tau_H\star(G F))\circ  
( L H\star( (\tau_G\star F)\circ( L G\star\tau_F))) \\ 
\eqref{eq:td_relation}&= \ (\tau_H\star(G F))\circ  
( L H\star( \tau_{G F}\circ(\td{G,F}\star\gamma_{\cC}))) \\ 
&= \ (\tau_H\star(G F))\circ  
( L H\star \tau_{G F})\circ ( L H \star\td{G,F}\star\gamma_\cC) \\ 
\eqref{eq:td_relation} &= \  \tau_{H G F}\circ( \td{H, G F}\star\gamma_\cC)\circ (L H\star \td{G,F}\star\gamma_\cC) \\
&= \  \tau_{H G F}\circ( (\td{H, G F}\circ (L H\star \td{G,F}))\star\gamma_\cC) \ .
\end{align*}
The uniqueness clause in the universal property of the pair $(L(H G F),\tau_{H G F})$
then implies the desired relation.
\end{proof}

\begin{rk}\label{rk:further coherence}
Proposition \ref{prop:derived cocycle} implies that
all coherence relations with respect to iterated composition
of derivable functors are automatically satisfied. 
Given $n$ composable, left derivable functors
$F_1,\dots,F_n$, for $n\geq 3$, we define a natural transformation
\[ \td{F_n,\dots, F_1}\ : \ L F_n\circ \dots\circ L F_1 \ \Longrightarrow \ 
L(F_n\circ\dots\circ F_1) \]
inductively by setting
\[ \td{F_n,F_{n-1},\dots, F_1}\ = \ 
 \td{F_n F_{n-1},F_{n-2},\dots, F_1}\circ
(\td{F_n,F_{n-1}}\star (L F_{n-1}\circ\dots\circ L F_1))\ , \]
assuming that all iterated composites of adjacent functors 
are left derivable.
Proposition \ref{prop:derived cocycle} implies that
we could have instead `spliced' at any other intermediate spot 
in the composition; 
more generally, we could have `collected adjacent factors' 
in any way we like, and get the same result.
More precisely, if for $i\leq j$ we write
\[ F_{[j,i]}\ = \ F_j\circ F_{j-1}\circ\dots\circ F_i\ , \]
and we set $\td{F_k}=\Id_{L F_k}$, then for all sequences of numbers
$1\leq m_1< m_2 <\dots < m_k < n$, we have
\begin{align*}
 &\td{F_n,\dots, F_1}\ = \\ 
 &\quad\td{F_{[n,m_k+1]},\dots, F_{[m_2,m_1+1]}, F_{[m_1,1]}}
\circ
\left( \td{F_n,\dots,F_{m_k+1}}\star \dots\star\td{F_{m_1},\dots,F_1}\right) \ .  
\end{align*}
\end{rk}
\index{left derived functor!of restriction|)}

The next Theorem \ref{thm:weak equivalence} is a homotopy invariance statement 
for proper equivariant stable homotopy theory;
it says, roughly speaking, that this homotopy theory
only depends on the Lie group `up to multiplicative weak equivalence'.
In the context of {\em compact} Lie groups, this statement does not
have much content: as we recall in the next proposition,
every multiplicative weak equivalence between compact Lie groups 
is already an isomorphism.

\begin{prop}\label{prop:compact rigid}
Let $\alpha:K\to H$ be a continuous homomorphism between compact Lie groups.
If $\alpha$ is a weak homotopy equivalence of underlying spaces,
then $\alpha$ is a diffeomorphism, and hence an isomorphism of Lie groups.  
\end{prop}
\begin{proof}
This result should be well-known, but we have not found a reference.
We owe this proof to George Raptis.  
Since $\alpha$ is a weak homotopy equivalence, 
it induces an isomorphism $\pi_0(\alpha):\pi_0(K)\to\pi_0(H)$
of component groups, and a weak homotopy equivalence
$\alpha^\circ:K^\circ\to H^\circ$ on the connected components of the identity elements.
So by restriction to identity components,
it suffices to treat the special case where $K$ and $H$ are path connected.

Since $\alpha$ is a weak homotopy equivalence, it
induces an isomorphism on mod-2 homology.
Since $K$ and $H$ are closed connected manifolds, 
their geometric dimension can be recovered as the 
largest dimension in which the mod-2 homology is non-trivial.
So $K$ and $H$ have the same dimension.

Since $\alpha$ induces an isomorphism on the top dimensional mod-2
homology groups, it must thus be surjective.
A continuous homomorphism between Lie groups is automatically smooth.
Since $K$ and $H$ have the same dimension,
the kernel of $\alpha$ is finite, and so $\alpha$ is a covering space projection.
Since $\alpha$ induces an isomorphism of fundamental groups,
this covering space projection must be a homeomorphism.
Hence $\alpha$ is an isomorphism of Lie groups.
\end{proof}

When we drop the compactness hypothesis, 
Proposition \ref{prop:compact rigid} ceases to hold.
For example, for every Lie group $G$, the projection $G\times\mR\to G$
is a continuous homomorphism and weak equivalence.

\begin{thm}\label{thm:weak equivalence}\index{G-equivariant stable homotopy category@$G$-equivariant stable homotopy category}
Let $\alpha:K\to G$ be a continuous homomorphism between Lie groups
that is also a weak equivalence of underlying spaces. Then the following hold:
\begin{enumerate}[\em (i)]
\item The homomorphism $\alpha$ is quasi-injective.
\item For every compact subgroup $H$ of $G$ there is a compact subgroup $L$
of $K$ such that $H$ is conjugate to $\alpha(L)$.
\item The restriction functor $L\alpha^*:\Ho(\Sp_G)\to\Ho(\Sp_K)$
is an equivalence of triangulated categories.\index{left derived functor!of restriction}
\end{enumerate}
\end{thm}
\begin{proof}
We start with the special case when the groups $K$ and $G$ are almost connected.
We choose a maximal compact subgroup $M$ of $K$
and a maximal compact subgroup $N$ of $G$ that contains the compact group $\alpha(M)$.
In the commutative diagram of continuous group homomorphisms
\[ \xymatrix{ 
M \ar[r]^-{\alpha|_M}\ar[d] & N\ar[d] \\
K \ar[r]_-{\alpha} & G } \]
the two vertical inclusions are then weak equivalence of underlying spaces,
compare \cite[Thm.\,A.5]{abels:parallelizability}.
Since $\alpha$ is also a weak equivalence, so is $\alpha|_M:M\to N$.
But $M$ and $N$ are compact, so $\alpha|_M:M\to N$ is an isomorphism
by Proposition \ref{prop:compact rigid}.
In particular, $\alpha$ is injective on $M$;
since every compact subgroup of $K$ is contained in a maximal compact subgroup,
$\alpha$ is quasi-injective.
This proves claim (i) in the special case.
Moreover, if $H$ is a compact subgroup of $G$, then there is an element $g\in G$
such that $^g H\subset N$.
Then $L=\alpha^{-1}(^g H)\cap M$ is a compact subgroup of $K$ such that
$\alpha(L)={^g H}$. This proves claim (ii) in the special case.

Now we contemplate the square of triangulated categories and exact functors
\[ \xymatrix@C=15mm{ 
\Ho(\Sp_M)  & \Ho(\Sp_N) \ar[l]_-{L(\alpha|_M^*)}\\
\Ho(\Sp_K) \ar[u]^{\res^K_M} & \Ho(\Sp_G)\ar[u]_{\res^G_N}\ar[l]^-{L\alpha^*}} \]
The two vertical functors are equivalences by Theorem \ref{thm:reduction},
and the upper functor is an equivalence because $\alpha|_M$ is an isomorphism.
Since the square commutes up to natural isomorphism, the functor $L\alpha^*$
is also an equivalence. This proves claim (iii) in the special case.

Now we treat the general case, i.e., $K$ and $G$ are allowed to have infinitely
many components.
For a compact subgroup $L$ of $K$, we write
\[ \td{L}\ = \ L\cdot K^\circ \]
for the subgroup generated by $L$ and the connected component $K^\circ$
of the identity. Another way to say this is that $\td{L}$
is the union of all path components of $K$ that have a non-empty
intersection with $L$.
By construction, the inclusion $L\to\td{L}$ induces a surjection
$\pi_0(L)\to \pi_0(\td{L})$ on path components. 
In particular, the component group of $\td{L}$ is finite because $L$ is compact.
Similarly, the group $\td{\alpha(L)}=\alpha(L)\cdot G^\circ$ is a closed almost
connected subgroup of $G$.
Then $\td{L}$ is a union of finitely many of the path components of $K$, 
and $\td{\alpha(L)}$ is the union of the corresponding path components of $G$;
the restriction $\bar\alpha:\td{L}\to \td{\alpha(L)}$
of the homomorphism $\alpha$ is thus again a weak equivalence of underlying spaces.
Property (i) for the restriction 
$\bar\alpha:\td{L}\to \td{\alpha(L)}$ shows that $\alpha$ is injective on $L$.
Since $L$ was an arbitrary compact subgroup of $K$, the homomorphism
$\alpha$ is quasi-injective. This proves claim (i) in the general case.

Now we let $H$ be a compact subgroup of $G$, 
and we consider the
almost connected subgroup $\td{H}=H\cdot G^\circ$ of $G$.
Because $\pi_0(\alpha):\pi_0(K)\to\pi_0(G)$ is an isomorphism, 
$\alpha^{-1}(\td{H})=\td{\alpha^{-1}(H)}=\alpha^{-1}(H)\cdot K^\circ$
is almost connected, 
and the restriction $\bar\alpha:\alpha^{-1}(\td{H})\to \td{H}$
of $\alpha$ is another weak equivalence of underlying spaces.
Property (ii) for the restriction 
$\bar\alpha$ provides a compact subgroup $L$ of $\alpha^{-1}(\td{H})$
and an element $g\in \td{H}\subset G$ such that $\alpha(L)={^g H}$.
This proves claim (ii) in the general case.

It remains to show that the functor $L\alpha^*:\Ho(\Sp_G)\to\Ho(\Sp_K)$
is an equivalence. In a first step we show that it detects isomorphisms. 
Since $L\alpha^*$ is an exact functor
of triangulated categories, it suffices to show
for every $G$-spectrum $X$ such that  $L\alpha^*(X)$ is a zero object in $\Ho(\Sp_K)$,
already $X=0$.
To see this we let $H$ be any compact subgroup of $G$. 
Part (ii) provides a compact subgroup $L$ of $K$ 
and an element $g\in G$ such that $\alpha(L)={^g H}$.
Since $\alpha$ is quasi-injective by part (i), it restricts
to an isomorphism $\alpha|_L:L\iso {^g H}$.
Since $L\alpha^*(X)$ is a zero object, we conclude that
\[ \pi_*^H(X) \ \xrightarrow[\iso]{\ c_g^*\ } 
\pi_*^{^g H}(X) \ \xrightarrow[\iso]{L(\alpha|_L)^*} 
\ \pi_*^L(L(\alpha|_L)^*(\res^G_{^g H}(X)))
\ \iso \ \pi_*^L(L\alpha^*(X))\ = \ 0 \ .\]
Since $H$ was an arbitrary compact subgroup of $G$, this proves that $X$
is a zero object in $\Ho(\Sp_G)$.
We have thus shown that  $L\alpha^*$ detects isomorphisms. 

The homomorphism $\alpha$ is quasi-injective by part (i),
so the derived functor $L\alpha^*$ has a left adjoint
$G\ltimes^L_\alpha-:\Ho(\Sp_K)\to\Ho(\Sp_G)$
by Theorem \ref{thm:homomorphism adjunctions spectra} (vi),
which is a total left derived functor of 
$G\ltimes_\alpha-:\Sp_K\to\Sp_G$.
We let $\cY$ denote the class of orthogonal $K$-spectra $Y$
such that the adjunction unit
\[ \eta_Y\ : \ Y\ \to \ L\alpha^*(G\ltimes^L_\alpha Y) \]
is an isomorphism in $\Ho(\Sp_K)$.
The adjunction unit is a natural transformation between two exact functors
that preserve arbitrary sums, so the class $\cY$ is a localizing subcategory
of $\Ho(\Sp_K)$.

Now we consider any almost connected closed subgroup $\bar K$ of $K$ 
with $(\bar K)^\circ=K^\circ$;
in other words, $\bar K$ is a finite union of path components of $K$.
We show that for every orthogonal $\bar K$-spectrum $Z$,
the induced spectrum $K\ltimes_{\bar K} Z$ belongs to the class $\cY$.
We let $\bar G=\td{\alpha(\bar K)}=\alpha(\bar K)\cdot G^\circ$
be the union of those path components of $G$ that correspond
to $\bar K$ under the isomorphism $\pi_0(\alpha):\pi_0(K)\to \pi_0(G)$.
Then $\bar\alpha=\alpha|_{\bar K}:\bar K\to\bar G$ is
a continuous homomorphism and weak equivalence between almost connected
Lie groups.
So property (iii) for the homomorphism $\bar\alpha$ shows that
the adjunction unit
\[ \eta_Z \ : \ Z \ \to \ L\bar\alpha^*(\bar G\ltimes^L_{\bar\alpha} Z) \]
is an isomorphism in $\Ho(\Sp_{\bar K})$.
Hence the left vertical morphism in the following commutative square 
is an isomorphism:
\begin{equation}  \begin{aligned}\label{eq:reduction square}
 \xymatrix@C=23mm{ 
K\ltimes_{\bar K} Z \ar[r]^-{\eta_{K\ltimes_{\bar K} Z}} \ar[d]^\iso_{K\ltimes_{\bar K}\eta_Z}& 
L\alpha^*(G\ltimes^L_\alpha(K\ltimes_{\bar K} Z))\ar[d]_\iso^{L\alpha^*(\mu)}\\
K\ltimes_{\bar K}(L\bar\alpha^*(\bar G\ltimes^L_{\bar\alpha} Z))
\ar[r]^-\iso_{\Ho(\lambda)(\bar G\ltimes^L_{\bar\alpha}Z)} &
  L\alpha^*(G\ltimes_{\bar G}(\bar G\ltimes^L_{\bar\alpha} Z)) 
}     
 \end{aligned}\end{equation}
Here $\mu:G\ltimes^L_\alpha (K\ltimes_{\bar K} Z)\to
G\ltimes_{\bar G}(\bar G\ltimes^L_{\bar\alpha} Z)$ is the mate (adjoint) of the isomorphism
\[
L\bar\alpha^*\circ \res^G_{\bar G}
\  \xrightarrow[\iso]{\td{\incl_{\bar G}^G,\bar\alpha}} \
L(\alpha|_{\bar K})^* 
\ \xleftarrow[\iso]{\td{\alpha,\incl_{\bar K}^K}} \ 
 \res^K_{\bar K}\circ L\alpha^*
\ .\]
Hence $\mu$ and $L\alpha^*(\mu)$ are isomorphisms.
The lower horizontal isomorphism needs some explanation.
On the pointset level, we can define a natural isomorphism of orthogonal
$K$-spectra
\[  
\lambda \ : \ 
K\ltimes_{\bar K}\bar\alpha^*(W)\ \xrightarrow{\ \iso \ }\   \alpha^*(G\ltimes_{\bar G} W ) 
\text{\quad by\quad}
\lambda(k\sm w)\ = \ \alpha(k)\sm w\ ,
\]
where $W$ is any orthogonal $\bar G$-spectrum.
To see that $\lambda$ is indeed an isomorphism, we observe that
the underlying non-equivariant spectrum of $K\ltimes_{\bar K}\alpha^*(W)$ 
is a wedge of copies of $W$ indexed by $K/\bar K$,
that  $\alpha^*(G\ltimes_{\bar G} W )$
is a wedge of copies of $W$ indexed by $G/\bar G$,
and that $\alpha$ induces a bijection of sets $K/\bar K\iso G/\bar G$
because $\pi_0(\alpha):\pi_0(K)\to\pi_0(G)$ is a group isomorphism.
The induction functors $K\ltimes_{\bar K}-$ and $G\ltimes_{\bar G}-$ 
are fully homotopical by Corollary \ref{cor:Gamma2G};
the restriction functors $\alpha^*$ and $\bar\alpha^*$
are fully homotopical by Theorem \ref{thm:adjunctions spectra pointset} (i) 
because $\alpha$ is quasi-injective by part (i).
So $\lambda$ is a natural isomorphism between homotopical functors, 
and hence it descends to a natural isomorphism 
\[ \Ho(\lambda)\ :\ (K\ltimes_{\bar K}-)\circ L\bar\alpha^*\ \iso\ 
L\alpha^*\circ (G\ltimes_{\bar G} -) \]
on the level of homotopy categories.
Now we can wrap up: since the other three morphisms in the commutative
square \eqref{eq:reduction square} are isomorphisms, so is the adjunction
unit $\eta_{K\ltimes_{\bar K} Z}$. In other words,
all $K$-spectra that are induced from subgroups of the form $\bar K$
are in the class $\cY$.

Now we let $L$ be any compact subgroup of $K$.
Then $\bar K=\td{L}=L\cdot K^\circ$ is an almost connected closed
subgroup of $K$ of the kind considered in the previous paragraph.
Since the suspension spectrum $\Sigma^\infty_+ K/L$
is isomorphic to $K\ltimes_{\bar K}(\Sigma^\infty_+ \bar K/L)$,
it is contained in the class $\cY$.
So the class $\cY$ contains all the preferred compact generators
of the triangulated category $\Ho(\Sp_K)$.
Since $\cY$ is a localizing subcategory of $\Ho(\Sp_K)$,
it contains all orthogonal $K$-spectra. In other words,
the adjunction unit $\eta_Y: Y\to L\alpha^*(G\ltimes^L_\alpha Y)$
is an isomorphism in complete generality.

Since the adjunction unit of the adjoint pair $(G\ltimes^L_\alpha-,L\alpha^*)$
is a natural isomorphism and the right adjoint functor $L\alpha^*$ detects isomorphisms,
we finally conclude that the adjunction is an adjoint equivalence
of categories. This proves property (iii) in general.
\end{proof}

Now we discuss another aspect of the homotopy invariance of 
proper equivariant stable homotopy theory, namely homotopy invariance of derived
restriction functors:
we will show that a homotopy through continuous homomorphisms
provides an isomorphism of derived restriction functors.
In the realm of {\em compact} Lie groups, the homotopy invariance
is a direct consequence of the conjugation invariance
of Example \ref{eg:conjugation invariance}:
a celebrated theorem of 
Montgomery and Zippin \cite[Thm.\,1 and Corollary]{montgomery-zippin:Lie}
says that in a Lie group `nearby compact subgroups are conjugate',
and this implies that 
two homotopic continuous homomorphisms from a compact Lie group
to a Lie group are already conjugate, 
compare \cite[III, Lemma 38.1]{conner-floyd:differentiable_periodic}.
If we drop the compactness hypothesis, this statement need not hold anymore:
the identity of the additive Lie group $\mR$ is homotopic, 
through continuous homomorphisms, to the zero homomorphism.
So the homotopy invariance of derived restriction functors
requires an additional argument.

\begin{con}[Homotopy invariance of derived restriction]
We let $\alpha,\beta:K\to G$ be two continuous homomorphisms between Lie groups.
We let 
\[ \omega\ : \ K\times [0,1]\ \to \ G \]
be a homotopy from $\alpha$ to $\beta$ through continuous homomorphisms,
i.e., such that $\omega(-,t):K\to G$ is a homomorphism for every $t\in[0,1]$.    
We define a functor
\[ \omega^* \ : \ \Sp_G \ \to \ \Sp_K  \text{\qquad by\qquad}
 \omega^*(X) \ = \ X\sm [0,1]_+\]
equipped with $K$-action by
\[ k\cdot (x\sm t)\ = \ (\omega(k,t)\cdot x)\sm t\ . \]
By hypothesis we have $\omega(k,0) = \alpha(k)$ and 
$\omega(k,1) = \beta(k)$,
so the two assignments
\begin{align*}
   a\ : \ \alpha^*(X)\ &\to \ \omega^*(X) \ , \quad a(x)\ = \ x\sm 0\text{\quad and\quad}\\
  b\ : \ \beta^*(X)\ &\to \ \omega^*(X) \ , \quad b(x)\ = \ x\sm 1
\end{align*}
define natural morphisms of orthogonal $K$-spectra.
\end{con}

\begin{thm}\label{thm:homotopic homomorphisms}
  Let $\omega: K\times [0,1]\to G$ be a homotopy of continuous homomorphisms
  from $\alpha=\omega(-,0)$ to $\beta=\omega(-,1)$.
  Then for every orthogonal $G$-spectrum $X$,  the morphisms
  \[ 
  a\ : \ \alpha^*(X)\ \to \ \omega^*(X) 
  \text{\qquad and\qquad}
  b\ : \ \beta^*(X)\ \to \ \omega^*(X) 
  \]
  are $\pi_*$-isomorphisms of orthogonal $K$-spectra.
  Hence $a$ and $b$ induce natural isomorphisms of total left derived functors
  \[ L\alpha^*\ \xrightarrow[\iso]{\ L a\ }\ L\omega^* \ 
  \xleftarrow[\iso]{\ L b\ }\ L\beta^* \ . \]
\end{thm}
\begin{proof}
We show that $a:\alpha^*(X)\to\omega^*(X)$ is a $\pi_*$-isomorphism;
the argument for $b$ is analogous.
We let $L$ be any compact subgroup of $K$.
We let $\Gamma=\alpha(L)\cdot G^\circ$ be the union of those path components of $G$ that are
in the image of $\pi_0(\alpha|_L):\pi_0(L)\to\pi_0(G)$.
Then $\alpha|_L$ has image in $\Gamma$, by construction,
and $\pi_0(\Gamma)$ is finite.
By continuity, each of the homomorphisms $\omega(-,t):K\to G$
also takes $L$ to $\Gamma$.
We can thus restrict $\omega$ to a path 
\[ \bar\omega \ = \ \omega|_{L\times [0,1]}\ : \ L\times [0,1]\ \to \ \Gamma \]
of continuous homomorphisms from  $\alpha|_L$ to $\beta|_L$.

We let $\hom(L,\Gamma)$ denote the space of continuous homomorphisms 
with the  subspace topology of the compact-open topology
(which coincides with the function space topology in the category $\bT$).
We let $\Gamma^\circ$ denote the identity path component of the group $\Gamma$;
then $\Gamma^\circ$ is also the identity path component of $G$.
Since $L$ is compact, the image of the continuous map
\[ \Gamma^\circ \ \to \ \hom(L,\Gamma)\ , \quad g \ \longmapsto \ c_g\circ \alpha|_L \]
is the entire path component $\hom(L,\Gamma;\alpha)$ of $\alpha|_L$, 
see \cite[Prop.\,A.25]{schwede:global}.
So the map factors over a continuous bijection
\[ \Gamma^\circ/ C \ \iso \ \hom(L,\Gamma;\alpha)\ , \]
where $C=\Gamma^\circ\cap C_{\Gamma}(\alpha(L))$ is the centralizer in $\Gamma$
of $\alpha(L)$, intersected with $\Gamma^\circ$.
Moreover, this map is a homeomorphism, 
for example by \cite[Thm.\,B.2]{luck_uribe:equivariant_principal}.
The projection $\Gamma^\circ\to \Gamma^\circ/C$
is a locally trivial fiber bundle, 
see for example \cite[I Thm.\,4.3]{brocker_tomdieck:representations_compact}.
So the path
\[ \omega|_L\ : \ [0,1]\ \to \  \hom(L,\Gamma;\alpha) \]
admits a continuous lift
\[ \lambda\ : \ [0,1]\ \to \ \Gamma^\circ \]
such that $\lambda(0)=1$ and $\omega|_L(t)=c_{\lambda(t)}\circ \alpha|_L$.
The map
\[ \omega^*(X)\ \to \ (\alpha|_L)^*(X)\sm [0,1]_+ \ , \quad x\sm t\ \longmapsto \
\lambda(t)\cdot x\sm t \]
is then an $L$-equivariant isomorphism of orthogonal $L$-spectra.
Moreover, under this isomorphism, the restriction to $L$ of
the morphism $a:\alpha^*(X)\to \omega^*(X)$
becomes the morphism $-\sm 0:(\alpha|_L)^*(X)\to (\alpha|_L)^*(X)\sm [0,1]_+$.
This proves that $a$ is an $L$-equivariant homotopy equivalence.
Since $L$ was an arbitrary compact subgroup of $K$, we have altogether
shown that the morphism $a$ is a $\pi_*$-isomorphism of orthogonal $K$-spectra.
\end{proof}

\begin{rk}\index{left derived functor!of restriction}
The left derived functor $L\alpha^*:  \Ho(\Sp_G)\to \Ho(\Sp_K)$
associated to a continuous homomorphism between Lie groups
is also compatible with derived smash products, i.e.,
it can be given a preferred strong symmetric monoidal structure.
The essential ingredient for this is the fact that the pointset level
restriction functor $\alpha^*$ preserves quasi-cofibrant spectra
and $\pi_*$-isomorphisms between quasi-cofibrant spectra,
and that quasi-cofibrant spectra are quasi-flat, i.e., the smash product
is fully homotopical on quasi-cofibrant spectra.
We will not go into more details about multiplicative
aspects of $L\alpha^*$.
\end{rk}

\chapter{Equivariant homotopy groups}

\section{\texorpdfstring{$G$}{G}-equivariant homotopy groups}

For a {\em compact} Lie group~$H$, 
the $H$-equivariant homotopy group $\pi_0^H(X)$
is defined as a colimit over all finite-dimensional $H$-sub\-represen\-tations $V$
of a complete $H$-universe, of the sets~$[S^V, X(V)]^H$. 
In this section we propose a generalization of these equivariant homotopy groups
to arbitrary Lie groups, not necessarily compact,
as the morphisms from the $G$-sphere spectrum $\mS_G$
in the triangulated homotopy category
$\Ho(\Sp_G)$, see Definition \ref{def:G homotopy groups}.
In contrast to the classical case of {\em compact} Lie groups,
in our context the $G$-equivariant homotopy groups need not
send wedges to direct sums.
Equivalently, the $G$-sphere spectrum need not be a small object
in the triangulated category $\Ho(\Sp_G)$.
The question of whether or not $\mS_G$ is small turns out to be
related to finiteness properties of $G$.
As an example we show in Proposition \ref{prop:EG small - S_G small}
that $\mS_G$ is small if the group admits a finite $G$-CW-model
for the universal $G$-space $\uEG$ for proper actions, i.e.,
one with only finitely many equivariant cells
(or equivalently, with compact orbit space).
Example \ref{eg:S_G not small} illustrates that for this particular purpose,
the existence of a {\em finite-dimensional} model for $\uEG$
is not sufficient.
A precise characterization of when $\mS_G$ is a small object of $\Ho(\Sp_G)$
appears in the work of B{\'a}rcenas and the first and fourth author,
see \cite[Thm\,5.1 and Thm.\,5.4]{barcenas-degrijse-patchkoria:stable_finiteness}.

We also discuss the natural structure that relates the $G$-equivariant
stable homotopy groups, specifically restriction homomorphisms
associated with continuous homomorphisms between Lie groups
(see Construction \ref{con:define restriction})
and transfer maps for finite index inclusions (see Construction \ref{con:transfer}).
Transfer maps are closely related to the Wirthm{\"u}ller isomorphism,
which we formulate and prove for finite index inclusions in 
Theorem \ref{thm:restriction adjoints}.
In the last part of this section we verify that the standard relations
between restriction homomorphisms and transfer maps generalize
from the classical context of compact Lie groups to our more general situation.\medskip

We refer to Remark \ref{rk:[k]} for our convention about the shifts $X[k]$
of an orthogonal $G$-spectrum, for $k\in\mZ$.

\index{equivariant homotopy groups|(}
\begin{defn}\label{def:G homotopy groups} 
Let~$G$ be a Lie group, $X$ an orthogonal $G$-spectrum
and~$k\in\mZ$.
The~$k$-th {\em $G$-homotopy group} is defined as
\[ \pi_k^G(X)\ = \ [ \mS_G,X[-k] ]^G \]
the group of morphisms, in the stable homotopy category $\Ho(\Sp_G)$,
from the $G$-sphere spectrum to the $(-k)$-fold shift of $X$.\index{G-equivariant stable homotopy category@$G$-equivariant stable homotopy category}
\end{defn}

If $G$ is compact, then the new definition of $\pi_*^G$ agrees with the old one, 
up to a specific natural isomorphism.
Indeed, for compact $G$, the representability result of Proposition~\ref{prop:rephtpy}  
provides a natural isomorphism
\[    [\mS_G,X[0]]\ \iso \ [\Sigma^{\infty}_+ G/G, X]^G \ \cong\ \pi_0^G(X)\ ,
\quad[f]\ \longmapsto \ f_*(1)\ .\]
Here we identified $\Sigma^\infty_+ G/G$ with $\mS_G$,
so that the tautological class $u_G\in\pi_0^G(\Sigma^\infty_+ G/G)$
becomes the class $1\in\pi_0^G(\mS_G)$ represented by the identity of $S^0=(\mS_G)(0)$.
For $k>0$, we combine this with the iterated loop isomorphism \eqref{eq:loop iso} 
into a natural isomorphism
\[    [\mS_G,X[-k]]\ =\ [\mS_G, \Omega^k X]^G \ \cong\ \pi_0^G(\Omega^k X)\ 
\iso \  \pi_k^G(X)\ .\]
For $k<0$, we instead use the 
iterated suspension isomorphism \eqref{eq:suspension iso} 
and obtain a natural isomorphism
\[    [\mS_G,X[-k]]\ =\ [\mS_G, X\sm S^{-k}]^G \ \cong\ \pi_0^G(X\sm S^{-k})\ 
\iso \  \pi_k^G(X)\ .\]

If $G$ is discrete and admits a finite $G$-CW-model for $\uEG$,
then the definition for compact groups generalizes, provided we replace
`$G$-representations' by `$G$-vector bundles over $\uEG$'.
Indeed, Theorem \ref{thm:rep2vect} (iv) below provides an isomorphism
\[  \pi_0^G(X)\ \iso \   X_G(\uEG)  \ \xrightarrow[\iso]{\mu_{\uEG}^X}\ X_G\gh{\uEG} \ ;\]
the right hand side is defined via fiberwise $G$-homotopy classes 
of $G$-maps $S^\xi\to X(\xi)$, for $G$-vector bundles $\xi$ over $\uEG$,
see Construction \ref{con:define alternative} below.

\begin{rk}
  Let $G$ be a discrete group.
  In Example \ref{eg:Mackey of G-spectrum}\index{G-Mackey functor@$G$-Mackey functor!of an orthogonal $G$-spectrum} below we introduce the graded homotopy group $G$-Mackey functor
  $\upi_*(X)$ of an orthogonal $G$-spectrum $X$.
  This algebraic object records the values of the $H$-equivariant homotopy groups
  for all finite subgroups $H$ of $G$, and the natural structure between them.
  If the group $G$ itself is infinite, then the groups $\pi_*^G(X)$
  are {\em not} encoded in the $G$-Mackey functor $\upi_*(X)$.

  In this situation, there is an Atiyah-Hirzebruch type spectral sequence
  converging to $\pi_*^G(X)$ whose $E_2$-term is the Bredon cohomology of\index{universal proper $G$-space}
  $\uEG$ with coefficients in the homotopy group Mackey functors of $X$.
  Indeed, the $G$-equivariant homotopy group $\pi_{-n}^G(X)$
  is canonically isomorphic to the group $X^n_G(\uEG)$.
  The Atiyah-Hirzebruch spectral sequence~\eqref{eq:AHSS},
  for $X=\uEG$ and for the proper cohomology theory $X^*_G$ 
  represented by $X$, thus takes the form\index{Atiyah-Hirzebruch spectral sequence}
  \begin{equation}\label{eq:AHSS_for_sphere}
    E_2^{p,q} \ = \ H^p_G(\uEG, \upi_{-q}(X))\ \Longrightarrow \ \pi_{-p-q}^G(X) \ .  
  \end{equation}
  If $G$ has an $n$-dimensional model for $\uEG$,
  or -- more generally -- an $n$-dimensional {\em stable} model for $\uEG$,
  then the Mackey functor cohomological dimension of $G$ is at most $n$,
  see for example \cite[Thm.\,1.2]{barcenas-degrijse-patchkoria:stable_finiteness}.
  So in this case, the $E_2$-term of the Atiyah-Hirzebruch spectral sequence 
  \eqref{eq:AHSS_for_sphere} vanishes for $p>n$, and the spectral sequence collapses at $E_{n+1}$.
  
  The 0-th Bredon cohomology group of $\uEG$ is the inverse limit over the
  $\Fin$-orbit category $\Or_G^{\Fin}$,
  so the edge homomorphism of the spectral sequence
  can be viewed as a homomorphism
  \[ \pi_k^G(X)\ \to \  H^0_G(\uEG, \upi_k(X))
    \ \cong \ {\lim}_{\Or_G^{\Fin}}\, \upi_k(X) \ .   \]
  A more detailed analysis would reveal that this edge homomorphism
  is given by the restriction maps
  \[ \res^G_H \ : \  \pi_k^G(X)\ = \ [\mS_G,X[-k]]^G \ \to \ 
    [\mS_H, X[-k]]^H \ \cong\ \pi_k^H(X) \]
  for all finite subgroups $H$ of~$G$.
  So a compatible system $\{x_H\}_{G/H\in\Or^{\Fin}_G}$ of homotopy classes in $\pi_k^H(X)$
  is the restriction of some class in $\pi_k^G(X)$ if any only if the 
  corresponding element of $H^0_G(\uEG,\upi_k(X))$ is a permanent cycle in
  the Atiyah-Hirzebruch spectral sequence~\eqref{eq:AHSS_for_sphere}.
  
  As a special case we consider a countable discrete group~$G$
  that is {\em locally finite}, i.e., every finitely generated subgroup of~$G$
  is finite.\index{locally finite group} 
  Such groups have a 1-dimensional model for $\uEG$.
  For example, if $G$ is locally finite and countable, then it is the union of
  an ascending sequence of finite subgroups, and 
  we described such a model in Example \ref{eg:1-dim for locally finite}.
  A general locally finite group $G$ is the filtered union of its finite subgroups.
  Group homology commutes with filtered unions, so the group homology
  and group cohomology of $G$ with coefficients 
  in any $\mQ G$-module vanish in positive dimensions.
  A theorem of Dunwoody \cite[Thm.\,1.1]{dunwoody:accessibility} then provides a
  1-dimensional model for $\uEG$.
  
  If $G$ has a 1-dimensional $\uEG$, 
  then the spectral sequence~\eqref{eq:AHSS_for_sphere} collapses at $E_2$
  and specializes to a short exact sequence
  \[ 0 \ \to \  {\lim}^1_{\Or_G^{\Fin}}\, \upi_{n+1}(X) \ \to \
    \pi_n^G(X) \ \to \ {\lim}_{\Or_G^{\Fin}}\, \upi_n(X) \ \to \ 0 \ . \]
  If $G$ is locally finite and countable, then this short exact sequence
  is a special case of Corollary~\ref{cor:lim_lim^1_sequence}.
\end{rk}

Essentially by definition, 
the $G$-equivariant homotopy groups take distinguished triangles 
to long exact sequences, and products to products.
One should beware, though, that infinite products of orthogonal
$G$-spectra are {\em not} generally 
products in the triangulated category~$\Ho(\Sp_G)$;
they are if all factors are $G$-$\Omega$-spectra.
However, in our more general context not all the `usual' properties 
of equivariant homotopy groups carry over from compact to
general Lie groups.
For example, the functor $\pi_*^G$ does {\em not} in general take
infinite wedges to direct sums, because the $G$-sphere spectrum $\mS_G$
need not be small in the triangulated category $\Ho(\Sp_G)$.

\begin{prop}\label{prop:EG small - S_G small}
  Let~$G$ be a Lie group that has a model for $\uEG$
  that admits a finite $G$-CW-structure.\index{G-equivariant stable homotopy category@$G$-equivariant stable homotopy category}\index{universal proper $G$-space}
  \begin{enumerate}[\em (i)]
  \item 
    Let $X$ be an orthogonal $G$-spectrum such that
    the $H$-spectrum $\res^G_H(X)$ is a small object in $\Ho(\Sp_H)$
    for every compact subgroup $H$ of $G$.
    Then the $G$-spectrum $X$ is a small object in the triangulated category $\Ho(\Sp_G)$.
  \item The $G$-sphere spectrum $\mS_G$ is a small object 
    in the triangulated category $\Ho(\Sp_G)$,
    and the functor $\pi_k^G:\Ho(\Sp_G)\to \cA b$ preserves all sums.
  \end{enumerate}
\end{prop}
\begin{proof}
  (i)
  The restriction functor $\res^G_H:\Ho(\Sp_G)\to\Ho(\Sp_H)$
  preserves coproducts, so its left adjoint $G\ltimes_H-$ preserves compact objects.
  So if $\res^G_H(X)$ is small in $\Ho(\Sp_H)$,
  then $X\sm G/H_+\iso G\ltimes_H\res^G_H(X)$ is small in $\Ho(\Sp_G)$.
  The class of small objects in a triangulated category is closed 
  under 2-out-of-3 in distinguished triangles. So induction over the
  number of equivariant cells shows that $X\sm A_+$ is small
  for every finite proper $G$-CW-complex $A$.
  Since $\uEG$ has a finite proper $G$-CW-model, $X\sm \uEG_+$ is small.
  But $X$ is isomorphic in $\Ho(\Sp_G)$ to $X\sm \uEG_+$, so $X$ itself is small.

  (ii) The restriction functor $\res^G_H:\Ho(\Sp_G)\to\Ho(\Sp_H)$ takes the
  $G$-sphere spectrum to the $H$-sphere spectrum, which is small if $H$ is compact.
  So the $G$-sphere spectrum is small in $\Ho(\Sp_G)$.
\end{proof}

As we just saw, a finite model for $\uEG$ implies that
the $G$-sphere spectrum is small in the triangulated homotopy category~$\Ho(\Sp_G)$.
The example below illustrates that this is not true in general,
and that even a finite-dimensional model for $\uEG$ 
does not imply smallness of $\mS_G$.
In fact, \cite[Thm.\,5.4]{barcenas-degrijse-patchkoria:stable_finiteness} 
shows that for a countable discrete group $G$, the $G$-sphere spectrum is small 
if and only if there exists a finite-dimensional stable model for $\uEG$,
and there exists a finite-type stable model for $\uEG$.
Moreover, by \cite[Thm.\,5.1]{barcenas-degrijse-patchkoria:stable_finiteness},
a finite-type stable model for $\uEG$ exists if and only if
there are only finitely many conjugacy classes of finite subgroups in $G$, 
and every Weyl group $W_G(H)=N_G(H)/H$ of a finite subgroup $H$ of $G$ 
is of homological type $F P_{\infty}$.

\begin{eg}[The $G$-sphere spectrum need not be small]\label{eg:S_G not small}
We let $F$ be any non-trivial finite group, and we define
\[ G \ = \ {\prod}'_{k\geq 1}\,  F \ , \]
an infinite weak product of copies of~$F$,
i.e., the subgroup of the product consisting of tuples with almost all
coordinates the neutral element.
We will now show that $\mS_G$ is not small in~$\Ho(\Sp_G)$. 
Example \ref{eg:1-dim for locally finite} exhibits a
1-dimensional $G$-CW-model for $\uEG$; 
Proposition \ref{prop:EG small - S_G small} shows that 
there cannot be a finite $G$-CW-model for $\uEG$.

We set $H_n=\prod_{k=1}^n F$.
Then $G$ is the ascending union of its finite subgroups $H_n$,
and we can apply Corollary \ref{cor:lim_lim^1_sequence}.
The inclusion $H_{n-1}\to H_n$ has a retraction $r:H_n\to H_{n-1}$
by a group homomorphism.
Now we let $X$ be an orthogonal spectrum, which we give the trivial $G$-action.
The inflation homomorphism $r^*:\pi_k^{H_{n-1}}(X)\to \pi_k^{H_n}(X)$
is defined in \eqref{eq:restriction_hom} below; it is a section to the
restriction $\res^{H_n}_{H_{n-1}}:\pi_k^{H_n}(X)\to \pi_k^{H_{n-1}}(X)$.
Since the restriction maps are surjective, 
the $\lim^1$ terms in the short exact sequence of Corollary \ref{cor:lim_lim^1_sequence} 
vanish, and we conclude that the map
\[ \pi_k^G(X)\ \to \ {\lim}\, \pi_k^{H_n}(X) \]
induced by restriction is an isomorphism.

In the commutative square
\begin{equation}\begin{aligned}\label{eq:smallness_counterexample}
 \xymatrix{ 
\bigoplus_{\mN} \, \pi_0^G(\mS_G)\ar[r]^-\cong \ar[d] &
\bigoplus_{\mN} \, \lim \, \pi_0^{H_n}(\mS_{H_n}) \ar[d] \\
\pi_0^G(\bigoplus_{\mN} \, \mS_G) \ar[r]_-\cong &
 \lim \, \pi_0^{H_n}(\bigoplus_\mN \mS_{H_n}) }     
\end{aligned}\end{equation}
the two horizontal maps are thus isomorphisms.
Since $H_n$ is finite, the group $\pi_0^{H_n}(\mS_{H_n})$ is
isomorphic to the Burnside ring $A(H_n)$, and the map
\[ \res^{H_n}_{H_{n-1}}\ : \ \mA(H_n) \ \to \ \mA(H_{n-1}) \]
is a split epimorphism between finitely generated free abelian groups. 
Since the group $F$ is non-trivial, the kernel of this restriction map 
is non-trivial.
The upper right corner of square~\eqref{eq:smallness_counterexample}
is thus a countably infinite sum of 
a countably infinite product of copies of~$\mZ$,
and the right vertical map is not surjective.
So the left vertical map is not surjective, 
and hence the $G$-sphere spectrum is not small for the particular
group under consideration.
\end{eg}

\begin{con}[Restriction homomorphisms]\label{con:define restriction}
We let $\alpha:K\to G$ be a continuous homomorphism between Lie groups.
As we shall now explain, such a homomorphism induces a {\em restriction homomorphism}
\begin{equation}\label{eq:restriction_hom}
 \alpha^* \ : \ \pi_k^G(X)\ \to \ \pi_k^K( (L\alpha^*)(X)) \ ,  
\end{equation}
natural for morphisms of orthogonal $G$-spectra, where $L\alpha^*$
is the total left derived functor of $\alpha^*:\Sp_G\to\Sp_K$,
compare Theorem \ref{thm:homomorphism adjunctions spectra}.
The construction exploits the two isomorphisms
\[  ((L\alpha^*)(\mS_G))\ \xrightarrow{\alpha_!^{\mS_G}}\
 \alpha^*(\mS_G)\ =\ \mS_K \text{\quad and\quad}
(L\alpha^*)(X)[-k]\ \iso\ (L\alpha^*)(X[-k])  \]
in $\Ho(\Sp_K)$; the first one is an isomorphism because $\mS_G$ is quasi-cofibrant,
and the second one is provided by part (iv) 
of Theorem \ref{thm:homomorphism adjunctions spectra}.
So we define the restriction homomorphism as the composite
\[ [\mS_G,X[-k]]^G\ \xrightarrow{\ L\alpha^*\ }\ 
[(L\alpha^*)(\mS_G),(L\alpha^*)(X[-k])]^G
\ \xrightarrow{\ \iso\ }\ [\mS_K,(L\alpha^*)(X)[-k]]^G\ . \]
\end{con}

Now we consider two composable continuous homomorphisms $\alpha:K\to G$ 
and $\beta:L\to K$. 
In \eqref{eq:td(a,b)} we exhibited a natural isomorphism
$\td{\alpha,\beta}: (L\beta^*)\circ (L\alpha^*)\Longrightarrow L(\alpha\beta)^*$ 
that relates the three derived functors.  

\begin{prop}\index{left derived functor!of restriction}
Let $\alpha:K\to G$ and $\beta:L\to K$ be composable continuous homomorphisms 
between Lie groups. Then for every orthogonal $G$-spectrum $X$, the composite
\begin{align*}
\pi_k^G(X) \ \xrightarrow{\ \alpha^*\ } \ &\pi_k^K( (L\alpha)^*(X)) \
             \xrightarrow{\ \beta^*\ } \\
&\pi_k^L( (L\beta)^*((L\alpha)^*(X)))   
\ \xrightarrow{\td{\alpha,\beta}^X_*} \ \pi_k^L(L(\alpha\beta)^*(X))
\end{align*}
coincides with the restriction homomorphism $(\alpha\beta)^*$.
\end{prop}
\begin{proof}
All maps are natural for $G$-maps in $X$ and compatible with
the suspension isomorphisms. So by naturality it suffices to prove
the claim for the universal example, the identity of $\mS_G$.
After unraveling all definitions, the universal example then comes down
to the relation
\[ \beta_!^{\mS_K}\circ (L\beta)^* (\alpha_!^{\mS_G})
= \ (\alpha\beta)_!^{\mS_G}\circ \td{\alpha,\beta}^{\mS_G}\ : \
(L\beta^*)((L\alpha^*)(\mS_G))\ \to \ \mS_L \ , \]
which is an instance of the defining property \eqref{eq:td(a,b)}
of the transformation $\td{\alpha,\beta}$.
\end{proof}
\index{equivariant homotopy groups|)}

\index{Wirthm{\"u}ller isomorphism|(}
Now we discuss the Wirthm{\"u}ller isomorphism for
finite index inclusions of Lie groups, and the transfer maps that it gives rise to.
We consider a closed subgroup $\Gamma$ of $G$
and we write $\res^G_\Gamma:\Sp_G\to\Sp_\Gamma$ for the restriction functor.
This restriction functor is fully homotopical, i.e., it takes
$\pi_*$-isomorphisms of orthogonal $G$-spectra to
$\pi_*$-isomorphisms of orthogonal $\Gamma$-spectra,
simply because every compact subgroup of $\Gamma$ is also a compact subgroup of $G$.
So we get an induced restriction functor on the homotopy categories
\[ \res^G_\Gamma\ \colon \ \Ho(\Sp_G) \ \to \ \Ho(\Sp_\Gamma) \]
from the universal property of localizations, 
for which we use the same name.
This functor satisfies $\res^G_\Gamma\circ\gamma_G=\gamma_\Gamma\circ \res^G_\Gamma$,
so it is in particular a total left derived functor of restriction.
The restriction functor $\res^G_\Gamma$ is 
both a left Quillen functor and a right Quillen functor for the two stable model
structures, by Corollary \ref{cor:restriction stable}. 
We write
\[  \coind^G_\Gamma  \colon \Ho(\Sp_\Gamma)\ \to \  \Ho(\Sp_G)  \]
for the right adjoint of the derived restriction functor,
which is a total right derived functor of the functor
$\map^\Gamma(G,-):\Sp_\Gamma\to\Sp_G$.
Moreover, the left adjoint $G\ltimes_\Gamma-$ is also fully homotopical,
so it, too, passes to a functor on the homotopy categories
\[ G\ltimes_\Gamma\ \colon \ \Ho(\Sp_\Gamma) \ \to \ \Ho(\Sp_G) \]
by the universal property of localizations, 
for which we also use the same name.

If we also assume that $\Gamma$ has finite index in $G$, then
the derived left and right adjoint to the restriction functor are in fact isomorphic;
this generalizes the classical `Wirthm{\"u}ller isomorphism' 
in equivariant homotopy theory of finite groups \cite{wirthmuller:equivariant_homology}.
Indeed, in this situation the group $G$ is the disjoint union 
of finitely many $\Gamma$-cosets.
If~$X$ is a based $\Gamma$-space, we can define a natural $G$-map
\[  w_X\ : \ G\ltimes_\Gamma X \ \to \ \map^\Gamma(G,X)\]
by sending $[g,x]$ to the $\Gamma$-equivariant map
\[ G \ \to \ X \ , \quad  g' \ \longmapsto \ 
\begin{cases}
  g' g x & \text{ for $g' g\in\Gamma$, and }\\
  \, \ast & \text{ for $g' g\not\in\Gamma$.}
\end{cases}\]
For an orthogonal $\Gamma$-spectrum $Y$ these maps are defined levelwise,
and they form a morphism of orthogonal $G$-spectra
 \[ w_Y \ : \  G\ltimes_\Gamma Y \ \to \ \map^\Gamma(G,Y)\ .\]

\begin{thm}\label{thm:restriction adjoints}\index{G-equivariant stable homotopy category@$G$-equivariant stable homotopy category}
  Let~$\Gamma$ be a closed subgroup of finite index of a Lie group $G$.
  For every orthogonal $\Gamma$-spectrum $Y$ the morphism
  $w_Y : G\ltimes_\Gamma Y \to \map^\Gamma(G,Y)$  is a $\pi_*$-isomorphism.
  Hence $w_Y$ descends to a natural isomorphism between the functors
  \[ G\ltimes_\Gamma \ , \
  \coind^G_\Gamma  \  \colon \ \Ho(\Sp_\Gamma)\ \to \  \Ho(\Sp_G)  \ .\]
\end{thm}
\begin{proof}
  We let $H$ be any compact subgroup of $G$. By our hypothesis,
  $G/\Gamma$ is a finite set and for every $g\in G$ the
  subgroup $H\cap{^g\Gamma}$ has finite index in~$H$.
  So for every orthogonal $(H\cap{^g\Gamma})$-spectrum~$X$ the  morphism
  \[ 
  w_X \ : \ H\ltimes_{H\cap{^g \Gamma}} X \ \to \ 
  \map^{H\cap{^g \Gamma}}(H, X)  \]
  is a $\pi_*$-isomorphism of $H$-spectra by the classical Wirthm{\"u}ller isomorphism
  for the finite index pair $(H,H\cap{^g\Gamma})$,
  see \cite{wirthmuller:equivariant_homology} or \cite[Thm.\,3.2.15]{schwede:global}.
  Moreover, there are double coset decompositions
  \[ \res^G_H(G\ltimes_\Gamma X) \ \cong \
  \bigvee_{[g]\in H\backslash G/\Gamma} 
  H\ltimes_{H\cap{^g \Gamma}}\left( c_g^*\left( \res^\Gamma_{H^g\cap\Gamma}(X)\right)\right)
  \]
  and
  \[ \res^G_H(\map^\Gamma(G, X)) \ \cong \
  \prod_{[g]\in H\backslash G/\Gamma} 
  \map^{H\cap{^g \Gamma}}\left( H, c_g^*\left( \res^\Gamma_{H^g\cap\Gamma}(X)\right)\right)
  \ .\]
  The morphism~$w_X$ respects these decomposition.
  Since finite wedges are $\pi_*$-iso\-morphic to finite products, we are done.
\end{proof}

An immediate consequence of Theorem \ref{thm:restriction adjoints}
is the isomorphism between the group $\pi_k^G(G\ltimes_\Gamma Y)$
and the group $\pi_k^\Gamma(Y)$, defined as the composite
\[ \pi_k^G(G\ltimes_\Gamma Y)\ \xrightarrow[\cong]{(\omega_Y)_*} \ 
\pi_k^G(\coind_\Gamma^G(Y)) \ \xrightarrow[\cong]{\text{adjunction}} \ 
  \pi_k^\Gamma( Y)\ . \]
This map is also the composite
\begin{equation}\label{eq:wirth}
 \Wirth_\Gamma^G \ : \ \pi_k^G(G\ltimes_\Gamma Y)\  \xrightarrow{\res^G_\Gamma} \ 
 \pi_k^\Gamma(G\ltimes_\Gamma Y)\  \xrightarrow{ (\text{pr}^G_\Gamma)_* } \ 
 \pi_k^\Gamma( Y)\ .   
\end{equation}
In the special case where $G$ (and hence $\Gamma$) are compact,
this isomorphism specializes to the {\em Wirthm{\"u}ller isomorphism}
\cite{wirthmuller:equivariant_homology}, 
see also \cite[Thm.\,3.2.15]{schwede:global}.
In our more general context, we also refer to the isomorphism \eqref{eq:wirth}
as the Wirthm{\"u}ller isomorphism.
\index{Wirthm{\"u}ller isomorphism|)}

\begin{con}[Transfer]\label{con:transfer}\index{transfer|(}
We continue to let $\Gamma$ be a finite index subgroup of a Lie group~$G$.
If $X$ is an orthogonal $G$-spectrum, then we can define a 
{\em transfer homomorphism} as the composite
\begin{equation}\label{eq:transfer}
 \tr_\Gamma^G \ : \ \pi_k^\Gamma(X)\ \xrightarrow[\cong]{(\Wirth_\Gamma^G)^{-1}} \ 
\pi_k^G(G\ltimes_\Gamma \res^G_\Gamma(X)) \ \xrightarrow{ (\text{act}^G_\Gamma)_*} \
\pi_k^G(X) \ ,
\end{equation}
where $\text{act}^G_\Gamma:G\ltimes_\Gamma \res^G_\Gamma(X)\to X$
is the action morphism (the counit of the adjunction).
\end{con}

Now we prove that the transfer maps satisfy the `usual properties'.
We start by studying how transfer maps interact with inflation, i.e., the
restriction homomorphism along a continuous epimorphism $\alpha:K\to G$.
We let $\Gamma$ be any closed subgroup of the Lie group $G$,
and we let $\Delta=\alpha^{-1}(\Gamma)$ be the inverse image,
a closed subgroup of $K$.
On the pointset level, the relation
\[  \res^K_\Delta \circ \alpha^*\ = \ (\alpha|_\Delta)^*\circ \res^G_\Gamma  \]
holds as functors from $\Sp_G$ to $\Sp_\Delta$.
On the level of homotopy categories, this relation becomes a natural isomorphism
between derived functors. 
Indeed, the isomorphisms
\[ \td{\incl^G_\Gamma,\alpha|_\Delta}\ : \ 
L(\alpha|_\Delta)^*\circ \res^G_\Gamma \ \Longrightarrow \ L(\incl^G_\Gamma\circ\alpha|_\Delta)^*\]
and
\[ \td{\alpha,\incl^K_\Delta}\ : \ 
\res^K_\Delta \circ L\alpha^*
\ \Longrightarrow \ L(\alpha\circ\incl^K_\Delta)^*\]
combine into a composite natural isomorphism
\begin{align*}
[\alpha,\Gamma]\ : \ 
L(\alpha|_\Delta)^*\circ \res^G_\Gamma 
\ \xrightarrow[\iso]{\td{\incl^G_\Gamma,\alpha|_\Delta}}
&\ L(\incl^G_\Gamma\circ\alpha|_\Delta)^* \\
  = \ &L(\alpha\circ\incl^K_\Delta)^*
\ \xrightarrow[\iso]{\td{\alpha,\incl^K_\Delta}^{-1}} \ 
  \res^K_\Delta \circ L\alpha^*\ .
\end{align*}

\begin{prop}\label{prop:transfer and inflation}
  Let $\alpha:K\to G$ be a continuous epimorphism between Lie groups,
  let $\Gamma$ be a closed subgroup of $G$ of finite index, 
  and set $\Delta=\alpha^{-1}(\Gamma)$.
  Then the following square commutes
  \[ \xymatrix@C=10mm{
    \pi_*^\Gamma(X)\ar[rr]^-{\tr^G_\Gamma}\ar[d]_{(\alpha|_\Delta)^*} &&
    \pi_*^G(X)\ar[d]^{\alpha^*} \\
    \pi_*^\Delta(L(\alpha|_\Delta)^*(\res^G_\Gamma(X))) \ar[r]_-{[\alpha,\Gamma]_*}^-\iso &
    \pi_*^\Delta( (L\alpha^*)(X) )\ar[r]_-{\tr^K_\Delta}&
    \pi_*^K( (L\alpha^*)(X) )
  } \]
  for every orthogonal $G$-spectrum $X$.
\end{prop}
\begin{proof}
We first consider a cofibrant orthogonal $\Gamma$-spectrum $Y$.
An isomorphism of orthogonal $K$-spectra
\[ u \ : \ 
K\ltimes_\Delta(\alpha|_\Delta)^*(Y)\ \xrightarrow{\ \iso \ } \ 
\alpha^*(G\ltimes_\Gamma Y) \]
is defined levelwise by $u[k,y] = [\alpha(k),y]$,
for $k\in K$ and $y\in Y(V)$.
Moreover, the composite
\[ K\ltimes_\Delta(\alpha|_\Delta)^*(Y)\ \xrightarrow{\ u \ } \ 
\alpha^*(G\ltimes_\Gamma Y) \  \xrightarrow{(\alpha|_\Delta)^*(\text{pr}^G_\Gamma)} 
\  (\alpha|_\Delta)^*(Y) \]
coincides with the morphism 
$\text{pr}^K_\Delta:K\ltimes_\Delta(\alpha|_\Delta)^*(Y)\to (\alpha|_\Delta)^*(Y)$.
Since $Y$ is cofibrant, we can calculate $(L\alpha^*)(Y)$ as $\alpha^*(Y)$.
Also, $G\ltimes_\Gamma Y$ is cofibrant as an orthogonal $G$-spectrum,
and we can calculate $(L\alpha^*)(G\ltimes_\Gamma Y)$ as
$\alpha^*(G\ltimes_\Gamma Y)$.
Similarly, the underlying $\Gamma$-spectrum of $G\ltimes_\Gamma Y$ is cofibrant,
so we can calculate $(L\alpha|_\Delta^*)(G\ltimes_\Gamma Y)$ as
$(\alpha|_\Delta)^*(G\ltimes_\Gamma Y)$.
The following diagram commutes by naturality and transitivity
of restriction maps:
\[ \xymatrix@C=16mm{ 
\pi_*^G(G\ltimes_\Gamma Y)\ar[r]_-{\res^G_\Gamma} \ar[d]_{\alpha^*}
\ar@/^1pc/[rr]^(.3){\Wirth^G_\Gamma}& 
 \pi_*^\Gamma(G\ltimes_\Gamma Y)\ar[r]_-{ (\text{pr}^G_\Gamma)_* } 
\ar[d]_{(\alpha|_\Delta)^*} &  \pi_*^\Gamma( Y) \ar[dd]^{(\alpha|_\Delta)^*}\\
\pi_*^K(\alpha^*(G\ltimes_\Gamma Y))\ar[r]^-{\res^K_\Delta} &
 \pi_*^\Delta( (\alpha|_\Delta)^*(G\ltimes_\Gamma Y))
\ar[dr]^{\ \ ((\alpha|_\Delta)^*(\text{pr}^G_\Gamma))_* } &  \\
\pi_*^K(K\ltimes_\Delta (\alpha|_\Delta)^*(Y))\ar[u]^{u_*}_\iso\ar[r]^{\res^K_\Delta} 
\ar@<-.6ex>@/_1pc/[rr]_(.7){\Wirth^K_\Delta}&
 \pi_*^\Delta( K\ltimes_\Delta (\alpha|_\Delta)^*(Y))\ar[r]^(.5){ (\text{pr}^K_\Delta)_* }
 \ar[u]^{u_*}_\iso&  \pi_*^\Delta((\alpha|_\Delta)^*(Y))
} \]
In formulas:
\begin{equation}  \label{eq:wirth2inflate}
 (\alpha|_\Delta)^*\circ\Wirth^G_\Gamma\ = \  \Wirth^K_\Delta \circ u^{-1}_*\circ \alpha^* \ .  
\end{equation}

Now we let $X$ be a cofibrant orthogonal $G$-spectrum.
Then the following square commutes:
\[ \xymatrix{
K\ltimes_\Delta(\alpha|_\Delta)^*(\res^G_\Gamma(X))\ar@{=}[r]\ar[d]_u &
K\ltimes_\Delta\res^K_\Delta(\alpha^*(X))\ar[d]^-{\text{act}^K_\Delta} \\
\alpha^*(G\ltimes_\Gamma X ) \ar[r]_-{\alpha^*(\text{act}^G_\Gamma)}  &
\alpha^*(X) &
} \]
The Wirthm{\"u}ller maps are isomorphisms, so we can deduce
\begin{align*}
\tr^K_\Delta\circ(\alpha|_\Delta)^*\ 
&= \   (\text{act}^K_\Delta)_*\circ (\Wirth^K_\Delta)^{-1}\circ (\alpha|_\Delta)^*\\
_\eqref{eq:wirth2inflate} 
&= \   (\text{act}^K_\Delta)_*\circ u^{-1}_*\circ \alpha^*\circ(\Wirth^G_\Gamma)^{-1}  \\
&= \   (\alpha^*(\text{act}^G_\Gamma))_*\circ\alpha^*\circ(\Wirth^G_\Gamma)^{-1}  \\
&= \   \alpha^*\circ(\text{act}^G_\Gamma)_*\circ(\Wirth^G_\Gamma)^{-1}  \
= \ \alpha^*\circ \tr^G_\Gamma\ .
\end{align*}
This proves the claim for cofibrant orthogonal $G$-spectra.
In $\Ho(\Sp_G)$, every object is isomorphic to a cofibrant $G$-spectrum,
so naturality concludes the argument.
\end{proof}
\index{transfer|)}

Now we spell out how transfers interact with the conjugation homomorphism.
For this purpose we let $\Gamma$ be any closed subgroup of a Lie group $G$
and $g\in G$. We let $\Gamma^g=g^{-1}\Gamma g$ be the conjugate subgroup
and denote by
\[ c_g \ : \ \Gamma \ \to \ \Gamma^g \ , \quad c_g(\gamma)\ = \ g^{-1}\gamma g \]
the conjugation homomorphism.
Restriction of group actions along $c_g$ is fully homotopical;
we abuse notation and write
\[  c_g^*\ =\ \Ho(c_g^*)\ :\  \Ho(\Sp_{\Gamma^g})\ \to\ \Ho(\Sp_{\Gamma}) \]
for the induced functor on homotopy categories.
This induced functor is then also a total left derived functor of $c_g^*$,
relative to the identity transformation.
As a special case of the restriction homomorphism \eqref{eq:restriction_hom},
the restriction functor thus induces an isomorphism
\[ c_g^*\ : \ \pi_k^{\Gamma^g}(Y)\ \to \  \pi_k^\Gamma(c_g^* (Y)) \]
for every orthogonal $\Gamma^g$-spectrum $Y$.
We call this the {\em conjugation isomorphism}.

Now we let $X$ be an orthogonal $G$-spectrum. Then left multiplication by $g$
is an isomorphism
\[ l_g \ : \ c_g^*(X) \ \to \ X \]
of orthogonal $G$-spectra, which induces an isomorphism on $\pi_k^\Gamma(-)$.
The composite
\[   \pi_k^{\Gamma^g}(X)\ \xrightarrow{\ c_g^*\ }\ \pi_k^\Gamma(c_g^* (X))
\ \xrightarrow{\ (l_g)_*\ }\ \pi_k^\Gamma(X) \]
is an `internal' conjugation isomorphism which we denote by
\begin{equation}\label{eq:define conjugation}
g_{\star}\ :\ \pi_k^{\Gamma^g}(X)\ \to\ \pi_k^\Gamma(X)\ .  
\end{equation}

\begin{rk}
The conjugation isomorphism has another interpretation as follows.  
The map
\[ l_g \ : \ G/\Gamma \ \to \ G/\Gamma^g \ : \ 
k\Gamma \ \longmapsto \ k g\Gamma^g \]
is an isomorphism of $G$-spaces, and it induces an isomorphism
of $G$-equivariant suspension spectra
\[ \Sigma^\infty_+ l_g \ : \ \Sigma^\infty_+ G/\Gamma \ \to \ 
\Sigma^\infty_+ G/\Gamma^g \ .\]
For every orthogonal $G$-spectrum, the derived adjunctions provide natural isomorphisms
\begin{align*}
\pi_0^{\Gamma^g}(X)\ = \ 
[\mS_{\Gamma^g},\res^G_{\Gamma^g}(X)]^{\Gamma^g} \ &\xrightarrow{\ \iso\ } \ 
[\Sigma^\infty_+ G/\Gamma^g,X]^G \text{\quad and}\\ 
\pi_0^{\Gamma}(X)\ = \ 
[ \mS_\Gamma,\res^G_{\Gamma}(X)]^\Gamma \ &\xrightarrow{\ \iso\ } \ 
[ \Sigma^\infty_+ G/\Gamma,X]^G \ .
\end{align*}
We omit the verification that under these isomorphisms, the conjugation
map $g_\star:\pi_k^{\Gamma^g}(X)\to\pi_k^\Gamma(X)$ corresponds to precomposition with
$\Sigma^\infty_+ l_g$.
\end{rk}

 \index{transfer|(}
\begin{prop}
  Let $G$ be a Lie group and $g\in G$.
  \begin{enumerate}[\em (i)]
  \item Let $\Delta\subset\Gamma$ be nested closed subgroups of $G$,
    such that $\Delta$ has finite index in $\Gamma$.
    Then
    \[ \tr_\Delta^\Gamma\circ g_{\star} \ = \ g_{\star} \circ \tr_{\Delta^g}^{\Gamma^g} \ : \ 
    \pi_*^{\Delta^g}(X)\ \to \ \pi_*^\Gamma(X) \]
    for every orthogonal $G$-spectrum $X$.
  \item The conjugation map $g_{\star}:\pi_*^G(X)\to \pi_*^G(X)$ is the identity.
\end{enumerate}
\end{prop}
\begin{proof}
(i) In the special case $\alpha=c_g:\Gamma\to \Gamma^g$, applied to the finite index subgroup
$\Delta^g$ of $\Gamma^g$, Proposition \ref{prop:transfer and inflation}
says that the following square commutes:
  \[ \xymatrix@C=10mm{
    \pi_*^{\Delta^g}(X)\ar[rr]^-{\tr^{\Gamma^g}_{\Delta^g}}\ar[d]_{c_g^*} &&
    \pi_*^{\Gamma^g}(X)\ar[d]^{c_g^*} \\
    \pi_*^\Delta( c_g^*(\res^G_{\Delta^g}(X))) \ar@{=}[r]&
    \pi_*^\Delta(\res^G_{\Delta}( c_g^*(X) ))\ar[r]_-{\tr^\Gamma_{\Delta}}&
    \pi_*^\Gamma( c_g^*(X))
  } \]
Naturality of restriction and transfer for the
morphism of orthogonal $G$-spectra $l_g:c_g^*(X)\to X$ then yields the
desired relation:
\begin{align*}
g_\star\circ\tr^{\Gamma^g}_{\Delta^g} \ &=\ 
(l_g)_*\circ c_g^*\circ\tr^{\Gamma^g}_{\Delta^g} \ = \  (l_g)_* \circ \tr^\Gamma_\Delta \circ c_g^* \
= \  \tr^\Gamma_\Delta \circ (l_g)_* \circ  c_g^* \ = \ 
\tr^\Gamma_\Delta \circ g_\star 
\end{align*}
For claim (ii) we exploit that the map $g_\star$ is natural for 
morphisms in $\Ho(\Sp_G)$ and commutes with the suspension isomorphism.
So it suffices to prove the claim in the universal example, the identity of $\mS_G$. 
Since $G$ acts trivially on $\mS_G$, we have
$c_g^*(\mS_G)=\mS_G$ and $l_g^{\mS_G}=\Id$.
So $c_g^*(\Id)=\Id$.
\end{proof}

Now we prove transitivity with respect to a
nested triple of finite index subgroups $\Gamma\leq\Delta\leq G$,
and the double coset formula.

\begin{prop}\label{prop:dcf}
Let $\Gamma$ be a closed finite index subgroup of a Lie group $G$.
\begin{enumerate}[\em (i)]
\item Let $\Delta\leq G$ be another closed subgroup with $\Gamma\leq \Delta$.
Then the transfer maps are transitive, i.e.,
\[ \tr_\Delta^G\circ\tr_\Gamma^\Delta \ = \ \tr_\Gamma^G \ : \ \pi_*^\Gamma(X)\ \to \ \pi_*^G(X)\]
for every orthogonal $G$-spectrum $X$.
\item Let $K$ be another closed subgroup of $G$.
Then for every orthogonal $G$-spectrum $X$ the relation
\[ \res^G_K\circ \tr_\Gamma^G \ = \ \sum_{[g]\in K\backslash G/\Gamma}\
\tr_{K\cap{^g \Gamma}}^K  \circ g_{\star}\circ \res^\Gamma_{K^g\cap\Gamma}\]
holds as maps $\pi_*^\Gamma(X)\to \pi_*^K(X)$.
Here the sum is indexed over a set of representatives of
the finite set of $K$-$\Gamma$-double cosets in $G$.
\end{enumerate}
\end{prop}
\begin{proof}
We reduce both properties to the special case of {\em finite} groups.
We set
\[ N \ = \ \bigcap_{g\in G}  \Gamma^g\ ,\]
the intersection of all $G$-conjugates of $\Gamma$.
Then $N$ is the largest normal subgroup of $G$ that is contained in $\Gamma$,
and it is the kernel of the translation action of $G$ on $G/\Gamma$.
Hence the quotient group $H=G/N$ acts faithfully on the finite set $G/\Gamma$;
in particular, the group $H$ is finite.
We let
\[ q\ : \ G\ \to \ G/N\ = \ H \]
denote the quotient map, which is a continuous epimorphism.
The morphism $q$ induces an isomorphism of finite $G$-sets
$G/\Gamma \iso q^*(H/I)$, where $I=q(\Gamma)=\Gamma/N$,
and hence an isomorphism of orthogonal $G$-spectra
from $G\ltimes_\Gamma\mS_\Gamma$ to $q^*(\Sigma^\infty_+ H/I)$.
The natural bijections
\[ \pi_0^\Gamma(X) \ = \ [\mS_\Gamma,\res^G_\Gamma(X)]^\Gamma\ \iso \ 
[G\ltimes_\Gamma \mS_\Gamma,X]^G\ \iso \ 
[q^*(\Sigma^\infty_+ H/I),X]^G \]
witness that the functor $\pi_0^\Gamma:\Ho(\Sp_G)\to\text{(sets)}$
is represented by the orthogonal $G$-spectrum $q^*(\Sigma^\infty_+ H/I)$;
the universal element is the class $q|_\Gamma^*(u_I)$
in the group $\pi_0^\Gamma(q^*(\Sigma^\infty_+ H/I))$,
where $u_I\in \pi_0^I(\Sigma^\infty_+ H/I)$
is the tautological class \eqref{eq:define_tautological}.
The formulas of parts (i) and (ii) are relations between natural
transformations of functors on $\Ho(\Sp_G)$ with source
the representable functor $\pi_0^\Gamma$.
By the Yoneda lemma, it thus suffices to prove the two formulas
applied to the universal example $(q^*(\Sigma^\infty_+ H/I),q|_\Gamma^*(u_I))$.

(i) We set $J=q(\Delta)$, another subgroup of the finite group $H$.
For the universal example we can then argue:
\begin{align*}
\tr_\Delta^G(\tr_\Gamma^\Delta (q|_\Gamma^* (u_I))) \ 
&= \ \tr_\Delta^G(q|_\Delta^* (\tr_I^J(u_I))) \\
&= \ q^*(\tr_J^H(\tr_I^J(u_I))) \ = \ q^*(\tr_I^H(u_I)) \ = \ 
\tr_\Gamma^G(q|_\Gamma^* (u_I)) \ .  
\end{align*}
The third equation is the transitivity property for transfers
in the realm of finite groups,
see for example \cite[Prop.\,3.2.9]{schwede:global}.
The other three equalities are instances of
the fact that transfers and inflations commute,
see Proposition \ref{prop:transfer and inflation}.

(ii)
We set $J=q(K)$, another subgroup of the finite group $H$.
For every $g\in G$ the following relation holds:
\begin{align}\label{eq:ind_summand}
 q|_K^*( \tr_{J\cap {^{q(g)} I}}^J ( q(g)_\star(\res^I_{J^{q(g)}\cap I}(u_I)))) \ 
&= \  \tr_{K\cap {^g\Gamma}}^K (q|_{K\cap{^g\Gamma}}^* ( q(g)_\star(\res^I_{J^{q(g)}\cap I}(u_I))))\\
&= \  \tr_{K\cap {^g\Gamma}}^K ( g_\star( q|_{K^g\cap\Gamma}^* ( \res^I_{J^{q(g)}\cap I}(u_I))))\nonumber\\
&= \  \tr_{K\cap {^g\Gamma}}^K ( g_\star(  \res^\Gamma_{K^g\cap \Gamma}(q|_\Gamma^* (u_I))))\ .\nonumber
\end{align}
The first equation is Proposition \ref{prop:transfer and inflation},
i.e., the fact that transfers and inflations commute.
The second and third equations are transitivity of restriction maps.
Now we deduce the double coset formula for the universal example:
\begin{align*}
  \res^G_K(\tr_\Gamma^G(q|_\Gamma^*(u_I)))\ &= \ 
  \res^G_K(q^*(\tr_I^H(u_I)))\ = \   q|_K^*(\res^H_J(\tr_I^H(u_I)))\\ 
&= \ \sum_{[h]\in J\bs H/I} \ q|_K^*( \tr_{J\cap {^h I}}^J (h_\star(\res^I_{J^h\cap I}(u_I)))) \\ 
_\eqref{eq:ind_summand} &= \ \sum_{[g]\in K\bs G/\Gamma} \  \tr_{K\cap {^g\Gamma}}^K ( g_\star(  \res^\Gamma_{K^g\cap \Gamma}(q|_\Gamma^* (u_I))))\ .
\end{align*}
The first equation is Proposition \ref{prop:transfer and inflation},
i.e., the fact that transfers and inflations commute.
The third equation is the classical double coset formula for the subgroups
$J=q(K)$ and $I=q(\Gamma)$ of the finite group $H$,
see for example \cite[Ex.\,3.4.11]{schwede:global}.
The fourth equation exploits that the epimorphism $q$ induces
a bijection from the set $K\bs G/\Gamma$ to the set $J\bs H/I$.
\end{proof}
 \index{transfer|)}

\section{Equivariant homotopy groups as Mackey functors}
\label{sec:Mackey}

If $G$ is a finite group and $X$ is an orthogonal $G$-spectrum, 
then the $H$-equivariant homotopy groups $\pi_0^H(X)$, 
for all subgroups $H$ of~$G$, form a $G$-Mackey functor,
see for example \cite[V.9]{lewis-may:steinberger:equivariant_stable}
or \cite[Sec.\,3.4]{schwede:global};
in this section we generalize this well-known fact to arbitrary discrete groups.
In Definition \ref{def:G-Mackey} we recall the notion of a $G$-Mackey functor
for discrete groups $G$, 
and in Example \ref{eg:Mackey of G-spectrum} we show
that the collection of equivariant homotopy groups 
of an orthogonal $G$-spectrum forms a graded $G$-Mackey functor.
Finally, in Theorem \ref{thm:embed_MG_into_Ho(Sp)} 
we identify the heart of the preferred t-structure on the equivariant
stable homotopy category~$\Ho(\Sp_G)$ with the abelian category of
$G$-Mackey functors.
A consequence is that every $G$-Mackey functor
has an Eilenberg-Mac\,Lane spectrum, i.e., an orthogonal $G$-spectrum
with equivariant homotopy groups concentrated in dimension 0,
where they realize the given $G$-Mackey functor, compare
Remark \ref{rk:general Eilenberg Mac Lane}.

\begin{con}[$G$-Mackey category]\index{G-Mackey category@$G$-Mackey category}
For a discrete group $G$, the pre-additive {\em Mackey category} $\bA_G$ 
has as objects all finite subgroups of $G$.
A $G$-set is {\em finitely generated} if it is generated by finitely many
of its element, or -- equivalently-- if it has finitely many orbits.
For two finite subgroups $H$ and $K$ of~$G$, 
a {\em span} is a triple $(S,\alpha,\beta)$ consisting of 
a finitely generated $G$-set $S$ and $G$-maps
\[ G/H \ \xleftarrow{\ \alpha\ } \ S \ \xrightarrow{\ \beta\ }  G/K\ . \]
An isomorphism of spans is an isomorphism of $G$-sets $\psi:S\to S'$
such that $\alpha'\circ\psi=\alpha$ and $\beta'\circ\psi=\beta$.
The isomorphism classes of spans form an abelian monoid under disjoint union, and
the morphism group $\bA_G(H,K)$ is defined as the Grothendieck
group of isomorphism classes of spans from $H$ to~$K$.
Composition
\[ \circ \ : \ \bA_G(K,L) \times \bA_G(H,K) \ \to \ \bA_G(H,L)\]
is induced by pullback of spans over the intermediate $G$-set $G/K$.
\end{con}

The following definition is taken from \cite[Sec.\,3]{martinezperez-nucinkis:cohomological_dimension}.

\begin{defn}\label{def:G-Mackey}
Let $G$ be a discrete group. A \emph{$G$-Mackey functor} is an additive\index{G-Mackey functor@$G$-Mackey functor}
functor from the Mackey category $\bA_G$ to the category of abelian groups.
A morphism of $G$-Mackey functors is a natural transformation.
We denote the category of $G$-Mackey functors by $\cM_G$.
\end{defn}

As a category of additive functors, $\cM_G$ 
is an abelian category with enough projectives and injectives. 
Monomorphisms, epimorphisms and exactness are detected objectwise.

As in the case of finite groups, $G$-Mackey functors 
also have a description via transfer, restriction and conjugation maps as follows.
Every $G$-set is the disjoint union of transitive $G$-sets,
so the group  $\bA_G(H,K)$ is a free abelian group with basis
the classes of those spans $(S,\alpha,\beta)$ where $G$ acts transitively on~$S$.
Up to isomorphism, every such `transitive span' is of the form
\begin{equation}\label{eq:Mackey_basis_element}
 G/H \ \xleftarrow{  g\gamma H \mapsfrom g L }  \  
G/L\ \xrightarrow{ g L\mapsto g K } \ G/K
\end{equation}
for some pair~$(L,\gamma)$ consisting of a subgroup~$L$ of~$K$ 
and an element $\gamma\in G$ such that $L\leq{^\gamma H}$.
Two such pairs~$(L,\gamma)$ and~$(L',\gamma')$ define isomorphic spans if and
only if there is an element~$k\in K$ such that $L'=L^k$ and
$\gamma^{-1}k\gamma'\in H$.
A different way to say the same thing is as an isomorphism
\[ \bA_G(H,K)\ \cong \ \bigoplus_{K \gamma H \in K \backslash G / H} A(K \cap {^\gamma H})\ , \]
where on the right hand side $A(-)$ is the Burnside ring functor for finite groups.

The bases of the morphism groups of $\bA_G$ 
lead to a more computational description of $G$-Mackey functors 
by `generators and relations'.
To specify a $G$-Mackey functor $M$, one has to give the following data: 
\begin{itemize}
\item an abelian group $M(H)$ for every finite subgroup $H$ of~$G$, 
\item a \emph{restriction homomorphism} $\res^H_K \colon M(H) \to M(K)$\index{restriction homomorphism!in a $G$-Mackey functor}
  and a \emph{transfer homomorphism} $\tr^H_K \colon M(K) \to M(H)$\index{transfer!in a $G$-Mackey functor}
for every pair of nested finite subgroups $K \leq H$ of $G$, and
\item  \emph{conjugation homomorphisms} $\gamma_\star\colon M(H^\gamma) \to M(H)$
for all $\gamma\in G$ and all finite subgroups~$H$ of~$G$.
\end{itemize}

These data must satisfy certain conditions which we do not recall here in details,
 but refer to \cite{martinezperez-nucinkis:cohomological_dimension}. 
We only summarize them briefly: restrictions, transfers and conjugations
are transitive; conjugations commute with the restriction and transfers; 
inner automorphisms act as the identity;
and finally the double coset formula holds. 
In the `generators-and-relations' description of Mackey functors,
the image of a basic transitive span~\eqref{eq:Mackey_basis_element} 
under a $G$-Mackey functor $M:\bA_G\to\cA b$ is the composite
\[  M(H)\ \xrightarrow{\res^H_{L^\gamma}} \
 M( L^\gamma)\ \xrightarrow{\ \gamma_\star\ } \
 M(L)\ \xrightarrow{\ \tr_L^K\ } \ M(K) \]
of the restriction map to~$L^\gamma$, the conjugation by the element~$\gamma$
and the transfer map to~$K$.

\begin{eg}\label{eg:G-Mackey examples}
(i) The {\em Burnside ring Mackey functor}\index{Burnside ring}
is the $G$-Mackey functor~$\mA$ given by
\[ \mA(H)\ = \ A(H)\ , \]
the Burnside ring of the finite subgroup~$H$ of~$G$.
If $G$ is finite, then $\mA$ is represented by the group~$G$ itself,
hence $\mA$ is then projective as a $G$-Mackey functor.
If $G$ is infinite, however, $\mA$ is neither representable nor projective.

(ii)
Given an abelian group $B$, the {\em constant $G$-Mackey functor} $\underline{B}$\index{G-Mackey functor@$G$-Mackey functor!constant}
is given by $\underline{B}(H)=B$,
and all restriction and conjugation maps are identity maps.
The transfer $\tr_K^H:\underline{B}(K)\to\underline{B}(H)$
is multiplication by the index $[H:K]$.

There is a well-known point set level model of an Eilenberg-Mac\,Lane spectrum~$B[\mS]$
that we recall in Example~\ref{eg:HM_constant};
for discrete groups, $B[\mS]$ is
an Eilenberg-Mac\,Lane spectrum for the constant $G$-Mackey functor $\underline{B}$.

(iii)
The {\em representation ring $G$-Mackey functor}\index{representation ring}
$R$ assigns to a finite subgroup $H$ of~$G$ the 
unitary representation ring $R(H)$, i.e., the Grothendieck
group of finite-dimensional complex $H$-representations,
with product induced by tensor product of representations.
The restriction maps are induced by restriction of representations.
The transfer maps are induced by induction of representations.

(iv)
Given any generalized cohomology theory $E$ (in the non-equivariant sense),
we can define a $G$-Mackey functor $\underline{E}$  by setting
\[   \underline{E}(H) \ = \ E^0(B H) \ , \]
the 0-th $E$-cohomology of a classifying space of the finite group $H$.
Restriction and conjugation maps come from the contravariant functoriality 
of classifying spaces in group homomorphisms.
The transfer map for a subgroup inclusion $K\leq H$ 
comes from the stable transfer map associated with the finite covering
\[ B K \ \simeq \ (E H)/K \ \to \  (E H)/ H \ = \ B H  \ .\]
As we will discuss in Examples~\ref{eg:Borel_cohomology}
and~\ref{eg:Borel global}, the $G$-Mackey functor $\underline{E}$ 
is realized by the 0th equivariant homotopy groups of a specific orthogonal $G$-spectrum,
the `$G$-Borel theory' associated with $E$.
\end{eg}

Now we link the purely algebraic concept of a $G$-Mackey functor
to the equivariant homotopy groups of orthogonal $G$-spectra.

\begin{con}\label{con:embed A_G}
We define an additive functor
\[ \Phi \ : \ \bA_G \ \to \ \Ho(\Sp_G)^{\op} \]
from the $G$-Mackey category to the opposite of the
triangulated homotopy category of orthogonal $G$-spectra.
On objects we set $\Phi(H)=\Sigma^\infty_+ G/H$.

We let $L\leq K$ be two nested finite subgroups of the discrete group $G$.
The preferred coset $e L$ is an $L$-fixed point of $G/L$,
so it defines an equivariant homotopy class
\[  u_L \ \in \ \pi_0^L(\Sigma^\infty_+ G/L) \ ,\]
compare \eqref{eq:define_tautological}.
By Proposition~\ref{prop:rephtpy} there is a unique morphism
\[ t_L^K \ : \ \Sigma^\infty_+ G/K \ \to \ \Sigma^\infty_+ G/L  \]
in the stable homotopy category~$\Ho(\Sp_G)$, characterized by the property
\[  (t_L^K)_*(u_K)\ =\ \tr_L^K(u_L) \]
in the group $\pi_0^K(\Sigma^\infty_+ G/L)$.
In other words, the morphism $t_L^K$ represents the 
transfer homomorphism $\tr_L^K:\pi_0^L(X)\to\pi_0^K(X)$ defined in~\eqref{eq:transfer}.
We emphasize that for $L\ne K$, the transfer morphism does {\em not} 
arise from an unstable $G$-map.
For finite subgroups~$H$ and~$K$ of~$G$ we can now define
\begin{equation}\label{eq:define Phi}
 \Phi \ : \  \bA_G(H,K) \ \to \ [\Sigma^\infty_+ G/K,\Sigma^\infty_+ G/H]^G 
\end{equation}
as the homomorphism that sends the basis element~\eqref{eq:Mackey_basis_element}
indexed by a pair~$(L,\gamma)$ to the composite morphism
\[ \Sigma^\infty_+ G/K \ \xrightarrow{\ t_L^K}
\Sigma^\infty_+ G/L \ \xrightarrow{\Sigma^\infty_+ \pi }\ 
\ \Sigma^\infty_+ G/H  \ . \]
Here $\pi:G/L\to G/H$ is the $G$-map defined by
\[ \pi( g L ) \ = \   g\gamma H\ .\]
The various properties of the transfer homomorphisms
translate into corresponding properties of the representing morphisms:
the normalization $t_K^K=\Id$, transitivity
\[ t_L^K\circ t_K^J  \ = \ t_L^J \]
for nested triples of finite subgroups~$L\leq K\leq J$,
and compatibility with conjugation
\[ (\Sigma^\infty_+ l_\gamma)\circ t_L^K \ = \ t_{L^\gamma}^{K^\gamma}\circ (\Sigma^\infty_+ l_\gamma) \ .\]
If $H$ and $L$ are both subgroups of~$K$,
we write $\rho^K_H:G/H\to G/K$ for the quotient map, which satisfies
$(\Sigma^\infty_+\rho^K_H)_*(u_H)=\res^K_H(u_K)$.
The double coset formula
Proposition \ref{prop:dcf} (ii)
for the orthogonal $K$-spectrum $\Sigma^\infty_+G/L$
yields
\begin{align}\label{eq:double coset tautological class}
  (t_L^K\circ \Sigma^\infty_+ \rho^K_H)_*(u_H)  \
  &= \   (t_L^K)_*((\Sigma^\infty_+\rho^K_H)_*(u_H))  \
    = \   (t_L^K)_*(\res^K_H(u_K)) \nonumber \\
  &= \       \res^K_H((t_L^K)_*(u_K))  \
    = \       \res^K_H(\tr_L^K(u_L))\nonumber  \\
  &= \  \sum_{K \gamma L \in H\backslash K/L}   
    \tr^H_{H\cap{^\gamma L}}( \gamma_\star( \res^L_{H^\gamma \cap L}(u_L))) \nonumber \\
        &= \  \sum_{K \gamma L \in H\backslash K/L}   
\left((\Sigma^\infty_+\rho^L_{H^\gamma\cap L})\circ (\Sigma^\infty_+ l_\gamma) \circ   t^H_{H\cap{^\gamma L}}\right)_*(u_H)\ . 
\end{align}
The last equality exploits the relations
\begin{align*}
  \tr^H_{H\cap{^\gamma L}}( \gamma_\star( \res^L_{H^\gamma \cap L}(u_L)))  \
  &= \ \tr^H_{H\cap{^\gamma L}}(\gamma_\star( (\Sigma^\infty_+ \rho^L_{H^\gamma \cap L})_*(u_{H^\gamma\cap L})))  \\
  &= \  \tr^H_{H\cap{^\gamma L}}((\Sigma^\infty_+\rho^L_{H^\gamma \cap L})_*( \gamma_\star (u_{H^\gamma\cap L})))  \\
  &= \ \tr^H_{H\cap{^\gamma L}}( (\Sigma^\infty_+\rho^L_{H^\gamma\cap L})_*((\Sigma^\infty_+ l_\gamma)_*( u_{H\cap{^\gamma L}}))) \\
  &= \ (\Sigma^\infty_+\rho^L_{H^\gamma\cap L})_*((\Sigma^\infty_+ l_\gamma)_*(\tr^H_{H\cap{^\gamma L}}(u_{H\cap{^\gamma L}}))) \\
  &= \ (\Sigma^\infty_+\rho^L_{H^\gamma\cap L})_*((\Sigma^\infty_+ l_\gamma)_*((t^H_{H\cap{^\gamma L}})_*(u_H))) \\
  &= \ \left((\Sigma^\infty_+\rho^L_{H^\gamma\cap L})\circ (\Sigma^\infty_+ l_\gamma) \circ   t^H_{H\cap{^\gamma L}}\right)_*(u_H) \ .
\end{align*}
By the representability property of Proposition \ref{prop:rephtpy},
the relation \eqref{eq:double coset tautological class} implies the relation
\[  t_L^K\circ \rho^K_H  \ = \ \sum_{K \gamma L \in H\backslash K/L} 
(\Sigma^\infty_+\rho^L_{H^\gamma\cap L})\circ (\Sigma^\infty_+ l_\gamma) \circ   t^H_{H\cap{^\gamma L}} 
\]
as morphisms $\Sigma^\infty_+G/H\to\Sigma^\infty_+ G/L$.
Altogether, these properties imply functoriality of the homomorphisms~$\Phi$.
\end{con}

\begin{eg}[$G$-Mackey functor of an orthogonal $G$-spectrum]\label{eg:Mackey of G-spectrum}\index{G-Mackey functor@$G$-Mackey functor!of an orthogonal $G$-spectrum}
  We can now associate a $G$-Mackey functor $\upi_0(X)$ to every orthogonal $G$-spectrum $X$,
  namely as the composite functor
  \[ \bA_G \ \xrightarrow{\ \Phi\ }\ \Ho(\Sp_G)^{\op}\ \xrightarrow{[-,X]^G}\ \cA b\ , \]
  where $\Phi$ was introduced in Construction \ref{con:embed A_G}.
  We take the time to translate this definition into the `explicit' description of Mackey functors
  in terms of restriction, conjugation and transfer homomorphisms.
  For every finite subgroup $H$ of $G$, evaluation at the class $u_H\in \pi_0^H(\Sigma^\infty_+ G/H)$
  is an isomorphism
  \[ \upi_0(X)(H) \ = \ [\Sigma^\infty_+ G/H,X]^G \ \iso \ \pi_0^H(X)\ ,\]
  see Proposition \ref{prop:rephtpy}.
  Now we let $K\leq H$ be nested finite subgroups of $G$. 
  Under the above identification, the restriction map $\res^H_K:\pi_0^H(X)\to\pi_0^K(X)$
  becomes a special case
  of the restriction homomorphism \eqref{eq:restriction_hom} for the inclusion $K\to H$.
  The transfer map $\tr^H_K:\pi_0^K(X)\to\pi_0^H(X)$ becomes the one
  defined in Construction \ref{con:transfer};
  the conjugation homomorphism $\gamma_\star:\pi_0^{H^\gamma}(X)\to\pi_0^H(X)$
  was defined in \eqref{eq:define conjugation}.
  Since $H$ and $K$ are finite, the groups $\pi_0^H(X)$ and $\pi_0^K(X)$
  have the explicit colimit descriptions \eqref{eq:define pi^H},
  and in this picture, restriction, conjugation and transfer are the `classical' ones in the
  context of equivariant homotopy theory of finite groups,
  see for example Constructions 3.1.5 and 3.2.7 of \cite{schwede:global}.
\end{eg}

In Corollary \ref{cor:t-structure} above
we established a non-degenerate t-structure
on the $G$-equivariant stable homotopy category $\Ho(\Sp_G)$,
in which the classes of connective and coconnective objects
are detected by equivariant homotopy groups at all compact subgroups of $G$.
The heart $\cH$ of this t-structure consist
of those orthogonal $G$-spectra~$X$ such that $\pi_n^H(X)=0$
for all compact subgroups~$H$ of~$G$ and all~$n\ne 0$.
In the special case of discrete groups, we will now identify the heart
with the abelian category of $G$-Mackey functors.

Part (i) of the following Theorem \ref{thm:embed_MG_into_Ho(Sp)} 
says that for every finite subgroup $H \leq G$, 
the $G$-Mackey functor $\upi_0( \Sigma^{\infty}_+ G/H)$ 
is a \emph{free $G$-Mackey functor} represented by the object $H$ of~$\bA_G$.
Part (ii) implies that every $G$-Mackey functor arises from an orthogonal $G$-spectrum,
see also Remark \ref{rk:general Eilenberg Mac Lane}.

\begin{thm}\label{thm:embed_MG_into_Ho(Sp)} 
Let $G$ be a discrete group.\index{t-structure!on the equivariant stable homotopy category} 
\begin{enumerate}[\em (i)]
\item 
The maps \eqref{eq:define Phi}
define a fully faithful functor $\Phi:\bA_G\to\Ho(\Sp_G)^{\op}$.
\item The functor 
\[ \upi_0 \ : \ \cH \ \to \ \cM_G \]
is an equivalence of categories from the heart of the t-structure 
on $\Ho(\Sp_G)$ to the category of $G$-Mackey functors.\index{G-Mackey functor@$G$-Mackey functor}
\end{enumerate}
\end{thm}
\begin{proof}
(i) The argument is essentially the same as for finite groups, so we will be brief.  
The maps $\Phi:\bA_G(H,K)\to[\Sigma^\infty_+ G/K,\Sigma^\infty_+ G/H]^G$
are additive, by definition, and they send the identity of~$H$ 
to the identity of~$\Sigma^\infty_+ G/H$.

To see that $\Phi$ is fully faithful it suffices,
by Proposition~\ref{prop:rephtpy}, to show that the map
\[ \bA_G(H,K)\ \to \ \pi_0^K(\Sigma^\infty_+ G/H) \]
sending the basis element~\eqref{eq:Mackey_basis_element} to the class
$\tr_L^K( \gamma_\star(\res^H_{L^\gamma}(u_H)))$ is an isomorphism.
By \cite[Thm.\,3.3.15 (i)]{schwede:global}, the group $\pi_0^K(\Sigma^\infty_+ G/H)$
is free abelian, with a basis given by the classes $\tr^K_L(\sigma^L(\gamma H))$,
where $L$ runs through conjugacy classes of subgroups of $K$,
and $\gamma H$ runs through $W_K L$-orbits of the set $(G/H)^L $,
and $\sigma^L(\gamma H)$ is the class in $\pi_0^L(\Sigma^\infty_+ G/H)$
represented by the $L$-map $S^0\to G/H_+=(\Sigma^\infty_+ G/H)(0)$
sending 0 to $\gamma H$.
The fact that $\gamma H$ is an $L$-fixed point of $G/H$ precisely means that
$L^\gamma\leq H$, and in our present notation we have
\[ \sigma^L(\gamma H)\ = \
   \gamma_\star(\sigma^{L^\gamma}(e H))\ = \ 
  \gamma_\star(\res^H_{L^\gamma}(u_H))
  \ . \]
So our claim follows from the fact that sending
$\gamma H\in (G/H)^L$ to the equivalence class of the span \eqref{eq:Mackey_basis_element}
passes to a bijection between the $W_K L$-orbits of $(G/H)^L$
and the equivalence classes of transitive span in which the middle term is
isomorphic to $G/L$.
Altogether, this shows that the functor~$\Phi$ 
takes the preferred basis of~$\bA_G(H,K)$
given by `transitive spans' to a basis of $\pi_0^K(\Sigma^\infty_+ G/H)$,
so it is an isomorphism.

(ii)
We denote by $\End$ the `endomorphism category' of the preferred
small generators, i.e., the full pre-additive subcategory of $\Ho(\Sp_G)$ 
with objects $\Sigma^\infty_+ G/H$ for all finite subgroups $H$ of $G$.
By an $\End$-module we mean an additive functor
\[ \End^{\op} \ \to \ \cA b \]
from the opposite category of $\End$.
The tautological functor
\begin{equation}\label{eq:tautological functor}
\Ho(\Sp_G) \ \to \ \text{mod-}\End
\end{equation}
takes an object $X$ to the restriction of the 
contravariant Hom-functor $[-,X]^G$ to the full subcategory $\End$.
By Proposition \ref{cor:Ho_G compactly generated},
the spectra $\Sigma^\infty_+ G/H$ form a set of small weak generators 
for the triangulated category $\Ho(\Sp_G)$; 
moreover, the group of maps from a generator to a positive shift of
any other generator is trivial, compare \eqref{eq:positivity}. 
So \cite[Thm.\,III.3.4]{beligiannis-reiten:torsion_theories}
applies and shows that the restriction 
of the tautological functor \eqref{eq:tautological functor}
to the heart of the t-structure is an equivalence of categories
\[ \ \cH \ \xrightarrow{\ \iso \ } \ \text{mod-}\End \ .\]
Part (i) shows that the functor $\Phi:\bA_G\to \End^{\op}$ is an isomorphism
of pre-additive categories, 
so it induces an isomorphism between the category of $\End$-modules
and the category of $G$-Mackey functors.
This equivalence turns the $\End$-module $[-,X]^G$
into the $G$-Mackey functor $\upi_0(X)$. This completes the proof.
\end{proof}

\begin{rk}[Eilenberg-Mac\,Lane spectra for $G$-Mackey functors]
\label{rk:general Eilenberg Mac Lane}
\index{Eilenberg-Mac\,Lane spectrum|(}
For discrete groups~$G$, part~(ii) of Theorem \ref{thm:embed_MG_into_Ho(Sp)} 
in particular provides an Eilenberg-Mac\,Lane spectrum
for every $G$-Mackey functor~$M$, i.e., an orthogonal $G$-spectrum $H\!M$
such that $\upi_k(H\!M)=0$ for all~$k\ne 0$
and such that the $G$-Mackey functor $\upi_0(H\!M)$ is isomorphic to $M$;
and these properties characterize $H\!M$ 
up to preferred isomorphism in~$\Ho(\Sp_G)$. 
Indeed, a choice of inverse to the equivalence~$\upi_0$ of
Theorem \ref{thm:embed_MG_into_Ho(Sp)}~(ii), composed with the inclusion of
the heart, provides an Eilenberg-Mac\,Lane functor
\[ H \ : \ \cM_G \ \to \ \Ho(\Sp_G) \]
to the stable $G$-homotopy category.
\end{rk}

The previous remark constructs
Eilenberg-Mac\,Lane spectra associated to $G$-Mackey functors;
the stable $G$-homotopy type is determined by the algebraic input data up 
to preferred isomorphism, but the construction is an abstract version
of `killing homotopy groups' and does not yield an explicit pointset level model.
In the next example we recall a well-known
pointset level construction that yields an 
Eilenberg-Mac\,Lane spectrum for the constant $G$-Mackey functor,
compare Example \ref{eg:G-Mackey examples} (ii).

\begin{eg}[Eilenberg-Mac\,Lane spectra for constant Mackey functors]\label{eg:HM_constant}
Let $B$ be an abelian group. The orthogonal Eilenberg-Mac\,Lane spectrum
$B[\mS]$ is defined at an inner product space~$V$ by
\[ B[\mS](V) \ = \ B[S^V] \ ,  \]
the reduced $B$-linearization of the $V$-sphere. 
The underlying set of this space consists of finite linear combinations
of elements of $S^V$ with coefficients in $B$, modulo
the subgroup of $B$-multiples of the basepoint.
The topology is as a quotient space of $\amalg_{n\geq 0} B^n\times (S^V)^n$.

The orthogonal group~$O(V)$ acts through the action on $S^V$ and
the structure map
$\sigma_{V,W}:S^V\sm B[\mS](W)\to B[\mS](V\oplus W)$ is given by
\[ S^V\sm B[S^W] \ \to\  B[S^{V\oplus W}] \ , \quad
  v\sm \big( {\sum}_i  b_i\cdot w_i\big) \ \longmapsto \
{\sum}_i\ b_i\cdot \, (v\sm w_i) \  . \]
The underlying non-equivariant space of $B[S^V]$ 
is an Eilenberg-Mac\,Lane space of type $(B,n)$, where $n=\dim(V)$. 
Hence the underlying non-equivariant homotopy type of $B[\mS]$ is that of
an Eilenberg-Mac\,Lane spectrum for $B$.
If $G$ is any Lie group, then $B[\mS]$ becomes an orthogonal $G$-spectrum
by letting $G$ act trivially.
We warn the reader that for compact Lie groups of positive dimension,
the equivariant homotopy groups of $B[\mS]$ are {\em not} generally concentrated
in dimension zero; for example, the group $\pi_1^{U(1)}(\mZ[\mS])$
is isomorphic to $\mQ$ by \cite[Thm.\,5.3.16]{schwede:global}.
Also, the group $\pi_0^G(B[\mS])$ may not be isomorphic to~$B$;
for example, the group $\pi_0^{S U(2)}(\mZ[\mS])$ has rank 2
by \cite[Ex.\,4.16]{schwede:equivariant_properties}.

However, for {\em discrete} groups~$G$, the equivariant behavior of $B[\mS]$ is as expected,
and the orthogonal $G$-spectrum $B[\mS]$ 
is an Eilenberg-Mac\,Lane spectrum
of the constant $G$-Mackey functor $\underline{B}$.\index{G-Mackey functor@$G$-Mackey functor!constant}
Indeed, $B[\mS]$ is obtained from a $\Gamma$-space $B[-]$
by evaluation on spheres. For every finite subgroup $K$ of $G$,
we can view this $\Gamma$-space as a $\Gamma$-$K$-space
by letting $K$ act trivially. For every finite $K$-set $S$, the map 
\[ P_S \ : \ B[S] \ \to \ B[1_+]^S\ = \ B^S \]
is then a homeomorphism, so in particular a $K$-homotopy equivalence,
and $B[-]$ is a very special $\Gamma$-$K$-space in the sense 
of Shimakawa~\cite[Def.\,1.3]{shimakawa:infinite_loop}.
Since~$\pi_0(B[1^+])$ is a group (as opposed to a monoid only),
Shimakawa's Theorem~B proves that the adjoint structure maps 
$\tilde\sigma_{V,W}: B[S^V]\to\map(S^W, B[S^{V\oplus W}])$
are $K$-weak equivalences, see also \cite[Thm.\,B.61]{schwede:global}.
Since $K$ was an arbitrary finite subgroup of $G$,
this shows that $B[\mS]$ is a $G$-$\Omega$-spectrum,
and an Eilenberg-Mac\,Lane spectrum for $\underline{B}$. 
\end{eg}
\index{Eilenberg-Mac\,Lane spectrum|)}

\section{Rational proper stable homotopy theory}

\index{orthogonal $G$-spectrum!rational|(} 
The purpose of this section is to give an algebraic model for the
rational proper $G$-equivariant stable homotopy category of a discrete group.
We call an orthogonal $G$-spectrum $X$ {\em rational}
if the equivariant homotopy groups $\pi_k^H(X)$ are uniquely divisible 
(i.e., $\mQ$-vector spaces) for all compact subgroups $H$ of~$G$.
For discrete groups $G$,
Theorem~\ref{thm:rational SH} below shows that the full triangulated
subcategory of rational $G$-spectra inside $\Ho(\Sp_G)$
is equivalent to the unbounded derived category of rational $G$-Mackey functors.
As in the case of finite groups, the abelian category of
rational $G$-Mackey functors can be simplified by
`dividing out transfers', see Proposition \ref{prop:rational divide transfer}.
However -- in contrast to the case of finite groups -- this
abelian category is in general {\em not} semisimple, see Remark \ref{rk:not semisimple}.

\begin{rk}
Let $G$ be a Lie group and $X$ a rational orthogonal $G$-spectrum.
For $n\in\mZ$ we let $n\cdot X\in [X,X]^G$ 
denote the $n$-fold sum of the identity morphism of $X$.
For every compact subgroup $H$ of $G$,
the morphism $n\cdot X$ induces multiplication by $n$ on $\pi_*^H(X)$,
which is invertible since $X$ is rational.
This means that $n\cdot X$ is an isomorphism in $\Ho(\Sp_G)$.
Hence the endomorphism ring $[X,X]^G$ of $X$ in $\Ho(\Sp_G)$
is a $\mQ$-algebra.
So all morphism groups in the full subcategory $\Ho^\mQ(\Sp_G)$ of rational spectra 
are uniquely divisible, i.e., $\Ho^\mQ(\Sp_G)$ is a $\mQ$-linear category.
\end{rk}

\begin{prop}\label{prop:rational in degree 0} 
  Let $H$ and $K$ be finite subgroups of a discrete group $G$. 
  Then the equivariant homotopy group $\pi_k^K(\Sigma^\infty_+ G/H)$ is torsion for every $k>0$.
\end{prop}
\begin{proof} 
  The underlying $K$-space of $G/H$ is the disjoint union of its $K$-orbits
  $K(g H)$;
  this becomes a wedge decomposition after passing to unreduced suspension spectra.
  Since equivariant homotopy groups takes wedges to direct sums,
  $\pi_k^K(\Sigma^\infty_+ G/H)$ is isomorphic to the direct sum,
  indexed by $K$-$H$-double coset, of the groups
  $\pi_k^K(\Sigma^\infty_+ K (g H) )$.
  The $K$-set $K(g H)$ is isomorphic to $K/(K\cap {^g H})$,
  so the Wirthm{\"u}ller isomorphism \cite{wirthmuller:equivariant_homology}, 
  see also \eqref{eq:wirth} or \cite[Thm.\,3.2.15]{schwede:global},
  provides an identification
  \[   \pi_*^K(\Sigma^\infty_+ K(g H)) \ \cong\ 
    \pi_*^K(\Sigma^\infty_+  K / (K\cap {^g H} ))
    \ \cong\ \pi_*^{K\cap{^g H}}(\mS_G) \ . \]
  The claim thus follows because for every finite group~$L$,
  the $L$-equivariant stable stems are finite in positive degrees.
  To see this, we can exploit the fact that
  the groups $\pi_k^L(\mS_G)$ can rationally be recovered
  as the product of the $W_L J$-fixed subgroup of the geometric fixed point 
  homotopy groups $\Phi_k^J(\mS_G)$, see for example \cite[Cor.\,3.4.28]{schwede:global}; 
  the latter groups are stable homotopy groups of spheres,
  which are finite in positive degrees.
\end{proof}

Before establishing an algebraic model for the rational stable $G$-homotopy category,
we first recall the two rational model structures to be compared.
We let $\cA$ be a pre-additive category, such as the 
Mackey category $\bA_G$.
We denote by $\cA\mo$ the category of additive functors from
$\cA$ to the category of $\mQ$-vector spaces.
This is an abelian category, and the rationalized represented functors
$\mQ\tensor \cA(a,-)$, for all objects $a$ of $\cA$, form a set
of finitely presented projective generators of $\cA\mo$.
The category of $\mZ$-graded chain complexes in the abelian category $\cA\mo$
then admits the {\em projective model structure} with the quasi-isomorphisms
as weak equivalences. The fibrations in the projective model structure
are those chain morphisms that are surjective in every chain complex degree
and at every object of $\cA$.
This projective model structure for complexes of $\cA$-modules is a
special case of \cite[Thm.~5.1]{christensen-hovey:relative}.

We also need the rational version of the stable model structure
on orthogonal $G$-spectra established in Theorem \ref{thm:stable}. 
We call a morphism $f:X\to Y$ of orthogonal $G$-spectra a
{\em rational equivalence}\index{rational equivalence!of orthogonal $G$-spectra}
if the map
\[ \mQ\otimes \pi_k^H(f)\ : \ \mQ\otimes\pi_k^H(X)\ \to\ \mQ\otimes\pi_k^H(Y) \]
is an isomorphism for all integers $k$ and all compact subgroups $H$ of $G$.

\begin{thm}[Rational stable model structure]\label{thm:rational stable} 
  Let $G$ be a Lie group.\index{stable model structure!for orthogonal $G$-spectra!rational}
  \begin{enumerate}[\em (i)]
  \item 
    The rational equivalences and the cofibrations 
    are part of a model structure on the category of orthogonal $G$-spectra, 
    the {\em rational stable model structure}.
  \item
    The fibrant objects in the rational stable model structure 
    are the rational $G$-$\Omega$-spectra.\index{G-Omega-spectrum@$G$-$\Omega$-spectrum!rational}
  \item
    The rational stable model structure is cofibrantly generated, 
    proper and topological. 
  \end{enumerate}
\end{thm}

Theorem \ref{thm:rational stable} is obtained
by Bousfield localization of the stable model structure on
orthogonal $G$-spectra, and one can use a similar proof as for the rational stable
model structure on sequential spectra in \cite[Lemma 4.1]{schwede-shipley:uniqueness}.
We omit the details. 

\index{G-equivariant stable homotopy category@$G$-equivariant stable homotopy category!rational|(} 
\begin{thm}\label{thm:rational SH} 
Let~$G$ be a discrete group.
There is a chain of Quillen equivalences between
the category of orthogonal $G$-spectra with the 
rational stable model structure and the category of chain complexes of
rational $G$-Mackey functors. In particular, this induces
an equivalence of triangulated categories
\[  \Ho^\mQ(\Sp_G)\ \xrightarrow{\ \iso \ } \ \cD\big( \cM_G^\mQ \big) \ . \]
The equivalence can be chosen so that the homotopy $G$-Mackey functor
on the left hand side corresponds to the homology $G$-Mackey functor 
on the right hand side.
\end{thm}
\begin{proof}
We prove this as a special case of the `generalized tilting theorem'
of Brooke Shipley and the fifth author. 
Indeed, by Corollary~\ref{cor:Ho_G compactly generated} 
the unreduced suspension spectra of the $G$-sets $G/H$ 
are small weak generators of the stable $G$-homotopy category
$\Ho(\Sp_G)$ as $H$ varies through all finite subgroups of~$G$.
So the rationalizations  $(\Sigma^\infty_+ G/H )_\mQ$ 
are small weak generators of the rational stable $G$-homotopy category $\Ho^\mQ(\Sp_G)$. 

By Proposition~\ref{prop:rephtpy} the evaluation map
\[ [\Sigma^\infty_+ G/H, X]^G_* \ \to \ \pi_*^H( X) \ , \quad 
[f] \longmapsto f_*(u_H)\]
is an isomorphism, where $u_H\in \pi_0^H(\Sigma^\infty_+ G/H)$
is the tautological class.
So the graded morphism groups between the small generators are given by
\begin{align*}
[ (\Sigma^\infty_+ G/K)_\mQ[k],\ (\Sigma^\infty_+ G/H )_\mQ] \ &\cong \
\pi_k^K((\Sigma^\infty_+  G/H )_\mQ) \ 
\cong \ \mQ\otimes \pi_k^K( \Sigma^\infty_+ G/H ) \\ 
&\cong \ 
\begin{cases}
  \mQ\otimes \bA_G(H,K) &\text{\quad for $k=0$, and}\\
\qquad 0 &\text{\quad for $k\ne0$.}
\end{cases}
\end{align*}
Here we have used Theorem \ref{thm:embed_MG_into_Ho(Sp)},
Proposition \ref{prop:rational in degree 0},
and the fact that the equivariant homotopy groups of 
unreduced suspension spectra vanish in negative dimensions,
see for example \cite[Prop.\,3.1.44]{schwede:global}.
The rational stable model structure on orthogonal $G$-spectra is 
topological (hence simplicial),
cofibrantly generated, proper and stable; 
so we can apply the `generalized tilting theorem' \cite[Thm.\,5.1.1]{schwede-shipley:stable_model}. 
This theorem yields a chain
of Quillen equivalences between orthogonal $G$-spectra 
in the rational stable model structure and the category of chain complexes of
$\mQ\otimes\bA_G$-modules, i.e., 
additive functors from the rationalized $G$-Mackey category $\mQ\otimes \bA_G$
to abelian groups. 
This functor category is equivalent to the category of additive functors
from $\bA_G$ to $\mQ$-vector spaces, and this proves the theorem.
\end{proof}
\index{orthogonal $G$-spectrum!rational|)}

\index{G-Mackey functor@$G$-Mackey functor!rational|(} 
There is a further algebraic simplification of the category of rational 
$G$-Mackey functors: Given a $G$-Mackey functor $M$ and a 
finite subgroup $H$ of $G$, we let $\tau(M)(H)$ denote the quotient of 
$M(H)$ by all transfers from proper subgroups of $H$. The values 
$\tau(M)(H)$ no longer assemble into a $G$-Mackey functor, but they 
inherit induced conjugation maps $g_{\star}:\tau(M)(H^g)\to \tau(M)(H)$ since 
conjugations and transfers of $G$-Mackey functors commute. We let 
$\Conj_G$ denote the conjugation category of $G$, i.e., the category 
with objects the finite subgroups of $G$ and morphisms 
$\Conj_G(H,K)=\{g\in G\ |\ ^g H=K\}/K$, the set of elements of $G$ which 
conjugate $H$ onto $K$, modulo conjugation by elements of $K$. Then $\tau(M)$ 
naturally forms a covariant functor from $\Conj_G$ to abelian groups.

It turns out that every {\em rational} $G$-Mackey functor $M$ can be 
reconstructed uniquely from the $\Conj_G$-functor $\tau(M)$:

\begin{prop}\label{prop:rational divide transfer}
  For every discrete group $G$,
  the functor $\tau:\cM_G^\mQ\to\cF(\Conj_G,\mQ)$
  is an equivalence of abelian categories.
\end{prop}
\begin{proof}
  We explain how to deduce the claim from the finite group case, 
  which can be found in \cite[App. A]{greenlees-may:generalized_Tate}
  or \cite[Thm.\,3.4.22]{schwede:global}. 
  When comparing to \cite[App. A]{greenlees-may:generalized_Tate},
  one must use that quotienting $M(H)$ by all proper transfers can be identified with 
  inverting the idempotent called $e_H$ 
  in \cite[App. A]{greenlees-may:generalized_Tate}.
  
  Since $\tau$ commutes with colimits and $\cM_G^\mQ$ is a functor category,
  $\tau$ has a right adjoint $R:\cF(\Conj_G,\mQ)\to \cM_G^\mQ$.
  The value of the right adjoint at a $\Conj_G$-functor $N$ is given by 
  \[ R(N)(H)\ =\ \text{Nat}_{\Conj_G}( \tau(\mQ\otimes\bA_G(H,-)),N) \ ,\] 
  where $H$ is a finite subgroup of $G$ and $\bA_G(H,-)$ denotes the 
  represented $G$-Mackey functor.
  We claim that, as a $\mQ$-vector space, $R(N)(H)$ only depends 
  on the underlying $\Conj_H$-functor of $N$. For this we note that $\tau 
  (\mQ\otimes \bA_G(H,-))(K)$ can be identified with the $\mQ$-linearization 
  of the set of $G$-equivariant maps from $G/K$ to $G/H$. This set 
  corresponds to the subset of elements $g\in G$ for which the conjugate 
  $K^g$ is a subgroup of $H$, modulo the right $H$-action.
  Stated differently, it is the disjoint union of the sets of all $g$
  such that $^g J$ is equal to $K$, where $J$ ranges
  through all subgroups of $H$, again modulo $H$.
  This yields an isomorphism of $\Conj_G$-modules
  \[
    \tau(\mQ\otimes \bA_G(H,-))\ \iso\ (\bigoplus_{J\subseteq H} \mQ[\Conj_G(J,-)])/H\ .
  \]
  Thus, for every rational $\Conj_G$-functor $N$ the morphism group 
  $\text{Nat}_{\Conj_G}(\tau(\mQ\otimes \bA_G(H,-)),N)$ is naturally isomorphic 
  to $(\bigoplus_{J\subseteq H}N(J))^H$. In particular, it only depends on 
  the underlying $\Conj_H$-functor of $N$, which proves the claim.
  
  By definition, the value $\tau(M)(H)$ also only depends on the 
  underlying $H$-Mackey functor of $M$. Both in $G$-Mackey functors and in 
  $\Conj_G$-functors isomorphisms are tested levelwise, so we can reduce 
  to the finite group case to see that unit and counit of the adjunction 
  are isomorphisms, compare \cite[Thm.\,3.4.22]{schwede:global}. This finishes the proof.
\end{proof}
\index{G-Mackey functor@$G$-Mackey functor!rational|)}

The category $\Conj_G$ is a groupoid and equivalent to the disjoint 
union of Weyl groups $W_G H=(N_G H)/H$, where~$H$ ranges through a system of 
representatives of conjugacy classes of finite subgroups. Hence the category of 
$\Conj_G$-functors is equivalent to the product of the 
$\mQ[W_G H]$-module categories. Since forming derived categories commutes 
with products of abelian categories, we get:

\begin{cor} Let $G$ be a discrete group.
The rational stable $G$-homotopy category is equivalent to 
the product of the derived categories of $\mQ[W_G H]$-modules, where $H$ 
ranges through a system of representatives of conjugacy classes of finite
subgroups of $G$. \end{cor}

The equivalence of the previous corollary is implemented by `geometric fixed points', 
see \cite[Prop.\,3.4.26]{schwede:global} for the precise statement.

\begin{rk} \label{rk:not semisimple}
There is an important homological difference between
rational $G$-Mackey functors for finite groups versus infinite discrete groups.
If $G$ is finite, all Weyl groups of subgroups are also finite and hence 
the abelian category of rational $G$-Mackey functors
is semisimple. So every object is projective and injective and the 
derived category is equivalent, by taking homology, to the category of 
graded rational Mackey functors over $G$.

This does not generalize to rational $G$-Mackey functors
for infinite discrete groups.
Indeed, already the simplest case $G=\mZ$ illustrates this.
Since the trivial subgroup is the only finite subgroup of~$\mZ$,
the category of $\mZ$-Mackey functors is equivalent to the
category of abelian groups with a~$\mZ$-action; this category in turn
is equivalent to the category of modules over the Laurent series 
ring~$\mZ[t,t^{-1}]$.
So $\Ho^\mQ(\Sp_\mZ)$ is equivalent to the derived category of
the ring $\mQ[t,t^{-1}]$ which has global dimension~1.
For example,~$\mQ$, with $t$ acting as the identity, is not projective.
\end{rk}

\begin{eg}\label{eg:rational Z-equivariant}
As an example we consider the $G$-sphere spectrum~$\mS_G$.
For every finite subgroup~$H$ of~$G$, 
the group~$\pi_k^H(\mS_G)$ is the $k$-th $H$-equivariant stable stem.
So this group is trivial for negative $k$, finite for positive~$k$,
and isomorphic to the Burnside ring of $H$ for~$k=0$. Hence
\[ \mQ\otimes \upi_k(\mS_G)\ \cong \ 
\begin{cases}
 \mQ\otimes \mA & \text{ for $k=0$, and}\\
 \quad 0  & \text{ for $k\ne0$.}
\end{cases} \]
Since the rationalized homotopy group $G$-Mackey functors are
concentrated in a single degree, the equivalence of categories of
Theorem~\ref{thm:rational SH} takes the $G$-sphere spectrum to
the $G$-Mackey functor~$\mQ\otimes \mA$, considered as a complex in
degree~0.

The equivalence of categories in particular induces isomorphisms
of the graded endomorphism rings of corresponding objects.
The graded endomorphism ring of a $G$-Mackey functor $M$ in the derived category
is its Ext algebra, i.e., the graded abelian group $\Ext^*_{\cM_G}(M,M)$
with multiplication by Yoneda product (splicing of exact sequences).
So we conclude that
\[ \mQ\otimes [\mS_G[k],\mS_G]^G \ \cong \ 
\cD(\cM_G)(\mQ\otimes \mA[k],\mQ\otimes \mA)\ \cong \ 
\Ext^{-k}_{\cM_G}(\mQ\otimes \mA,\mQ\otimes \mA)\ . \]
For infinite groups~$G$, the Burnside ring $G$-Mackey functor
is typically neither projective nor injective, and has non-trivial
Ext groups in non-zero degrees. So
for infinite groups~$G$, the $G$-sphere spectrum typically has
non-trivial stable self-maps of {\em negative} degrees.

Again, the simplest case $G=\mZ$ already illustrates this phenomenon.
As we explained in the previous remark, the category of~$\mZ$-Mackey functors
is equivalent, by evaluation at the trivial subgroup, to the 
category of~$\mZ[t,t^{-1}]$-modules, and the Burnside ring $\mZ$-Mackey functor
corresponds to~$\mZ$ with~$t$ acting as the identity.
The Ext algebra of this $\mZ[t,t^{-1}]$-module is an exterior algebra
on a class in $\Ext^1_{\mZ[t,t^{-1}]}(\mZ,\mZ)$.
So the rationalized algebra is an exterior algebra over~$\mQ$ 
on one generator of (cohomological) degree~1.

The exterior generator in $\mQ\otimes [\mS_\mZ,\mS_\mZ[1]]^\mZ$ is realized
by the universal cover of the circle, in the following sense.
The real line~$\mR$ is a $\mZ$-space by translation:
\[ \mZ\times \mR\ \to \ \mR \ , \quad (n,x)\ \longmapsto \ n+x \ .\]
In fact, this action makes $\mR$ into a universal space
for the group $\mZ$ (for proper actions or, equivalently, for free actions).
Since~$\mR$ is non-equivariantly contractible, 
the unique map~$\mR\to\ast$ is a $\Com$-weak equivalence, so we obtain
a weak $\mZ$-map from a point to $S^1$ as the composite
\[ \ast \ \xleftarrow{\ \simeq\ } \ \mR \ \xrightarrow{\quad} \ S^1 \ , \]
where the right map is a universal cover.
Adding disjoint basepoints to~$\ast$ and~$\mR$,
passing to suspension spectra and going into the stable homotopy category
$\Ho(\Sp_\mZ)$ produces a non-trivial self map of $\mS_\mZ$ of degree~$-1$.
\end{eg}
\index{G-equivariant stable homotopy category@$G$-equivariant stable homotopy category!rational|)}

\chapter{Proper equivariant cohomology theories}

\section{Excisive functors from \texorpdfstring{$G$}{G}-spectra}
\label{sec:excisive}

In this section we discuss `excisive functors',
i.e., contravariant homotopy functors defined on finite proper $G$-CW-complexes
that satisfy excision for certain pushouts, see Definition \ref{def:excisive functor}.
Excisive functors are the components of proper $G$-cohomology theories,
to be studied in Section \ref{sec:proper cohomology} below.
In Construction \ref{con:naive G-Omega} we recall the classical procedure
to define an excisive functor from a sequential $G$-spectrum.
Remark \ref{rk:Davis-Luck} explains why the excisive functors
represented by sequential $G$-spectra are precisely the ones
represented by `$G$-orbit spectra' in the sense of
Davis and the third author \cite[Def.\,4.1]{davis-luck:spaces_assembly}, i.e., by
contravariant functors from the $\Fin$-orbit category of $G$ to spectra.

In Definition \ref{def:proper from spectrum} we explain that orthogonal $G$-spectra
also define excisive functors by taking morphism groups in
the triangulated stable homotopy category $\Ho(\Sp_G)$ from unreduced suspension spectra.
As we show in the proof of Proposition \ref{prop:represented is excisive},
every such `represented' cohomology theory is also
represented by a sequential $G$-spectrum, namely the underlying sequential
$G$-spectrum of a $\pi_*$-isomorphic orthogonal $G$-$\Omega$-spectrum.
While the represented functor $[\Sigma^\infty_+(-),E]^G$ is easily seen
to extend to a proper $G$-cohomology theory, it does not come with explicit `cycles'
that represent cohomology classes. This makes it difficult
to compare the represented $G$-cohomology theory with other theories,
such as equivariant cohomotopy or equivariant K-theory.
To remedy this, Construction \ref{con:define alternative}
introduces a more down-to-earth description, in the case of discrete groups,
based on parameterized equivariant homotopy theory, of the excisive functor
represented by an orthogonal $G$-spectrum $E$.
The construction generalizes the equivariant cohomotopy groups of
the third author \cite[Sec.\,6]{luck:burnside_ring}, which is the special case $E=\mS_G$
of the equivariant sphere spectrum; many of the arguments are inspired by that special case.
We then show in Theorem \ref{thm:rep2vect} that for discrete groups,
the new theory agrees with the represented theory.

\index{excisive functor|(}
\begin{defn}\label{def:excisive functor}
  Let $G$ be a Lie group. A functor
    \[ \cH\ : \ \text{(finite proper $G$-CW-complexes})^{\op}\ \to \ \cA b \]
    is {\em excisive} if it satisfies the following conditions:
  \begin{enumerate}[(i)]
  \item (Homotopy invariance) Let $f,f':Y\to X$ be continuous $G$-maps 
    between finite proper $G$-CW-complexes
    that are equivariantly homotopic.
    Then $\cH(f)=\cH(f')$.
  \item (Additivity) For all finite proper $G$-CW-complexes $X$ and $Y$,
    the map
    \[ (i_X^*,i_Y^*)\ : \ \cH(X\amalg Y)\ \to \ \cH(X)\times\cH(Y) \]
    is bijective, where $i_X:X\to X\amalg Y$ and $i_Y:Y\to X\amalg Y$ are the
    summand inclusions.
  \item (Excision)
    Let $(X,A)$ and $(Y,B)$ be two finite proper $G$-CW-pairs, and let
    \begin{equation}\begin{aligned}\label{eq:generic_pushout}
      \xymatrix{ A\ar[r]^-i\ar[d]_f & X\ar[d]^g\\ B \ar[r]_-j & Y}          
      \end{aligned}\end{equation}
    be a pushout square of $G$-spaces, where the horizontal maps are inclusions,
    and $f$ and $g$ are cellular maps. 
    Then for all 
    $(b,x)\in \cH(B)\times \cH(X)$ such that $f^*(b)=i^*(x)$ in $\cH(A)$,
    there is an element $y\in \cH(Y)$ such that $j^*(y)=b$ and $g^*(y)=x$.
  \end{enumerate}
\end{defn}

Now we develop some general features of excisive functors, in particular
a 5-term Mayer-Vietoris sequence, see Proposition \ref{prop:Mayer-Vietoris}.
For a finite $G$-CW-pair $(X,A)$ we use the notation
\[ \cH(X|A)\ = \ \ker( i^*:\cH(X)\to \cH(A))\ , \]
where $i:A\to X$ is the inclusion.

\begin{prop}\label{prop:relative iso}
  Let $G$ be a Lie group and let $f:A\to Y$ be a cellular $G$-map between finite proper $G$-CW-complexes.
  Then for every excisive functor $\cH$, the canonical map
  $A\times S^1\to Y\cup_{A\times\infty,f} A\times S^1$ induces an isomorphism
  \[ \cH(Y\cup_{A\times\infty} A\times S^1|Y) \ \xrightarrow{\ \iso \ }\
    \cH(A\times S^1|A\times\infty) \ .  \]
\end{prop}
\begin{proof}
  We start with the special case where $f$ is the inclusion of a subcomplex.
  Then both squares in the commutative diagram
  \[ \xymatrix{
      A\ar[d] \ar[r]^-{(-,\infty)}& A\times S^1\ar[d]\ar[r]^-{\proj}&  A\ar[d]\\
        Y \ar[r]&  Y\cup_{A\times\infty}A\times S^1\ar[r]_-{\proj}& Y
    } \]
  are pushouts, where all vertical maps are inclusions.
  Excision for the left square is the surjectivity of the map in question.

  For injectivity we consider a class $w\in\cH(Y\cup_{A\times\infty}A\times S^1|Y)$
  that restricts to 0 on $A\times S^1$.
  Excision for the right pushout square provides a class $y\in \cH(Y)$
  such that $\proj^*(y)=w$.
  Since the projection restricts to the identity on $Y$,
  we obtain the relation
  \[   y \ = \ \proj^*(y)|_Y\ = \ w|_Y \ = \ 0\ ,\]
  and hence also $w=\proj^*(y)=0$.
  
  Now we treat the general case where $f:A\to Y$ is an arbitrary cellular $G$-map.
  We let $Z=A\times[0,1]\cup_{A\times 1,f}Y$ be the mapping cylinder of $f$.
  The map $(-,0):A\to Z$ is the inclusion of a subcomplex,
  and the map $q:Z\to Y$ that projects $A\times[0,1]$ to $A$ and is the identity on $Y$
  is an equivariant homotopy equivalence. So the induced map
  \[
    q\cup(A\times S^1)\ : \ Z\cup_{A\times\infty} A\times S^1\ \to \ Y\cup_{A\times\infty} A\times S^1
  \]
  is also an equivariant homotopy equivalence, by the gluing lemma.  The two induced maps
  \begin{align*}
    \cH(Y\cup_{A\times\infty} A\times S^1|Y) \ \xrightarrow{(q\cup(A\times S^1))^*}\
                                  &\cH(Z\cup_{A\times\infty} A\times S^1|Z) \\
    \xrightarrow{\ (-,0)^* \ }\ &\cH(A\times S^1|A\times\infty) \ .  
  \end{align*}
  are then isomorphisms by homotopy invariance and the previous paragraph, respectively. 
  So the composite is an isomorphism, which proves the claim.  
\end{proof}

As we shall now explain, the excision property extends to a Mayer-Vietoris sequence
for a pushout square of $G$-spaces \eqref{eq:generic_pushout},
where $(X,A)$ and $(Y,B)$ are finite proper $G$-CW-pairs. 
We define a {\em connecting homomorphism}
\begin{equation}\label{eq:define partial}
 \partial \ : \ \cH(A\times S^1|A\times\infty)\ \to \ \cH(Y)   
\end{equation}
as the composite
\begin{align*}
     \cH(A\times S^1|A\times\infty)\ &\xrightarrow{\ \iso\ } \
      \cH(Y\cup_{A\times \infty}A\times S^1|Y)\ \xrightarrow{\incl} \
                                       \cH(Y\cup_{A\times \infty}A\times S^1)\\
                                     &\xrightarrow{(j\cup(A\times t)\cup g)^*} \
  \cH(B\cup_{A\times 0,f}A\times[0,1]\cup_{A\times 1}X)\ \xrightarrow[\iso]{(q^*)^{-1}} \ \cH(Y)\ .   
\end{align*}
The first isomorphism is the one provided by Proposition \ref{prop:relative iso}.
The quotient map $t:[0,1]\to S^1$ was defined in \eqref{eq:define_t}.
The map 
\[
  q\ =\ j\cup(g|_A\circ\proj)\cup g\ :\ B\cup_{A\times 0,f}A\times[0,1]\cup_{A\times 1}X\ \to\ Y
\]
is an equivariant homotopy equivalence, so it induces an
isomorphism in the homotopy functor $\cH$.

\begin{prop}\label{prop:Mayer-Vietoris}
  Let $G$ be a Lie group and $\cH$ an excisive functor.
  Let $(X,A)$ and $(Y,B)$ be two finite proper $G$-CW-pairs, and let
  \eqref{eq:generic_pushout} be a pushout square of $G$-spaces,
  where $f$ and $g$ are cellular maps.
  Then the following sequence is exact:
  \begin{align*}
    \cH(B\times S^1|B\times\infty)\times\cH(X\times S^1|X\times\infty)\
    &\xrightarrow{(f\times S^1)^*-(i\times S^1)^*} \
      \cH(A\times S^1|A\times\infty)\\
    \xrightarrow{\quad\partial\quad} \  \cH(Y)& \ \xrightarrow{(j^*,g^*)} \
      \cH(B)\times \cH(X)\ \xrightarrow{f^*-i^*}\    \cH(A)    
  \end{align*}
\end{prop}
\begin{proof}
  Exactness at $\cH(B)\times \cH(X)$ is the excision property for the functor $\cH$.
  For exactness at $\cH(Y)$ we consider the pushout square:
  \begin{equation}\begin{aligned}\label{eq:shifting_po}
   \xymatrix{ B\amalg X \ar[r]^-\iota \ar[d]_{j+g} &
      B\cup_{A\times 0,f}A\times[0,1]\cup_{A\times 1}X\ar[d]^{j\cup (A\times t)\cup g}\\
      Y \ar[r]_-{\incl} & Y\cup_{A\times\infty}A\times S^1}  
    \end{aligned}  \end{equation}
  Excision for this square provides an exact sequence
  \[
    \cH(Y\cup_{A\times\infty}A\times S^1|Y)\ \xrightarrow{(j\cup(A\times t)\cup g)^*} \
    \cH(B\cup_{A\times 0,f}A\times[0,1]\cup_{A\times 1}X)\ \xrightarrow{\ \iota^*\ }  \cH(B\amalg X)\ .
 \]
 Proposition \ref{prop:relative iso} identifies the group
 $\cH(Y\cup_{A\times\infty}A\times S^1|Y)$ with $\cH(A\times S^1|A\times\infty)$;
 the equivariant homotopy equivalence $q:B\cup_{A\times 0,f}A\times[0,1]\cup_{A\times 1}X\to Y$ 
 induces an isomorphism from $\cH(Y)$ to $\cH(B\cup_{A\times 0,f}A\times[0,1]\cup_{A\times 1}X)$;
 and additivity identifies the group $\cH(B\amalg X)$ with $\cH(B)\times\cH(X)$.
 These substitutions prove exactness of the original sequence at $\cH(Y)$.

 To establish exactness at $\cH(A\times S^1|A\times\infty)$,
 we employ exactness at the target of the connecting homomorphism,
 but for the pushout square \eqref{eq:shifting_po} instead of the original square.
 The result is an exact sequence
 \begin{align*}
   \cH((B\amalg X)\times S^1|(B\amalg X)\times\infty)\
   \xrightarrow{\ \partial\ }& \  \cH(Y\cup_{A\times\infty}A\times S^1)\\
   \xrightarrow{(\incl^*,(j\cup(A\times t)\cup g)^*)}& \
   \cH(Y)\times \cH(B\cup_{A\times 0,f} A\times[0,1]\cup_{A\times 1} X) \ .
 \end{align*}
 The map $\incl^*:  \cH(Y\cup_{A\times\infty}A\times S^1)\to \cH(Y)$
 is a split epimorphism, so passing to kernels gives another exact sequence
 \begin{align*}
   \cH((B\amalg X)\times S^1|(B\amalg X)\times\infty)\
   &\xrightarrow{\ \partial\ } \  \cH(Y\cup_{A\times\infty}A\times S^1|Y)\\
   &\xrightarrow{(j\cup(A\times t)\cup g)^*} \   \cH(B\cup_{A\times 0,f} A\times[0,1]\cup_{A\times 1} X) \ .
 \end{align*}
 We use additivity to identify the first group with the product of
 $\cH(B\times S^1|B\times\infty)$ and $\cH(X\times S^1|X\times\infty)$;
 we use Proposition \ref{prop:relative iso} to identify the middle group with
 $\cH(A\times S^1|A\times\infty)$;
 and we use the equivariant homotopy equivalence $q:B\cup_{A\times 0,f}A\times[0,1]\cup_{A\times 1}X\to Y$
 to identify the third group with $\cH(Y)$.
 Because the following square commutes
 \[ \xymatrix@C=18mm{
     \cH((B\amalg X)\times S^1|(B\amalg X)\times\infty)\ar[d]_{((i_B\times S^1)^*,(i_X\times S^1)^*)}
     \ar[r]^-\partial &
     \cH(Y\cup_{A\times\infty}A\times S^1|Y)\ar[d]^\iso\\
     \cH(B\times S^1|B\times\infty)\times\cH(X\times S^1|X\times\infty)
     \ar[r]_-{(f\times S^1)^*-(i\times S^1)^*}& \cH(A\times S^1|A\times\infty)
   } \]
 these substitutions result in the desired exactness at
 $\cH(A\times S^1|A\times\infty)$ of the original sequence. 
\end{proof}

The Mayer-Vietoris sequence yields a convenient criterion for checking
that a natural transformation between excisive functors is an isomorphism.

\begin{prop}\label{prop:excisive on orbits}
  Let $G$ be a Lie group and let $\Psi:\cH\to \cH'$
  be a natural transformation between excisive functors.
  Suppose that for every compact subgroup $H$ of $G$
  and every non-equivariant finite CW-complex $X$,
  the homomorphism $\Psi_{G/H\times X}:\cH(G/H\times X)\to \cH'(G/H\times X)$
  is an isomorphism.
  Then $\Psi_Y:\cH(Y)\to\cH'(Y)$ is an isomorphism for every finite proper $G$-CW-complex $Y$.
\end{prop}
\begin{proof}
  We show by induction over the number of cells in an equivariant CW-structure on $Y$
  that for every non-equivariant finite CW-complex $L$, the map $\Psi_{Y\times L}:\cH(Y\times L)\to\cH'(Y\times L)$ is an isomorphism.
  Taking $L$ to be a point proves the claim.

  If there are no equivariant cells, then $Y$ and $Y\times L$ are empty,
  and hence $\cH(Y\times L)=\cH'(Y\times L)=0$.
  If $Y$ is non-empty we choose a pushout square of $G$-spaces:
  \[   \xymatrix{ G/H\times S^{n-1}\ar[r]^-i\ar[d]_f & G/H\times D^n\ar[d]^g\\ B \ar[r]_-j & Y}   \]
  Here $H$ is a compact subgroup of $G$, and $B$ is a $G$-subcomplex of $Y$ with one
  fewer cell.
  Taking product with $L$ yields another pushout square.
  We also know by induction that $\Psi$
  is an isomorphism for $B\times L$ and $B\times L\times S^1$,
  and by hypothesis for $G/H\times S^{n-1}\times L\times S^1$,
  $G/H\times S^{n-1}\times L$, $G/H\times D^n\times L\times S^1$ and $G/H\times D^n\times L$.
  So the natural exact sequence
  \begin{align*}
    \cH(B\times L\times S^1&|B\times L\times\infty)\times \cH(G/H\times D^n\times L\times S^1|G/H\times D^n\times L\times \infty)\\
    \to\ &\cH(G/H\times S^{n-1}\times L\times S^1|G/H\times S^{n-1}\times L\times \infty)\
           \to \ \cH(Y\times L)\\
    \to\ &\cH(B\times L)\times \cH(G/H\times D^n\times L)\ \to\  \cH(G/H\times S^{n-1}\times L)    
  \end{align*}
  provided by Proposition \ref{prop:Mayer-Vietoris} and the five lemma
  show that the map $\Psi_{Y\times L}:\cH(Y\times L)\to \cH'(Y\times L)$ is an isomorphism. 
\end{proof}
\index{excisive functor|)}

Now we discuss three different ways to define an excisive functor
on finite proper $G$-CW-complexes from an orthogonal $G$-spectrum $E$.
\begin{enumerate}[(i)]
\item The functor $E_G\td{X}$ is defined as the colimit, over $n\geq 0$,
  of the sets $[S^n\sm X_+,E(\mR^n)]^G$, compare \eqref{eq:define_sequential_F_coho}.
\item The functor $E_G(X)=[\Sigma^\infty_+X, E]^G$ is represented by $E$
  in the triangulated stable homotopy category $\Ho(\Sp_G)$, compare \eqref{eq:define_E_G(A)}.
\item The functor $E_G\gh{X}$ is defined via parameterized homotopy classes
  indexed by $G$-vector bundles over $X$, see Construction \ref{con:define alternative}.
\end{enumerate}
All three constructions can be extended to $\mZ$-graded proper cohomology theories
by replacing $E$ by its shifts $E[k]$, for $k\in\mZ$, as defined in Remark \ref{rk:[k]}.
The first functor $E_G\td{-}$ only depends on the underlying sequential $G$-spectrum of $E$,
and can also be defined via `$G$-orbit spectra' in the sense of
\cite{davis-luck:spaces_assembly}, see Remark \ref{rk:Davis-Luck}. 
The second construction $E_G(-)$ defines `genuine' proper cohomology theories
in complete generality, for all Lie groups; we refer to Remark \ref{rk:RO(G) grading}
below for an explanation of the adjective `genuine' in this context.
The third functor $E_G\gh{-}$ is excisive for all discrete groups $G$,
but {\em not} generally for positive dimensional Lie groups with infinitely many components.
We explain in Remark \ref{rk:restriction finite} why the restriction
to discrete groups arises here.
For discrete groups, $E_G\gh{X}$ is naturally isomorphic
to the represented theory $E_G(X)$, see Theorem \ref{thm:rep2vect};
this isomorphism is the link to comparing the represented cohomology theories with
other theories, such as equivariant cohomotopy or equivariant K-theory.

\begin{defn}\index{sequential $G$-spectrum}
  Let $G$ be a Lie group. A {\em sequential $G$-spectrum} $E$ consists of
  a sequence of based $G$-spaces $E_n$, for $n\geq 0$, and based continuous $G$-maps
  $\sigma_n:S^1\sm E_n\to E_{1+n}$.
\end{defn}

Every orthogonal $G$-spectrum $X$ has an {\em underlying sequential $G$-spectrum}
with terms $X_n=X(\mR^n)$ and structure maps $\sigma_n=\sigma_{\mR,\mR^n}:S^1\sm X_n\to X_{1+n}$.

\begin{con}\label{con:naive G-Omega}
  Let $G$ be a Lie group and let $E$ be a sequential $G$-spectrum.
  For every based $G$-space $Y$, we define the 
  $G$-equivariant $E$-cohomology group as  
  \[ \tilde E_G\td{Y} \ = \ \colim_n\, [S^n\sm Y,  E_n]_*^G \ ,   \]
  where $[-,-]^G_*$ denotes the set of equivariant homotopy classes of based $G$-maps.
  The colimit is taken over the poset of natural numbers, along the maps
  \[ [S^n\sm Y,  E_n ]_*^G \ \xrightarrow{S^1\sm -}\
    [S^{1+n}\sm Y,  S^1\sm E_n]_*^G \ \xrightarrow{(\sigma_n)_*}\
    [S^{1+n}\sm Y,  E_{1+n} ]_*^G \ .  \]
  For $n\geq 2$, the set $[S^n\sm Y,  E_n]_*^G$ is an abelian group under
  the pinch sum.
  The stabilization maps are group homomorphisms, so the colimit inherits an
  abelian group structure.
  The group $\tilde E_G\td{-}$ is contravariantly functorial, by precomposition,
  for continuous based $G$-maps in $Y$.
  
  If $X$ is a finite proper $G$-CW-complex (without a basepoint),
  we define the (unreduced) $G$-equivariant $E$-cohomology group as  
  \begin{equation}  \label{eq:define_sequential_F_coho}
    E_G\td{X}\ = \ \tilde E_G\td{X_+}\ .  
  \end{equation}
\end{con}

\begin{prop}\label{prop:td is excisive}
Let $G$ be a Lie group and $E$ a sequential $G$-spectrum.
Then the functor $E_G\td{-}$ is excisive.
\end{prop}
\begin{proof}
(i) Homotopy invariance is clear by the very definition,
since each of the functors $[S^n\sm X_+,E_n]^G$ sends $G$-homotopic maps
in $X$ to the same map.

(ii) The universal property of a disjoint union, applied
to $G$-maps and $G$-equivariant homotopies, shows that for fixed $n$, the map
\[ (i_X^*,i_Y^*)\ : \ [S^n\sm (X\amalg Y)_+,E_n]^G_*\ \to \
  [S^n\sm X_+,E_n]^G_*\times [S^n\sm Y_+,E_n]^G_* \]
is bijective. Sequential colimits commute with finite products, so additivity
follows by passing to colimits over $n$.

(iii)
For the excision property we consider two finite proper $G$-CW-pairs $(X,A)$ and $(Y,B)$
and a pushout square of $G$-spaces \eqref{eq:generic_pushout}.
We consider two classes $b\in E_G\td{B}$ and $x\in E_G\td{X}$ such that
$f^*(b)=i^*(x)$ in $E_G\td{A}$. 
Then $b$ and $x$ can be represented by
continuous based $G$-maps $\bar b:S^n\sm B_+\to E_n$ and $\bar x:S^n\sm X_+\to E_n$,
for some $n\geq 0$.
The hypothesis $f^*(b)=i^*(x)$ means that, possibly after increasing $n$, 
the two based $G$-maps
$\bar b\circ(S^n\sm f_+),\bar x\circ(S^n\sm i_+):S^n\sm A_+\to E_n$ are equivariantly homotopic.
Since $(X,A)$ is a $G$-CW-pair, the inclusion $S^n\sm A_+\to S^n\sm X_+$
has the $G$-equivariant homotopy extension property for
continuous based $G$-maps. So we can modify $\bar x$ into an
equivariantly homotopic based $G$-map $\tilde x:S^n\sm X_+\to E_n$
such that $\bar b\circ(S^n\sm f_+)=\tilde x\circ(S^n\sm i_+)$.
Then $\bar b$ and $\tilde x$ glue to a continuous based $G$-map $S^n\sm Y_+\to E_n$;
this $G$-map represents a class $y\in E_G\td{Y}$ such that $j^*(y)=b$  and $g^*(y)=x$.
\end{proof}

\begin{rk}[Cohomology theories from spectra over the orbit category]\label{rk:Davis-Luck}\index{G-orbit spectrum@$G$-orbit spectrum}
We let $G$ be a discrete group. 
We denote by $\Or^{\Fin}_G$ the $\Fin$-orbit category of $G$, i.e.,\index{Fin-orbit category@$\Fin$-orbit category}
the full subcategory of the category of $G$-sets with objects $G/H$ 
for all finite subgroups $H$ of $G$.
Davis and the third author explain in \cite[Def.\,4.1]{davis-luck:spaces_assembly}
how to construct a proper cohomology theory from a
`$G$-orbit spectrum', i.e., a functor
\[ \bE\ : \ \left(\Or^{\Fin}_G\right)^{\op} \ \to \ \Sp^\mN \]
to the category of non-equivariant sequential spectra.
We claim that the proper cohomology theories represented by
$G$-orbit spectra are precisely the ones represented by sequential $G$-spectra
as in \eqref{eq:define_sequential_F_coho};
we sketch this comparison without giving complete details.

We recall the construction from \cite{davis-luck:spaces_assembly}.
For a $G$-space $X$, we denote by 
\[ \Phi(X)\ : \ \left(\Or^{\Fin}_G\right)^{\op} \ \to \ \bT \]
the fixed point functor, i.e.,
\[ \Phi(X)(G/H)\ = \ \map^G(G/H,X) \ .\]
Evaluation at the preferred coset $e H$ is a homeomorphism
$\Phi(X)(G/H)\iso X^H$ to the $H$-fixed point space.
The $\bE$-cohomology groups are then defined as
\[ \bE_G^k(X)\ = \ \pi_{-k}( \map^{\Or^{\Fin}_G}(\Phi(X),\bE)) \ , \]
the $(-k)$-th homotopy group of the spectrum
$\map^{\Or^{\Fin}_G}(\Phi(X),\bE)$ of natural transformations
from $\Phi(X)$ to $\bE$.
By \cite[Lemma 4.4]{davis-luck:spaces_assembly},
this indeed defines a proper $G$-cohomology theory
on finite proper $G$-CW-complexes. 
Moreover, if $\bE$ happens to take values in sequential $\Omega$-spectra,
then the cohomology theory also takes arbitrary disjoint unions to products.

We shall now explain how sequential $G$-spectra give
rise to $G$-orbit spectra in such a way that the
Davis-L{\"u}ck cohomology theory recovers the cohomology theory
as in \eqref{eq:define_sequential_F_coho}.
As in the case of $G$-spaces, a sequential $G$-spectrum $E$
gives rise to a fixed point diagram of sequential spectra
\[ \Phi(E)\ : \ \left(\Or^{\Fin}_G\right)^{\op} \ \to \ \Sp^\mN \]
by setting
\[ \Phi(E)(G/H)\ = \ \map^G(G/H,E) \ \iso \ E^H\ .\]
The fixed point functor $\Phi$
from $G$-spaces to spaces over the $\Fin$-orbit category is fully faithful,
so it induces a bijection
\[  [S^n\sm X_+,E_n]^G_* \ \xrightarrow[\iso]{\ \Phi \ }\ 
  [S^n\sm \Phi(X)_+,\Phi(E_n)]^{\Or^{\Fin}_G}_*  = \
  \pi_n(\map^{\Or^{\Fin}_G}(\Phi(X),\Phi(E_n)))\ .\]
Passing to colimits over $n$ yields an isomorphism
\[ E_G\td{X} \ \iso \ \Phi(E)^0_G(X)\ .\]
So every sequential $G$-spectrum gives rise to a $G$-orbit spectrum
that represents the same proper cohomology theory.
The converse is also true. A $G$-orbit spectrum
can be viewed as a sequential spectrum internal to the category
of based spaces over the $\Fin$-orbit category.
Elmendorf's theorem \cite{elmendorf:system_fixed}
can be adapted to a Quillen equivalence between
the $\Com$-model structure on the category of based $G$-spaces
and the category of based spaces over the $\Fin$-orbit category
(with the objectwise, or projective, model structure).
So every spectrum of based $\Or_G^{\Fin}$-spaces is levelwise
equivalent to $\Phi(E)$ for some sequential $G$-spectrum $E$.
\end{rk}

The next definition is based on the triangulated stable homotopy category $\Ho(\Sp_G)$,
so it makes essential use of our entire theory.

\index{proper $G$-cohomology theory!represented|(} 
\begin{defn}\label{def:proper from spectrum}
Let~$G$ be a Lie group and~$E$ an orthogonal $G$-spectrum.
For every $G$-space $X$, we define the represented $G$-equivariant $E$-cohomology group as
\begin{equation}\label{eq:define_E_G(A)}
 E_G(X) \ = \ [\Sigma^\infty_+ X,  E]^G \ ,  
\end{equation}
the group of morphisms in $\Ho(\Sp_G)$
from the unreduced suspension spectrum of $X$ to $E$.\index{G-equivariant stable homotopy category@$G$-equivariant stable homotopy category}\index{suspension spectrum}
\end{defn}

The group $E_G(X)$ is contravariantly functorial for continuous $G$-maps in $X$.
For the one-point $G$-space $X=\ast$, the cohomology group already has another name:
\[ E_G(\ast)\ = \ [\mS_G, E]^G \ \iso \ \pi_0^G(E)\ , \]
the 0-th $G$-equivariant homotopy group of $E$.

\begin{prop}\label{prop:represented is excisive}
  Let $G$ be a Lie group and $E$ an orthogonal $G$-spectrum.
  Then the functor $E_G(-)$ is excisive.
\end{prop}
\begin{proof}
  We start with the special case where $E$ is an orthogonal $G$-$\Omega$-spectrum.
  Since $E$ is fibrant in the stable model structure on $\Sp_G$,
  the derived adjunction isomorphism stemming from the Quillen adjoint functor pair
  \[ \xymatrix{ \Sigma^\infty_+ \ : \ G\bT\ \ar@<.4ex>[r] & \ \Sp_G \ : \ \ar@<.4ex>[l] \ (-)(0)}  \]
  provides a bijection
  \[ [X,E(0)]^{G\bT} \ \iso\  [\Sigma^\infty_+ X,E]^G  =  E_G(X) \ .\]
  An orthogonal $G$-$\Omega$-spectrum is in particular a naive $\Omega$-spectrum, i.e.,
  the adjoint structure map 
  \[ \tilde\sigma_{\mR,\mR^n}\ : \ E(\mR^n)\ \to \ \Omega E(\mR^{1+n}) \]
  is a $\Com$-equivalence for every $n\geq 0$. For every finite proper $G$-CW-complex $X$,
  the based $G$-space $S^n\sm X_+$ is $\Com$-cofibrant, so the map
  \[ (\tilde\sigma_{\mR,\mR^n})_*  \ : \  [S^n\sm X_+,E(\mR^n)]^G_* \ \to\
    [S^n\sm X_+,\Omega E(\mR^{1+n})]^G_* \]
  is bijective. Hence also the stabilization maps in the colimit system
  defining $E_G\td{X}$ are bijective. So the canonical map
  \[ [X,E(0)]^{G\bT}\iso [X_+,E(0)]^G_* \ \to \ \colim_{n\geq 0}\, [S^n\sm X_+,E(\mR^n)]^G_* \ = \ E_G\td{X} \]
  is bijective. Altogether we have exhibited a bijection between $E_G(X)$ and $E_G\td{X}$
  that is natural for $G$-maps in $X$. The functor $E_G\td{-}$ is excisive
  by Proposition \ref{prop:td is excisive}.
  Since the homotopy invariance, additivity and excision properties do not use the abelian
  group structure and only refer to the underlying set-valued functor,
  we conclude that the functor $E_G(-)$ is also excisive.

  In the general case the stable model structure of Theorem \ref{thm:stable} provides
  a $\pi_*$-isomorphism of orthogonal $G$-spectra $q:E\to F$ whose target
  $F$ is an orthogonal $G$-$\Omega$-spectrum. Then $\gamma_G(q)$ is an isomorphism
  in $\Ho(\Sp_G)$, and hence induces a natural isomorphism $E_G(-)\iso F_G(-)$.
  The latter functor is excisive by the previous paragraph, hence so is the former.
\end{proof}
\index{proper $G$-cohomology theory!represented|)} 

In order to compare the `genuine' equivariant cohomology theories
represented by orthogonal $G$-spectra
with other theories, such as equivariant cohomotopy \cite[Sec.\,6]{luck:burnside_ring}
or equivariant K-theory \cite[Sec.\,3]{luck-oliver:completion}, we provide another description
of the excisive functor represented by a $G$-spectrum $E$.
This alternative description $E_G\gh{X}$ is in terms of parameterized equivariant
homotopy theory over $X$, see Construction \ref{con:define alternative}.
The construction generalizes the equivariant cohomotopy groups of
the third author \cite[Sec.\,6]{luck:burnside_ring}, which is the special case $E=\mS_G$
of the equivariant sphere spectrum; many of the arguments are inspired by that special case.

While the definition of the group $E_G\gh{X}$ makes sense for all Lie groups,
the excision property established in Theorem \ref{thm:excisive - bundle} below
does not hold in that generality.
Consequently, various results that depend on excision for the theory $E_G\gh{-}$
are only formulated for {\em discrete} groups;
we explain in Remark \ref{rk:restriction finite} why the restriction
to discrete groups arises.
The new theory $E_G\gh{X}$ calculates the represented theory $E_G(X)$
for {\em discrete} groups and finite proper $G$-CW-complexes~$X$,
compare Theorem \ref{thm:rep2vect} (iv).

\begin{defn} Let $G$ be a Lie group and $X$ a $G$-space.
  A \emph{retractive $G$-space}\index{retractive $G$-space}\index{G-space@$G$-space!retractive}
  over $X$ is a triple $(E,p,s)$,  where $E$ is a $G$-space, and
  \[p \ : \ E\ \to\ X\text{\qquad and\qquad}s \ : \ X\ \to\ E\]
  are continuous $G$-maps that satisfy $p\circ s=\Id_X$. 
  If $(E',p',s')$ is another retractive $G$-space over $X$, then
  a {\em morphism} from $(E,p,s)$ to $(E',p',s')$ is a continuous $G$-map
  $f:E\to E'$ such that $p'\circ f=p$ and $s'=f\circ s$.
  A \emph{parameterized homotopy} between two such morphisms 
  is a continuous $G$-map $H \colon E \times I \to E'$
  such that $H(-,t)$ is a morphism of retractive $G$-spaces over $X$
  for every $t\in[0,1]$.
\end{defn}

\begin{con}
  We let $G$ be a Lie group, $E$ an orthogonal $G$-spectrum and $\xi:B\to X$
  a euclidean $G$-vector bundle over a $G$-space~$X$.\index{G-vector bundle@$G$-vector bundle} 
  We define a $G$-space $E(\xi)$ as follows. 
  If $\xi$ has constant rank $n$, then we denote by~$\cF_n(\xi)$ 
  the frame bundle of~$\xi$, 
  i.e., the principal $O(n)$-bundle whose fiber over~$x\in X$
  is the space of $n$-frames (orthonormal bases) of $\xi_x$.
  We now form the space
  \[ E(\xi)\ = \ \cF_n(\xi)\times_{O(n)} E(\mR^n) \]
  endowed with the diagonal $G$-action from the action on~$\xi$ and on~$E(\mR^n)$.
  If the bundle does not have constant rank, then we let $X_{(n)}$
  be the subspace of those $x\in X$ such that $\dim(\xi_x)=n$.
  The subspaces $X_{(n)}$ are open by local triviality, and they are $G$-invariant.
  We define
  \[ E(\xi)\ = \ \coprod_{n\geq 0} E(\xi|_{X_{(n)}})\ . \]
  The space $E(\xi)$ comes with a projection to $X$ which is a locally trivial fiber bundle,
  with fiber $E(\mR^n)$ over $X_{(n)}$. 
  The $(G\times O(n))$-fixed basepoint of~$E(\mR^n)$ gives a preferred section
  \[ s \ : \ X \ \to \ E(\xi) \ .\]
  The projection and section
  are $G$-equivariant; so we can -- and will -- consider $E(\xi)$
  as a retractive $G$-space over $X$.
  
  An important special case of this construction is when $E=\mS_G$ is the sphere spectrum.
  Here we write
  \[ S^\xi \ = \ \mS_G(\xi) \ = \ \cF_n(\xi)\times_{O(n)} S^n\]
  for the locally trivial bundle with fiber $S^n$ over $X_{(n)}$.
  In this situation,  the section $s\colon X\to S^\xi$
  is a $\Com$-cofibration of~$G$-spaces by Proposition \ref{prop:s_infty is cofibration}. 
  The quotient space $\Th(\xi)= S^\xi / s(X)$ 
  is the Thom space of~$\xi$ as defined in \eqref{eq:define_Thom_space} above.
  
  The structure maps of the orthogonal $G$-spectrum $E$
  can be used to relate the spaces defined from different vector bundles.
  Given another $G$-vector bundle~$\eta$ over~$X$ of dimension~$m$, 
  a frame in~$\eta_x$ and a frame in~$\xi_x$ concatenate into a frame 
  in~$\eta_x\oplus\xi_x=(\eta\oplus\xi)_x$, and the resulting map
  \[ \cF_m(\eta)\times_X\cF_n(\xi) \ \to \ \cF_{m+n}(\eta\oplus\xi) \]
  is $(G\times O(m)\times O(n))$-equivariant.
  Using the $(O(m)\times O(n))$-equivariant 
  structure map~$\sigma_{m,n}:S^m\sm E(\mR^n)\to E(\mR^{m+n})$ we obtain a continuous $G$-map
  \begin{align*}
    \sigma_{E,\eta,\xi}\ : \ S^\eta \sm_X E(\xi)\ = \ 
    & (\cF_m(\eta)\times_{O(m)} S^m)\sm_X (\cF_n(\xi)\times_{O(n)} E_n)   \\
    \cong \ 
    &(\cF_m(\eta)\times_X\cF_n(\xi))\times_{O(m)\times O(n)} (S^m\sm E(\mR^n))   \\
    \xrightarrow{ (\sigma_{m,n})_*}
    \  &(\cF_m(\eta)\times_X\cF_n(\xi))\times_{O(m)\times O(n)} E(\mR^{m+n})  \\
    \to \  &\cF_{m+n}(\eta\oplus\xi)\times_{O(m+n)} E(\mR^{m+n}) \ =\ E(\eta\oplus\xi)
  \end{align*}
  of retractive $G$-spaces over~$X$.
  If the bundles do not have constant rank, we perform these constructions
  separately over the components $X_{(n)}$. If $E=\mS_G$ is the sphere spectrum,
  then these structure maps are isomorphisms 
  \[ S^\eta \sm_X S^\xi\ \iso\ S^{\eta\oplus\xi} \]
  of retractive $G$-spaces over $X$.
\end{con}

  \begin{defn}
We call a morphism $\psi:\xi\to\eta$ of euclidean $G$-vector bundles
an {\em isometric embedding} if $\psi$ is fiberwise a linear isometric embedding
of inner product spaces.  
\end{defn}

\index{proper $G$-cohomology theory!$E_G\gh{-}$|(} 
\begin{con}\label{con:define alternative}
  We let $G$ be a Lie group, $E$ an orthogonal $G$-spectrum,
  and $X$ a finite, proper $G$-CW-complex.
  We define an abelian group $E_G\gh{X}$ as follows.
  Elements of $E_G\gh{X}$ are equivalence classes of pairs $(\xi,u)$, where
  \begin{enumerate}[(i)]
  \item $\xi$ is a euclidean $G$-vector bundle over $X$, 
  \item $u:S^{\xi}\to E(\xi)$ 
    is a map of retractive $G$-spaces over $X$.
  \end{enumerate}
  To explain the equivalence relation we let $\psi:\xi\to\eta$ be an isometric embedding
  of euclidean $G$-vector bundles over $X$.
  We write $\gamma$ for the orthogonal complement of the image of $\psi$ in $\eta$.
  Then $\gamma$ is another $G$-vector bundle over $X$, and the map
  \begin{equation}\label{eq:sum_identification}
    \gamma\oplus\xi \ \to \ \eta  \ , \quad (x,y)\ \longmapsto \ x + \psi(y) 
  \end{equation}
  is an isomorphism.
  If $u:S^\xi\to E(\xi)$ is a map of retractive $G$-spaces over $X$,
  we write $\psi_*(u)$ for the map of retractive $G$-spaces
\[ 
    S^\eta \ \iso \ S^{\gamma}\sm_X S^{\xi}\ \xrightarrow{S^\gamma\sm_X u} \ 
    S^{\gamma}\sm_X E(\xi) \ \xrightarrow{\sigma_{E,\gamma,\xi}} \ 
    E(\gamma\oplus \xi) \ \iso \ E(\eta)\ .     
 \]
  We will refer to $\psi_*(u)$ as the {\em stabilization}
  of $u$ along $\psi$.
  The two isomorphisms are induced by the bundle isomorphism \eqref{eq:sum_identification}.
  We call two pairs $(\xi,u)$ and $(\xi',v)$ {\em equivalent} if there is a $G$-vector bundle
  $\eta$ over $X$ and isometric embeddings $\psi:\xi\to \eta$ and $\psi':\xi'\to \eta$
  such that the two maps of retractive $G$-spaces 
  \[ \psi_*(u)\ , \ \psi'_*(v)\ : \ S^\eta \ \to \ E(\eta) \]
  are parameterized $G$-equivariantly homotopic.
  We omit the verification that this relation is reflexive, symmetric and
  transitive, and hence an equivalence relation.

  We suppose that $\psi:\xi\to\eta$ is an equivariant isometric isomorphism
  of euclidean $G$-vector bundles over $X$.
  Then $\psi$ is in particular an isometric embedding. So for every map of retractive $G$-spaces
  $u:S^\xi\to E(\xi)$ over $X$, the pair $(\xi,u)$ is equivalent to the pair
  \[ (\eta,\psi_*(u))\ = \ (\eta, E(\psi)\circ u\circ S^{\psi^{-1}})\ .\]
  Informally speaking, this says that conjugation by an isometric isomorphism
  does not change the class in $E_G\gh{X}$.

  The isomorphism classes of $G$-vector bundles over~$X$ form a set,
  hence so do the equivalence classes. We write 
  \[ [\xi,u] \ \in \ E_G\gh{X}\]
  for the equivalence class of a pair $(\xi,u)$.
  This finishes the definition of the set $E_G\gh{X}$.
\end{con}

\begin{prop}\label{prop:normal form}
  Let $G$ be a Lie group, $X$ a finite proper $G$-CW-complex and $\xi$
  a euclidean $G$-vector bundle over $X$.
  \begin{enumerate}[\em (i)]
  \item Let  $a,b:\xi\to \nu$ be two equivariant isometric embeddings
    of euclidean $G$-vector bundles over $X$.
    Then $i_1\circ a, i_1\circ b:\xi\to\nu\oplus\nu$ are homotopic
    through $G$-equivariant isometric embeddings,
    where $i_1:\nu\to\nu\oplus\nu$ is the embedding as the first summand.
  \item
    Let $u,v:S^\xi\to E(\xi)$ be two
    maps of retractive $G$-spaces.
    Then  $[\xi, u] =  [\xi,v]$ in $E_G\gh{X}$ if and only if there is an
    isometric embedding $\psi:\xi\to\eta$ of euclidean $G$-vector bundles over $X$
    such that $\psi_*(u), \psi_*(v):S^\eta\to E(\eta)$ are parameterized $G$-homotopic.
\end{enumerate}
\end{prop}
\begin{proof}
  (i)
  We let $i_1,i_2:\nu\to\nu\oplus\nu$ be the embeddings as the first and second summand,
  respectively.
  Then the images of the isometric embeddings
  $i_1\circ a\ , \ i_2\circ b:\xi\to\nu\oplus\nu$ are orthogonal.
  So there is an equivariant homotopy $H:\xi\times[0,1]\to\nu\oplus\nu$
  through isometric embeddings from $i_1\circ a$ to $i_2\circ b$, for example
  \[ H(x,t)\ = \ (\sqrt{1-t^2}\cdot a(x),\ t\cdot b(x))\ .  \]
  For $a=b$ this in particular shows that $i_1\circ b$ and $i_2\circ b$ are $G$-homotopic
  through equivariant isometric embeddings.
  Altogether, $i_1\circ a$ and $i_1\circ b$ are $G$-homotopic
  through isometric embeddings.

  (ii)
  The `if' part of the claim holds by definition of the equivalence relation
  that defines $E_G\gh{X}$.
  Now we suppose that conversely, $[\xi, u] =  [\xi,v]$ in $E_G\gh{X}$.
  Then there are two isometric embeddings $a,b:\xi\to \nu$ of euclidean $G$-vector bundles over $X$
  such that $a_*(u)$ and $b_*(v)$ are parameterized equivariantly homotopic.
  Part (i) provides a homotopy between $i_1\circ a$ and $i_1\circ b$
  through $G$-equivariant isometric embeddings.
  The homotopy induces a parameterized equivariant homotopy between the two maps
  $(i_1\circ a)_*(v):S^{\nu\oplus\nu}\to E(\nu\oplus\nu)$ and $(i_1\circ b)_*(v)$.
  So we obtain a chain of parameterized equivariant homotopies
  \begin{align*}
    (i_1\circ a)_*(u) \ &= \ (i_1)_*( a_*(u) )\ \simeq \ (i_1)_*( b_*(v) )\
                            =\  (i_1\circ b)_*(v)\ \simeq \ (i_1\circ a)_*(v)\ .
  \end{align*}
  So the isometric $G$-embedding $\psi=i_1\circ a:\xi\to \nu\oplus\nu=\eta$
  has the desired property.
\end{proof}

Now we define an abelian group structure on the set $E_G\gh{X}$.
We let $\nabla:S^1\to S^1\vee S^1$ be a pinch map;
for definiteness, we take the same map as in \cite[6.2]{luck:burnside_ring},
namely
\[ \nabla(x) \ = \
  \begin{cases}
    \quad \ln(x) \text{ in the first copy of $S^1$ }
    & \text{ if $x\in(0,\infty)$,}\\
    -\ln(-x)  \text{ in the second copy of $S^1$ } & \text{ if $x\in(-\infty,0)$, and}\\
    \quad\infty  & \text{ if $x\in\{0,\infty\}$.}
  \end{cases}
\]
By Proposition \ref{prop:normal form}
we can represent any two given classes of $E_G\gh{X}$
by pairs $(\xi,u)$ and $(\xi,v)$, 
defined on the same $G$-vector bundle $\xi$ over $X$.
To add the classes we stabilize the representative by the trivial
line bundle $\ul{\mR}$ and then form the `pinch sum', i.e., the composite
\begin{align*}
  \nabla(u,v)\ : \  S^{\ul{\mR}\oplus\xi}\
  &\iso \ S^1\sm_X S^\xi \ \xrightarrow{\ \nabla\sm_X S^\xi} \
    (S^1\vee S^1)\sm_X  S^\xi  \\ 
&\iso \ (S^1\sm_X S^\xi )\vee_X (S^1\sm_X S^\xi) 
\ \xrightarrow{\ (S^1\sm u) +_X (S^1\sm v)} \\
  &\iso \ S^1 \sm_X E(\xi)
\ \xrightarrow{\sigma_{E,\ul{\mR},\xi}} \ E(\ul{\mR}\oplus\xi) \ .
\end{align*}
This way of adding representatives is compatible with parameterized homotopy
and stabilization along isometric embeddings, so we get a well-defined map
\[ + \ : \ E_G\gh{X}\times E_G\gh{X}\ \to \ E_G\gh{X}\ , \quad
[\xi,u]+[\xi,v]\ = \ [\xi\oplus\ul{\mR},\nabla(u,v)]\ . \]
The pinch map is coassociative and counital up to homotopy,
and has an inverse up to homotopy; this implies that the binary operation thus defined is 
associative and unital, with the class of the trivial map (with values the respective basepoints)
as unit, and that inverses exist. After stabilizing one additional time by $\ul{\mR}$,
the Eckmann-Hilton argument shows that the binary operation is commutative.
So we have indeed defined an abelian group structure on the set $E_G\gh{X}$.

The groups $E_G\gh{X}$ are clearly covariantly functorial for 
morphisms of orthogonal $G$-spectra in~$E$.
A contravariant functoriality in the $G$-space $X$
arises from pullback of vector bundles.
We let  $f:Y\to X$ be a continuous $G$-map between 
two finite proper $G$-CW-complexes.
We let $\xi$ be a euclidean $G$-vector bundle over $X$ and $u:S^{\xi}\to E(\xi)$ 
a map of retractive $G$-spaces over $X$.
Then $f^*(\xi)$ is a euclidean $G$-vector bundle over $Y$, and
\[ f^*(u)\ : \ S^{f^*\xi}\ = \ f^*(S^{\xi}) 
\ \to \ f^*(E(\xi))\ = \ E(f^*\xi)\]
a map of retractive $G$-spaces over $Y$.
The pullback construction respects parameterized homotopies
and is compatible with stabilization along isometric embeddings.
So we can define 
\[ f^*\ = \ E_G\gh{f} \ : \ E_G\gh{X} \ \to \ E_G\gh{Y} 
\text{\quad by \quad} f^*[\xi,u]\ = \ [f^*(\xi),f^*(u)]\ .\]

\begin{eg}
  We consider $X=G/H$, the homogeneous $G$-space for a compact subgroup $H$
  of the Lie group $G$. Every orthogonal $H$-representation $V$
  gives rise to a euclidean $G$-vector bundle 
  \[ \xi_V \ : \ G\times_H V \ \to \ G/H \ , \quad [g,v]\ \longmapsto \ g H\ . \]
  Moreover, every euclidean $G$-vector bundle over $G/H$ is isomorphic 
  to a bundle of this form.
  
  We have $\cF_n(\xi_V)=G\times_H \bL(\mR^n,V)$, where $n=\dim(V)$. So
  \[ E(\xi_V)\ = \ G\times_H E(V) \]
  for every orthogonal $G$-spectrum $E$. 
  In particular, $S^{\xi_V}=G\times_H S^V$.
  Every map of retractive $G$-spaces over $G/H$
  \[  S^{\xi_V}\ = \ G\times_H S^V \ \to \ G\times_H E(V)\ = \ E(\xi_V)  \]
  is of the form $G\times_H f$
  for a unique based $H$-map $f:S^V\to E(V)$, and this correspondence passes
  to a bijection of homotopy classes.
  If we let $V$ exhaust the finite-dimensional subrepresentations 
  of a complete $H$-universe, these bijections assemble into an isomorphism
  \begin{align*}
    \pi_0^H( E)\ = \ \colim_V \, [S^V,E(V)]^G_* \ &\xrightarrow{\ \iso \ } \ 
                                               E_G\gh{G/H}\\
    [f:S^V\to E(V)]\ &\longmapsto \ [G\times_H V, G\times_H f]\ .
           \end{align*}
\end{eg}

\begin{eg}[Compact Lie groups]\label{eg:bundle for compact}
  We let $H$ be a compact Lie group, $X$ a finite $H$-CW-complex,
  and $E$ an orthogonal $H$-spectrum.
  We define an isomorphism
  \[  \omega\ : \ \colim_{V\in s(\cU_H)}\, [S^V\sm X_+,E(V)]^H_* \ \to \  E_H\gh{X}
    \text{\quad by\quad}
    [f]\ \longmapsto \ [X\times V,f^\sharp]\ .\]
  Here $\cU_H$ is a complete $H$-universe, and $s(\cU_H)$ is the poset, under inclusion,
  of finite-dimensional $H$-subrepresentations of $\cU_H$.
  Moreover, for a continuous based $H$-map $f:S^V\sm X_+\to E(V)$,
  we write  $X\times V$ for the trivial $H$-vector bundle over $X$ with fiber $V$;
  a map of retractive $H$-spaces over $X$
  \begin{equation}\label{eq:define_f^sharp}
    f^\sharp \ : \ S^{X\times V}\ =\ X\times S^V\ \to \ X\times E(V)\ = \ E(X\times V)    
  \end{equation}
  is defined by $f^\sharp(x,v)= (x,f(v\sm x))$.
  Since $H$ is compact Lie, every euclidean $H$-vector bundle over the compact $H$-space $X$ embeds
  into a trivial bundle of the form $X\times V$, for some $H$-representation $V$,
  for example by \cite[Prop.\,2.4]{segal:equivariant_K};
  we can suppose that $V$ is a subrepresentation of the complete $H$-universe $\cU_H$.
  Moreover, every map of retractive $H$-spaces $X\times S^V\to X\times E(V)$ is of the form
  $f^\sharp$ for a unique based $H$-map $f:S^V\sm X_+\to E(V)$.
  So the map $\omega$ is surjective.
  
  For injectivity we exploit that $\omega$ is a group homomorphism
  for the group structure on the source arising from the identification
  with $\pi_0^H(\map(X,E))$.
  We suppose that $\omega[f]=[X\times V,f^\sharp]=0$.
  There is then an isometric embedding $\psi:X\times V\to\eta$ 
  of euclidean $H$-vector bundles over $X$ such that $\psi_*(f^\sharp)$ is
  parameterized equivariantly null-homotopic.
  We let $\gamma$ be the orthogonal complement of the image of $X\times V$ in $\eta$.
  We can embed $\gamma$ into the trivial $H$-vector bundle $X\times W$ 
  associated with another $H$-representation $W$,
  and we can then embed $V\oplus W$ into the complete $H$-universe $\cU_H$
  in a way that extends the inclusion $V\to\cU_H$.
  So we can altogether assume that $\psi=\incl:X\times V\to X\times \bar V$
  for $V\subset\bar V$ in the poset $s(\cU_H)$.
  After stabilizing $f:S^V\sm X_+\to E(V)$ along the inclusion of $V$ into $\bar V$,
  we can assume without loss of generality that the map $f^\sharp$ is
  parameterized $H$-null-homotopic.
  Maps of retractive $H$-spaces $S^{X\times V}\to E(X\times V)$ biject
  with continuous based $H$-maps $S^V\sm X_+\to E(V)$, so we conclude that
  the map $f$ is based $H$-null-homotopic. Thus $f$ represents the zero element
  in the source of $\omega$, and we have shown that the map $\omega$ is also injective.
\end{eg}

\begin{eg}\label{eg:induction iso}\index{induction isomorphism}
  We let $\Gamma$ be a closed subgroup of a Lie group $G$,
  and we let $Y$ be a finite proper $\Gamma$-CW-complex.
  We let $E$ be an orthogonal $G$-spectrum.
  We define an induction homomorphism
  \begin{equation}\label{eq:ind bundle}
    \ind \ : \ E_\Gamma\gh{Y}\ \to \ E_G\gh{G\times_\Gamma Y}      
  \end{equation}
  as follows.  We let $(\xi,u)$ represent a class
  in $E_\Gamma\gh{Y}$. Then $G\times_\Gamma\xi$ is a $G$-vector bundle over
  $G\times_\Gamma Y$, with frame bundle
  \[ \cF_n(G\times_\Gamma \xi)\ =\ G\times_\Gamma \cF_n(\xi) \ ,\]
  where $n=\dim(\xi)$. Hence
  \begin{align*}
    E(G\times_\Gamma \xi)\ &=\ \cF_n(G\times_\Gamma \xi)\times_{O(n)} E(\mR^n)\\
       &= \ G\times_\Gamma \cF_n(\xi)\times_{O(n)} E(\mR^n)\ = \ G\times_\Gamma E(\xi) 
  \end{align*}
  as retractive $G$-spaces over $G\times_\Gamma Y$. Moreover,
  \[ G\times_\Gamma u \ : \ S^{G\times_\Gamma \xi}\ = \ G\times_\Gamma S^\xi \ \to \
    G\times_\Gamma E(\xi)\ = \ E(G\times_\Gamma \xi) \]
  is a map of retractive $G$-spaces. We can thus define the induction homomorphism by
  \[ \ind[\xi,u]\ = \ [G\times_\Gamma\xi,G\times_\Gamma u]\ . \]
  Every euclidean $G$-vector bundle $\eta$ over $G\times_\Gamma Y$
  is isomorphic to a bundle of the form $G\times_\Gamma\xi$: we can take
  $\xi$ as the restriction of $\eta$ along the $\Gamma$-equivariant map
  \[ Y \ \to \ G\times_\Gamma Y \ , \quad y \ \longmapsto \ [1,y]\ . \]
  Similarly, every $G$-map $v:G\times_\Gamma S^\xi \to  G\times_\Gamma E(\xi)$
  is of the form $G\times_\Gamma u$ for a unique $\Gamma$-map
  $u:S^\xi \to E(\xi)$. Hence the induction homomorphism \eqref{eq:ind bundle}
  is an isomorphism.
\end{eg}

\begin{rk}[Vector bundles versus representations]
  Now is a good time to explain why we build the theory $E_G\gh{X}$
  using $G$-vector bundles, as opposed to just $G$-representations.
  One could contemplate a variation $\hat E_G\gh{X}$
  where elements are represented by classes $(V,v)$, where $V$
  is a $G$-representation and $v:S^V\sm X_+\to E(V)$ is a based $G$-map.
  Two pairs $(U,u)$ and $(V,v)$ represent the same class in $\hat E_G\gh{X}$
  if and only if there is a $G$-representation $W$ and
  $G$-equivariant linear isometric embeddings $\psi:U\to W$ and $\psi':V\to W$
  such that the two stabilizations $\psi_*(u) , \psi'_*(v): S^W\sm X_+\to E(W)$
  are based $G$-homotopic. This construction provides a homotopy functor from
  the category of based $G$-spaces to abelian groups.
  A $G$-representation $V$ gives rise to the trivial $G$-vector bundle
  $X\times V$ over any $G$-space $X$, and a based $G$-map $v:S^V\sm X_+\to E(V)$
  gives rise to a map of retractive $G$-spaces $v^\sharp:S^{X\times V}\to E(X\times V)$,
  compare \eqref{eq:define_f^sharp}.
  This assignment is compatible with the equivalence relations, and provides
  a natural group homomorphism $\hat E_G\gh{X}\to E_G\gh{X}$.
  Example \ref{eg:bundle for compact} can be rephrased as saying that
  for {\em compact} Lie groups, this homomorphism is an isomorphism.
  
  However, the construction $\hat E_G\gh{-}$
  based on $G$-representations (as opposed to $G$-vector bundles)
  does not in general have induction isomorphisms.\index{induction isomorphism}
  Our represented equivariant cohomology theories
  support induction isomorphisms, so this shows that the functor $\hat E_G\gh{-}$
  is not in general represented by an orthogonal $G$-spectrum.
  The case $E=\mS_G$ of stable cohomotopy, already considered in \cite[Rk.\,6.17]{luck:burnside_ring},
  can serve to illustrate the lack of induction isomorphisms.
  We let $G$ be any discrete group with the following two properties:
  \begin{enumerate}[(a)]
  \item Every finite-dimensional $G$-representation is trivial, and
  \item the group $G$ has a non-trivial finite subgroup $H$.
  \end{enumerate}
  As explained in \cite[Rk.\,6.17]{luck:burnside_ring},
  for every finite $G$-CW-complex $X$, the cohomotopy group
  $\hat \mS_G\gh{X}$ based on $G$-representations (as opposed to $G$-vector bundles over $X$)
  is isomorphic to the non-equivariant cohomotopy group $\pi^0_e(X/G)$
  of the $G$-orbit space. In particular,
  \[ \hat \mS_G\gh{G/H}\ \iso \ \pi^0_e(\ast)\ \iso\ \mZ \ . \]
  On the other other hand, $\hat \mS_H\gh{\ast}=\pi^0_H(\ast)$ is isomorphic
  to the Burnside ring of the finite group $H$, which has rank bigger than one
  since $H$ is non-trivial.

  An explicit example of a group satisfying (a) and (b)
  is Thompson's group $T$, see for example \cite{cannon-floyd} and the references given therein.
  The group $T$ is infinite, finitely presented,
  and simple (i.e., the only normal subgroups are $\{e\}$ and $T$).
  As explained in  \cite[Sec.\,2.5]{luck:burnside_ring},
  an infinite, finitely generated simple group
  does not have non-trivial finite-dimensional representations over any field;
  in particular, every finite-dimensional $\mR$-linear representation of Thompson's group $T$
  is trivial. On the other hand, $T$ has plenty of finite subgroups.
\end{rk}

The following proposition is a slight refinement of \cite[Lemma 3.7]{luck-oliver:completion}.

\begin{prop}\label{prop:relative LO}
  Let $G$ be a discrete group.
  Let $h:A\to Y$ be a continuous $G$-map between finite proper $G$-CW-complexes.
  Let $\zeta$ be a euclidean $G$-vector bundle over $Y$ and
  $\psi:h^*(\zeta)\to \kappa$ an isometric embedding of euclidean $G$-vector bundles over $A$.
  Then there is an isometric embedding $\varphi:\zeta\to \omega$
  of euclidean $G$-vector bundles over $Y$ and 
  an isometric embedding $\lambda:\kappa\to h^*(\omega)$ of euclidean $G$-vector bundles over $A$
  such that the composite
  \[ h^*(\zeta) \ \xrightarrow{\ \psi\ } \ \kappa \ \xrightarrow{\ \lambda\ }\  h^*(\omega) \]
  coincides with $h^*(\varphi)$.
\end{prop}
\begin{proof}
  We let $\gamma$ be the orthogonal complement of the image of $h^*(\zeta)$ under $\psi$;
  this is a euclidean $G$-vector subbundle of $\kappa$.
  By \cite[Lemma 3.7]{luck-oliver:completion}, there is a $G$-vector bundle $\nu$
  over $Y$ and an isometric embedding $j:\gamma\to h^*(\nu)$ of $G$-vector bundles over $A$.
  We let $\varphi:\zeta\to\nu\oplus\zeta=\omega$
  be the embedding of the second summand.  
  We let $\lambda:\kappa=\gamma\oplus \psi(h^*(\zeta))\to h^*(\nu)\oplus h^*(\zeta)=h^*(\omega)$
  be the internal direct sum of $j$ and the
  inverse of the isomorphism $\psi:h^*(\zeta)\iso \psi(h^*(\zeta))$.
\end{proof}

Now we can prove the main result about the functor $E_G\gh{-}$.

\begin{thm}\label{thm:excisive - bundle} 
  Let $G$ be a discrete group. For every orthogonal $G$-spectrum $E$,
  the functor $E_G\gh{-}$ is excisive.
\end{thm}
\begin{proof}
  Homotopy invariance of the functor $E_G\gh{-}$ can be proved 
  in the same way as \cite[Lemma 6.6]{luck:burnside_ring},
  which is the special case $E=\mS_G$.
  For excision we consider two finite proper $G$-CW-pairs
  $(X,A)$ and $(Y,B)$ and a pushout square of $G$-spaces: 
  \begin{equation}  \begin{aligned}\label{eq:inproof_pushout}
      \xymatrix{ A\ar[r]^-i\ar[d]_f & X\ar[d]^g\\ B \ar[r]_-j & Y}             
    \end{aligned}\end{equation}
  We consider $(b,x)\in E_G\gh{B}\times E_G\gh{X}$ such that $f^*(b)=i^*(x)$ in $E_G\gh{A}$.

  Case 1:
  We suppose that there is a $G$-vector bundle $\zeta$ over $Y$ such that the
  given classes can be represented as  $b=[\zeta|_B,u]$ and $x=[g^*(\zeta),v]$.
  We observe that $f^*(\zeta|_B)=g^*(\zeta)|_A$ as $G$-vector bundles over $A$.
  We assume moreover that the maps
  \[ f^*(u)\ , \  v|_A\ :\ S^{g^*(\zeta)|_A}\ \to\  E(g^*(\zeta)|_A) \]
  are parameterized $G$-homotopic.
  We let
  \[ D \ = \ s(X)\cup S^{g^*(\zeta)|_A}\]
  be the $G$-subspace of $S^{g^*(\zeta)}$
  given by the union of the image of the section at infinity
  $s:X\to S^{g^*(\zeta)}$ and the part sitting over $A$.
  Proposition \ref{prop:s_infty is cofibration} implies that
  the inclusion $D\to S^{g^*(\zeta)}$ is a $\Com$-cofibration of $G$-spaces.
  The bundle projection $p:E(g^*(\zeta))\to X$
  is locally trivial in the equivariant sense, and hence a $\Com$-fibration of $G$-spaces.
  So the inclusion 
  \[ S^{g^*(\zeta)}\times 0\cup_{D \times 0} D\times[0,1] \to S^{g^*(\zeta)}\times[0,1] \]
  has the left lifting property with respect to the bundle projection $p:E(g^*(\zeta))\to X$.
  We can thus replace $v$
  by a map of retractive $G$-spaces $\bar v:S^{g^*(\zeta)}\to E(g^*(\zeta))$
  over $X$ that is equivariantly parameterized homotopic to $v$, and such that
  \[ f^*(u)\ = \ \bar v|_A \ .\]
  The two maps
  \[  S^{\zeta|_B} \ \xrightarrow{\ u \ } \  E(\zeta|_B)\
    \xrightarrow{\incl} \ E(\zeta) \text{\quad and\quad}
    S^{g^*(\zeta)}\ \xrightarrow{\ \bar v \ } \ E(g^*(\zeta))\ \xrightarrow{E(\bar g)}\ E(\zeta) \]
  are then compatible over $S^{f^*(\zeta|_B)}=S^{g^*(\zeta)|_A}$,
  where $\bar g:g^*(\zeta)\to\zeta$
  is the bundle morphism covering $g:X\to Y$. So these maps glue to a map of retractive $G$-spaces
  over $Y$
  \[  w \ = \ (\incl\circ u)\cup(E(\bar g)\circ \bar v)\ : \
    S^{\zeta} \ = \ S^{\zeta|_B}\cup_{S^{g^*(\zeta)|_A}} S^{g^*(\zeta)}\ \to \ E(\zeta) \ .\]
  The pair $(\zeta,w)$ then represents a class in $E_G\gh{Y}$ that satisfies
  $j^*[\zeta,w]=[\zeta|_B,u]=b$ and $g^*[\zeta,w]=[g^*(\zeta),\bar v]=[g^*(\zeta),v]=x$.

  Case 2:
  We suppose that there is a $G$-vector bundle $\zeta$ over $Y$ such that the
  given classes can be represented as  $b=[\zeta|_B,u]$ and $x=[g^*(\zeta),v]$.
  In contrast to the previous Case 1,
  we make no further assumptions on the maps $u$ and $v$.
  The hypothesis yields the relation
  \begin{align*}
    [f^*(\zeta|_B),f^*(u)] \ &= \ f^*(b)\ =\ i^*(x) \
                        = \ [g^*(\zeta)|_A,v|_A]  \ = \   [f^*(\zeta|_B),v|_A]
  \end{align*}
  in $E_G\gh{A}$.
  So Proposition \ref{prop:normal form}
  provides an isometric embedding $\psi:f^*(\zeta|_B)\to\kappa$
  of euclidean $G$-vector bundles over $A$
  such that the maps $\psi_*(f^*(u))$ and $\psi_*(v|_A)$ are parameterized $G$-homotopic.
  Proposition \ref{prop:relative LO} for $h=g i=j f$
  provides an isometric embedding $\varphi:\zeta\to\omega$
  of $G$-vector bundles over $Y$ and 
  an isometric embedding $\lambda:\kappa\to f^*(\omega|_B)$ of $G$-vector bundles over $A$
  such that the composite
  \[ 
    f^*(\zeta|_B) \ \xrightarrow{\ \psi\ } \ \kappa \ \xrightarrow{\ \lambda\ }\  f^*(\omega|_B) \]
  coincides with $f^*(\varphi|_B)$.
  So the maps 
  \begin{align*}
    f^*((\varphi|_B)_*(u)) \ &= \ (f^*(\varphi|_B))_*(f^*(u))\ =\ \lambda_*(\psi_*(f^*(u)))
                               \text{\quad and\quad}\\  
    ((g^*(\varphi))_*(v))|_A\ &= \ (f^*(\varphi|_B))_*(v|_A)\ =\ \lambda_*(\psi_*(v|_A))
  \end{align*}
  are parameterized $G$-homotopic.
  We have
  \[  b \ = \ [\zeta|_B,u] \ = \ [\omega|_B, (\varphi|_B)_*(u)]\text{\quad and\quad}
    x \ = \ [g^*(\zeta),v] \ = \ [g^*(\omega), (g^*(\varphi))_*(v)] \ ,\]
  and the new representatives $(\omega|_B, (\varphi|_B)_*(u))$ and
  $(g^*(\omega), (g^*(\varphi))_*(v))$
  satisfy the hypotheses of Case 1, so we are done.

  Case 3: Now we treat the general case.
  We consider two pairs $(\xi,u)$ and $(\eta,v)$ that represent classes 
  in $E_G\gh{B}$ and $E_G\gh{X}$, respectively, and
  such that $f^*[\xi,u]=i^*[\eta,v]$ in $E_G\gh{A}$.
  By \cite[Lemma 3.7]{luck-oliver:completion} there are $G$-vector bundles
  $\omega,\omega'$ over $Y$ such that $\xi$ is a direct summand in $\omega|_B$,
  and $\eta$ is a direct summand in $g^*(\omega')$. We set $\zeta=\omega\oplus\omega'$.
  Then there are isometric embeddings
  \[ a : \ \xi \ \to \ \zeta|_B \text{\qquad and\qquad} b\ : \ \eta\ \to \ g^*(\zeta) \]
  of $G$-vector bundles over $B$ and $X$, respectively. Hence
  \[  b \ = \ [\xi,u]\ = \ [\zeta|_B, a_*(u)] \text{\qquad and\qquad}
    x \ = \ [\eta,v]\ = \  [g^*(\zeta), b_*(v)]\ .\]
  The new representatives satisfy the hypotheses of Case 2, so we are done.
  This completes the proof of excision for the functor $E_G\gh{-}$.
  
  It remains to prove additivity. Excision for a pushout square
  \eqref{eq:inproof_pushout} with $A=\emptyset$ and $Y=B\amalg X$ shows that the map
  \[ (i_B^*,i_X^*)\ : \ E_G\gh{B\amalg X}\ \to \ E_G\gh{B}\times E_G\gh{X} \]
  is surjective. For injectivity we let $(\xi,u)$ represent a class in 
  $E_G\gh{B\amalg X}$ such that $i_B^*\gh{\xi,u}=i_X^*\gh{\xi,u}=0$.
  Then after stabilizing along some isometric embedding, if necessary, we can assume that
  the restriction of $u$ to $B$ is parameterized $G$-null-homotopic, and
  the restriction of $u$ to $X$ is parameterized $G$-null-homotopic.
  The total space of the sphere bundle $S^{\xi}$ over $B\amalg X$
  is the disjoint union of the total spaces of $S^{\xi|_B}$ and $S^{\xi|_X}$,
  so the two null-homotopies combine into a parameterized $G$-null-homotopy of $u$.
\end{proof}

\begin{rk}[Hilbert bundles versus vector bundles]\label{rk:restriction finite}
We would like to clarify where the restriction to discrete groups 
in Theorem \ref{thm:excisive - bundle}, and hence in all subsequent results
regarding the functor $E_G\gh{-}$, comes from.
While the definition of $E_G\gh{X}$ in Construction \ref{con:define alternative}
makes sense for all Lie groups $G$, the proof of excision
needs a crucial fact, proved in \cite[Lemma 3.7]{luck-oliver:completion}:
when $G$ is discrete and $\varphi:X\to Y$ is a continuous $G$-map
between finite proper $G$-CW-complexes,
then every $G$-vector bundle over $X$ is a summand of $\varphi^*(\xi)$
for some $G$-vector bundle $\xi$ over $Y$. 
As explained in Section 5 of \cite{luck-oliver:completion},
this fact does {\em not} generalize from discrete groups to Lie groups.
For {\em compact} Lie groups, excision still holds,
as a consequence of Example \ref{eg:bundle for compact}.

In the larger generality, Phillips \cite{phillips:equivariant_proper}
has defined equivariant K-theory for proper actions by using
suitable Hilbert $G$-bundles instead of finite-dimensional $G$-vector bundles.
One can speculate whether Phillips' approach can be adapted to generalize
our results about the functor $E_G\gh{-}$ from discrete groups to non-compact Lie groups,
but we will not go down that avenue in this monograph.
\end{rk}

\begin{prop}\label{prop:pi_*-iso2cohomology iso}
  Let $G$ be a discrete group and $X$ a finite proper $G$-CW-complex.
  \begin{enumerate}[\em (i)]
  \item For every $\pi_*$-isomorphism $f:E\to F$ of orthogonal $G$-spectra,
    the induced homomorphism $f_*:E_G\gh{X}\to F_G\gh{X}$ is an isomorphism.\index{pistar-isomorphism@$\pi_*$-isomorphism}
  \item
    For all orthogonal $G$-spectra $E$ and $F$, the maps
    \[ (E\vee F)_G\gh{X}\ \xrightarrow{\ \kappa_* \ } \ 
      (E\times F)_G\gh{X}\ \xrightarrow{(p^E_*,p^F_*)} \ 
      E_G\gh{X}\times F_G\gh{X} \]
    are isomorphisms, where $\kappa:E\vee F\to E\times F$ is the canonical map
    and $p^E:E\times F\to E$ and $p^F:E\times F\to F$ are the projections.
  \end{enumerate}
\end{prop}
\begin{proof}
  (i) We start with the special case when $G$ is finite, in which case we denote it by $H$ instead.
  The horizontal maps in the commutative square
  \[ \xymatrix{
      \colim_{V\in s(\cU_H)}\, [S^V\sm X_+,E(V)]^H_* \ar[r]_-\iso^-\omega\ar[d]_{f_*} &
      E_H\gh{X}\ar[d]^{f_*} \\
      \colim_{V\in s(\cU_H)}\, [S^V\sm X_+,F(V)]^H_* \ar[r]_-\omega^-\iso  &  F_H\gh{X}
    } \]
  are isomorphisms by Example \ref{eg:bundle for compact}.
  Adjointness identifies $[S^V\sm X_+,E(V)]^H_*$ with $[S^V,\map(X,E(V))]^H_*$.
  So the upper left group is isomorphic to
  \[  \colim_{V\in s(\cU_G)} [S^V,\map(X,E(V))]^H_* \ = \ \pi_0^H(\map(X,E)) \ , \]
  and hence invariant under $\pi_*$-isomorphisms in $E$, by \cite[Prop.\,3.1.40]{schwede:global}.
  This proves the proposition for finite groups.
  
  Now we let $G$ be any discrete group, and we suppose that $X=G/H\times K$
  for a finite subgroup $H$ of $G$ and a finite non-equivariant CW-complex $K$.
  Since $f$ is a $\pi_*$-isomorphism of orthogonal $G$-spectra,
  the underlying morphism of orthogonal $H$-spectra is a $\pi_*$-isomorphism.
  The induction isomorphisms \eqref{eq:ind bundle}  participate in a commutative diagram:
  \[ \xymatrix@C=12mm{
      E_H\gh{K} \ar[d]_{f_*}\ar[r]_-\iso^-{\ind}  & E_G\gh{G/H\times K}\ar[d]^{f_*} \\
      F_H\gh{K} \ar[r]_-{\ind}^-\iso  & F_G\gh{G/H\times K}
    } \]
  The left vertical map is an isomorphism by the previous paragraph,
  hence so is the right vertical map. This proves the special case $X=G/H\times K$.
  Since $E_G\gh{-}$ and $F_G\gh{-}$ are excisive by Theorem \ref{thm:excisive - bundle},
  the general case is now taken care of by Proposition \ref{prop:excisive on orbits}.
  
  (ii)
  The morphism $\kappa:E\vee F\to E\times F$ is a $\pi_*$-isomorphism,
  for example by \cite[Cor.\,3.1.37 (iii)]{schwede:global};
  so the first map $\kappa_*$ is an isomorphism by part (i).
  
  The morphisms of orthogonal $G$-spectra $i^E=(\Id_E,\ast):E\to E\times F$
  and $i^F=(\ast,\Id_F):F\to E\times F$ induce a homomorphism
  \[     E_G\gh{X}\times F_G\gh{X}\ \to\ (E\times F)_G\gh{X} \ ,\quad
    (x,y)\ \longmapsto \  i^E_*(x) + i^F_*(y) \ .\]
  This homomorphism splits the homomorphism $(p^E_*,p^F_*)$, which is thus surjective.
  Now we consider a pair $(\xi,u)$ that represents a class in the kernel of
  $(p^E_*,p^F_*):(E\times F)_G\gh{X} \to E_G\gh{X} \times F_G\gh{X}$.
  Then after stabilizing the representative along an isometric embedding,
  if necessary, we can assume that the composites
  \[ S^\xi \ \xrightarrow{\ u \ }\ (E\times F)(\xi) \ \xrightarrow{p^E(\xi)}\ E(\xi)
    \text{\quad and\quad}    S^\xi \ \xrightarrow{\ u \ }\ (E\times F)(\xi) \ \xrightarrow{p^F(\xi)}\ F(\xi)
  \]
  are parameterized equivariantly null-homotopic.
  The canonical map
  $(E\times F)(\xi)\to E(\xi)\times_X F(\xi)$ is an isomorphism,
  where the target is the fiber product over $X$; so the map $u$ itself is parameterized
  equivariantly null-homotopic. Hence $[\xi,u]=0$, and the homomorphism $(p^E_*,p^F_*)$
  is also injective.
\end{proof}

We have completed the construction of the excisive functor $E_G\gh{-}$.
Now we compare it to the functor that is represented by $E$
in the stable $G$-homotopy category $\Ho(\Sp_G)$.

\begin{con}\label{con:bundle class}
We consider a morphism of orthogonal $G$-spectra $f:\Sigma^\infty_+ X\to E$.
Such a morphism represents a class $\gamma_G(f)$ in $E_G(X)=[\Sigma^\infty_+ X,E]^G$,
where $\gamma_G:\Sp_G\to\Ho(\Sp_G)$ is the localization functor.
The value of $f$ at the zero vector space is a map of based $G$-spaces
$f(0):X_+=(\Sigma^\infty_+ X)(0)\to E(0)$, and we let $f^\flat$ be the map
of retractive $G$-spaces over $X$ 
\[ f^\flat\ :\ S^{\ul{0}}\ = \ S^0\times X \ \to \ E(0)\times X \ = \ E(\ul{0}) \]
that sends $(0,x)$ to $(f(0)(x),x)$.
The pair $(\ul{0}, f^\flat)$ then represents a class in the group $E_G\gh{X}$. 
\end{con}

  By Proposition \ref{prop:pi_*-iso2cohomology iso} (i), the functor
  \[
    (-)_G\gh{X} \ : \ \Sp_G \ \to \ \text{(sets)} \ , \quad E \ \longmapsto \ E_G\gh{X}
  \]
  takes $\pi_*$-isomorphisms of orthogonal $G$-spectra to bijections.
  So the functor factors uniquely through the localization $\gamma_G:\Sp_G\to\Ho(\Sp_G)$.
  We abuse notation and also write $E_G\gh{X}$ for the resulting functor
  defined on the homotopy category $\Ho(\Sp_G)$.
  
\index{proper $G$-cohomology theory!represented|(} 
\begin{thm}\label{thm:rep2vect}
  Let $G$ be a discrete group and $X$ a finite proper $G$-CW-complex.
  \begin{enumerate}[\em (i)]
  \item 
    There is a unique natural transformation
    \[  \mu_X^E\ : \ E_G(X)  \ \to\ E_G\gh{X} \]
    of covariant functors in $E$ from $\Ho(\Sp_G)$ to abelian groups
    with the following property:
    for every morphism of orthogonal $G$-spectra $f:\Sigma^\infty_+ X\to E$,
    the relation
    \[ \mu_X^E(\gamma_G(f)) \ = \ [\ul{0}, f^\flat]  \]
    holds in $E_G\gh{X}$.
  \item For every orthogonal $G$-spectrum $E$, the maps $\mu^E_X$ are natural in continuous $G$-maps
    $\varphi:Y\to X$ between finite proper $G$-CW-complexes.
  \item Let $\Gamma$ be a subgroup of $G$, and let $Y$ be a finite proper $\Gamma$-CW-complex. 
    Then the composite\index{induction isomorphism}
    \[ E_G(G\times_\Gamma Y)\ \xrightarrow{\text{\em adjunction}}\
      E_\Gamma(Y)\ \xrightarrow{\ \mu_Y^E\ }\  E_\Gamma\gh{Y}\ \xrightarrow[\eqref{eq:ind bundle}]{\ind}\
      E_G\gh{G\times_\Gamma Y}    \]
    coincides with $\mu_{G\times_\Gamma Y}^E$.
  \item
    For every orthogonal $G$-spectrum $E$, the map $\mu_X^E$ is an isomorphism of abelian groups.
\end{enumerate}
\end{thm}
\begin{proof}
  (i)  We can apply Construction \ref{con:bundle class}
  to the identity of the orthogonal $G$-spectrum $\Sigma^\infty_+ X$;
  it yields a class
  \[  [\ul{0},\Id^\flat]    \ \in \ (\Sigma^\infty_+ X)_G\gh{X}\ . \]
  Since $\Sigma^\infty_+ X$ represents the functor
  $(-)_G(X)$ on $\Ho(\Sp_G)$,
  the Yoneda lemma provides a unique natural transformation
  $\mu_X:(-)_G(X)=[\Sigma^\infty_+ X,-]^G\to (-)_G\gh{X}$ such that
  \[ \mu_X^{\Sigma^\infty_+X}(\Id_{\Sigma^\infty_+ X})\ = \   [\ul{0},\Id^\flat]\ . \]
  This transformation then satisfies the relation stated in part (i),
  by naturality:
  \[  \mu_X^E(\gamma_G(f)) \ = \ (\gamma_G(f))_*(\mu_X^E(\Id_{\Sigma^\infty_+ X}))
    = \ (\gamma_G(f))_*[\ul{0}, \Id^\flat] \
    = \                              [\ul{0}, f^\flat] \ . \]
  The natural transformation $\mu_X:(-)_G(X)\to (-)_G\gh{X}$ is a priori only set-valued,
  and we must prove that it is additive.
  The two functors $(-)_G(X)=[\Sigma^\infty_+ X,-]^G$ and $(-)_G\gh{X}$
  are {\em reduced}, i.e., they send the trivial orthogonal $G$-spectrum
  to the trivial abelian group. 
  Proposition \ref{prop:pi_*-iso2cohomology iso} (ii) says that the 
  target functor is also additive in $E$.
  As shown in \cite[Prop.\,2.2.12]{schwede:global},
  every set-valued natural transformation between reduced additive functors
  is automatically additive, so this proves that $\mu_X^E$ is a homomorphism
  of abelian groups.
  
  (ii)
  We must prove the commutativity of the following square:
  \[
    \xymatrix@C=15mm{ 
      [\Sigma^\infty_+ X,E]^G\ar[r]^-{\mu^E_{X}}\ar[d]_{\varphi^*} &
      E_G\gh{X}\ar[d]^{\varphi^*} \\
      [\Sigma^\infty_+ Y,E]^G \ar[r]_-{\mu^E_Y} & E_G\gh{Y}    }
  \]
  We let $E$ vary in $\Ho(\Sp_G)$.
  Then the Yoneda lemma reduces the claim to the universal example,
  the identity of $E=\Sigma^\infty_+ X$. The universal case is straightforward:
  \begin{align*}
    \varphi^*(\mu^{\Sigma^\infty_+ X}_X(\Id_{\Sigma^\infty_+ X}))\ = \ \varphi^*[\ul{0},\Id^\flat]\ 
    &= \  [\ul{0},\varphi^*(\Id^\flat)] \ 
      = \ [\ul{0},(\Sigma^\infty_+ \varphi)^\flat] \\
    &= \ \mu^E_Y(\gamma_G(\Sigma^\infty_+ \varphi)) \
      = \ \mu^E_Y(\varphi^*(\Id_{\Sigma^\infty_+ X})) \ .
  \end{align*}

  (iii)
  If we let $E$ vary in $\Ho(\Sp_G)$,
  the Yoneda lemma reduces the claim to the universal example,
  the identity of $E=\Sigma^\infty_+ (G\times_\Gamma Y)$.
  We consider the $\Gamma$-equivariant map $[1,-]:Y\to G\times_\Gamma Y$,
  the unit of the adjunction between restriction and extension of scalars.
  Then
  \begin{align*}
    \ind(\mu^{\Sigma^\infty_+ (G\times_\Gamma Y)}_Y(\text{adj}(\Id_{\Sigma^\infty_+ (G\times_\Gamma Y)})))
    \ &= \     \ind(\mu^{\Sigma^\infty_+ (G\times_\Gamma Y)}_Y(\gamma_\Gamma(\Sigma^\infty_+ [1,-])))\\
    &= \     \ind[\ul{0},(\Sigma^\infty_+ [1,-])^\flat]\\
    &= \ [\ul{0},\Id_{\Sigma^\infty_+ (G\times_\Gamma Y)}^\flat] \\
    & = \ \mu^{\Sigma^\infty_+ (G\times_\Gamma Y)}_{G\times_\Gamma Y}(\Id_{\Sigma^\infty_+ (G\times_\Gamma Y)}) \ .
  \end{align*}
  The third equation exploits that extending the zero $\Gamma$-vector bundle from $Y$ to $G\times_\Gamma Y$
  yields the zero $G$-vector bundle, and that the two maps of
  retractive $G$-spaces over $G\times_\Gamma Y$
  \[ G\times_\Gamma(\Sigma^\infty_+ [1,-])^\flat\ , \Id_{\Sigma^\infty_+ (G\times_\Gamma Y)}^\flat
    \ : \ S^0\times (G\times_\Gamma Y) \ \to \ (G\times_\Gamma Y)_+\times (G\times_\Gamma Y) \]
  are equal. 
  This proves the universal case, and hence the claim. 

  (iv)
  Source and target of the transformation $\mu_X^E$ send $\pi_*$-isomorphisms
  in $E$ to isomorphisms of abelian groups, by Proposition \ref{prop:pi_*-iso2cohomology iso}
  and by construction, respectively. So by appeal to the stable model structure
  of Theorem \ref{thm:stable} we can assume without loss of generality that
  $E$ is a $G$-$\Omega$-spectrum. 

  We start with the special case of a {\em finite} group, which we denote by $H$ (instead of $G$).
  We let $X$ be a finite $H$-CW-complex, and $E$ an orthogonal $H$-spectrum.
  We define a homomorphism
  \[ \nu \ : \ \colim_{W\in s(\cU_H)}\, [S^W\sm X_+,E(W)]^H_*\ \to\ [\Sigma^\infty_+ X,E]^H =  E_H(X)\ . \]
  The construction uses the shift $\sh^W X$ of an orthogonal $H$-spectrum $X$
  by an $H$-representation $W$, defined by
  \[ (\sh^W X)(V)\ = \ X(V\oplus W) \]
  with structure maps $\sigma^{\sh^W X}_{U,V}=\sigma^X_{U,V\oplus W}$.
  Here the $H$-action on $X(V\oplus W)$ is diagonal, from the $H$-actions on $X$ and on $W$.
  A natural morphism of orthogonal $H$-spectra
  \[ \lambda_X^W \ : \ X\sm S^W \ \to \ \sh^W X  \]
  is defined at $V$ as the composite
  \[ X(V)\sm S^W \ \xrightarrow{\text{twist}} \ S^W\sm X(V)\ \xrightarrow{\sigma_{W,V}^X} \
    X(W\oplus V)\ \xrightarrow{X(\tau_{W,V})}\ X(V\oplus W)\ .\]
  The morphism $\lambda_X^W$ and its adjoint $\tilde\lambda^W_X:X\to \Omega^W\sh^W X$
  are $\pi_*$-isomorphisms by \cite[Prop.\,3.1.25]{schwede:global}.
  The transformation $\nu$ now takes the class
  represented by a continuous based $H$-map
  $ x : S^W\sm X_+ \to E(W)$ to the composite morphism 
  \[ \Sigma^\infty_+ X\ \xrightarrow{\ \gamma_H(\tilde x)\ } \Omega^W \sh^W E\
    \xrightarrow[\iso]{\gamma_H(\tilde\lambda^W_E)^{-1}}\  E\]
  in $\Ho(\Sp_H)$. Here $\tilde x:\Sigma^\infty_+ X\to \Omega^W \sh^W E$
  is the morphism of orthogonal $H$-spectra adjoint to $x$.

  We must argue that $\nu$ is well-defined. If we vary the map $x:S^W\sm X_+\to E(W)$ 
  by a based equivariant homotopy, then $\tilde x$ changes by a homotopy of
  morphisms of orthogonal $H$-spectra, and its image $\gamma_H(\tilde x)$
  in $\Ho(\Sp_H)$ remains unchanged. Now we stabilize $x:S^W\sm X_+\to E(W)$
  along an inclusion $\iota:W\to \bar W$  in the poset $s(\cU_H)$
  to a new representative $\iota_*(x):S^{\bar W}\sm X_+\to E(\bar W)$, the composite
  \begin{align*}
 S^{\bar W}\sm X_+\ = \ 
    S^{\bar W- W}\sm S^W\sm X_+\ &\xrightarrow{S^{\bar W-W}\sm x}\
    S^{\bar W- W}\sm E(W)\\ &\xrightarrow{\sigma_{\bar W-W,W}}\
    E((\bar W- W)\oplus W)\ = \ E(\bar W)\ .      
  \end{align*}
  We let $\Psi:\Omega^W \sh^W E\to\Omega^{\bar W} \sh^{\bar W} E$
  be the morphism of orthogonal $H$-spectra defined at an inner product space $V$ as
  \[  \Psi(V)\ :\ \map_*(S^W,E(V\oplus W))\ \to\ \map_*(S^{\bar W},E(V\oplus\bar W))\ ,
    \quad f \longmapsto \ \iota_*(f)\ .\]
  Then the various morphisms of orthogonal $H$-spectra participate
  in a commutative diagram:
  \[  \xymatrix@C=20mm{
      \Sigma^\infty_+ X\ar[r]^-{\tilde x} \ar[dr]_{\widetilde{\iota_*(x)}} &
      \Omega^W \sh^W E\ar[d]^\Psi & E \ar[l]_-{\tilde\lambda^W_E}\ar[dl]^-{\tilde\lambda^{\bar W}_E}\\
&  \Omega^{\bar W}\sh^{\bar W} E &
    }\]
  Hence we conclude that
  \begin{align*}
    \gamma_H(\tilde\lambda^W_E)^{-1}\circ\gamma_H(\tilde x)\ 
        &= \  \gamma_H(\tilde\lambda^{\bar W}_E)^{-1}\circ\gamma_H(\Psi)\circ \gamma_H(\tilde x)\\
        &= \  \gamma_H(\tilde\lambda^{\bar W}_E)^{-1}\circ\gamma_H(\widetilde{\iota_*(x)} )\ .
  \end{align*}
  So the homomorphism $\nu$ is well-defined.

  Now we contemplate three composable maps:
    \[
    [X_+,E(0)]_*^H \to \colim_{V\in s(\cU_H)} [S^V\sm X_+,E(V)]^H_* \xrightarrow{\ \nu\ }
    E_H(X) \xrightarrow{\ \mu_X^E\ } E_H\gh{X}\ .
  \]
  The first map is the canonical map to the colimit, and it is an isomorphism because $E$
  is an $H$-$\Omega$-spectrum.   
  The composite map from $[X_+,E(0)]^H $ to $E_H(X)$
  is the derived adjunction for the Quillen adjoint functor pair
  \[ \xymatrix{ \Sigma^\infty_+ \ : \ H\bT\ \ar@<.4ex>[r] & \ \Sp_H \ : \ \ar@<.4ex>[l] \ (-)(0)}\ ;\]  
  it is thus bijective. So the map $\nu$ is also an isomorphism.
  The composite $\mu_X^E\circ\nu:\colim_V[S^V\sm X_+,E(V)]^H_* \to E_H\gh{X}$
  is the homomorphism $\omega$ discussed in Example \ref{eg:bundle for compact},
  which is thus an isomorphism. Since $\nu$ and $\omega$ are isomorphisms,
  so is the map $\mu_X^E$. This completes the proof of part (iv) for finite groups.
  
  Now we let $G$ be any discrete group, $H$ a finite subgroup of $G$, and we consider
  $X=G/H\times K$, for a non-equivariant finite CW-complex $K$.
  By part (iii), the map $\mu_{G/H\times K}^E$ factors as the composite
  \[
    E_G(G/H\times  K)\ \xrightarrow{\text{adjunction}}\
    E_H(K)\ \xrightarrow{\ \mu_K^E\ }\  E_H\gh{K}\ \xrightarrow{\ind}\
    E_G\gh{G/H\times K}\ .
  \]
  All three maps are isomorphisms,
  by the derived adjunction (see Corollary \ref{cor:restriction stable}),
  the special case of the previous paragraph, and by Example \ref{eg:induction iso}, respectively.
  So $\mu_{G/H\times K}^E$ is an isomorphism.
  Since $E_G(-)$ and $E_G\gh{-}$ are both excisive functors,
  Proposition \ref{prop:excisive on orbits} then finishes the argument.
\end{proof}
\index{proper $G$-cohomology theory!represented|)} 
\index{proper $G$-cohomology theory!$E_G\gh{-}$|)} 

\section{Proper cohomology theories from \texorpdfstring{$G$}{G}-spectra}
\label{sec:proper cohomology}

In this section we discuss how orthogonal $G$-spectra
give rise to `proper equivariant cohomology theories' for $G$-spaces, where $G$ is any Lie group.
For our purposes, a proper $G$-cohomology theory is a $\mZ$-indexed family
of excisive functors, linked by suspension isomorphisms, compare Definition \ref{def:proper cohomology}.
Our main example is the proper $G$-cohomology theory
represented by an orthogonal $G$-spectrum $E$, defined by taking morphism groups in
the triangulated stable homotopy category $\Ho(\Sp_G)$ into the shifts of $E$,
see Construction \ref{con:cohomology from E}.
One of the main points of this book is that the
equivariant cohomology theories represented by orthogonal $G$-spectra
have additional structure not generally present in the equivariant cohomology
theories arising from sequential $G$-spectra or $G$-orbit spectra.
This additional structure manifests itself in different forms,
such as a `$KO_G(\uEG)$-grading' 
or transfer maps that extend the homotopy group coefficient system to
a $G$-Mackey functor, compare Remark \ref{rk:RO(G) grading}.

For discrete groups, we also discuss how two prominent
equivariant cohomology theories
are represented by orthogonal $G$-spectra: 
equivariant stable cohomotopy in the sense of the third author
\cite{luck:burnside_ring} is represented by the $G$-sphere spectrum,
see Example \ref{eg:represent cohomotopy};
and Bredon cohomology with coefficients in a $G$-Mackey functor
is represented by an Eilenberg-MacLane spectrum,
see Example \ref{eg:HM_represents_Bredon}.

\begin{defn}\label{def:proper cohomology}\index{proper $G$-cohomology theory} 
  Let $G$ be a Lie group. A {\em proper $G$-cohomology theory}
  consists of a collection  of excisive functors
  \[ \cH^k\ : \ \text{(finite proper $G$-CW-complexes})^{\op}\ \to \ \cA b \]
  and  a collection of natural isomorphisms
  \[ \sigma \ : \ \cH^{k-1}(X)\ \xrightarrow{\ \iso \ } \ \cH^k(X\times S^1|X\times\infty) \ ,\]
  for $k\in\mZ$.
\end{defn}

Some previously studied proper $G$-cohomology theories
are Bredon cohomology,
Borel cohomology,
equivariant stable cohomotopy \cite{luck:burnside_ring},
and equivariant K-theory \cite{luck-oliver:completion, phillips:equivariant_proper}.
One of the main motivations for the present work was
to provide a general method for constructing proper $G$-cohomology theories
from $G$-spectra, see Theorem \ref{thm:represented equals bundle} below.
As reality check, we will show that for discrete groups $G$,
Bredon cohomology,
Borel cohomology,
equivariant stable cohomotopy,
and equivariant K-theory
are indeed represented by orthogonal $G$-spectra,
see Example \ref{eg:HM_represents_Bredon},
Example \ref{eg:Borel_cohomology},
Example \ref{eg:represent cohomotopy}, 
and Theorem \ref{thm:repr K-theory}, respectively.

We let $(\cH^k,\sigma)_{k\in\mZ}$ be a proper $G$-cohomology theory,
and we consider a pushout square of $G$-spaces
\[    \xymatrix{ A\ar[r]^-i\ar[d]_f & X\ar[d]^g\\ B \ar[r]_-j & Y}           \]
where $(X,A)$ and $(Y,B)$ are finite proper $G$-CW-pairs,
and $f$ and $g$ are cellular maps.
Then the 5-term Mayer-Vietoris sequence
established in Proposition \ref{prop:Mayer-Vietoris} extends to a long exact sequence
as follows.
We define a {\em connecting homomorphism}
\[ \delta^{k-1} \ : \ \cH^{k-1}(A)\ \to \ \cH^k(Y) \]
as the composite
\[ \cH^{k-1}(A)\ \xrightarrow[\iso]{\ \sigma \ } \ \cH^k(A\times S^1|A\times \infty)
  \ \xrightarrow{\ \partial \ } \   \cH^k(Y)\ , \] 
where the homomorphism $\partial$ was defined in \eqref{eq:define partial}.
Then we can splice the various exact sequences for the functor $\cH^k$
into an exact sequence
\begin{align}\label{eq:long_M V}
  \cdots \ \to \ &\cH^{k-1}(A)\  \xrightarrow{\ \delta^{k-1}\ } \  \cH^k(Y)\ \xrightarrow{(j^*,g^*)} \\
                        &\cH^k(B)\times\cH^k(X)\ \xrightarrow{f^*-i^*}\  \cH^k(A)\
                        \xrightarrow{\ \delta^k\ } \  \cH^{k+1}(Y)\ \to \ \cdots \ .\nonumber
  \end{align}

In Section \ref{sec:excisive} we used orthogonal $G$-spectra $E$ to
define three excisive functors $E_G\td{-}$, $E_G(-)$ and $E_G\gh{-}$.
The functor $E_G\td{-}$ is less relevant for us, among other things because
it does not in general extend to a `genuine' cohomology theory.
We now upgrade the two constructions  $E_G(-)$ and $E_G\gh{-}$
to proper $G$-cohomology theories by using the shifts of $E$.

\index{proper $G$-cohomology theory!represented|(}
\index{proper $G$-cohomology theory!$E_G\gh{-}$|(}
\begin{con}[Proper cohomology theories from orthogonal $G$-spectra]\label{con:cohomology from E}
  We let $G$ be a Lie group and $E$ an orthogonal $G$-spectrum.
  If $k$ is any integer, we recall from Remark \ref{rk:[k]} 
  that the $k$-fold shift $E[k]$ is defined as
  \[ E[k]\ = \
    \begin{cases}
      E\sm S^k & \text{\ for $k\geq 0$, and}\\
      \Omega^{-k} E & \text{\ for $k< 0$.}
    \end{cases}
  \]
  We define functors $E^k_G(-)$ and $E_G^k\gh{-}$ by
  \[ E_G^k(X)\ = \ E[k]_G(X)\ = \ [\Sigma^\infty_+X,E[k]]^G\text{\qquad and\qquad}
    E_G^k\gh{X}\ = \ E[k]_G\gh{X}\ .  \]
  The functor $E^k_G(-)$ is excisive by Proposition \ref{prop:represented is excisive}.
  If $G$ is discrete, then the functor $E^k_G\gh{-}$ is excisive by
  Theorem \ref{thm:excisive - bundle}. 

  \index{suspension isomorphism!for represented $G$-cohomology}
  Now we link these excisive functors by suspension isomorphisms.
  We define a suspension homomorphism
  \begin{equation} \label{eq:define_represented_Sigma}
    \sigma\ : \  E^k_G(X) \ \to \ E^{k+1}_G(X\times S^1|X\times\infty)    
  \end{equation}
  as the composite
  \[  [\Sigma^\infty_+ X,  E[k] ]^G \ \xrightarrow{\ [1]\ }\   
    [ \Sigma^\infty X_+\sm S^1,  E[k][1] ]^G 
    \ \xrightarrow{[\Sigma^\infty q,s_k]}\  
    [\Sigma^\infty_+ (X\times S^1),  E[k+1] ]^G \ ,  \]
  where $q:(X\times S^1)_+\to X_+\sm S^1$ is the projection,
  and the isomorphism $s_k:E[k][1]\iso E[k+1]$ was defined in \eqref{eq:define_s_k}.
  
  We define a suspension homomorphism\index{suspension isomorphism!for $E_G\gh{-}$}
  \begin{equation}\label{eq:suspension hom} 
    \Sigma^E\ : \ E_G\gh{X}\  \to \  (E\sm S^1)_G\gh{ X\times S^1|X\times\infty}
  \end{equation}
  as follows. We let $(\xi,u)$ represent a class in $E_G\gh{X}$.
  We pull back the vector bundle $\xi$ along the projection $X\times S^1\to X$
  to obtain a vector bundle $\xi\times S^1$ over $X\times S^1$.
  We define
  \[ \Sigma(u)\ : \ S^{\xi\times S^1}\ = \ S^{\xi}\times S^1 \ \to \
    ( E(\xi)\sm_X S^1)\times S^1 \ = \ (E\sm S^1)(\xi\times S^1)   \]
  by $\Sigma(u)(z,x)=(u(z)\sm x, x)$.
  We can then define the suspension homomorphism \eqref{eq:suspension hom} by
  \[  \Sigma^E[\xi,u]\ = \ [ \xi\times S^1,\Sigma(u)]\ .\]
  By construction, the map $\Sigma(u)$ is trivial over $X\times\infty$,
  so the image of $\Sigma^E$ indeed lies in the kernel of the restriction map
  $i^*:(E\sm S^1)_G\gh{ X\times S^1}\to (E\sm S^1)_G\gh{X\times\infty}$.
  The suspension homomorphisms
  \begin{equation}\label{eq:define bundle Sigma}
    \sigma \ : \   E^k_G\gh{X}\ \xrightarrow{\ \iso \ } \
    E^{k+1}_G\gh{X\times S^1|X\times\infty}
  \end{equation}
  for the theory $E^*_G\gh{-}$ are then defined as the composite
  \begin{align*} 
    E^k_G\gh{X}\ = \ E[k]_G\gh{X}\ \xrightarrow[\iso]{\Sigma^{E[k]}} \ 
& \ ( E[k][1])_G\gh{X\times S^1|X\times\infty}\nonumber\\ 
\xrightarrow{\ (s_k)_*\ }\ & E[k+1]_G\gh{X\times S^1|X\times\infty} 
\ = \ E^{k+1}_G\gh{X\times S^1|X\times\infty} \ .
  \end{align*} \end{con}

\begin{thm}\label{thm:represented equals bundle}
Let $G$ be a Lie group and $E$ an orthogonal $G$-spectrum.
\begin{enumerate}[\em (i)]
\item The functors $E^k_G(-)$ 
and the suspension homomorphisms \eqref{eq:define_represented_Sigma}
form a proper $G$-cohomology theory.
\item If $G$ is discrete, then the functors $E^k_G\gh{-}$ 
and the suspension homomorphisms \eqref{eq:define bundle Sigma}
form a proper $G$-cohomology theory.
\item If $G$ is discrete, then the natural transformations
  $\mu^{E[k]}:E^k_G(-)\to E^k_G\gh{-}$ 
form a natural isomorphism of proper $G$-cohomology theories.
\end{enumerate}
\end{thm}
\begin{proof}
(i)  The functors $E^k_G(-)$ are excisive by Proposition \ref{prop:represented is excisive}.
  It remains to show that the suspension homomorphism \eqref{eq:define_represented_Sigma}
  is an isomorphism. This is a direct consequence of the fact that the shift functor
  in a triangulated category is fully faithful, and that precomposition with the
  projection $q:(X\times S^1)_+\to X_+\sm S^1$ is an isomorphism from the group
  $[ \Sigma^\infty X_+\sm S^1,  E[1] ]^G$ to the kernel of the split epimorphism
  \[  i^* \ : \ [ \Sigma^\infty_+ (X\times S^1),  E[1] ]^G \ \to \
    [ \Sigma^\infty_+ X,  E[1] ]^G \ . \]

  (ii) We show first that the transformations $\mu_X^E$
  commute with the suspension homomorphisms.
  This amounts to checking that the following square commutes
  for all orthogonal $G$-spectra $E$
  and all finite proper $G$-CW-complexes $X$:
  \begin{equation}
    \begin{aligned}\label{eq:mu and sigma}
      \xymatrix@C=15mm{ 
        E_G(X)\ar[r]^-{\mu^E_X}\ar[d]_{\sigma} &
        E_G\gh{X}\ar[d]^{\Sigma^E} \\
        (E\sm S^1)_G(X\times S^1|X\times\infty) \ar[r]_-{\mu^{E\sm S^1}_{X\times S^1}} &
        (E\sm S^1)_G\gh{ X\times S^1|X\times\infty }
      }
    \end{aligned}
  \end{equation}
  We start with a class of the form $\gamma_G(f)$,
  for some morphism of orthogonal $G$-spectra $f:\Sigma^\infty_+ X\to E$.
  Then
  \begin{align*}
    \Sigma^E(\mu^E_X(\gamma_G(f)))\ &= \   \Sigma^E[\ul{0},f^\flat]\ 
                                      = \  [\ul{0},\Sigma(f^\flat)] \ = \ [\ul{0},((f\sm S^1)\circ\Sigma^\infty q)^\flat] \\
                                    &= \ \mu^{E\sm S^1}_{X\times S^1}(\gamma_G( (f\sm S^1)\circ\Sigma^\infty q)) \
                                      = \ \mu^{E\sm S^1}_{X\times S^1}(\sigma(\gamma_G(f))) \ .
  \end{align*}  
  The third equation exploits that pulling back the zero vector bundle
  yields the zero vector bundle,
  and that the functions $\Sigma(f^\flat)$ and $((f\sm S^1)\circ\Sigma^\infty q)^\flat$ are equal
  on the nose. 

  Now we let $E$ vary in $\Ho(\Sp_G)$;
  the Yoneda lemma reduces the commutativity of the square \eqref{eq:mu and sigma}
  to the universal example, the identity of $E=\Sigma^\infty_+ X$. This universal class is of the form
  considered in the previous paragraph, so we are done.

  The suspension homomorphism \eqref{eq:define_represented_Sigma}
  for the theory $E^*_G(-)$ is an isomorphism by part (i),
  and the maps $\mu_X^E$ and $\mu_{X\times S^1}^{E\sm S^1}$
  are isomorphisms by Theorem \ref{thm:rep2vect} (iv).
  So the commutative square \eqref{eq:mu and sigma}
  shows that the suspension homomorphism \eqref{eq:define bundle Sigma}
  for the theory $E^*_G\gh{-}$ is an isomorphism.
  The functors $E^k_G\gh{-}$ are excisive by Theorem \ref{thm:excisive - bundle},
  so this shows part (ii).
  
  We have now shown that the transformations $\mu_X^{E[k]}$
  commute with the suspension homomorphisms, so they form a morphism of proper $G$-cohomology theories.
  Since $\mu^{E[k]}$ is a natural isomorphism, this proves claim (iii).
\end{proof}
\index{proper $G$-cohomology theory!$E_G\gh{-}$|)}
\index{proper $G$-cohomology theory!represented|)}

\begin{eg}[Equivariant stable cohomotopy]\label{eg:represent cohomotopy}\index{equivariant stable cohomotopy}
  Let $G$ be a discrete group.
  In \cite[Sec.\,6]{luck:burnside_ring}, the third author introduced
  an equivariant cohomology theory $\pi^*_G(X)$
  on finite proper $G$-CW-complexes $X$,
  called {\em equivariant stable cohomotopy}.
  By the main result of \cite{luck:segal_infinite},
  this particular equivariant version of stable cohomotopy satisfies a completion
  theorem, generalizing Carlsson's theorem (previously known as the {\em Segal conjecture})
  for finite groups. 

  The definition of the group $\pi^k_G(X)$ is precisely the special case
  of $E^k_G\gh{X}$ for $E=\mS_G$ the $G$-sphere spectrum:
  \[ \pi^k_G(X)\ = \ \mS^k_G\gh{X}\ . \]
  In fact, Section 6 of \cite{luck:burnside_ring} is a blueprint
  for much of what we do here,
  and the definitions and proofs involving $E^*_G\gh{X}$
  were inspired by \cite{luck:burnside_ring}.
  
  Anyhow, for $E=\mS_G$, Theorem \ref{thm:represented equals bundle} (iii)
  shows that equivariant stable cohomotopy is represented in $\Ho(\Sp_G)$
  by the $G$-sphere spectrum, i.e., the map
  \[ \mu_X^{\mS_G[k]}\ : \ \mS^k_G(X)\ =\ [\Sigma^\infty_+X,\mS_G[k]]^G \ \to \ \pi^k_G(X) \]
  is an isomorphism for every finite proper $G$-CW-complex $X$ and every integer $k$.
\end{eg}

\begin{rk}[Genuine versus naive cohomology theories]\label{rk:RO(G) grading}\index{genuine cohomology theory}
As we argued in the proof of Proposition \ref{prop:represented is excisive},
every proper equivariant cohomology theory arising from an orthogonal $G$-spectrum
as in Definition \ref{def:proper from spectrum}
is also represented by a sequential $G$-spectrum as in \eqref{eq:define_sequential_F_coho}.
Hence these cohomology theories are also proper equivariant cohomology theories
in the sense of Davis and the third author
\cite[Def.\,4.1]{davis-luck:spaces_assembly}, compare Remark \ref{rk:Davis-Luck}.

As we already indicated a number of times, the equivariant cohomology theories represented
by orthogonal $G$-spectra
are a lot richer than those arising from sequential $G$-spectra or $G$-orbit spectra.
When $G$ is {\em compact},
then these special cohomology theories are known under the names
{\em genuine} or {\em $R O(G)$-graded} cohomology theories.\index{representation ring}
In our more general context, one could informally refer
to the extra structure as a `$K O_G(\uEG)$-grading',\index{KOG-grading@$K O_G(\uEG)$-grading}
where $K O_G(\uEG)$ is the Grothendieck group of\index{G-vector bundle@$G$-vector bundle}
isomorphism classes of real $G$-equivariant vector bundles over~$\uEG$.
If $G$ happens to be compact, then $\uEG$ can be taken to be
a point, so an equivariant vector bundle is just a representation,
and then the group $K O_G(\uEG)$ is isomorphic to
the real representation ring~$R O(G)$.

The $K O_G(\uEG)$-grading can be encoded as follows.
Given a $G$-vector bundle $\xi$ over~$\uEG$,\index{Thom space}
we continue to denote by $\Th(\xi)=S^\xi / s_\infty(\uEG)$ its Thom space.
We define the $\xi$-th $G$-equivariant $E$-cohomology group as
\[ E^\xi_G(X)\ = \  [\Sigma^\infty_+ X, E\sm \Th(\xi) ]^G \ .\]
If~$\xi$ is the trivial vector bundle of rank~$k$,
its Thom space is $\uEG_+\sm S^k$, which is
$\Com$-equivalent to $S^k$; so in that case we recover the group~$E^k_G(X)$.
If~$\eta$ is another such $G$-vector bundle, a suspension homomorphism
\[ \sigma^\eta  \colon  E^\xi_G(X)\ \to \  
E^{\xi\oplus\eta}_G(X_+\sm \Th(\eta)) \]
is essentially defined by derived smash product with the $G$-space $\Th(\eta)$.
Since smashing with the Thom space~$\Th(\eta)$ is invertible
(compare Proposition~\ref{prop:Thom is invertible}~(iv)), 
this suspension homomorphism is an isomorphism.
So up to a non-canonical isomorphism, the group $E^\xi_G(X_+\sm \Th(\eta))$
only depends on the class $[\xi]-[\eta]$ in the Grothendieck group $K O_G(\uEG)$.

As a concrete example, let us consider the infinite dihedral group $D_\infty=\mZ\ltimes\mZ/2$.
In Example \ref{eg:ED_infty}\index{infinite dihedral group}
we exhibited a $D_\infty$-CW-structure on the real line $\mR$
that is a model for $\underline{E}D_\infty$.
There are 2 equivariant 0-cells whose isotropy groups the two-element subgroups
$H_1$ generated by $(0,1+2\mZ)$ and $H_2$ generated by $(1,1+2\mZ)$, respectively;
and there is one free equivariant 1-cell.
The resulting Mayer-Vietoris sequence 
shows that the restriction homomorphisms combine into a monomorphism of rings
\[ (\res^{D_\infty}_{H_1},\res^{D_\infty}_{H_2}) \ : \ 
K O_{D_\infty}(\underline{E}D_\infty)\ \to \ R O(H_1)\times R O(H_2)\]
whose image consists of pairs of virtual representations with the same virtual dimension.
The rings $R O(H_1)$ and $R O(H_2)$ are isomorphic to $\mZ[\sigma]/(\sigma^2-1)$,
where $\sigma$ is the class of the sign representation.
We let $\sigma_1,\sigma_2$ denote the classes in $K O_{D_\infty}(\underline{E}D_\infty)$
that restrict to $(\sigma,1)$ and to $(1,\sigma)$, respectively.
Then the elements $1, \sigma_1$ and $\sigma_2$ form an additive basis of 
$K O_{D_\infty}(\underline{E}D_\infty)$,
and the multiplicative structure is determined by
$\sigma_1^2=\sigma_2^2=1$ and $\sigma_1\sigma_2=\sigma_2\sigma_1=\sigma_1+\sigma_2-1$.
In this special case, the three generators $1,\sigma_1$ and $\sigma_2$
are actually all represented by 1-dimensional representations of the group $D_\infty$.
For a general infinite discrete group $G$,
not all classes in $K O_G(\uEG)$ will be virtual $G$-representations.

Another piece of structure consists of transfers. Every proper cohomology
theory gives rise to a coefficient system by restriction to orbits.
If the cohomology theory arises from a sequential $G$-spectrum $E$, 
the coefficient system is simply the composite
\[ \left(\Or^{\Fin}_G\right)^{\op} \ \xrightarrow{\ \Phi(E)\ } \ 
\Sp^\mN\ \xrightarrow{\ \pi_0\ }\ \cA b \ ,\]
where $\Phi(E)(G/H)=E^H$ is the spectrum of $H$-fixed points.
Every coefficient system arises in this way from a
$G$-orbit spectrum, by postcomposition with a suitable
pointset Eilenberg-Mac\,Lane functor $H:\cA b\to\Sp^\mN$,
and then assembling the resulting $G$-orbit spectrum into
a sequential $G$-spectrum as indicated in Remark \ref{rk:Davis-Luck}.

If the cohomology theory arises from a genuine stable $G$-homotopy type $E$
(i.e., an object in $\Ho(\Sp_G)$), then the coefficient system 
is given by the equivariant homotopy groups $\pi_0^H(E)$
for finite subgroups $H$ of $G$;
the transfers discussed in Construction \ref{con:transfer} 
then extend the coefficient system to a $G$-Mackey functor,\index{G-Mackey functor@$G$-Mackey functor}
see Example \ref{eg:Mackey of G-spectrum}.
Moreover, every Mackey functor arises in this way from
an orthogonal $G$-spectrum, by Theorem \ref{thm:embed_MG_into_Ho(Sp)}.
In \cite[Sec.\,5]{luck:equivariant_chern}, the presence of transfers
is exploited to construct rational splittings of proper cohomology theories.
\end{rk}

We found it convenient to formulate our results
in terms of {\em absolute} equivariant cohomology theories.
As we shall now explain, every proper $G$-cohomology theory
can be extended in a specific way to a {\em relative} theory
defined on finite proper $G$-CW-pairs.

\begin{defn}(Relative cohomology groups)\label{def:relative}\index{relative cohomology group}\index{proper $G$-cohomology theory!relative}
  We let $G$ be a Lie group and $\{\cH^k,\sigma\}_{k\in\mZ}$ a proper $G$-cohomology theory.
  For every finite proper $G$-CW pair $(X,A)$ we define 
  the {\em relative $\cH$-cohomology groups} by
  \[ \cH^k(X,A) \ = \ \text{ker}( i_2^*\ : \ \cH^k(X\cup_A X)\ \to\ \cH^k(X) )\ ,\]
  where $i_2:X\to X\cup_A X$ is the embedding of the second copy of $X$.
\end{defn}

The relative groups are contravariantly functorial for continuous
$G$-maps of pairs.
We define a natural homomorphism $r:\cH^k(X,A)\to \cH^k(X)$
as the composite
\[ \cH^k(X,A)\ \xrightarrow{\incl}\ \cH^k(X\cup_A X)\ \xrightarrow{\ i_1^*\ }\
 \cH^k(X)\ , \]
where $i_1:X\to X\cup_A X$ is the embedding of the first copy of $X$.
The Mayer-Vietoris sequence \eqref{eq:long_M V} for the pushout square
\[ \xymatrix{  A \ar[r]^-i\ar[d]_i & X\ar[d]^{i_2} \\ X\ar[r]_-{i_1} & X\cup_A X}   \]
has the form
\begin{align*}
    \cdots\ \to\ 
    \cH^{k-1}(A)\ &\xrightarrow{\ \delta^{k-1}\ } \ 
    \cH^k(X\cup_A X)\  \xrightarrow{(i_1^*,i_2^*)}\\
    &\cH^k(X)\times \cH^k(X)\ \xrightarrow{(x,x')\mapsto i^*(x)-i^*(x')}
    \  \cH^k(A) \ \xrightarrow{\ \delta^k\ } \ \cdots \ .
\end{align*}
So the connecting homomorphism $\delta^{k-1}:\cH^{k-1}(A)\to \cH^k(X\cup_A X)$
lands in the subgroup $\cH^k(X,A)$, and the exact sequence splits off
a long exact sequence
\begin{equation}\label{eq:les proper coho}
    \cdots\ \to\ 
    \cH^{k-1}(A)\ \xrightarrow{\ \delta^{k-1}\ } \ 
    \cH^k(X,A)\  \xrightarrow{\ r\ }\
    \cH^k(X)\ \xrightarrow{\ i^*\ }\  \cH^k(A) \ \to \ \cdots\ .
\end{equation}

\begin{rk}
  For every orthogonal $G$-spectrum $E$, the functor $E_G(-)=[-,E]^G\circ \Sigma^\infty_+$
  is excisive by Proposition \ref{prop:represented is excisive}.
  For every finite proper $G$-CW-pair $(X,A)$,
  the relative group $E_G(X,A)$ is defined
  as in Definition \ref{def:relative}.
  This relative group can in fact be described more directly
  as a morphism group in $\Ho(\Sp_G)$: 
  we let $q:X\cup_A X\to X/A$ denote the map
  that sends the second copy of $X$ to the basepoint and that is the projection $X\to X/A$ 
  on the first copy. Since the inclusion of the second summand has a retraction,
  the cofiber sequence
  \[ X_+ \ \xrightarrow{\ i_2\ }\ (X\cup_A X)_+\ \xrightarrow{\ q\ }\  X/A \]
  induces a short exact sequence
  \[ 0 \ \to \ [\Sigma^\infty X/A,E]^G \ \xrightarrow{\ q^* \ } \ 
    E_G(X\cup_A X) \ \xrightarrow{\ i_2^* \ } \ E_G(X) \ \to \ 0 \]
  and hence an isomorphism
  \[  [\Sigma^\infty X/A, E ]^G \ \iso \  E_G(X,A)  \ .\]
  With a little more work one can show that under this identification, 
  the long exact sequence \eqref{eq:les proper coho}
  \[ \dots\ \to \ E_G^{k-1}(A) \ \xrightarrow{\ \delta^{k-1}\ }\  
    E_G^k(X,A)\  \to \  E_G^k(X)\ \to\ E_G^k(A)
    \  \to  \dots  \]
  is the effect of applying $[-,E[k]]^G$ 
  to the distinguished triangle in $\Ho(\Sp_G)$
  \[ \Sigma^\infty_+ A \ \xrightarrow{\Sigma^\infty \incl} \
    \Sigma^\infty_+ X \ \xrightarrow{\Sigma^\infty \proj} \
    \Sigma^\infty X/A \ \xrightarrow{(\Sigma^\infty p)\circ(\Sigma^\infty q)^{-1}} \
    (\Sigma^\infty_+ A)\sm S^1 \]
  and its rotations. Here $q:C A\cup_A X\to X/A$ is the projection from
  the mapping cone of the inclusion, and $p:C A\cup_A X\to A_+\sm S^1$
  was defined in Construction \ref{con:define triangles}.

  For discrete groups $G$, the functor $E_G\gh{-}$ defined in Construction \ref{con:define alternative}
  via parameterized equivariant homotopy theory
  is excisive by Theorem \ref{thm:excisive - bundle}.
  The relative groups $E_G\gh{X,A}$ defined as in 
  Definition \ref{def:relative} can also be described more directly
  by relative parameterized homotopy classes, in much the same way as for $E=\mS_G$ in
  \cite[Sec.\,6.2]{luck:burnside_ring}; we won't dwell on this any further. 
\end{rk}

A formal consequence of the properties of an equivariant cohomology theory
is an Atiyah-Hirzebruch type spectral sequence that starts from Bredon cohomology.
This spectral sequence is a useful calculational tool
and a systematic generalization of various previous statements,
so we take the time to spell out the details.
The following discussion is a special case of the Atiyah-Hirzebruch spectral sequence
of \cite[Thm.\,4.7]{davis-luck:spaces_assembly},
for the $\Fin$-orbit category of a discrete group.\index{Fin-orbit category@$\Fin$-orbit category}

\index{Bredon cohomology|(}
\index{Atiyah-Hirzebruch spectral sequence|(}
\begin{con}[Atiyah-Hirzebruch spectral sequence]\label{con:equivariant AHSS}
  We let $G$ be a discrete group,
  and we let $(\cH^k,\sigma)_{k\in\mZ}$
  be a proper $G$-cohomology theory.
  We consider a proper $G$-CW-complex $X$ with equivariant skeleton filtration
  \[ \emptyset = X^{-1} \subset X^0 \subset X^1 \subset \ldots  \subset X^n \subset \dots \ .\]
  If the cohomology theory is only defined on {\em finite} proper $G$-CW-complexes
  as in our original Definition \ref{def:proper cohomology}, then we must insist
  that $X$ has only finitely many cells. However, the following arguments
  apply without this size restriction if $\cH^k$ is defined on all $G$-CW-complexes 
  (i.e., possibly infinite),
  satisfies conditions (i), (ii) and (iii) of
  Definition \ref{def:excisive functor} in this larger category,
  and additionally, $\cH^k$ takes coproducts (possibly infinite)
  to products.
  
  From the cohomology theory and the CW-structure 
  we define an exact couple in the standard way by setting
  \[ D_1^{p,q}\ = \ \cH^{p+q}(X^p) \text{\qquad and\qquad} 
    E_1^{p,q}\ = \ \cH^{p+q}(X^p,X^{p-1}) \ ,\]
  see for example \cite[Sec.\,2.2, p.\,37ff]{mccleary:users_guide}.
  These groups are linked by the homomorphisms
  \begin{align*}
    \cH^{p+q}(\incl) \ = \ i\ &: \ D_1^{p,q}\ = \ \cH^{p+q}(X^p) \ \to \ \cH^{p+q}(X^{p-1})\ = \  D_1^{p-1,q+1} \\
    \delta\quad \ = \ j \ &: \ D_1^{p,q}\ = \ \cH^{p+q}(X^p) \ \to \ \cH^{p+q+1}(X^{p+1},X^p)\ = \  E_1^{p+1,q} \\
    r\quad\ = \ k \ &: \ E_1^{p,q}\ = \ \cH^{p+q}(X^p,X^{p-1}) \ \to \ \cH^{p+q}(X^p)\ = \  D_1^{p,q}   \ .
  \end{align*}
  Now we identify the $E_1$-term with the Bredon cohomology complex, and hence
  the $E_2$-term with Bredon cohomology.
  Bredon cohomology is defined for {\em $G$-coefficient systems},
  i.e., functors
  \[ M \ : \ \left(\Or^{\Fin}_G\right)^{\op} \ \to \ \cA b \ . \]
  Here $\Or^{\Fin}_G$ is the $G$-orbit category
  with finite stabilizers: the objects are the $G$-sets $G/H$ for all finite subgroups~$H$
  of $G$, and morphisms are $G$-maps.
  The $G$-CW-structure gives rise to a cellular cochain complex $C^*(X,M)$
  as follows.
  For every finite subgroup~$K$ of $X$, the fixed point space~$X^K$ is
  a non-equivariant CW-complex with respect to the skeleton filtration
  \[ (X^0)^K\ \subset\ (X^1)^K\ \subset\ \ldots\  \subset (X^n)^K\  \subset\ \dots \ .\]
  So $X^K$ has a cellular chain complex $C_*(X^K)$ with $n$-th chain group
  \[ C_n(X^K)\ = \ H_n\left((X^n)^K,(X^{n-1})^K,\mZ\right) \ ,\]
  the relative integral homology of the pair~$((X^n)^K,(X^{n-1})^K)$.
  Since $X^K$ is the space of $G$-maps from $G/K$ to~$X$, every morphism
  $f:G/K\to G/H$ in the orbit category $\Or_G^{\Fin}$ induces a cellular map 
  $f^*:X^H \to X^K$, and hence a morphism of cellular chain complexes
  \[ f^*\ : \  C_*(X^H) \ \to \ C_*(X^K)\ .\]
  These maps make the complexes $\{C_*(X^K)\}_{K\in\Fin}$ into a
  contravariant functor from~$\Or_G^{\Fin}$ to the category of chain complexes.
  Equivalently, we can consider $C_*(X^{\bullet})=\{C_*(X^K)\}_{K\in\Fin}$ as a chain complex
  of coefficient systems. Thus we can define a cochain complex of abelian
  groups by mapping into the given coefficient system, i.e., we set
  \[ C^n(X,M) \ = \ \Hom_{G\text{-coeff}}( C_n(X^{\bullet}),M)\ ,\]
  the group of natural transformations of coefficient systems.
  The cellular differential $C_{n+1}(X^{\bullet})\to C_n(X^\bullet)$
  induces a differential $C^n(X,M)\to C^{n+1}(X,M)$. The {\em Bredon cohomology}
  of $X$ with coefficients in $M$ is then given by
  \[ H_G^n(X,M)\ = \ H^n(C^*(X,M)) \ .\]
  
  The $G$-spaces $G/H$ for finite subgroups of $G$ are
  finite proper $G$-CW-complexes. So the proper $G$-cohomology theory
  gives rise to a $G$-coefficient system $\ul{\cH}^k$, namely the composite
  \[ \left(\Or^{\Fin}_G\right)^{\op} \ \xrightarrow{\incl} \ 
    \text{(finite proper $G$-CW-complexes})^{\op}\ \xrightarrow{\ \cH^k\ } \ \cA b \ .\]
  We describe an isomorphism of abelian groups
  \[ 
    C^p(X, \ul{\cH}^q)\ =\ 
    \Hom_{G\text{-coeff}}( C_p(X^{\bullet}),\ul{\cH}^q)\ \cong \  
    \cH^{p+q}(X^p,X^{p-1})\ = \ E_1^{p,q}\ .\]
  We choose a presentation of how $X^p$ is obtained by attaching equivariant 
  $p$-cells, in the form of a pushout of $G$-spaces:
  \[  \xymatrix{ 
      \coprod_I G/H_i\times \partial D^p \ar[r]\ar[d] & X^{p-1}\ar[d]\\
      \coprod_I G/H_i\times D^p \ar[r] & X^p }    \]
  These data induce an isomorphism
  \[ \bigoplus_{i\in I} H_p( (G/H_i)^K\times D^p,(G/H_i)^K\times \partial D^p,\mZ)\ \cong \ 
    H_p( (X^p)^K,(X^{p-1})^K,\mZ) \ = \ C_p(X^K)\ .\]
  Moreover,
  \[ H_p( (G/H_i)^K\times D^p,(G/H_i)^K\times \partial D^p,\mZ)\ \cong \ 
    \mZ[\Or_G^{\Fin}(G/K, G/H_i)] \]
  is the coefficient system represented by the coset~$G/H_i$, so 
  \[  C_p(X^{\bullet})\ \cong \ \bigoplus_{i\in I}\, \mZ[\Or_G^{\Fin}(-,G/H_i)] \ .   \]
  This isomorphism induces an isomorphism of Bredon cochain groups
  \begin{align*}
    C^p(X, \ul{\cH}^q )\ &= \ \Hom_{G\text{-coeff}}(C_p(X^\bullet),\ul{\cH}^q)\
                           \cong \ \prod_{i\in I} \cH^q( G/H_i )\\ 
&\iso \ \prod_{i\in I} \cH^{p+q}( G/H \times D^p, G/H\times \partial D^p )\ \iso 
        \  \cH^{p+q}(X^p,X^{p-1}) \ = \  E_1^{p,q}\ .
  \end{align*}
  The same argument as for the classical Atiyah-Hirzebruch spectral sequence
  identifies the Bredon cohomology differential
  \[ d\ :\ C^p(X, \ul{\cH}^q )\to  C^{p+1}(X, \ul{\cH}^q ) \]
  with  the $d_1$-differential 
  \[ d_1 = j\circ k\ :\ E^{p,q}_1\ \to \ E^{p+1,q}_1 \ .\]
  The $E_2^{p,q}$-term of the exact couple is thus
  the $p$-th Bredon cohomology group of $X$ with coefficients in $\ul{\cH}^q$.
  So the exact couple gives rise to a 
  conditionally convergent half plane spectral sequence
  \begin{equation}\label{eq:AHSS}
    E_2^{p,q} \ = \ H^p_G(X, \ul{\cH}^q)\ \Longrightarrow \ \cH^{p+q}(X) \ .  
  \end{equation}
  The $d_r$-differential has bidegree $(r,1-r)$.
  
  In the case where $X=G/H$ is a single orbit
  for a finite subgroup $H$ of $G$, 
  the $E_2$-term of the spectral sequence~\eqref{eq:AHSS}
  is concentrated in bidegrees $(0,q)$. So the spectral sequence collapses at $E_2$
  and recovers the isomorphism between $\cH^q(G/H)$ and~$H^0_G(G/H,\ul{\cH}^q)$.
\end{con}
\index{Atiyah-Hirzebruch spectral sequence|)}

\begin{eg}[Eilenberg-Mac\,Lane spectra represent Bredon cohomology]\label{eg:HM_represents_Bredon}
As before we let $G$ be a discrete group.\index{Eilenberg-Mac\,Lane spectrum}
Every $G$-Mackey functor $M$ has an associated Eilenberg-Mac\,Lane
$G$-spectrum $H\!M$, compare Remark \ref{rk:general Eilenberg Mac Lane}.
Since the homotopy group Mackey functors of $H\!M$ are concentrated
in degree zero,
the $E_2$-term of the Atiyah-Hirzebruch spectral sequence~\eqref{eq:AHSS}
for the proper $G$-cohomology represented by $H\!M$
is concentrated in bidegrees $(p,0)$. So the spectral sequence collapses at $E_2$
and yields an isomorphism
\[ H_G^*(X,M)\ \cong \  (H\!M)^*_G(X)\ .    \]
In this sense, Bredon cohomology is represented by an Eilenberg-Mac\,Lane spectrum.

The reader should beware, however, that Bredon cohomology is
defined for $G$-coefficients systems, whereas the construction of an
Eilenberg-Mac\,Lane spectrum requires a full-fledged $G$-Mackey functor.
Not every $G$-coefficient system can be extended to a $G$-Mackey functor,
and if an extension exists, it need not be unique. 
Different extensions of a $G$-coefficient system to a $G$-Mackey functor 
give orthogonal $G$-spectra that are non-isomorphic in $\Ho(\Sp_G)$. 
As we just argued, the $\mZ$-graded cohomology theory on $G$-CW-complexes
only depends on the underlying coefficient system, and so it does not `see'
the extension to a $G$-Mackey functor. The extension is visible, however,
if we extend the grading for the cohomology theory.
Indeed, the cohomology theory represented by an orthogonal $G$-spectrum 
can be indexed on equivariant vector bundles over~$\uEG$, 
see Remark~\ref{rk:RO(G) grading}. Different extensions to a Mackey
functor will typically lead to non-isomorphic `$K O_G(\uEG)$-graded' cohomology
theories.
\end{eg}
\index{Bredon cohomology|)}

\section{Global versus proper stable homotopy types}
\label{sec:global}

In this section we explain how global stable homotopy types
give rise to proper $G$-homotopy types,
and we identify some extra structure that is present in this case.
Several specific examples of such global equivariant theories
feature throughout this monograph,
for example equivariant stable cohomotopy
(Example \ref{eg:represent cohomotopy}, arising from the global sphere spectrum),
Borel cohomology (Example \ref{eg:Borel global}), 
and equivariant K-theory
(Section \ref{sec:equiv K},
arising from the periodic global K-theory spectrum $\bKU$).

In \cite{schwede:global} the fifth author developed a framework
for {\em global stable homotopy theory}, i.e., equivariant stable homotopy theory
where all compact Lie groups act simultaneously and in a compatible way.
The technical realization of this slogan is via a certain
{\em global model structure} on the category of orthogonal spectra
in which the weak equivalences are the `global equivalences'
of \cite[Def.\,4.1.3]{schwede:global},
see also Definition \ref{def:global equivalence} below.\index{global spectrum}
A direct consequence of the definition is a
`forgetful functor' from the global stable homotopy category
to $\Ho(\Sp_G)$ for every Lie group~$G$, see Theorem \ref{thm:change to groups}.
This forgetful functor is an exact functor of triangulated categories
that admits a left adjoint and a right adjoint.
In a sense made precise by Theorem \ref{thm:U_G and L alpha^*},
the forgetful functors for different Lie groups are compatible with
derived restriction along a continuous homomorphism
introduced in Section \ref{sec:change of groups}.
One specific benefit of coming from a global stable homotopy type is
that the associated equivariant cohomology theories
satisfy an induction isomorphism, see Proposition \ref{prop:global has induction}.
At the end of this section, we introduce the Borel $G$-spectrum
associated with a non-equivariant spectrum (see Example \ref{eg:Borel_cohomology}),
and we show that the Borel spectra in fact underlie a global homotopy type
(see Example \ref{eg:Borel global}).\medskip

For every Lie group $G$ we can consider the functor
\[ (-)_G \ : \ \Sp \ \to \ \Sp_G \ , \quad X \ \longmapsto \ X_G \]
from orthogonal spectra to orthogonal $G$-spectra
given by endowing an orthogonal spectrum with the trivial $G$-action.

\begin{defn}\label{def:global equivalence}
  A morphism~$f:X\to Y$ of orthogonal spectra is a {\em global equivalence}\index{global equivalence}
  if the map
  \[ \pi_k^H(f_H)\ : \ \pi_k^H(X_H)\ \to \ \pi_k^H(Y_H) \]
  is an isomorphism for every compact Lie group~$H$ and every integer~$k$.
\end{defn}

\index{global stable homotopy category|(}
We denote by $\GH = \Ho^{\gl}(\Sp)$
the category obtained by formally inverting 
the global equivalences of orthogonal spectra, 
and we refer to this as the {\em global stable homotopy category}.
We write 
\[ \gamma_{\gl}\ :\ \Sp\ \to \ \Ho^{\gl}(\Sp) \ = \ \GH\]
for the localization functor.
By \cite[Thm.\,4.3.18]{schwede:global}, 
the global equivalences are part of a stable model structure.
The global stable homotopy category is a compactly generated triangulated category, 
and a specific set of compact generators
is given by the suspension spectra of the `global classifying spaces'
of all compact Lie groups, see \cite[Thm.\,4.4.3]{schwede:global}.

By the very definition, the functor~$(-)_G$
takes global equivalences of orthogonal spectra 
to $\pi_*$-isomorphisms of orthogonal $G$-spectra.
So we obtain a `forgetful' functor on the homotopy categories
\begin{equation} \label{eq:define_U_G}
  U_G\ = \ \Ho( (-)_G)\ : \ \GH \ \to \ \Ho(\Sp_G)    
\end{equation}
from the universal property of localizations.
In other words, $U_G$ is the unique functor that satisfies 
\[ U_G\circ\gamma_{\gl}\ =\ \gamma_G\circ (-)_G\ . \]
Moreover, $U_G$ is canonically an exact functor of triangulated categories:
the pointset level equality
\[ X_G\sm S^1 \ = \ (X\sm S^1)_G \]
of functors $\Sp\to\Sp_G$ descends to an equality
\begin{align*}
  U_G\circ [1]\ &= \ \Ho((-)_G)\circ \Ho(-\sm S^1)
                  \ = \  \Ho((-)_G\circ (-\sm S^1)) \\ 
                &= \  \Ho((-\sm S^1)\circ (-)_G)\
                  \ = \  \Ho(-\sm S^1)\circ\Ho( (-)_G)\ = \ [1]\circ U_G\ .   
\end{align*}
Since distinguished triangles are defined in exactly the same way in 
$\GH$ and $\Ho(\Sp_G)$, the functor $U_G$ preserves them.

The functor $X\mapsto X_G$ is fully faithful on the pointset level, 
but its  derived functor $U_G$ is typically {\em not} fully faithful.
A hint is the fact that the equivariant homotopy groups
of a global homotopy type, restricted to $G$ and its subgroups,
have more structure than is available for a general $G$-homotopy type,
and satisfy certain relations that do not hold for general
orthogonal $G$-spectra.

\begin{thm}\label{thm:change to groups}
  For every Lie group~$G$ the forgetful functor
  \[ U_G \ : \ \GH \ \to \ \Ho(\Sp_G) \]
  preserves all set-indexed sums and products, and it has a left adjoint and a right adjoint.
\end{thm}
\begin{proof}
  Sums in $\GH$ and $\Ho(\Sp_G)$ are represented 
  in both cases by the pointset level wedge.
  On the pointset level, the forgetful functor preserves wedges,
  so the derived forgetful functor preserves sums.
  The existence of the right adjoint is an abstract consequence
  of the fact that~$\GH$ is compactly generated 
  and that functor~$U$ preserves sums, 
  compare~\cite[Cor.\,4.4.5 (iv)]{schwede:global}.

  The forgetful functor also preserves infinite products,
  but the argument here is slightly more subtle because products in~$\GH$
  are not generally represented by the pointset level product,
  and because equivariant homotopy groups do not in general commute
  with infinite pointset level products.
  We let~$\{X_i\}_{i\in I}$ be a set of orthogonal spectra.
  By replacing each factor by a globally equivalent spectrum, if necessary,
  we can assume without loss of generality that each $X_i$ is a global $\Omega$-spectrum
  in the sense of~\cite[Def.\,4.3.8]{schwede:global}.
  Since global~$\Omega$-spectra are the fibrant objects
  in a model structure underlying~$\GH$, 
  the pointset level product $\prod_{i\in I} X_i$ then represents the product in~$\GH$.
  
  Even though~$X_i$ is a global $\Omega$-spectrum,
  the underlying orthogonal $G$-spectrum $(X_i)_G$ need {\em not}
  be a $G$-$\Omega$-spectrum. However, as spelled out in the proof of
  \cite[Prop.\,4.3.22 (ii)]{schwede:global}, the natural map
  \[ \pi_k^H\left({\prod}_{i\in I} X_i\right) \ \to \ {\prod}_{i\in I} \pi_k^H(X_i)\]
  is an isomorphism for all compact Lie groups~$H$ and all integers~$k$.
  This implies that in this situation, 
  the pointset level product is also a product in~$\Ho(\Sp_G)$.
  So the derived forgetful functor preserves products.
    
  The existence of the left adjoint is then again an abstract consequence
  of the fact that~$\GH$ is compactly generated 
  and that the functor~$U$ preserves products, 
  compare~\cite[Cor.\,4.4.5 (v)]{schwede:global}.
\end{proof}
\index{global stable homotopy category|)}

\index{left derived functor!of restriction|(}
\index{G-equivariant stable homotopy category@$G$-equivariant stable homotopy category|(}
We let $ \alpha \colon K \to G$ be a continuous homomorphism between Lie groups.
In Theorem \ref{thm:homomorphism adjunctions spectra} 
we discussed various properties of the total left derived functor
\[ L\alpha^* \ : \ \Ho(\Sp_G)\ \to \ \Ho(\Sp_K) \]
of the restriction functor $\alpha^*:\Sp_G\to\Sp_K$,
with $\alpha_!:L\alpha^*\circ\gamma_G\Longrightarrow \gamma_K\circ \alpha^*$
the universal natural transformation.
For example, the functor $L\alpha$ is exact and has a right adjoint.
For another homomorphism $\beta:J\to K$, we constructed a 
specific exact natural isomorphism
\[   \td{\alpha,\beta}\ : \ 
(L\beta^*)\circ (L\alpha^*) \ \Longrightarrow \ L(\alpha\beta)^*  \]
in \eqref{eq:td(a,b)}. 
The data of the functors $L\alpha^*$ and the transformations $\td{\alpha,\beta}$
form a pseudo-functor from the category of Lie groups and continuous homomorphisms
to the 2-category of triangulated categories, 
exact functors, and exact transformations.

Now we discuss how the derived restriction
functors interact with the passage from global to proper homotopy theory.
If $X$ is any orthogonal spectrum, then on the pointset level, we have
$\alpha^*(X_G)=X_K$, because $K$ acts trivially on both sides.
However, $X_G$ will typically not be cofibrant as an orthogonal $G$-spectrum,
so the relationship between the derived functors is more subtle:
the universal property of the derived functor $U_K$ provides 
a unique natural transformation
\begin{equation}\label{eq:define_sharp}
 \alpha^\sharp\ : \ L\alpha^*\circ U_G \ \Longrightarrow \ U_K    
\end{equation}
of functors $\GH\to\Ho(\Sp_K)$ that satisfies the relation
\[  \alpha^\sharp\star \gamma_{\gl}\ = \ \alpha_! \star(-)_G  \]
as transformations from the functor 
$L\alpha^*\circ U_G\circ \gamma_{\gl}=L\alpha^*\circ \gamma_G\circ (-)_G$
to the functor 
$U_K\circ\gamma_{\gl}=\gamma_K\circ \alpha^*\circ (-)_G$.

We recall that a continuous homomorphism between Lie groups
is {\em quasi-injective} if the restriction
to every compact subgroup of the source is injective.

\begin{thm}\label{thm:U_G and L alpha^*}
Let $\alpha:K\to G$ be a continuous homomorphism between Lie groups.
\begin{enumerate}[\em (i)]
\item If $\alpha$ is quasi-injective, then the natural transformation 
$\alpha^\sharp: L\alpha^*\circ U_G\Longrightarrow U_K$ is an isomorphism.
\item If $\beta:J\to K$ is another continuous homomorphism, then 
\[ (\alpha\beta)^\sharp \circ (\td{\alpha,\beta}\star U_G) \ = \ 
\beta^\sharp\circ (L\beta^*\star \alpha^\sharp)
 \]
as natural transformations $L\beta^*\circ L\alpha^*\circ U_G\Longrightarrow U_J$.
\end{enumerate}
\end{thm}
\begin{proof}
(i) We let $X$ be any orthogonal spectrum.
We choose a $\pi_*$-isomorphism of orthogonal $G$-spectra $\psi:Y\to X_G$
whose source is cofibrant.
We obtain a commutative diagram in $\Ho(\Sp_K)$:
\[ \xymatrix@C=20mm{ 
(L\alpha^*)(Y) \ar[r]_-\iso^-{(L\alpha^*)(\gamma_G(\psi))} \ar[d]^-\iso_{\alpha_!} &
(L\alpha^*)(X_G) \ar[r]^-{\alpha^\sharp}  \ar[dr]_{\alpha_!}& X_K \ar@{=}[d]\\
\alpha^*(Y) \ar[rr]^\iso_-{\gamma_K(\alpha^*(\psi))} && \alpha^*(X_G)
} \]
The left vertical morphism $\alpha_!$ is an isomorphism
by Theorem \ref{thm:homomorphism adjunctions spectra} (ii).
Since $\alpha$ is quasi-injective, the functor $\alpha^*$ is homotopical 
by Theorem \ref{thm:adjunctions spectra pointset} (i),
so the lower horizontal morphism is an isomorphism as well.
This shows that $\alpha^\sharp:(L\alpha^*)(X_G)\to X_K$ is an isomorphism
in $\Ho(\Sp_K)$.

Part (ii) is Proposition \ref{prop:derived cocycle},
applied to the left derivable functors $(-)_G:\Sp\to \Sp_G$, 
$\alpha^*:\Sp_G\to\Sp_K$ and $\beta^*:\Sp_K\to\Sp_J$.
\end{proof}
\index{left derived functor!of restriction|)}
\index{G-equivariant stable homotopy category@$G$-equivariant stable homotopy category|)} 

As we just explained, every global homotopy type
gives rise to a $G$-homotopy type for every Lie group~$G$.
The `global' nature is also reflected in the $G$-equivariant cohomology
theories represented by the $G$-spectra.
The following proposition says that for every orthogonal spectrum~$E$,
the collection of equivariant cohomology theories
$E^*_G$ for varying $G$ form an `equivariant cohomology theory'
in the sense of \cite[5.2]{luck:burnside_ring}.

\begin{con}[Restriction maps for global homotopy types]
If all the equivariant cohomology theories~$E^*_G$
arise from a global homotopy type (i.e., from a single orthogonal spectrum),
then there is extra structure in the form of \emph{restriction homomorphisms}
\[ \alpha^* \ : \ E^*_G(X) \ \to \ E^*_K(\alpha^*(X))\]
associated with every continuous homomorphism~$\alpha:K\to G$
between Lie groups.
Here $X$ is any $\Com$-cofibrant $G$-space, 
so that $\Sigma^\infty_+ X$ is a cofibrant orthogonal $G$-spectrum;
hence the morphism 
\[ \alpha_!\ :\ (L\alpha^*)(\Sigma^\infty_+ X) \ \to \ 
\alpha^*(\Sigma^\infty_+ X) \ = \ \Sigma^\infty_+ \alpha^*(X) \]
is an isomorphism in $\Ho(\Sp_K)$, 
compare Theorem \ref{thm:homomorphism adjunctions spectra} (ii).
We then define $\alpha^*$ as the composite
\begin{align*}
 E^0_G(X)\ = \ [\Sigma^\infty_+ X, E_G]^G \ &\xrightarrow{L\alpha^*} \ 
[(L\alpha)^*(\Sigma^\infty_+ X), (L\alpha^*)(E_G) ]^K \\ 
&\xrightarrow{[\alpha_!^{-1},\alpha^\sharp]} \ 
[\Sigma^\infty_+ \alpha^*(X),  E_K ]^K \ = \ E^0_K(\alpha^*(X))\ .   
\end{align*}
The natural transformation $\alpha^\sharp:L\alpha^*\circ U_G\Longrightarrow U_K$
was defined in \eqref{eq:define_sharp}.
\end{con}

We let $\alpha:K\to G$ be a continuous homomorphism between Lie groups.
As before, for a $K$-space~$X$ we denote the {\em induced $G$-space} by
\[ G\times_\alpha X \ = \ (G\times X) / 
(g\cdot\alpha(k),x) \sim (g,k\cdot x)\ . \]
The functor~$G\times_\alpha-$ is left adjoint to restriction along~$\alpha$,
and the map
\[ \eta_X\ : \ X \ \to \ \alpha^*(G\times_\alpha X)\ , \quad x \ \longmapsto \ [1,x] \]
is the unit of the adjunction.

Now we let $E$ be an orthogonal spectrum.
If $X$ is a $\Com$-cofibrant $K$-space, 
then $G\times_\alpha X$ is a $\Com$-cofibrant $G$-space,
by Proposition \ref{prop:homomorphism adjunctions spaces} (ii). 
We define the {\em induction map}
\begin{equation}\label{eq:induction_map}
 \ind_\alpha \ : \  E^*_G(G\times_\alpha X) \ \to \ E^*_K(X)    
\end{equation}
as the composite
\[ E^*_G(G\times_\alpha X) \ \xrightarrow{\alpha^*} \ E^*_K(\alpha^*(G\times_\alpha X))
\ \xrightarrow{\ \eta_X^*\ }\ E^*_K(X)\ . \]

\begin{prop}\label{prop:global has induction}
  Let~$E$ be an orthogonal spectrum and~$\alpha:K\to G$
  a continuous homomorphism between Lie groups. 
  Let $X$ be a proper $K$-CW-complex
  on which the kernel of $\alpha$ acts freely. 
  Then the induction map~\eqref{eq:induction_map}
  is an isomorphism.
\end{prop}
\begin{proof}
The functor~$G\times_\alpha-$ preserves equivariant homotopies 
and commutes with wedges and mapping cones.
So the functor $E^*_G(G\times_\alpha-)$ from the category of 
$K$-spaces to graded abelian groups is a proper cohomology theory. 
The induction maps form a transformation
of cohomology theories, so it suffices to check the claim on orbits
of the form  $X=K/L$, for all compact subgroups $L$ of~$K$,
on which the kernel of $\alpha$ acts freely. 
The freeness condition precisely means that the restriction
\[ \bar\alpha = \alpha|_L\ : \ L \ \to \ G \]
of $\alpha$ to~$L$ is injective. The $G$-map
\[ \psi \ : \ G\times_{\bar\alpha} (L/L) \ \to \  G\times_\alpha (K/L)\ , \quad
[g, e L]\ \longmapsto \ [g, e L] \]
is a homeomorphism, and it makes the following square commute:
\[ \xymatrix{
L/L \ar[r]^-{\incl}\ar[d]_{\eta_{L/L}} & K/L\ar[d]^{\eta_{K/L}}\\
G\times_{\bar\alpha} (L/L) \ar[r]_-{\psi} &G\times_\alpha (K/L)
} \]
So the following diagram of equivariant cohomology groups commutes as well:
\[ \xymatrix@C=12mm{ 
& E^*_G( G\times_\alpha (K/L) ) \ar@/_1pc/[ddl]_-{\ind_\alpha}\ar[r]^-{\psi^*}_-\iso \ar[d]^{\bar\alpha^*} &
E^*_G( G\times_{\bar\alpha} L/L ) \ar[d]_{\bar\alpha^*} \ar@/^2pc/@<8ex>[dd]^(.3){\text{adj}}_(.3)\iso\\
&
E^*_L(\bar\alpha^*(G\times_\alpha( K/L))) \ar[r]^-{\psi^*}_-\iso 
\ar[d]^-{\eta_{K/L}^*} &
E^*_L(\bar\alpha^*(G\times_{\bar\alpha} L/L))
\ar[d]_{\eta_{L/L}^*} \\
E^*_K( K/L) \ar[r]^-{\res^K_L} \ar@/_2pc/[rr]_(.35){\text{adj}}^(.35)\iso&
E_L^*(K/L) \ar[r]^-{\incl^*}  &
E_L^*( L/L)} \]
Commutativity of the left part uses
the relation $\res^K_L\circ\alpha^*=\bar\alpha^*$
and the naturality of restriction from $K$ to $L$.
The lower horizontal composite and the right vertical composite 
are adjunction bijections.
This proves that the induction map is an isomorphism for $X=K/L$.
\end{proof}

The following corollary is the special case of
Proposition \ref{prop:global has induction}
for the unique homomorphism $K\to e$.
Part (ii) also uses that whenever $K$ has no non-trivial compact subgroups, 
then $K$ acts freely on the universal proper $K$-space $\ul{E}K$.

\begin{cor}
  Let $E$ be an orthogonal spectrum and $K$ a Lie group.
  \begin{enumerate}[\em (i)]
  \item
    For every free $K$-CW-complex $X$, the induction map~\eqref{eq:induction_map}
    \[   \ind \ : \  E_e^*(X/K) \ \to \ E^*_K(X)     \]
    for the unique homomorphism $K\to e$ is an isomorphism.
  \item 
    If $K$ has no non-trivial compact subgroups, 
    then the induction map~\eqref{eq:induction_map}
    \[   \ind \ : \  E^*_e( B K) \ \to \ E^*_K(\ul{E}K) \ \iso \ \pi^K_{-*}(E)     \]
    for the unique homomorphism $K\to e$ is an isomorphism. 
  \end{enumerate}
\end{cor}

\index{Borel cohomology|(}
\begin{eg}[Borel cohomology]\label{eg:Borel_cohomology}
We let $F$ be a non-equivariant generalized cohomology theory.
For a $G$-space~$A$, the associated {\em Borel cohomology theory} is given 
\[ F^*(E G\times_G A) \ , \] 
the $F$-cohomology of the Borel construction.
This Borel cohomology theory is realized by an orthogonal $G$-spectrum.
For this purpose we represent the given cohomology theory by an orthogonal
$\Omega$-spectrum $X$ (in the non-equivariant sense).
We claim that then the orthogonal $G$-spectrum 
\[ b X\ = \ \map(E G,X) \ ,   \]
obtained by taking the space of unbased maps from $E G$ levelwise
(see Construction \ref{con:smash with G-space}), represents Borel cohomology.
\end{eg}

\begin{prop}\label{prop:global homotopy of b X}
Let $G$ be a Lie group and $X$ an orthogonal $\Omega$-spectrum.
\begin{enumerate}[\em (i)]
\item The orthogonal $G$-spectrum $\map(E G,X)$ is a $G$-$\Omega$-spectrum. 
\item For every $\Com$-cofibrant $G$-space $A$,
there is an isomorphism
\[ \map(E G,X)^k_G(A)\ \iso \ X^k( E G\times_G A) \]
that is natural for $G$-maps in $A$. In particular,
\[ \pi_{-k}^G(\map(E G, X)) \ \cong \  X^k(B G)\ . \]
\end{enumerate}
\end{prop}
\begin{proof} (i) 
We let $H$ be any compact subgroup of~$G$, and let~$V$ and~$W$ be two $H$-representations.
Then the adjoint structure map
\[  \tilde\sigma_{V,W}^X\ : \  X(W) \ \to \ \map_*(S^V, X(V\oplus W)) \]
is $H$-equivariant and a weak equivalence on underlying non-equivariant spaces.
The underlying $H$-space of $E G$ is a free cofibrant $H$-space,
so applying $\map^H(E G,-)$ to $\tilde\sigma_{V,W}^X$ returns a weak equivalence
\begin{align*}
  \map^H(E G ,\tilde\sigma_{V,W}^X) \ : \  
(\map(E G, X)(W))^H \ =  \ &\map^H(E G, X(W)) \\ 
\to \ &\map^H(E G, \map_*(S^V, X(V\oplus W))) \ .  
\end{align*}
The target of this map is homeomorphic to 
\[  \map^H_*(S^V,\map(E G,X(V\oplus W))) \ = \   \map^H_*(S^V,\map(E G,X)(V\oplus W)) \] 
in such a way that $\map^H(E G,\tilde\sigma_{V,W}^X)$
becomes the $H$-fixed points of the adjoint structure map of $\map(E G,X)$.
So $\map(E G, X)$ is a $G$-$\Omega$-spectrum.

(ii) Because $A$ is $\Com$-cofibrant, its unreduced suspension spectrum
$\Sigma^\infty_+ A$ is cofibrant in the stable model structure of orthogonal $G$-spectra.
On the other hand, the spectrum $\map(E G,X)$ is a $G$-$\Omega$-spectrum by part~(i),
hence it is fibrant in the stable model structure of orthogonal $G$-spectra.
So morphisms from 
$\Sigma^\infty_+ A$ to $\map(E G,X)$ in $\Ho(\Sp_G)$ can be calculated as homotopy
classes of morphisms of orthogonal $G$-spectra. 
Combining this with various adjunction bijections yields the desired isomorphism
for $k=0$:
\begin{align}\label{eq:Borel iso}
 \map(E G, X)^0_G(A)\ &= \ [\Sigma^\infty_+ A, \map(E G,X)]^G \\ 
&\cong \ \Sp_G(\Sigma^\infty_+ A, \map(E G,X)) / \sim \nonumber\\
&\cong \ \pi_0(\map^G(A,\map(E G,X)(0)))  \nonumber\\
&\cong \ \pi_0(\map(E G\times_G A,  X(0))) \ \cong\
X^0(E G\times_G A)\ .\nonumber
\end{align}
Here the symbol `$\sim$' stands for the homotopy relation.
For $k>0$ we exploit that the shifted spectrum  $\sh^k X$
is again an $\Omega$-spectrum. 
Proposition 3.1.25 of \cite{schwede:global} provides a
$\pi_*$-isomorphism 
\[ \lambda_{\map(E G,X)}^k\ :\ \map(E G,X)\sm S^k\ \to\ \sh^k(\map(E G,X))
\ = \  \map(E G,\sh^k X)  \]
which induces a natural isomorphism
\begin{align*}
  \map(E G,X)^k_G(A)\ &= \ [\Sigma^\infty_+ A, \map(E G,X)\sm S^k]^G \\
                      &\iso \ [\Sigma^\infty_+ A, \map(E G,\sh^k X)]^G \\ 
                      &= \  \map(E G,\sh^k X)^0_G(A)  \\
  _\eqref{eq:Borel iso}  &\iso  \ (\sh^k X)^0_G(E G\times_G A)  
= \  X^k_G(E G\times_G A) \ .
\end{align*}
Similarly, the looped spectrum  $\Omega^k X$
is another $\Omega$-spectrum. So we get natural isomorphisms
\begin{align*}
  \map(E G,X)^{-k}_G(A)\ &= \ [\Sigma^\infty_+ A, \Omega^k \map(E G,X)]^G \\
  &\iso \ [\Sigma^\infty_+ A, \map(E G,\Omega^k X)]^G \\ 
&= \  \map(E G,\Omega^k X)^0_G(A)  \\ 
_\eqref{eq:Borel iso} \ &\iso \ (\Omega^k X)^0_G(E G\times_G A)  \
= \  X^{-k}_G(E G\times_G A) \ .
\end{align*}
The last claim is the special case where $A=\uEG$,
in which case the $\Com$-equivalence $\uEG\to \ast$
induces an isomorphism
\begin{align*}
 \map(E G,X)^k_G(\uEG)\ &= \ 
[\Sigma^\infty_+ \uEG, \map(E G,X)[k]]^G \\ 
&\cong \ [\Sigma^\infty_+ \mS_G, \map(E G,X)[k]]^G \ = \ \pi_{-k}^G(\map(E G,X))\ .  
\end{align*}
On the other hand, the projection
$E G\times \uEG \to E G$
is a $G$-equivariant homotopy equivalence, so the
induced map on orbits
\[ E G\times_G \uEG \ \to \ E G/G \ = \ B G \]
is a homotopy equivalence.
\end{proof}

\begin{eg}[Borel cohomology is global]\label{eg:Borel global}
In Proposition \ref{prop:global homotopy of b X} we showed that the Borel cohomology
theory associated with a (non-equivariant) cohomology theory is
represented by an orthogonal $G$-spectrum. We will now argue that the
Borel cohomology theories are in fact `global', i.e., can be represented
by an orthogonal $G$-spectrum with trivial $G$-action.
The global version of the Borel construction actually models the right
adjoint to the forgetful functor $U:\GH\to\Ho(\Sp)$ from the global
to the non-equivariant stable homotopy category, see \cite[Prop.\,4.5.22]{schwede:global}.
The following construction
is taken from~\cite[Con.\,4.5.21]{schwede:global}.

We start with an orthogonal spectrum~$X$ (in the non-equivariant sense)
that represents a (non-equivariant) cohomology theory $X^*(-)$. 
We define a new orthogonal spectrum $b X$ as follows.
For an inner product space $V$, we write
$\bL(V,\mR^\infty)$ for the contractible space of linear isometric embeddings
of $V$ into $\mR^\infty$. We set
\[ (b X)(V)\ = \ \map(\bL(V,\mR^\infty),X(V)) \ ,   \]
the space of all continuous maps from $\bL(V,\mR^\infty)$ to $X(V)$.
The orthogonal group $O(V)$ acts on this mapping space by conjugation, through
its actions on $\bL(V,\mR^\infty)$ and on $X(V)$.
We define structure maps $\sigma_{V,W}:S^V\sm (b X)(W)\to (b X)(V\oplus W)$
as the composite
\begin{align*}
 S^V\sm \map(\bL(W,\mR^\infty),X(W)) \quad \xrightarrow{\text{assembly}} \qquad 
&\map(\bL(W,\mR^\infty),S^V\sm X(W)) \\ 
\xrightarrow{\map(\res_W,\sigma_{V,W}^X)}\ &\map(\bL(V\oplus W,\mR^\infty),X(V\oplus W)) 
\end{align*}
where $\res_W:\bL(V\oplus W,\mR^\infty)\to \bL(W,\mR^\infty)$
is the map that restricts an isometric embedding from $V\oplus W$ to~$W$.
\end{eg}

The endofunctor $b$ on the category of orthogonal spectra 
comes with a natural transformation
\[ i_X \ : \ X \ \to \ b X  \]
whose value at an inner product space $V$ sends a point
$x\in X(V)$ to the constant map $\bL(V,\mR^\infty)\to X(V)$ with value $x$.
With the help of the morphism $i_X$ we can now compare the spectrum $b X$
to the $G$-spectrum $\map(E G,X)$ defined in Example~\ref{eg:Borel_cohomology} 
via the two natural morphisms of orthogonal $G$-spectra
\begin{equation}\label{eq:compare_Borel}
 \map(E G,X)  \ \xrightarrow{\map(E G,i_X)} \ 
 \map(E G,b X) \ \xleftarrow{\ \text{const}\ }\ b X \ .  
\end{equation}
Both maps are morphism of orthogonal $G$-spectra, where $b X$ is
endowed with trivial $G$-action.

\begin{prop}
  For every orthogonal $\Omega$-spectrum~$X$ the two morphisms
  \eqref{eq:compare_Borel} are $\pi_*$-isomorphisms of orthogonal $G$-spectra.
  So the orthogonal spectrum $b X$, endowed with trivial $G$-action,
  represents the Borel $G$-cohomology theory associated with $X$.
\end{prop}
\begin{proof} 
  Since the space $\bL(V,\mR^\infty)$ is contractible,
  the morphism $i_X:X\to b X$ is a non-equivariant level equivalence.
  So applying $\map(E G,-)$ takes it to a level equivalence
  of orthogonal $G$-spectra. Since level equivalences are in particular
  $\pi_*$-isomorphisms, this takes care of the morphism $\map(E G,i_X)$.

For the second morphism we consider a compact subgroup~$H$ of~$G$
and a {\em faithful} $H$-representation~$V$.
Then the $H$-space $\bL(V,\mR^\infty)$ is cofibrant, non-equivariantly
contractible
and has a free $H$-action, i.e., it is a model for $E H$,
see for example \cite[Prop.\,1.1.26]{schwede:global}.
The underlying $H$-space of $E G$ is also a model for $E H$,
so the projection
\[ \proj\ : \ E G\times \bL(V,\mR^\infty) \ \to \ \bL(V,\mR^\infty)\]
is an $H$-equivariant homotopy equivalence.
The induced map
\begin{align*}
  \map^H(\proj,X(V))\ : \ 
  (b X(V))^H \ &= \ \map^H(\bL(V,\mR^\infty), X(V))\\
  &\to \  \map^H(E G\times \bL(V,\mR^\infty),X(V)) 
\end{align*}
is thus a homotopy equivalence.
The target of this map is isomorphic to the space
$\map^H(E G,b X)(V)$, and under this isomorphism the map 
$\map^H(\proj,X(V))$ becomes the `constant' map.
So the second morphism is a level equivalence of orthogonal $H$-spectra
when restricted to {\em faithful} $H$-representations.
Since faithful representations are cofinal in all $H$-representations,
the second morphism induces an isomorphism of $H$-homotopy groups.
Since $H$ is any compact subgroups of $G$, this shows that the
second morphism is a $\pi_*$-isomorphism of orthogonal $G$-spectra.
\end{proof}
\index{Borel cohomology|)}

\section{Equivariant K-theory}
\label{sec:equiv K}

\index{equivariant K-theory|(}

In this final section we show that for discrete groups,
the equivariant K-theory defined from $G$-vector bundles
is representable by the global equivariant K-theory spectrum $\bKU$.
Before going into more details, we give a brief overview of the main players.
In \cite[Thm.\,3.2]{luck-oliver:completion}, the third author and Oliver show that the
classical way to construct equivariant K-theory still works for
discrete groups $G$ and finite proper $G$-CW-complexes $X$:
the group $K_G(X)$, defined as the Grothendieck group of $G$-vector bundles
over $X$, is excisive in $X$. Moreover, the functor $K_G(X)$ is Bott periodic,
and can thus be extended to a $\mZ$-graded theory.
Also, the theory $K_G^*(X)$ supports Thom isomorphisms for hermitian $G$-vector bundles
\cite[Thm.\,3.14]{luck-oliver:completion},
and it satisfies a version of the Atiyah-Segal completion theorem, 
see \cite[Thm.\,4.4]{luck-oliver:completion}.

In \cite[Def.\,3.6]{joachim:higher_coherence}, Joachim introduced a model $\bKU$
for periodic global K-theory 
that is based on spaces of homomorphisms of $\mathbb{Z}/2$-graded $C^*$-algebras,
see also \cite[Sec.\,6.4]{schwede:global}.
This is a commutative orthogonal ring spectrum,
and the underlying orthogonal $G$-spectrum $\bKU_G$ represents
equivariant K-theory for all compact Lie groups, 
see \cite[Thm.\,4.4]{joachim:higher_coherence} or
\cite[Cor.\,6.4.13]{schwede:global}.
For a general Lie group $G$, the underlying $G$-spectrum $\bKU_G$ represents
a proper $G$-cohomology theory by Theorem \ref{thm:represented equals bundle}.

The purpose of this final section is to establish a natural isomorphism between
the two proper $G$-cohomology theories $K_G^*$ and $\bKU_G^*$ for discrete groups $G$,
stated in Theorem \ref{thm:repr K-theory}.
Since the two theories are defined very differently,
the main issue is to construct a natural transformation of cohomology
theories in one direction, and most of our work goes into this.
Our strategy is to first compare the excisive functors $K_G(X)$ and $\bKU_G(X)$,
see Theorem \ref{thm:0 k-theory iso}.
This comparison passes through the theory $\bku_G\gh{X}$,
represented by the {\em connective global K-theory spectrum} $\bku$ in the sense of
the fifth author \cite[Con.\,6.3.9]{schwede:global},
a variation of Segal's configuration space model
for K-homology \cite{segal:Khomology}.
We use $\bku_G\gh{X}$ as a convenient target of an explicit map
$\td{-}: \Vect_G(X)\to\bku_G\gh{X}$ that turns
 a $G$-vector bundle into an equivariant homotopy class, see \eqref{eq:bundle2ku}.
The fact that the map $\td{-}$ is additive
and multiplicative is not entirely obvious, and this verification involves some work,
see Propositions \ref{prop:ample bundle} and \ref{prop:kappa multiplicative}.
The connective and periodic global K-theory spectra are related by
a homomorphism $j:\bku\to\bKU$ of commutative orthogonal ring spectra,
defined in \cite[Con.\,6.4.13]{schwede:global}. The induced transformation of
excisive functors $j_*:\bku_G\gh{X}\to\bKU_G\gh{X}$ is then additive and multiplicative.
Because the isomorphism $K_G(X)\iso\bKU_G\gh{X}$ is suitably multiplicative,
it matches the two incarnations of Bott periodicity, and can thus be extended to the
$\mZ$-graded periodic theories in a relatively formal (but somewhat tedious) way,
see the final Theorem \ref{thm:repr K-theory}.

The main results in this section require the group $G$ to be discrete
(as opposed to allowing general Lie groups).
The restriction arises from the vector bundle side of the story:
Example 5.2 of \cite{luck-oliver:completion} shows that
the theory made from isomorphism classes of $G$-vector bundles is
{\em not} in general excisive in the context of non-discrete Lie groups.
We think of the represented theory $\bKU_G^*(X)$ as the `correct'
equivariant K-theory in general; indeed, it restricts to
equivariant K-theory for compact Lie groups by
\cite[Thm.\,4.4]{joachim:higher_coherence} or \cite[Cor.\,6.4.23]{schwede:global},
and the point of this book is precisely that proper $G$-cohomology theories represented
by orthogonal $G$-spectra always have the desired formal properties,
for all Lie groups.
Phillips \cite{phillips:equivariant_proper} has defined equivariant K-theory
for second countable locally compact topological groups $G$,
defined on proper locally compact $G$-spaces. His construction
is based on Hilbert $G$-bundles instead of finite-dimensional $G$-vector bundles.
It seems plausible that in the common realm of Lie groups, our represented theory 
$\bKU_G^*(X)$ ought to be isomorphic to Phillips' theory, but we have not attempted
to define an isomorphism.

\begin{con}[Connective global K-theory]\index{connective global K-theory}\index{global K-theory!connective}
We recall the definition of the orthogonal spectrum $\bku$
from \cite[Sec.\,6.3]{schwede:global}. 
For a real inner product space~$V$, we let $V_\mC=\mC\otimes_\mR V$
denote the complexification which inherits a unique hermitian inner product
$(-,-)$ characterized by
\[ (1\otimes v,1\otimes w)\ = \ \langle v,w\rangle \]
for all $v,w\in V$.
The symmetric algebra $\Sym(V_\mC)$ of the complexification inherits
a preferred hermitian inner product in such a way that the canonical algebra isomorphism
\[ \Sym(V_\mC)\otimes_\mC \Sym(W_\mC) \ \iso \
\Sym( (V\oplus W)_\mC) \]
becomes an isometry, compare \cite[Prop.\,6.3.8]{schwede:global}. 
The $V$-th space of the orthogonal spectrum $\bku$ is the
value on $S^V$ of the $\Gamma$-space of finite-dimensional, pairwise orthogonal subspaces
of $\Sym(V_\mC)$.
More explicitly, we define the $V$-th space of the orthogonal spectrum~$\bku$ as
the quotient space
\[ \bku(V) = \Big(  
\coprod_{m\geq 0} \mathrm{Gr}_{\langle m \rangle}(\Sym(V_\mC)) \times (S^V)^m 
  \Big) \Big /{\sim_V}   \]
where $S^V$ is the one-point compactification of $V$ 
and $\mathrm{Gr}_{\langle m \rangle}(\Sym(V_\mC))$ is the space of $m$-tuples 
of pairwise orthogonal subspaces in $\Sym(V_\mC)$.
Here, the equivalence relation $\sim_V$  makes the following identifications
\begin{itemize}
\item[(i)] The tuple $(E_1,\ldots,E_m;\,v_1,\ldots,v_m)$, 
where $(E_1,\ldots,E_m)$ is an $m$-tuple of pairwise orthogonal subspaces 
of $\Sym(V_\mC)$ and $v_1,\ldots,v_m \in S^V$, 
is identified with $(E_{\sigma(1)},\ldots,E_{\sigma(m)};\,v_{\sigma(1)},\ldots,v_{\sigma(m)})$ 
for any $\sigma$ in the symmetric group $\Sigma_m$. 
This implies that we can represent equivalence classes
as formal sums $\sum_{k=1}^m v_k E_k$.
\item[(ii)] If $v_i = v_j$ for some $i\ne j$, then
\[ \sum_{k=1}^m v_k E_k\ = \ v_i(E_i\oplus E_j) +\sum_{\substack{k=1 \\ k\neq i,  j}}^m v_k E_k 
\ .\]
\item[(iii)] If $E_i=0$ is the trivial subspace or
$v_i=\infty$ is the basepoint of $S^V$ at infinity, then
\[  \sum_{k=1}^m v_k E_k   = \sum_{\substack{k=1 \\ k\neq i}}^m v_k E_k\ .   \]
\end{itemize}
Hence, the topology of $\bku(V)$ is such that, informally speaking, 
the labels $E_i$ and $E_j$ add up inside $\Sym(V_\mC)$ 
whenever the two points $v_i$ and $v_j$ collide, and the label $E_i$ disappears 
when $v_i$ reaches the basepoint at infinity. 
The action of the orthogonal group $O(V)$ 
on $S^V$ and $\Sym(V_\mC)$ induces a based continuous $O(V)$-action 
on $\bku(V)$. We define, for all inner product spaces~$V$ and~$W$,
an $O(V)\times O(W)$-equivariant multiplication map
\begin{align*}
 \mu_{V,W}\ : \  \bku(V) \wedge \bku(W)\qquad &\to \qquad \bku(V\oplus W) \\ 
\Big(  \sum v_i E_i \Big) \sm \Big(  \sum w_j F_j \Big) \ &\longmapsto \ 
\sum (v_i \sm w_j)\cdot(E_i\otimes F_j)     
\end{align*}
where the canonical isometry $\Sym(V_\mC)\otimes \Sym(W_\mC)\cong \Sym((V\oplus W)_\mC)$ 
is used to interpret $E_i\otimes F_j$ as a subspace of 
$\Sym((V\oplus W)_\mC)$. 
Finally, we define  the $O(V)$-equivariant unit map 
\[  \nu_V : S^V \to  \bku(V)\ , \quad v \mapsto v \mC\ , \]
where $\mC$ refers to the `constants' in the symmetric algebra~$\Sym(V_\mC)$,
i.e., the subspace spanned by the multiplicative unit~$1$.
The maps $\{\mu_{V, W}\}$ together with the maps $\{\nu_V\}$ 
turn $\bku=\{ \bku(V)\}$ into a commutative orthogonal ring spectrum.

For a Lie group $G$,
 the \emph{connective $G$-equivariant K-theory spectrum $\bku_G$} 
 is the orthogonal spectrum $\bku$ equipped with trivial $G$-action.
It is relevant for our purposes that $\bku_G$ arises from a global stable homotopy type, 
i.e., it is obtained by the forgetful functor of Section \ref{sec:global},
applied to $\bku$.
\end{con}

\begin{con}
  For this construction, $G$ is any Lie group.\index{G-vector bundle@$G$-vector bundle|(}
  We write $\Vect_G(X)$ for the abelian monoid of isomorphism classes of
  hermitian $G$-vector bundles over a $G$-space $X$.
  We introduce a natural homomorphism of abelian monoids
  \begin{equation}  \label{eq:Vect2ku}
    \td{-}\ : \  \Vect_G(X) \to \bku_G\gh{X}  
  \end{equation}
  for any finite proper $G$-CW-complex $X$.
  The construction is a parameterized version 
  of the construction in \cite[Thm.\,6.3.31]{schwede:global},
  and proceeds as follows. 
  For a hermitian $G$-vector bundle $\xi$ over $X$, 
  we let $u\xi$ denote the underlying
  euclidean vector bundle.
  We denote by $(u \xi)_\mC$ the complexification of the latter.
  Then the maps 
  \[ j_x \ : \ \xi_x \ \to \ (u\xi_x)_\mC \ , \quad v \ \longmapsto \
    1/\sqrt{2}\cdot(1\otimes v - i\otimes (i v))\]
  are $\mC$-linear isometric embeddings of each fiber that vary 
  continuously with $x\in X$. Altogether, these define an isometric embedding
  of hermitian $G$-vector bundles
  \[ j \ : \ \xi \ \to \ (u\xi)_\mC \ .\]
  By design, the fiber of the retractive $G$-space $\bku(u\xi)$
  over $x\in X$ is $\bku(u\xi_x)$.
  So we can define a map of retractive $G$-spaces over $X$
  \[ \braced{\xi}\ : \ S^{u\xi} \ \to \ \bku(u\xi) \text{\qquad by\qquad}
    \braced{\xi}(x,v)\ = \ [j_x(\xi_x); v].\]
  In more detail: we view $j_x(\xi_x)$ as sitting in the linear summand
  in the symmetric algebra $\Sym((u\xi_x)_\mC)$,
  and $[j_x(\xi_x); v]$ as the configuration in $\bku(u\xi_x)$
  consisting of the single vector $v$ labeled by the vector space
  $j_x(\xi_x)$. 
  The point $ [j_x(\xi_x); v]\in \bku(u\xi)$ varies continuously with
  $(x, v) \in S^{u\xi}$, and altogether this defines the $G$-equivariant map $\braced{\xi}$.
  If $\psi:\xi\to\eta$ is an isomorphism of hermitian $G$-vector bundles
  over $X$, then the maps $\braced{\xi}$ and $\braced{\eta}$ are conjugate by $\psi$,
  and so they represent the same class in $\bku_G^0\gh{X}$.
  So we obtain a well-defined map
  \begin{equation}\label{eq:bundle2ku}
  \td{-}\ : \  \Vect_G(X)\ \to\ \bku_G\gh{X}  \text{\qquad by\qquad}
    \td{\xi}\ = \ [u\xi,\braced{\xi}] \ .    
  \end{equation}
  For every continuous $G$-map $f:Y\to X$ we have
  $f^*(u\xi)=u(f^*\xi)$ as euclidean $G$-vector bundles over $Y$.
  Moreover, the two maps
  \[ f^*\braced{\xi}\ : \  S^{f^*(u\xi)} \ \to \ \bku(f^*(u\xi))
    \text{\quad and\quad}
    \braced{f^*\xi}\ : \  S^{u(f^*\xi)} \ \to \ \bku(u(f^*\xi))
  \]
  are equal. So
  \[ \td{f^*[\xi]}\
    = \ [u(f^*\xi),\braced{f^*\xi}] \  = \ [f^*(u\xi),f^*\braced{\xi}]  \ = \ f^*[u\xi,\braced{\xi}]\ = \
    f^*\td{\xi} \ ,   \]
  i.e., for varying $G$-spaces $X$,
  the maps $\td{-}$ constitute a natural transformation.
\end{con}

The following proposition provides additional freedom in the passage \eqref{eq:bundle2ku}
from $G$-vector bundles to $\bku$-cohomology classes:
it lets us replace the embedding
\[ \xi \ \xrightarrow{\ j\ } \ (u \xi)_{\mC} \ \xrightarrow{\text{linear summand}}
\ \Sym( (u\xi)_\mC)\]
used in the definition of $\td{\xi}$
by any other equivariant isometric embedding of $\xi$ into
the complexified symmetric algebra of any other euclidean vector bundle.

\begin{prop}\label{prop:freedom of embedding}
  Let $G$ be a Lie group, $X$ a $G$-space and $\xi$ a hermitian $G$-vector bundle over $X$.  
  Let $\mu$ be a euclidean $G$-vector bundle over $X$ and 
  \[ J \ : \ \xi \ \to \ \Sym(\mu_\mC) \]
  a $G$-equivariant $\mC$-linear isometric embedding.
  We define a map of retractive $G$-spaces
  \[ \lambda(J)\ : \ S^{\mu}\ \to \ \bku(\mu)\text{\qquad by\qquad}
  \lambda(J)(x,v)\ = \ [ J_x(\xi_x), v]\ .\]
  Then $\td{\xi}$ coincides with the class of $(\mu,\lambda(J))$.
\end{prop}
\begin{proof}
  The two composites around the (non-commutative!) square
  \[ \xymatrix@C=15mm{
      \xi \ar[r]^-J\ar[d]_{i\circ j^\xi} &\Sym(\mu_\mC) \ar[d]^{\Sym(i_2)} \\
      \Sym((u\xi)_\mC) \ar[r]_-{\Sym(i_1)} &\Sym((u\xi)_\mC\oplus \mu_\mC)    } \]
  are $G$-equivariant isometric embeddings whose images are
  orthogonal inside the hermitian vector bundle $\Sym((u\xi)_\mC\oplus \mu_\mC)$.
  Here $i:(u\xi)_\mC\to\Sym((u\xi)_\mC)$ is the embedding of the linear summand.
  The diagram thus commutes up to equivariant homotopy of linear embeddings,
  fiberwise given by the formula
  \[ H(t,v)\ = \ t \cdot (\Sym(i_1)\circ i\circ j^{\xi})(v)\ + \
  \sqrt{1-t^2}\cdot  (\Sym(i_2)\circ J)(v) \ .\]
  Such a homotopy induces an equivariant homotopy of maps of retractive $G$-spaces
  over $X$ between $\lambda(\Sym(i_1)\circ i\circ j^\xi)$ and  $\lambda(\Sym(i_2)\circ J)$.
  Hence
  \begin{align*}
    \td{\xi}\
    = \ [u\xi,\braced{\xi}] \ = \ [u\xi,\lambda(i\circ j^\xi)] \
    & = \ [u\xi\oplus\mu,\lambda(i_1\circ i\circ j^\xi)] \\
    & = \ [u\xi\oplus\mu,\lambda(i_2\circ J)] \ = \ [\mu,\lambda(J)] \ .\qedhere
  \end{align*}
\end{proof}

Our next aim is to establish additivity of the map
$\td{-}: \Vect_G(X) \to \bku_G\gh{X}$ defined in \eqref{eq:Vect2ku}.  
This relation is slightly subtle because the sum in
$\bku_G\gh{X}$ is defined by addition via a fiberwise pinch map,
which a priori has no connection to Whitney sum of vector bundles.
The concept of `ample bundle' we are about to introduce will serve
as a tool in the proof of the additivity relation.

\begin{defn}
  Let $G$ be a Lie group and $X$ a proper $G$-space.
  A hermitian $G$-vector bundle $\zeta$ over $X$ is {\em ample}
  if the following property holds:
  for every point $x\in X$ the infinite-dimensional
  unitary $G$-representation $\Sym(\zeta_x)$ is a complete complex $G_x$-universe,
  where $G_x$ is the stabilizer group of $x$.
  In other words, every finite-dimensional unitary $G_x$-representation
  admits a $G_x$-equivariant
  linear isometric embedding into $\Sym(\zeta_x)$.
\end{defn}

\begin{prop}\label{prop:universal bundle} 
  Let $G$ be a Lie group and $X$ a finite proper $G$-CW-complex.
  \begin{enumerate}[\em (i)]
  \item Let $\zeta$ be an ample hermitian $G$-vector bundle over $X$.
    Let $\xi$ be a hermitian $G$-vector bundle over $X$, of finite or countably infinite dimension.
    Then there is a $G$-equivariant linear isometric embedding of $\xi$ into $\Sym(\zeta)$ over $X$.
  \item If $G$ is discrete, then $X$ has an ample $G$-vector bundle. 
  \end{enumerate}
\end{prop}
\begin{proof}
  (i)
  We prove a more general relative version of the claim: given a $G$-subcomplex $A$ of $X$,
  every $G$-equivariant linear isometric embedding of $\xi|_A$ into $\Sym(\zeta)|_A$ over $A$
  can be extended to a $G$-equivariant linear isometric embedding of $\xi$ into $\Sym(\zeta)$ over $X$.
  The case $A=\emptyset$ then proves the proposition.

  Induction over the number of relative $G$-cells reduces the claim to
  the case where $X$ is obtained from $A$ by attaching a single $G$-cell with compact isotropy group $H$.
  Hence we may assume that $X=G/H\times D^n$ and $A=G/H\times S^{n-1}$, for some $n\geq 0$.
  Since $H$ is a compact Lie group, every hermitian $G$-vector bundle $\zeta$
  over $G/H\times D^n$
  is of the form $\zeta=(G\times_H W)\times D^n$, for some unitary $H$-representation $W$,
  projecting away from $W$, compare \cite[Lemma 1.1 (a)]{luck-oliver:completion} or
  \cite[Prop.\,1.3]{segal:equivariant_K}.
  Since $\zeta$ is an ample bundle, $W$ must be an ample $H$-representation,
  i.e., the symmetric algebra $\Sym(W)$ is a complete complex $H$-universe.
  Similarly, we may assume that $\xi=(G\times_H V)\times D^n$
  where $V$ is a unitary $H$-representation of finite or countably infinite dimension.
  
  Every linear isometric embedding of the bundle $\xi|_{S^{n-1}}=(G\times_H V)\times S^{n-1}$
  into the bundle $\Sym(\zeta|_{S^{n-1}})=(G\times_H \Sym(W))\times S^{n-1}$ 
  is of the form
 \begin{align*}
   (G\times_H V)\times S^{n-1}\ &\to \ (G\times_H \Sym(W))\times S^{n-1}\\
   ([g,v],x) \quad &\longmapsto \quad ([g,\psi(x)(v)],x)
 \end{align*}
 for some continuous map $\psi:S^{n-1}\to\bL^H(V,\Sym(W))$ into
 the space of $H$-equivariant $\mC$-linear isometric embeddings from $V$ into $\Sym(W)$.
 Because $\Sym(W)$ is a complete complex $H$-universe, the space
 $\bL^H(V,\Sym(W))$ is weakly contractible: when $V$ is finite-dimensional, this is
 the complex analog of \cite[Prop.\,1.1.21]{schwede:global};
 Proposition A.10 of \cite{schwede:orbispaces} (or rather its complex analog)
 reduces the infinite-dimensional case to the finite-dimensional case.
 Because $\bL^H(V,\Sym(W))$ is weakly contractible, $\psi$ admits a continuous extension to a map
 $D^n\to\bL^H(V,\Sym(W))$, which yields the desired linear isometric embedding
 $\xi\to\Sym(\zeta)$ by the same formula as for $\psi$.

 (ii) We let $X$ be a finite proper $G$-CW-complex. Since $X$ has
  only finitely many $G$-cells, there are only finitely many
  conjugacy classes of finite subgroups of $G$ that occur as isotropy groups
  of points of $X$. In particular, the isotropy groups of $X$ have bounded order.
  Since $G$ is discrete, Corollary 2.7 of \cite{luck-oliver:completion} thus provides
  a hermitian $G$-vector bundle $\zeta$ over $X$ such that for every point
  $x\in X$, the fiber $\zeta_x$ is a multiple of the regular representation
  of the isotropy group $G_x$. In particular, the $G_x$-action on $\zeta_x$ is
  faithful, and hence $\Sym(\zeta_x)$ is a complete complex $G_x$-universe,
  compare \cite[Rk.\,6.3.22]{schwede:global}. So the bundle $\zeta$ is ample.
\end{proof}

\begin{prop}\label{prop:ample bundle}
  Let $G$ be a discrete group and $X$ a finite proper $G$-CW-complex.
  Then the map  $\td{-}: \Vect_G(X)\to\bku_G\gh{X}$ defined in \eqref{eq:bundle2ku}
  is additive.
\end{prop}
\begin{proof}
  The sum in the group $\bku_G\gh{X}$ is defined by addition via a fiberwise
  pinch map.
  We will relate the pinch sum to a different binary operation, the `bundle sum', by an
  Eckmann-Hilton argument.

  Proposition \ref{prop:universal bundle} (ii) provides an ample hermitian $G$-vector bundle
  $\zeta$ over $X$. We let $\mu=u\zeta$ denote the underlying euclidean vector bundle.
  Because $\zeta$ embeds into $(u\zeta)_\mC=\mu_\mC$, the  hermitian $G$-vector bundle
  $\mu_\mC$ is also ample.
  By adding a trivial complex line bundle to $\zeta$, if necessary,
  we can moreover assume that
  there exists a $G$-equivariant linear isometric embedding $j:X\times \underline{\mR}\to \mu=u\zeta$
  of the trivial $\mR$-line bundle over $X$.
  The embedding $j$ parameterizes a trivial 1-dimensional summand in $\mu$,
  and hence a pinch map $p:S^\mu\to S^\mu\vee_X S^\mu$.
  The `pinch sum' on the set  $[S^{\mu},\bku_G(\mu)]^G_X$
  of parameterized homotopy classes is given by
  \begin{equation} \label{eq:define_pinch_sum}
     [f]\vee[g]\ = \ [(f+g)\circ p]\ , 
  \end{equation}
  where $f+g:S^\mu\vee_X S^\mu\to\bku_G(\mu)$ is given by $f$ and $g$ on the respective
  summands.

  Because $\mu_\mC$ is ample, Proposition \ref{prop:universal bundle} (i)
  provides a $G$-equivariant linear isometric embedding
  \[    \theta\ :\  \Sym(\mu_{\mC})\oplus \Sym(\mu_{\mC})\ \to\ \Sym(\mu_{\mC})  \]
  of bundles over $X$.
  The embedding $\theta$ in turn yields a map of retractive $G$-spaces over $X$
  \[  \oplus\ :\ \bku_G(\mu) \times_X \bku_G(\mu)\ \to \ \bku_G(\mu) \]
  defined fiberwise by 
  \begin{align*} 
    [E_1&,\ldots,E_k; v_1,\ldots,v_k]\oplus[F_1,\ldots, F_l; w_1,\ldots,w_l]  \\
        &= [\theta(E_1\oplus 0),\ldots ,\theta(E_k\oplus 0),
          \theta(0 \oplus F_1),\ldots ,\theta(0\oplus F_l); v_1,\ldots,v_k,w_1,\ldots,w_l] \\
        &=    \sum_{i=1}^k v_i \theta(E_i \oplus 0) + \sum_{j=1}^l w_j \theta(0\oplus F_j)\ .  
  \end{align*}
  The `bundle sum' on the set  $[S^{\mu},\bku_G(\mu)]^G_X$ is given by
  \begin{equation} \label{eq:define_bundle_sum}
    [f]\oplus[g]\ = \ [\oplus\circ(f,g)]\ , 
  \end{equation}
  where $(f,g):S^\mu\to\bku_G(\mu)\times_X\bku_G(\mu)$ has components $f$ and $g$, respectively.
  The pinch sum \eqref{eq:define_pinch_sum} and the bundle sum \eqref{eq:define_bundle_sum}
  share the same neutral element, and they satisfy the interchange relation
  \[      ( [f]\oplus [g])\vee ([h]\oplus[k]) \  = ( [f]\vee [h])\oplus ([g]\vee[k])  \ .\]
  The Eckmann-Hilton argument then applies: taking $[g]$ and $[h]$
  as the common neutral element shows that the pinch sum and the bundle sum
  on the set $[S^{\mu},\bku_G(\mu)]^G_X$ coincide.

  Now we prove additivity.
  We let $\xi$ and $\eta$ be two hermitian $G$-vector bundles over $X$.
  Proposition \ref{prop:universal bundle} (i)
  provides $G$-equivariant linear isometric embeddings 
  \[ \varphi\ :\ \xi\ \to\ \Sym(\mu_\mC)\text{\qquad and\qquad}
    \psi\ :\ \eta\ \to\ \Sym(\mu_\mC) \]
  of hermitian $G$-vector bundles over $X$.
  The map
  \[ \theta\circ(\varphi\oplus\psi)\ : \ \xi\oplus\eta \ \to \ \Sym(\mu_\mC) \]
  is another equivariant isometric embedding.
  Proposition \ref{prop:freedom of embedding} turns these bundle embeddings
  into maps of retractive $G$-spaces over $X$
  \[ \lambda(\varphi)\ , \ \lambda(\psi)\ , \ \lambda(\theta\circ(\varphi\oplus\psi))\ : \
    S^\mu \ \to \ \bku_G(\mu)\ .\]
  Moreover, the relation
  \[ \lambda(\theta\circ(\varphi\oplus\psi))\ = \  \oplus\circ(\lambda(\varphi),\lambda(\psi)) \]
  holds by design.
Proposition \ref{prop:freedom of embedding} then yields
  \begin{align*}
    \td{\xi}+\td{\eta}\
    &= \  [\lambda(\varphi)]\vee[\lambda(\psi)] \
    = \  [\lambda(\varphi)]\oplus[\lambda(\psi)] \\
    &= \  [\oplus\circ(\lambda(\varphi),\lambda(\psi))] \
    = \  [\lambda(\theta\circ(\varphi\oplus\psi))] \
    = \ \td{\xi\oplus\eta}\ .\qedhere
  \end{align*}
\end{proof}

We have now constructed a well-defined monoid homomorphism
\[\td{-}\ :\ \Vect_G(X)\ \to\ \bku_G\gh{X}\ .\] 
We write $K_G(X)$ for the group completion (Grothendieck group)
of the abelian monoid $\Vect_G(X)$. The universal property of the Grothendieck group
extends $\td{-}$ to a unique group homomorphism
\begin{equation}  \label{eq:K_G2ku}
    \kappa_X\ : \ K_G(X) \ \to \ \bku_G\gh{X}\ .  
\end{equation}
Since the maps $\td{-}$ are natural for $G$-maps in $X$,
so are the extensions $\kappa_X$.

\medskip

Equivariant K-groups admit products induced from tensor product of vector bundles.
The cohomology groups represented by $\bku$
admit products arising from the ring spectrum structure.
Our next aim is to show that the homomorphisms \eqref{eq:K_G2ku}
are suitably multiplicative.
We start by formally introducing the relevant pairings in the represented $\bku$-cohomology,
in somewhat larger generality.

\begin{con}
We let $E$ be an orthogonal ring spectrum and $M$ a left $E$-module spectrum.
Given a Lie group $G$, a finite proper $G$-CW-complex $X$
and a finite CW-complex $Y$, we now construct natural pairings 
\begin{equation}\label{eq:pairing}
\cup \ : \  E_G\gh{X} \times M\gh{Y}\ \to \ M_G\gh{ X\times Y}\ ;  
\end{equation}
here $M\gh{Y}$ is the {\em non-equivariant}
special case of the Construction \ref{con:define alternative},
i.e., for $G$ a trivial group. 
We write
\[ \alpha_{m,n}\ :\ E(\mR^m) \sm M(\mR^n)\ \to\ M(\mR^{m+n})\]
for the $(O(m)\times O(n))$-equivariant component of the 
action morphism $\alpha:E\sm M\to M$.
We let $\eta$ and $\xi$ be vector bundles over $X$ and $Y$,
of dimension $m$ and $n$, respectively.
We write $\eta\times\xi$ for the exterior product bundle over $X\times Y$,
and $-\triangle-$ for the external smash product of retractive spaces, with fibers
\[ (\eta\times\xi)_{(x,y)}\ = \ \eta_x\sm \xi_y\ .\]
Then the multiplication maps give rise to a map of retractive
spaces over $X\times Y$
\[ \xymatrix{ E_G(\eta) \triangle M(\xi) = 
    \Big( \cF_{n}(\eta) \times_{O(m)} E(\mR^m) \Big) \triangle
    \Big( \cF_n(\xi) \times_{O(n)} M(\mR^n) \Big) \ar[d]^{\cong} \\
\Big ( \cF_m(\eta)\times \cF_n(\xi)  \Big) \times_{O(m)\times O(n)} \Big(E(\mR^m)\wedge M(\mR^n) \Big)
\ar[d]^{\psi \times \alpha_{m,n}} \\
 \cF_{m+n}(\eta\times \xi) \times_{O(m+n)} M(\mR^{m+n})= E_G(\eta \times \xi)}\]
that we denote by $\alpha_{\eta,\xi}$.
The map $\psi:\cF_m(\eta)\times \cF_n(\xi)\to \cF_{m+n}(\eta\times \xi)$
takes direct products of frames, i.e., it is given by
\[ \psi((x_1,\dots,x_m),(y_1,\dots,y_n))\ = \ ((x_1,0),\dots,(x_m,0),(0,y_1),\dots,(0,y_n)) \ .\]
Now we let $(\eta,u)$ represent a class in $E_G\gh{X}$,
and we let $(\xi,v)$ represent a class in $M\gh{Y}$.
Then the composite 
\[ u\cup v \ : \ S^{\eta\times \xi} \ \iso \ S^{\eta}  \triangle S^{\xi} \ 
\xrightarrow{u\triangle v}\ 
 E_G(\eta) \triangle M(\xi) \ \xrightarrow{\alpha_{\eta,\xi}} \ M_G(\eta\times\xi)  \]
is an equivariant map of retractive $G$-spaces over $X\times Y$,
where $G$ acts trivially on $Y$ and on $\xi$.
The construction passes to equivalence classes under fiberwise homotopy
and stabilization, so we can define the pairing \eqref{eq:pairing} by
\[  [\eta,u]\cup [\xi,v]\ = \ [\eta\times\xi, u\cup v]\ .   \]
\end{con}

The following naturality properties of the cup product construction are straightforward from
the definitions, and we omit the formal proofs.

\begin{prop}\label{prop:cup properties}
  Let $E$ be an orthogonal ring spectrum and $G$ a Lie group.
  The pairing \eqref{eq:pairing} is natural 
  for morphisms in the global homotopy category of $E$-module spectra in the variable $M$,
  for continuous $G$-maps in $X$, and for continuous maps in $Y$.
\end{prop}

The cases we mostly care about are the orthogonal ring spectra $\bku$
and $\bKU$, each acting on itself by multiplication.
We can now state and prove the multiplicativity property of the
homomorphisms \eqref{eq:K_G2ku}.
In the next proposition, the upper horizontal pair is induced by
exterior tensor product of vector bundles.

\begin{prop}\label{prop:kappa multiplicative}
  For every discrete group $G$, every proper finite $G$-CW-complex $X$
  and every finite CW-complex $Y$, the diagram
  \[ \xymatrix@C=15mm{
      K_G(X)\times K(Y) \ar[r]^-{\tensor}\ar[d]_{\kappa_X\times \kappa_Y}& 
      K_G(X\times Y) \ar[d]^{\kappa_{X \times Y}} \\
      \bku_G\gh{X} \times \bku\gh{Y} \ar[r]_-{\cup} & \bku_G\gh{X\times Y} }\] 
  commutes.
\end{prop}
\begin{proof} 
Since both pairings are biadditive, it suffices to check the
commutativity for classes represented by actual vector bundles
(as opposed to virtual vector bundles).
We let $\eta$ be a hermitian $G$-vector bundle over $X$,
and $\xi$ a hermitian vector bundle over $Y$.
The class $\td{\eta}\cup\td{\xi}$ is represented
by the map of retractive $G$-spaces over $X\times Y$
\begin{align*}
  S^{u(\eta)\times u(\xi)}\
  &\iso \ S^{u\eta}\triangle S^{u\xi}\ \xrightarrow{\braced{\eta}\triangle\braced{\xi}} \
    \bku(u\eta)\triangle \bku(u\xi)\
  \xrightarrow{\mu_{\eta,\xi}} \ \bku( (u \eta)\times (u\xi))\ .
\end{align*}
The multiplication in $\bku$ ultimately stems from the tensor product of
hermitian vector spaces. Unraveling the definition of $\mu_{\eta,\xi}$
shows that the above composite coincides with the map of retractive $G$-spaces
$\lambda(J):S^{ (u\eta)\times (u\xi)}\to \bku( (u\eta)\times (u\xi))$,
associated with the isometric embedding
\[ J \ : \ u(\eta\tensor \xi)\ \to \ \Sym( ((u\eta)\times (u\xi))_\mC) \  ,
\quad J_{(x,y)}(v\tensor w)\ = \ (j^\eta_x(v),0)\cdot (0,j^\xi_y(w)) \ .\]
So the image of $J$ belongs to $\Sym^2(u(\eta)\times u(\xi))$,
the quadratic summand in the complexified symmetric algebra
of $u(\eta)\times u(\xi)$.
We emphasize that $J$ is {\em not} the isometric embedding used in the definition
of the class $\td{\eta\tensor\xi}$: the defining isometric embedding 
takes values in the linear summand of the symmetric algebra
of the exterior tensor product $u(\eta\tensor\xi)$.
However, Proposition \ref{prop:freedom of embedding}
shows that the map $\lambda(J)$ also represents the class
$\td{\eta\tensor\xi}$. So we conclude that
\[ \td{\eta}\cup\td{\xi}\ = \ \td{\lambda(J)}\ = \ \td{\eta\tensor\xi}\ . \qedhere\]
\end{proof}
\index{G-vector bundle@$G$-vector bundle|)}

Now we consider the periodic global K-theory spectrum $\bKU$
introduced by Joachim in \cite{joachim:higher_coherence}
and later studied in \cite[Sec.\,6.4]{schwede:global}.
The definition of $\bKU$ is based on spaces of homomorphisms of
$\mZ/2$-graded $C^*$-algebras, and can be found in 
\cite[Sec.\,4]{joachim:higher_coherence} and
\cite[Con.\,6.4.9]{schwede:global}.\index{periodic global K-theory}\index{global K-theory!periodic}
For our purposes, we can (and will) use $\bKU$ as a black box;
the main properties we use is that $\bKU$ is a commutative orthogonal ring spectrum,
that it receives a ring spectrum homomorphism $j:\bku\to\bKU$
(see \cite[Con.\,6.4.13]{schwede:global}),
that the homomorphism $j$ sends the Bott class in $\pi_2^e(\bku)$
to a unit in the graded ring $\pi_*^e(\bKU)$
(see \cite[Thm.\,6.4.29]{schwede:global})
and that $\bKU$ represents equivariant K-theory for compact Lie groups
(see \cite[Thm.\,4.4]{joachim:higher_coherence} or \cite[Cor.\,6.4.23]{schwede:global}).

Now we let $G$ be a discrete group.
The functor $(-)_G$ from Section \ref{sec:global}
yields a $G$-equivariant commutative orthogonal ring spectrum $\bKU_G$.
The morphism of commutative orthogonal ring spectra $j: \bku \to \bKU$
defined in \cite[Con.\,6.4.13]{schwede:global} induces a morphism
of commutative orthogonal $G$-ring spectra $j_G: \bku_G \to \bKU_G$.
For $X$ a finite proper $G$-CW-complex, we write $c_X:K_G(X)\to\bKU_G\gh{X}$ for the composite
\[
  K_G(X) \ \xrightarrow{\kappa_X} \ \bku_G\gh{X}\ \xrightarrow{(j_G)_*} \
  \bKU_G\gh{X} \ .\]
Source and target of this natural transformation are excisive functors in $X$
by \cite[Lemma 3.8]{luck-oliver:completion}
and by Theorem \ref{thm:excisive - bundle}, respectively.

\begin{thm} \label{thm:0 k-theory iso}
  For every discrete group $G$
  and every finite proper $G$-CW-complex $X$,
  the homomorphism
  \[ c_X\ : \ K_G(X)\ \to \ \bKU_G\gh{X} \]
  is an isomorphism.
\end{thm}
\begin{proof}
  In the special case when the group $G$ is finite,
  the map $c_X$ factors as the composite
  \[  K_G(X)\ \xrightarrow{\ \iso\ }\ \bKU_G(X) \ \xrightarrow{\ \mu_X^{\bKU_G}}\ \bKU_G\gh{X}\ ,\]
  where the first map is the isomorphism established in \cite[Cor.\,6.4.23]{schwede:global},
  and the second map is the isomorphism constructed in Theorem \ref{thm:rep2vect}.
  This proves the claim for finite groups.

  Now we let $G$ be any discrete group, and we consider
  the special case $X=G/H \times K$ for a finite subgroup $H$ of $G$
  and a finite non-equivariant CW-complex $K$.
  Lemma 3.4 of \cite{luck-oliver:completion} and our Example \ref{eg:induction iso}
  show that vertical induction maps in the commutative square
  \[
    \xymatrix@C=22mm{
      K_H(K) \ar[d]_{\cong}^\ind \ar[r]^-{c_K} &  \bKU_H\gh{K} \ar[d]^{\cong}_\ind\\
      K_G(G/H \times K) \ar[r]_-{c_{G/H \times K}} &  \bKU_G\gh{G/H \times K}
         }  \]
  are isomorphisms.
  Hence $c_{G/H \times K}$ is an isomorphism,
  and Proposition \ref{prop:excisive on orbits} concludes the proof.
\end{proof}

The rest of this section is devoted to extending the natural isomorphism 
$c_X :  K_G(X)\to \bKU_G\gh{X}$
to an isomorphism of $\mZ$-graded proper $G$-cohomology theories,
using different incarnations of Bott periodicity for source and
target of $c_X$.

\begin{con}
In \cite[Sec.\,3]{luck-oliver:completion}, the third author and Oliver
use Bott periodicity to extend the excisive functor $K_G(X)$ to
a $\mZ$-graded proper $G$-cohomology theory. We quickly recall the relevant definitions.
We let $L$ and $\underline{\mC}$ denote the tautological line bundle
and the trivial line bundle, respectively, over the complex projective line $\mathbb{C}P^1$.
Their formal difference is a reduced K-theory class
\[ [L] - [\underline{\mC}]\ \in \  K(\mC P^1|\infty) \ = \
  \Ker(K(\mC P^1)\to K(\{\infty\}))  \ .\]
We identify $S^2=\mC\cup\{\infty\}$ with $\mC P^1$ 
by sending $\lambda\in \mC$ to the point $[\lambda:1]$.
The {\em Bott class} $b\in K(S^2|\infty)$ is the image of $ [L] - [\underline{\mC}]$
under the induced isomorphism $K(\mC P^1|\infty)\iso K(S^2|\infty)$.
The reduced K-group $K(S^2|\infty)$ is infinite cyclic, and the Bott class $b$
is a generator.

Now we let $G$ be a discrete group and $X$ a finite proper $G$-CW-complex.
Exterior tensor product of vector bundles induces the exterior product map
\[  -\tensor b \ : \  K_G(X)\ \to\ K_G(X\times S^2)\ ; \]
because $b$ is a reduced K-theory class, this map takes values
in the relative group $K_G(X\times S^2|X\times\infty)$.
Equivariant K-theory is Bott periodic
in the sense that this refined exterior product map
\[  -\tensor b \ : \  K_G(X)\ \xrightarrow{\ \iso \ }\ K_G(X\times S^2|X\times \infty)\]
is an isomorphism for every finite proper $G$-CW-complex $X$, see \cite[Thm.\,3.12]{luck-oliver:completion}.
For an integer $m$, the third author and Oliver define
  \[ K_G^m(X) \ = \
    \begin{cases}
     \quad  K_G(X)& \text{\ for $m$ even, and}\\
      K_G(X\times S^1|X\times\infty)& \text{\ for $m$ odd.}
    \end{cases}
\]
The suspension isomorphism
\[ \sigma\ : \ K_G^m(X)\ \xrightarrow{\ \iso\ } \ K_G^{m+1}(X\times S^1|X\times\infty) \]
is the identity when $m$ is odd. When $m$ is even, the suspension isomorphism
is the composite
\begin{align*}
 K_G(X)\ \xrightarrow[\iso]{-\tensor b} \ K_G(X\times S^2|X\times\infty)
  \ \xrightarrow[\iso]{(X\times q)^*}\ K_G(X\times S^1\times S^1|X\times(S^1\vee S^1)) \ ,
\end{align*}
where  $q:S^1\times S^1\to S^1\sm S^1\iso S^2$ is the composite of
the projection and the canonical homeomorphism.
\end{con}

\begin{con}
  The Bott class $b\in K(S^2)$ is a generator of the reduced K-group $K(S^2|\infty)$,
  which is infinite cyclic. The homomorphisms
  $\kappa_{S^2}:K(S^2)\to\bku\gh{S^2}$ and   $\kappa_\ast:K(\ast)\to\bku\gh{\ast}$
  are isomorphisms by \cite[Thm.\,6.3.31 (iii)]{schwede:global}.
  So the class $\kappa_{S^2}(b)$ is a generator of the infinite cyclic group
  \[ \bku\gh{S^2|\infty}\ \iso \ \pi_2^e(\bku)\ . \]
  The composite $c_{S^2}=j_*\circ\kappa_{S^2}:K(S^2|\infty)\to\bKU\gh{S^2|\infty}$
  is an isomorphism by Theorem \ref{thm:0 k-theory iso},
  so the class
  \[ \beta \ = \ c_{S^2}(b)\ = \ j_*(\kappa_{S^2}(b)) \ \in \ \bKU\gh{S^2|\infty}\]
  is a generator.
  We represent $\beta$ by a morphism $\hat\beta:\Sigma^\infty S^2\to\bKU$
  in the global stable homotopy category $\GH$, i.e., such that
   $1\sm S^2\in\pi_2^e(\Sigma^\infty S^2)$ maps to the class
  corresponding to $\beta$ under the isomorphism
  \[ \bKU\gh{S^2|\infty}\ \iso \ \pi_2^e(\bKU)\ . \]
  We define $\bar\beta:\bKU\sm S^2\to \bKU$ as the free extension to a morphism
  of $\bKU$-module spectra, i.e., the composite in $\GH$
  \[ \bKU\sm S^2\ \xrightarrow{\bKU\sm \hat\beta}\ \bKU\sm \bKU \
    \xrightarrow{\ \mu^{\bKU}}\ \bKU\ . \]
  By \cite[Thm.\,6.4.29]{schwede:global},
  the homomorphism of orthogonal ring spectra $j:\bku\to\bKU$
  sends each of the additive generators of $\pi_2^e(\bku)$ to a unit of degree 2
  in the graded ring $\pi_*^e(\bKU)$.
  In particular, the class $u=\hat\beta_*(1\sm S^2)$ in $\pi_2^e(\bKU)$
  corresponding to $\beta\in\bKU\gh{S^2|\infty}$ is a graded unit.
  
  For every compact Lie group $G$, the effect of $\bar\beta$
  on $G$-equivariant homotopy groups is multiplication by the class
  $p^*(u)\in\pi_2^G(\bKU)$, where $p^*:\pi_*^e(\bKU)\to\pi_*^G(\bKU)$
  is inflation along the unique group homomorphism $p:G\to e$.
  Since $p^*$ is a ring homomorphism, $p^*(u)$ is unit in the
  graded ring $\pi_*^G(\bKU)$. So $\bar\beta$ induces an isomorphism
  of $G$-equivariant stable homotopy groups. Since $G$ was any compact
  Lie group, the morphism $\bar\beta:\bKU\sm S^2\to\bKU$ is a global equivalence.
  We apply the `forgetful' functor $U_G : \GH \to \Ho(\Sp_G)$
  discussed in \eqref{eq:define_U_G} to obtain an isomorphism
  \[ \tilde\beta\ =\ U_G(\bar\beta)\ :\ \bKU_G\sm S^2\ \xrightarrow{\ \iso \ }\ \bKU_G \]
  in the homotopy category of orthogonal $G$-spectra.
\end{con}

In the following proposition, we write again $q:S^1\times S^1\to S^1\sm S^1\iso S^2$
for the composite of the projection and the canonical homeomorphism.

\begin{prop}\label{prop:tilde b realizes beta}
  Let $G$ be a discrete group and $X$ a finite proper $G$-CW-complex.
  The following square commutes:
  \[
    \xymatrix@C=8mm{
        \bKU_G\gh{X}  \ar[d]_--{\Sigma^{(\bKU\sm S^1)_G}\circ\Sigma^{\bKU_G}} \ar[r]^-{-\cup \beta}  &
   \bKU_G\gh{X\times S^2|X\times\infty} \ar[d]^{(X\times q)^*}\\
   (\bKU\sm S^2)_G\gh{X\times S^1\times S^1|X\times(S^1\vee S^1)}\ar[r]_-{\tilde\beta_*} &
   \bKU_G\gh{X\times S^1\times S^1|X\times(S^1\vee S^1)}     }
\]
\end{prop}
\begin{proof}
  We let $1\in \bKU\gh{\ast}$ denote the class represented by
  the trivial 0-dimensional vector bundle over a point,
  and the map of based spaces
  \[ \eta\ : \  S^0\ \to \  \bKU(0) \ ,\]
  the unit of the ring spectrum structure of $\bKU$.
  We write $\iota_2\in  (\bKU\sm S^2)\gh{S^2|\infty}$
  for the unique class satisfying
  \[ q^*(\iota_2)\ = \  \Sigma^{\bKU\sm S^1}(\Sigma^{\bKU}(1)) \ \in \
    (\bKU\sm S^2)\gh{S^1\times S^1|S^1\vee S^1}\ ,  \]
  where $\Sigma^{\bKU}$ is the suspension homomorphism \eqref{eq:suspension hom}.
  Then for every class $x\in \bKU_G\gh{X}$, the relation
  \begin{align*}
    \Sigma^{(\bKU\sm S^1)_G}(\Sigma^{\bKU_G}(x)) \
    &= \ x\cup (\Sigma^{\bKU\sm S^1}(\Sigma^{\bKU}(1)))    \\
    &= \ x\cup q^*(\iota_2)    
    \ = \ (X\times q)^*(x\cup\iota_2)
  \end{align*}
  holds in $(\bKU\sm S^2)_G\gh{X\times S^1\times S^1|X\times(S^1\vee S^1)}$.
  The relation
    \[ \beta \ = \ \tilde\beta_*(\iota_2) \]
  holds in $\bKU\gh{S^2|\infty}$, by construction of the morphism $\tilde\beta:\bKU\sm S^2\to\bKU$.
  The naturality properties of the cup product pairing,
  recorded in Proposition \ref{prop:cup properties}, thus provide the relations
  \begin{align*}
  \tilde\beta_*(\Sigma^{\bKU_G\sm S^1}(\Sigma^{\bKU_G}(x)))\
  &= \ \tilde\beta_*( (X\times q)^*(x\cup\iota_2))\
    = \  (X\times q)^*(\tilde\beta_*( x\cup\iota_2))\\
  &= \  (X\times q)^*(x\cup \tilde\beta_*(\iota_2))\
  = \   (X\times q)^*(x\cup \beta)\ .
\end{align*}  
The third equality exploits that $\tilde\beta$ is underlying a morphism
in the global homotopy category of left $\bKU$-module spectra.
\end{proof}

Now we define the periodicity isomorphisms of the proper $G$-cohomology theory $\bKU^*_G\gh{-}$,
essentially as the effect of the $\pi_*$-isomorphism $\tilde\beta:\bKU_G[2]=\bKU_G\sm S^2\to\bKU_G$.
We recall that Proposition \ref{prop:coherent isos} specifies a
natural isomorphism 
\[ t_{2,m}\ : \ \bKU_G[2][m]\ \xrightarrow{\ \iso \ }\ \bKU_G[2+m] \]
in $\Ho(\Sp_G)$, for every integer $m$.
We define a natural isomorphism
\begin{equation}\label{eq:define B m}
  B^{[m]}\ = \ (\tilde\beta[m]\circ t_{2,m}^{-1})_* \ :\
  \bKU_G^{2+m}\gh{X}\ \xrightarrow{\ \iso \ } \bKU_G^m\gh{X} \ , 
\end{equation}
the effect of the composite isomorphism
\begin{align*}
  \bKU_G[2+m]\  \xrightarrow[\iso]{t_{2,m}^{-1}}\
  \bKU_G[2][m]\  \xrightarrow[\iso]{\tilde\beta[m]}\ \bKU_G[m]\ .
  \end{align*}

\begin{prop}\label{prop:periodicity isos for KU}
  Let $G$ be a discrete group and $X$ a finite proper $G$-CW-complex.
  Then the square  \[ \xymatrix@C=18mm{
        \bKU_G^{2+m}\gh{X}\ar[r]^-{B^{[m]}} \ar[d]_\sigma & \bKU_G^m\gh{X}\ar[d]^\sigma \\
        \bKU_G^{2+m+1}\gh{X\times S^1|X\times\infty}\ar[r]_-{B^{[m+1]}} &
          \bKU_G^{m+1}\gh{X\times S^1|X\times\infty}
      } \]
    commutes for every $m\in\mZ$.
\end{prop}
\begin{proof}
  The associativity and naturality property of the natural isomorphisms $t_{k,l}$
  stated in Proposition \ref{prop:coherent isos}
  imply that the following two squares commute:
 \[ \xymatrix@C=12mm{
     \bKU_G[2+m][1]\ar[r]^-{ t_{2,m}^{-1}[1]}
     \ar[d]_{t_{2+m,1}} &
          \bKU_G[2][m][1]\ar[r]^-{ \tilde\beta[m][1]}
     \ar[d]_{t_{m,1}} &
      \bKU_G[m][1]\ar[d]^{t_{m,1}} \\
      \bKU_G[2+m+1]\ar[r]_-{t_{2,m+1}^{-1}} & \
      \bKU_G[2][m+1]\ar[r]_-{\tilde\beta[m+1]} & \bKU_G[m+1]
    } \]
  The suspension isomorphisms $\sigma:E^k_G\gh{X}\to E^{k+1}_G\gh{X\times S^1|X\times\infty}$
  defined in \eqref{eq:define bundle Sigma}
  are natural in the variable $E$ for morphisms in $\Ho(\Sp_G)$.
  The two facts together provide the desired commutativity.
\end{proof}

Now we can put all the ingredients together and prove the main result
of this section, identifying vector bundle K-theory 
with the proper $G$-cohomology theory
represented by the orthogonal $G$-spectrum
underlying the global K-theory spectrum $\bKU$.
We let $G$ be a discrete group.
Theorems \ref{thm:rep2vect} and \ref{thm:0 k-theory iso}
together provide a natural isomorphism 
\[  d_X\ = \ (\mu_X^{\bKU_G})^{-1}\circ c_X\ : \ K_G(X)  \ \to\ \bKU_G(X) \]
of excisive functors on finite proper $G$-CW-complexes.

\begin{thm}\label{thm:repr K-theory}
  Let $G$ be a discrete group. The natural isomorphism of excisive functors
  \[ d_X\ :\  K_G(X)\ \xrightarrow{\ \iso \ }\ \bKU_G(X) \]
  extends to an isomorphism $K^*_G\iso \bKU_G^*$
  of proper $G$-cohomology theories on finite proper $G$-CW-complexes 
  from the equivariant K-theory in these sense of \cite{luck-oliver:completion}
  to the $G$-cohomology theory represented by the orthogonal $G$-spectrum $\bKU_G$.
\end{thm}
\begin{proof}
  We define natural isomorphisms
  \[ \psi^X_{2 k} \ : \ K_G(X)\ \xrightarrow{\ \iso \ }\  \bKU_G^{2 k}\gh{X}  \]
  and
  \[ \psi^X_{2 k-1}\ : \ K_G(X\times S^1|X\times\infty)\ \xrightarrow{\ \iso\ } \  \bKU_G^{2 k -1 }\gh{X} \ , \]
  for all integers $k$, compatible with the suspension isomorphisms.
  We start by setting $\psi_0^X=c_X: K_G(X)\to\bKU_G\gh{X}$.
  For $k<0$, we define $\psi_{2 k}$ inductively as the composite
  $\psi_{2 k}= B^{[2 k]}\circ \psi_{2+2 k}$,
  where $B^{[m]}:\bKU_G^{2+m}\gh{X}\to\bKU_G^m\gh{X}$
  is the natural isomorphism \eqref{eq:define B m}.
  For $k>0$, we define $\psi_{2 k}$ inductively as the composite
  $\psi_{2 k}= (B^{[2 k-2]})^{-1}\circ\psi_{2k-2}$.
  In odd dimensions, we define $\psi_{2k-1}$ as the composite
  \begin{align*}
    K_G(X\times S^1|X\times\infty)\ \xrightarrow{\psi^{X\times S^1}_{2 k}} \
    \bKU_G^{2 k}\gh{X\times S^1|X\times\infty}\ \xrightarrow{\ \sigma^{-1}\ } \ 
    \bKU_G^{2 k -1}\gh{X}  \ . 
  \end{align*}
  With these definitions, the relation
  \[ \psi_m\ = \ B^{[m]}\circ  \psi_{2+m} \]
  holds for all integers $m$, by definition in even dimension,
  and by Proposition \ref{prop:periodicity isos for KU} in odd dimensions.

  Now we show that the isomorphisms $\psi_m$ are compatible with
  suspension isomorphisms.
  We first observe that the following diagram commutes:
  \[
    \xymatrix@C=18mm{
      K_G(X) \ar[d]^{\kappa_X} \ar[r]^-{-\tensor b}\ar@<-4ex>@/_2pc/[dd]_{c_X} &
      K_G(X\times S^2) \ar[d]_{\kappa_{X\times S^2}} \ar@<6ex>@/^2pc/[dd]^{c_{X\times S^2}}\\
      \bku_G\gh{X} \ar[d]^{(j_G)_*}\ar[r]_-{-\cup \kappa_{S^2}(b)} & \bku_G\gh{X \times S^2}\ar[d]_{(j_G)_*}\\
      \bKU_G\gh{X} \ar[r]_-{-\cup \beta} & \bKU_G\gh{X \times S^2}
     }\]
   Indeed, the upper square commutes by Proposition \ref{prop:kappa multiplicative},
   and the lower square commutes because $j:\bku\to\bKU$ is a morphism
   of ultra-commutative ring spectra,
   and the class $\beta$ was defined as $c_{S^2}(b)=j_*(\kappa_{S^2}(b))$.
   Now we contemplate the following diagram of isomorphisms:
   \[ \xymatrix@C=3mm{
      K_G(X) \ar[d]^{-\tensor b} \ar@<-4ex>@/_4pc/[dddd]_{\psi_0^X=c_X}\ar@/^2pc/[dr]^(.3)\sigma &&\\
       K_G(X\times S^2|X\times\infty) \ar[r]_-{(X\times q)^*}\ar[d]^{c_{X\times S^2}} &
      K_G(X\times S^1\times S^1|X\times(S^1\vee S^1)) \ar[d]^{c_{X\times S^1\times S^1}}
      \ar@<15ex>@/^3pc/[ddd]^{\psi_1^X}       \\
       \bKU_G\gh{X\times S^2|X\times\infty}\ar[r]_-{(X\times q)^*}\ar[dd]^(.7){(-\cup \beta)^{-1}}&
      \bKU_G\gh{X\times S^1\times S^1|X\times(S^1\vee S^1)}\\
      &
      (\bKU\sm S^2)_G\gh{X\times S^1\times S^1|X\times(S^1\vee S^1)} \ar[u]_{\tilde\beta_*}^\iso\\
      \bKU_G\gh{X}     \ar[r]_-{\Sigma^{\bKU_G}}       &
       (\bKU\sm S^1)_G\gh{X\times S^1|X\times\infty} \ar[u]_{\Sigma^{(\bKU\sm S^1)_G}}^\iso
     } \]
   As we argued above, the left vertical composite coincides with the map $\psi^X_0=c_X$.
   The middle square commutes by naturality of $c_X$;
   and the lower part of the diagram commutes by
   Proposition \ref{prop:tilde b realizes beta}.
   This proves the relation $\sigma\circ\psi_0=\psi_1\circ\sigma$,
   i.e., the degree 0 instance of compatibility with the suspension isomorphisms.
   
  For compatibility in other even dimensions we consider the following diagram:
  \[ \xymatrix@C=3mm{
      K_G(X) \ar[r]_-{-\tensor b} \ar[d]^{\psi_{2+2 k}^X}\ar@/^2pc/[rr]^(.3)\sigma       
      \ar@<-3ex>@/_2pc/[dd]_{\psi_{2 k}^X}
      & K_G(X\times S^2|X\times\infty) \ar[r]_-{(X\times q)^*} &
      K_G(X\times S^1\times S^1|X\times(S^1\vee S^1)) \ar[d]_{\psi_{2+2 k+1}^{X\times S^1}}
      \ar@<9ex>@/^2pc/[dd]^{\psi_{2 k+1}^{X\times S^1}} \\
      \bKU_G^{2+2 k}\gh{X}      \ar[rr]^-{\sigma} \ar[d]^{B^{[2 k]}}&&
      \bKU_G^{2+2 k+1}\gh{X\times S^1|X\times\infty} 
      \ar[d]_{B^{[2 k+1]}}\\
      \bKU_G^{2 k}\gh{X}\ar[rr]_-{\sigma} && \bKU_G^{2k+1}\gh{X\times S^1|X\times\infty}
    } \]
  The lower square commutes by Proposition \ref{prop:periodicity isos for KU}.
  So the outer diagram commutes if and only if the upper part commutes.
  In other words: compatibility with the suspension isomorphisms
  holds in dimension $2 k$ if and only if it holds in dimension $2+2 k$.
  We already showed compatibility in dimension 0, so we
  conclude that compatibility with the suspension isomorphisms holds in all even dimensions.
  In odd dimensions, compatibility with the suspension isomorphisms was built into
  the definition of suspension isomorphism in $K_G^*$ and the maps $\psi_{2k -1}$.
  This completes the construction of the isomorphism
  from the vector bundle K-theory $K_G^*$ to the theory $\bKU_G^*\gh{-}$.

  Theorem \ref{thm:represented equals bundle} provides another isomorphism of
  proper $G$-cohomology theories from 
  the represented theory $\bKU_G^*(-)$ to the theory $\bKU_G^*\gh{-}$,
  given by $\mu^{\bKU_G}_X:K_G(X)\to\bKU_G\gh{X}$ in dimension 0.
  This concludes the proof.
\end{proof}

We leave it to the interested reader to verify that the isomorphisms
of Theorem \ref{thm:repr K-theory} are compatible with restriction to
finite index subgroups, with the induction isomorphisms \eqref{eq:ind bundle},
and with graded products.

\index{equivariant K-theory|)}

\backmatter

\printindex

\end{document}